\documentclass[11pt]{article}
\usepackage[utf8]{inputenc} % allow utf-8 input
\usepackage[T1]{fontenc}    % use 8-bit T1 fonts
\usepackage{url}            % simple URL typesetting
\usepackage{booktabs}       % professional-quality tables
\usepackage{nicefrac}       % compact symbols for 1/2, etc.
\usepackage{microtype}      % microtypography
\usepackage{multirow}
\usepackage{array}
\usepackage[noblocks]{authblk}
\usepackage{amssymb,arydshln}
\usepackage[nodayofweek]{datetime}

\usepackage[shortlabels]{enumitem}
\setlist[itemize]{itemindent=0ex,itemsep=-0.5ex,leftmargin=3ex,topsep=5pt}
\setlist[enumerate]{label={\arabic*)},itemindent=0ex,itemsep=-0.5ex,leftmargin=3ex,topsep=5pt}

\usepackage[font=small,labelfont=bf]{caption}
\usepackage{amsmath,amsfonts,nccmath,mathtools,mathrsfs}
\usepackage{tcolorbox}
\usepackage{xcolor}
\usepackage[colorlinks = true,
            linkcolor = blue,
            urlcolor  = blue,
            citecolor = blue,
            anchorcolor = blue]{hyperref}

\usepackage{graphicx}

\usepackage{epstopdf}

\usepackage[margin=1in]{geometry}

\usepackage{caption}
\usepackage[numbers]{natbib}
\usepackage{amsthm}
\usepackage{graphicx,color}
\usepackage{subcaption,setspace}

\newtheorem{theorem}{Theorem}[section]
\newtheorem{proposition}{Proposition}[section]
\newtheorem{lemma}{Lemma}[section]
\newtheorem{corollary}{Corollary}[section]

\newtheorem{assumption}{Assumption}

\theoremstyle{definition}
\newtheorem{definition}{Definition}
\newtheorem{example}{Example}

\theoremstyle{remark}
\newtheorem{remark}{Remark}

\usepackage[nameinlink]{cleveref}
\crefname{equation}{}{}
\crefname{theorem}{Theorem}{Theorems}
\crefname{corollary}{Corollary}{Corollaries}
\crefname{example}{Example}{Examples}
\crefname{assumption}{Assumption}{Assumptions}
\crefname{lemma}{Lemma}{Lemmas}
\crefname{proposition}{Proposition}{Propositions}
\crefname{figure}{Figure}{Figures}
\crefname{table}{Table}{Tables}
\crefname{section}{Section}{Sections}
\crefname{appendix}{Appendix}{Appendices}
\Crefname{equation}{}{}
\Crefname{theorem}{Theorem}{Theorems}
\Crefname{corollary}{Corollary}{Corollaries}
\Crefname{example}{Example}{Examples}
\Crefname{lemma}{Lemma}{Lemma}
\Crefname{proposition}{Proposition}{Proposition}
\Crefname{figure}{Figure}{Figures}
\Crefname{table}{Table}{Tables}
\Crefname{section}{Section}{Sections}
\Crefname{appendix}{Appendix}{Appendices}

\newcommand{\tr}{{{\mathsf T}}}
\newcommand{\mK}{{\mathsf{K}}}
\newcommand{\mZ}{{\mathsf{Z}}}

\makeatletter
\newcommand{\removelatexerror}{\let\@latex@error\@gobble}
\makeatother

\title{\bf Analysis of the Optimization Landscape of Linear Quadratic Gaussian (LQG) Control
\thanks{Y. Zheng and Y. Tang contributed to this work equally. This work is supported by NSF career 1553407, AFOSR Young Investigator Program, and ONR Young Investigator Program. Emails: zhengy@g.harvard.edu; yujietang@seas.harvard.edu; nali@seas.harvard.edu.}
}
\author[1]{Yang Zheng}
\author[1]{Yujie Tang}
\author[1]{Na Li}

\affil[1]{\small School of Engineering and Applied Sciences,
Harvard University}
\date{\today} %{}

\begin{document}

\maketitle
\begin{abstract}

This paper revisits the classical Linear Quadratic Gaussian (LQG) control from a modern optimization perspective. We analyze two aspects of the optimization landscape of the LQG problem: 1) connectivity of the set of stabilizing controllers $\mathcal{C}_n$; and 2) structure of stationary points.  It is known that similarity transformations do not change the input-output behavior of a dynamical controller or LQG cost. This inherent symmetry by similarity transformations makes the landscape of LQG very rich. We show that 1) the set of stabilizing controllers $\mathcal{C}_n$ has at most two path-connected components and they are diffeomorphic under a mapping defined by a similarity transformation; 2) there might exist many \emph{strictly suboptimal stationary points} of the LQG cost function over $\mathcal{C}_n$ and these stationary points are always \emph{non-minimal}; 3) all \emph{minimal} stationary points are globally optimal and they are identical up to a similarity transformation. These results shed some light on the performance analysis of direct policy gradient methods for solving the LQG problem.

%\vspace{5pt}
%\noindent{\bf Keywords: }
\end{abstract}

\section{Introduction}

As one of the most fundamental optimal control problems, Linear Quadratic Gaussian (LQG) control has been studied for decades. Many structural properties of the LQG problem have been established in the literature, such as existence of the optimal controller, separation principle of the controller structure, and no guaranteed stability margin of closed-loop LQG systems~\citep{zhou1996robust,bertsekas1995dynamic,doyle1978guaranteed}. %In particular, the LQG problem enjoys the celebrated principle of separation of estimation and control: the optimal solution admits an elegant closed-loop form, combining a Kalman filter together with a linear–quadratic regulator (LQR)~\cite[Chapter 14]{zhou1996robust}.
Despite the non-convexity of the LQG problem, a globally optimal controller can be found by solving two algebraic Riccati equations~\citep{zhou1996robust}, or a convex semidefinite program based on a change of variables~\citep{gahinet1994linear,scherer1997multiobjective}.

While extensive results on LQG have been obtained in classical control, its optimization landscape is less studied, i.e., viewing the LQG cost as a function of the controller parameters and studying its analytical and geometrical properties. On the other hand, recent advances in reinforcement learning (RL) have revealed that the landscape analysis of another benchmark optimal control problem, linear quadratic regulator (LQR), can lead to fruitful and profound results, especially for model-free controller synthesis~\citep{fazel2018global,malik2019derivative,mohammadi2019convergence,tu2019gap,li2019distributed,umenberger2019robust,zhang2019policy}. For instance, it is shown that the set of static stabilizing feedback gains for LQR is connected, and that the LQR cost function is coercive and enjoys an interesting property of gradient dominance~\citep{fazel2018global,bu2019topological}. These properties are fundamental to establish convergence guarantees for gradient-based algorithms for solving LQR and their model-free extensions for RL \citep{malik2019derivative,mohammadi2019convergence}. We note that recent studies have also contributed to establishing performance guarantees of model-based RL techniques for LQR (see e.g.,~\cite{dean2019sample,wang2020exact}) as well as LQG control~\citep{tu2017non,boczar2018finite,zheng2020sample,simchowitz2020improper}.

This paper aims to analyze the optimization landscape of the LQG problem. Unlike LQR that deals with \emph{fully observed} linear systems whose optimal solution is a static feedback policy, the LQG problem concerns \emph{partially observed} linear systems driven by additive Gaussian noise and its optimal controller is no longer static.  We need to search over dynamical controllers for LQG problems. This makes its optimization landscape richer and yet much more complicated than LQR. Indeed, the set of stabilizing static state feedback policies is connected, but the set of stabilizing static output feedback policies can be highly disconnected~\citep{feng2020connectivity}. The connectivity of stabilizing dynamical output feedback policies, i.e., the feasible region of LQG control, remains unclear. % and has not been addressed in the literature.
Furthermore, LQG has a natural symmetry structure induced by similarity transformations that do not change the input-output behavior of dynamical controllers, which is not the case for LQR.

Some recent studies~\citep{sun2018geometric,chi2019nonconvex,li2019symmetry,qu2019analysis,ge2017optimization} have demonstrated that symmetry properties play a key role in rendering a large class of non-convex optimization problems in machine learning tractable; see also~\cite{zhang2020symmetry} for a recent review. For the LQG problem, we can expect the inherent symmetry by similarity transformations to bring some important properties of its non-convex optimization landscape. We also note that the notion of \emph{minimal controllers} (a.k.a.~controllable and observable controllers; see \cref{App:control_basics}) is a unique feature in controller synthesis of \emph{partially observed} dynamical systems, making the optimization landscape of LQG distinct from many machine learning problems.

\subsection{Our contributions}

In this paper, we view the classical LQG problem from a modern optimization perspective, and study two aspects of its optimization landscape. First,
we characterize the connectivity of the feasible region of the LQG problem, i.e., the set of strictly proper stabilizing dynamical controllers, denoted by $\mathcal{C}_n$ ($n$ is the state dimension). We prove that %the feasible region of LQG controllers
$\mathcal{C}_n$ can be disconnected, but has at most two path-connected components (\cref{Theo:disconnectivity}) that are diffeomorphic under a similarity transformation (\cref{Theo:Cn_homeomorphic}). This brings positive news to gradient-based local search algorithms for the LQG problem, since it makes no difference to search over either path-connected component even if $\mathcal{C}_n$ is disconnected.
 We further present a sufficient condition under which $\mathcal{C}_n$ is always connected, and this condition becomes necessary for LQG problems with a single input or a single output (\cref{Theo:connectivity_conditions}). The sufficient condition naturally holds for any open-loop stable system, thus its set of strictly proper stabilizing dynamical controllers is always connected (\cref{corollary:stable}).  %One main idea in the proofs of our results relies on a classical change of variables for output-feedback dynamical controller design that allows for equivalent convex reformulations of non-convex controller synthesis in various setups~\citep{gahinet1994linear,scherer1997multiobjective}. We combine the change of variables with some basic tools in differential geometry to finish the proofs.

Second, we investigate structural properties of the stationary points of the LQG cost function. %By exploiting the symmetry induced by similarity transformations,
%One natural consequence is that
%we show that %the globally optimal solutions to the LQG problem are not unique, not isolated, and can be disconnected in the state-space domain.
%for a class of LQG problems, the set of globally optimal solutions is a submanifold of dimension $n^2$ and has two path-connected components (\cref{proposition:sim_trans_submanifold}). % The tangent space at a given stabilizing controller is established in~\cref{proposition:tangent_orbit}.
%In addition, we present an example for which the LQG problem is not coercive, i.e., the LQG cost may not diverge to infinity when the controller approaches the boundary of the feasible region (\cref{example:non-coercivity}).
%
When characterizing the stationary points, the notion of \emph{minimal controllers} plays an important role. By exploiting the symmetry induced by similarity transformations, we show that the LQG cost is very likely to have many \emph{strictly suboptimal stationary points}, and these stationary points are always non-minimal (\cref{theorem:non_globally_optimal_stationary_point}). For LQG with an open-loop stable plant, we explicitly construct a family of non-minimal stationary points and further establishes a criterion for checking whether the corresponding Hessian is indefinite or vanishing (\cref{theorem:zero_stationary_hessian}).
In contrast, we prove that all \emph{minimal} stationary points are globally optimal to the LQG problem (\cref{theo:stationary_points_globally_optimal}), and form a submanifold of dimension $n^2$ that has two path-connected components (\cref{proposition:sim_trans_submanifold}). These minimal stationary points are identical up to similarity transformations. % This is expected from the classical result that the globally optimal LQG controller is unique in the frequency domain~\cite[Theorem 14.7]{zhou1996robust}.
This result implies that if local search iterates converge to a stationary point that corresponds to a controllable and observable controller, then the algorithm has found a globally optimal solution to the LQG problem (\cref{{corollary:Gradient_Descent_Convergence}}).
Finally, we construct an example showing that the second-order shape of the LQG cost function can be ill-behaved around a minimal stationary point in the sense that its Hessian has a huge condition number (see~\cref{proposition:Hessian_example}).

\subsection{Related work}

\paragraph{Optimization landscape of LQR} The classical Linear-Quadratic Regulator (LQR), one of the most well-studied optimal control problems, has re-attracted increasing interest~\citep{fazel2018global, dean2019sample, malik2019derivative, umenberger2019robust,tu2018gap, recht2019tour}
in the study of RL techniques for control systems.
%. In particular, LQR has been used as a benchmark problem to study how machine learning methods interacts with continuous control.
For model-free policy optimization methods, the optimization landscape of LQR is essential to establish their performance guarantees. In~\cite{fazel2018global,malik2019derivative,mohammadi2019convergence}, it is shown that both continuous-time and discrete-time LQR problems enjoy the gradient dominance property, and that model-free gradient-based algorithms converge to the optimal LQR controller under mild conditions. %\cite{zhang2019policy} has
The authors in~\cite{zhang2019policy} have examined the optimization landscape of a class of risk-sensitive state-feedback control problems and the convergence of corresponding policy optimization methods. Furthermore, it is shown in \cite{furieri2020learning} that a class of finite-horizon output-feedback linear quadratic control problems also satisfies the gradient dominance property. % of , and gradient algorithms can reach the globally optimal solutions~.
%
%Since the optimal LQR controller is a static feedback policy,
Some recent studies have examined the connectivity of stabilizing static output feedback policies~\cite{feng2020connectivity, bu2019topological,fatkhullin2020optimizing}. It is shown in~\cite{feng2020connectivity} that the set of stabilizing static output feedback policies can be highly disconnected,
% and that there are instances with an exponential number of connected components. This disconnectivity
which poses a significant challenge for decentralized LQR problems. % with a subspace constraint.
For general decentralized LQR, policy optimization methods can only be guaranteed to reach some stationary points~\cite{li2019distributed}.

We note that many landscape properties of LQR are derived using classical control tools~\cite{mohammadi2019convergence,zhang2019policy,furieri2020learning,fatkhullin2020optimizing}. % and there are an extensive literature on decentralized control (see, e.g.,~\cite{rotkowitz2005characterization,furieri2020sparsity,anderson2019system} and the references therein).
%; it is promising to adapt classical decentralized LQR results for the understanding of their optimization landscape.
Our work leverages ideas from classical control tools~\cite{zhou1996robust,gahinet1994linear,scherer1997multiobjective} to analyze the optimization landscape of the LQG problem.

\paragraph{Reinforcement learning for LQG and controller parameterization} %Motivated by the advances of RL techniques on LQR problems,
Recent studies have also started to investigate % address performance guarantees for
LQG with unknown dynamics, including offline robust control~\cite{tu2017non,boczar2018finite,zheng2020sample} and online adaptive control~\cite{simchowitz2020improper,lale2020logarithmic,lale2020regret}. The line of studies on offline robust control first estimates a system model as well as a bound on the estimation error % using non-asymptotic analysis
(see, e.g.,~\cite{tu2017non,oymak2019non,zheng2020non}), and then design a robust LQG controller that stabilizes the plant against model uncertainty. % One challenge is on establishing explicit performance degradation with respect to model uncertainty for the robust LQG synthesis procedure~\cite{tu2017non,boczar2018finite,zheng2020sample}.
For online adaptive control, the recent work~\cite{simchowitz2020improper} has introduced an online gradient descent algorithm to update LQG controller parameters with a sub-linear regret; % in various settings;
see~\cite{lale2020logarithmic,lale2020regret} for further developments. For both lines of works, a convex reformulation of the LQG problem is essential for algorithm design as well as performance analysis. For example, the works~\cite{simchowitz2020improper,lale2020logarithmic,lale2020regret} employ the classical Youla parameterization~\cite{youla1976modern}, while the works~\cite{boczar2018finite,zheng2020sample} adopt the recent system-level parameterization (SLP)~\cite{wang2019system} and input-output parameterization (IOP)~\cite{furieri2019input}, respectively. The Youla parameterization, SLP, and IOP recast the LQG problem into equivalent convex formulations in the frequency domain~\cite{zheng2020equivalence}, but they all rely on the underlying system dynamics explicitly. Thus, a system identification procedure is required \emph{a priori} in~\cite{tu2017non,boczar2018finite,zheng2020sample, simchowitz2020improper}, and these methods are all model-based.

In this work, we consider a natural model-free controller parameterization  for the LQG problem in the state-space domain. This parameterization does not depend on the system dynamics explicitly but leads to a non-convex formulation. Our results contribute to the understanding of this non-convex optimization landscape, which shed light on performance analysis of model-free RL methods for solving LQG control.
%
% of the closed-loop system since the LQG formulation is non-convex in terms of
% tension between convexity and non-convexity.
%
%convexity is important -- parameterization
%
%our parameterization is model-free, but lead to a non-convex optimization problem.
%
%
% We contribute the understanding of its landscape.

\paragraph{Non-convex optimization with symmetry} %Non-convex optimization has been widely employed in machine learning:
% Many important problems in machine learning
%, from matrix factorization to deep learning,
% rely on optimization over Non-convex functions, and
Recent works~\cite{zhang2020symmetry,li2019symmetry} have revealed the significance of symmetry properties in understanding the geometry of many non-convex optimization problems in machine learning. For example, the phase retrieval~\cite{sun2018geometric} and low-rank matrix factorization~\cite{chi2019nonconvex,li2019symmetry} problems have rotational symmetries, while sparse dictionary learning~\cite{qu2019analysis} and tensor decomposition~\cite{ge2017optimization} exhibit discrete symmetries; see~\cite{zhang2020symmetry} for a recent survey. %Under some technical assumptions,
These symmetries enable identifying the local curvature of stationary points, and contribute to the tractability of the associated non-convex optimization problems. In this paper, we highlight a new symmetry defined by similarity transformations in the LQG problem. This symmetry appears only in dynamical output-feedback controller synthesis. %, and thus the LQR problem does not have this symmetry property.
In addition, a notion of \emph{minimal controllers} is unique in control problems, making the study of the landscape of LQG distinct from other machine learning problems~\cite{sun2018geometric,chi2019nonconvex,qu2019analysis,ge2017optimization,zhang2020symmetry}.

\subsection{Paper outline}
The rest of this paper is organized as follows. \cref{section:problem_statement} presents the problem statement of Linear Quadratic Gaussian (LQG) control. We introduce our main results on the connectivity of stabilizing controllers in~\cref{sec:connectivity}, and present our main results on the structure of stationary points of LQG problems in~\cref{sec:stationary_points}. Some numerical results on gradient descent algorithms for LQG are reported in~\cref{section:numerical_results}. We conclude the paper in~\cref{section:conclusion}. Appendices contain preliminaries in control and differential geometry, proofs of auxiliary results, a connectivity result of proper stabilizing controllers, and analogous results for discrete-time systems.

\subsection*{Notations}
We use $\mathbb{R}$ and $\mathbb{N}$ to denote the set of real and natural numbers, respectively. The set of  $k\times k$ real symmetric matrices is denoted by $\mathbb{S}^k$, and the determinant of a square matrix $M$ is denoted by $\det M$. We denote the set of $k \times k$ real invertible matrices by $\mathrm{GL}_k=\{T\in\mathbb{R}^{k\times k} \mid \det T\neq 0\}$. Given a matrix $M \in \mathbb{R}^{k_1 \times k_2}$, $M^\tr$ denotes the transpose of $M$, and $\|M\|_F$ denotes the Frobenius norm of $M$. For any $M_1,M_2\in\mathbb{S}^k$, we use $M_1\prec M_2$ and $M_2\succ M_1$ to mean that $M_2-M_1$ is positive definite, and use $M_1\preceq M_2$ and $M_2\succeq M_1$ to mean that $M_2-M_1$ is positive semidefinite. We use $I_k$ to denote the $k\times k$ identity matrix, and use $0_{k_1\times k_2}$ to denote the $k_1\times k_2$ zero matrix; we sometimes omit their subscripts if the dimensions can be inferred from the context.

\section{Problem Statement} \label{section:problem_statement}
In this section, we first introduce the Linear Quadratic Gaussian control problem, and then present the problem statement of our work.

\subsection{The Linear Quadratic Gaussian (LQG) problem}
Consider a continuous-time\footnote{We only consider the continuous-time case in the main text. The results for discrete-time systems can be found in~\cref{appendix:discrete_time}.} linear dynamical system
\begin{equation}\label{eq:Dynamic}
\begin{aligned}
\dot{x}(t) &= Ax(t)+Bu(t)+w(t), \\
y(t) &= Cx(t)+v(t),
\end{aligned}
\end{equation}
where $x(t) \in \mathbb{R}^n$ represents the vector of state variables, $u(t)\in \mathbb{R}^m$ the vector of control inputs, $y(t) \in \mathbb{R}^p$ the vector of measured outputs available for feedback, and $w(t) \in \mathbb{R}^n, v(t)  \in \mathbb{R}^p$ are system process and measurement noises at time $t$. It is assumed that $w(t)$ and $v(t)$ are white Gaussian noises with intensity matrices $W \succeq 0$ and $V \succ 0$. For notational simplicity, we will drop the argument $t$ when it is clear in the context.

The classical linear quadratic Gaussian (LQG) problem is defined as
\begin{equation} \label{eq:LQG}
    \begin{aligned}
        \min_{u(t)} \quad & J := \lim_{T \rightarrow \infty }\frac{1}{T}\mathbb{E} \left[\int_{t=0}^T \left(x^\tr Q x + u^\tr R u\right)dt\right] \\
        \text{subject to} \quad & ~\eqref{eq:Dynamic},
    \end{aligned}
\end{equation}
where $Q \succeq 0$ and $R\succ 0$. In~\eqref{eq:LQG}, the input $u(t)$ is allowed to depend on all past observation $y(\tau)$ with $\tau < t$.
Throughout the paper, we make the following standard assumption of \emph{minimal} systems in the sense of Kalman (see~\cref{App:control_basics} for a review of these notions).
\begin{assumption} \label{assumption:stabilizability}
$(A,B)$ and $(A,W^{1/2})$ are controllable, and $(C,A)$ and $(Q^{1/2},A)$ are observable.
% The dynamics of the plant \eqref{eq:Dynamic} are controllable and observable. %In addition, $Q$ and $R$ are positive definite.
\end{assumption}
%A state-space system is called \emph{minimal} if it is controllable and observable
%(see \cref{App:control_basics} for a review of these notions).

Unlike the problem of linear quadratic regulator (LQR), static feedback policies in general do not achieve optimal values of the cost function, and we need to consider a class of dynamical controllers in the form of
\begin{equation}\label{eq:Dynamic_Controller}
    \begin{aligned}
        \dot \xi(t) &= A_{\mK}\xi(t) + B_{\mK}y(t), \\
        u(t) &= C_{\mK}\xi(t),
    \end{aligned}
\end{equation}
where $\xi(t) \in \mathbb{R}^q$ is the internal state of the controller, and $A_{\mK},B_{\mK},C_{\mK}$ are matrices of proper dimensions that specify the dynamics of the controller. We refer to the dimension $q$ of the internal control variable $\xi$ as the order of the dynamical controller~\eqref{eq:Dynamic_Controller}. A dynamical controller is called a \emph{full-order dynamical controller} if its order is the same as the system dimension, i.e., $q = n$; if $q < n$, we call~\eqref{eq:Dynamic_Controller} a \emph{reduced-order} or \emph{lower-order controller}. We shall see later that it is unnecessary to consider dynamical controllers with order beyond the system dimension $n$.

The LQG problem~\eqref{eq:LQG} admits the celebrated separable principle and has a closed-form solution by solving two algebraic Riccati equations~\cite[Theorem 14.7]{zhou1996robust}. Indeed, the optimal solution to~\eqref{eq:LQG} is $u(t) = -K \xi(t)$ with a fixed $p \times n$ matrix $K$ and $\xi(t)$ is the state estimation based on the Kalman filter. Precisely, the optimal controller is given by
\begin{equation} \label{eq:LQGcontroller}
\begin{aligned}
    \dot \xi &= (A-BK) \xi + L(y - C\xi), \\
    u &= -K \xi.
\end{aligned}
\end{equation}
In~\eqref{eq:LQGcontroller}, the matrix $L$ is called the \textit{Kalman gain}, computed as $L = PC^\tr V^{-1}$ where $P$ is the unique {positive semidefinite} solution (see e.g.,~\cite[Corollary 13.8]{zhou1996robust}) to
\begin{subequations}\label{eq:Riccati}
\begin{equation} \label{eq:Riccati_P}
    AP + PA^\tr - PC^\tr V^{-1}CP + W = 0,
\end{equation}
and the matrix $K$ is called the \textit{feedback gain}, computed as $K = R^{-1}B^\tr S$ where $S$ is the unique {positive semidefinite} solution to
\begin{equation} \label{eq:Riccati_S}
    A^\tr S + SA - SB R^{-1}B^\tr S + Q = 0.
\end{equation}
\end{subequations}

We can see that the optimal LQG controller~\eqref{eq:LQGcontroller} can be written into the form of~\eqref{eq:Dynamic_Controller} with
    \begin{equation} \label{eq:LQGstatespace}
     A_{\mK} = A - BK - LC,\quad  B_{\mK} = L, \quad C_{\mK} = -K.
    \end{equation}
     Thus, the solution from Ricatti equations~\eqref{eq:Riccati} is always full-order, i.e., $q = n$. We note that two dynamical controllers with the same transfer function
     $
     \mathbf{K}(s)=
     C_{\mK}(sI-A_{\mK})^{-1}B_{\mK}
     $
     lead to the same LQG cost. In general, the optimal LQG controller is only unique in the frequency domain~\cite[Theorem 14.7]{zhou1996robust} but not unique in the state-space domain~\eqref{eq:Dynamic_Controller}; any similarity transformation on~\eqref{eq:LQGstatespace} leads to another optimal solution that achieves the global minimum cost\footnote{This is a well-known fact and can be verified easily; see~\cref{lemma:Jn_invariance}.}.

\subsection{Parameterization of Dynamical Controllers and the LQG Cost Function}
The controller~\eqref{eq:LQGcontroller} explicitly depends on the plant's parameters $A, B, C$, and it may not be straightforward to compute~\eqref{eq:LQGcontroller} if $A,B$ and $C$ are not available. Recently, model-free policy gradient methods have been applied in a range of control problems, such as LQR in discrete-time~\cite{fazel2018global} and continuous-time~\cite{mohammadi2019convergence}, finite-horizon discrete-time LQG problem~\cite{furieri2020learning}, and state-feedback risk-sensitive control~\cite{zhang2019policy}. These methods view classical control problems from a modern optimization perspective, and directly optimize control policies based on system observations, without explicit knowledge of the underlying model. To avoid the explicit dependence on model parameters $A, B, C$, we consider the class of dynamical controllers in~\eqref{eq:Dynamic_Controller}, parameterized by $(A_{\mK}, B_{\mK}, C_{\mK})$. As we will see later, this allows us to view LQG~\eqref{eq:LQG} from a model-free optimization perspective.

%{\color{red}To do: Add motivation of model-free control and viewing LQG learning from an optimization perspective.}

In order to formulate the cost in \eqref{eq:LQG} as a function of the parameterized dynamical controller $(A_{\mK}, B_{\mK}, C_{\mK})$, we first need to specify its domain. By combining~\eqref{eq:Dynamic_Controller} with~\eqref{eq:Dynamic}, we get the closed-loop system
\begin{equation} \label{eq:closed-loop_LQG}
\begin{aligned}
    \frac{d}{dt} \begin{bmatrix} x \\ \xi \end{bmatrix} &= \begin{bmatrix}
    A & BC_{\mK} \\
    B_{\mK}C & A_{\mK}
    \end{bmatrix} \begin{bmatrix} x \\ \xi \end{bmatrix}  +  \begin{bmatrix} I & 0 \\ 0 & B_{\mK}  \end{bmatrix}\begin{bmatrix} w \\ v \end{bmatrix}, \\
    \begin{bmatrix} y \\ u \end{bmatrix}& = \begin{bmatrix} C & 0 \\ 0& C_{\mK} \end{bmatrix} \begin{bmatrix} x \\ \xi \end{bmatrix} + \begin{bmatrix} v \\0 \end{bmatrix}.
\end{aligned}
\end{equation}
It is known from classical control theory~\cite[Chapter 13]{zhou1996robust} that under~\cref{assumption:stabilizability}, the LQG cost is finite if the closed-loop matrix
\begin{equation}
\label{eq:closedloopmatrix}
\begin{bmatrix}
    A & BC_{\mK} \\
    B_{\mK}C & A_{\mK}
\end{bmatrix} = \begin{bmatrix}
    A & 0 \\
    0 & 0
\end{bmatrix} + \begin{bmatrix}
    B & 0 \\
    0 & I
\end{bmatrix}\begin{bmatrix}
    0 & C_{\mK} \\
    B_{\mK} & A_{\mK}
\end{bmatrix}\begin{bmatrix}
    C & 0 \\
    0 & I
\end{bmatrix}
\end{equation}
is stable~\cite{zhou1996robust}, i.e., the real parts of all its eigenvalues are negative; dynamical controllers satisfying this condition is said to \emph{internally stabilize} the plant~\eqref{eq:Dynamic}. Furthermore, it is a known fact in control theory that the optimal controller \eqref{eq:LQGstatespace} obtained by solving the Riccati equations internally stabilizes the plant. We therefore parameterize the set of stabilizing controllers with order $q\in\mathbb{N}$ by\footnote{
We explicitly include the zero matrix $D_{\mK}$ in the definition of $\mathcal{C}_q$, which corresponds to the set of \emph{strictly proper} dynamical controllers. If we allow $D_{\mK}$ to be non-zero, we will have a \emph{proper} dynamical controller; see \cref{appendix:proper_controller}. In this definition, when $q=0$, we have $\mathcal{C}_q=\{0_{m\times p}\}$ if the plant \eqref{eq:Dynamic} is open-loop stable, and $\mathcal{C}_q=\varnothing$ otherwise.
}\textsuperscript{,}\footnote{In~\eqref{eq:internallystabilizing}, for notational simplicity, we lumped the controller parameters into a single matrix; but it should be interpreted as a dynamical controller, represented by~\eqref{eq:Dynamic_Controller}. Note that this definition allows us to apply block-wise matrix operations; see \emph{e.g.},~\eqref{eq:def_sim_transform}.}
\begin{equation} \label{eq:internallystabilizing}
    \mathcal{C}_{q} := \left\{
    \left.\mK=\begin{bmatrix}
    D_{\mK} & C_{\mK} \\
    B_{\mK} & A_{\mK}
    \end{bmatrix}
    \in \mathbb{R}^{(m+q) \times (p+q)} \right|\;  D_{\mK} = 0_{m\times p},~\eqref{eq:closedloopmatrix}~\text{is stable}\right\},
\end{equation}
and let $J_q(\mK):\mathcal{C}_q\rightarrow\mathbb{R}$ denote the function that maps a parameterized dynamical controller in $\mathcal{C}_q$ to its corresponding LQG cost for each $q\in\mathbb{N}$.
%
%We now have the following result characterizing when $\mathcal{C}_{q}$ is non-empty.
% \begin{lemma}[{\cite{brasch1970pole}}]
%   Under \cref{assumption:stabilizability}, $\mathcal{C}_{q} \neq \emptyset$ if $q = n-1$ or $q = n$.
% \end{lemma}
%
It can be shown that the set of \emph{full-order} stabilizing controllers $\mathcal{C}_{n}$ is nonempty as long as~\cref{assumption:stabilizability} holds~\cite{zhou1996robust}, and since it also contains the optimal LQG controller to~\eqref{eq:LQG}, we will mainly focus on the set of \emph{full-order} stabilizing controllers $\mathcal{C}_{n}$ in this paper. We will abbreviate $J_n(\mK)$ as $J(\mK)$ when no confusions occur.

The following lemma shows that the set $\mathcal{C}_q$ can be treated as an open set when it is nonempty.
This is a direct consequence of the fact that the Routh–Hurwitz stability criterion returns a set of strict polynomial inequalities in terms of the elements of $(A_{\mK}, B_{\mK}, C_{\mK})$.
\begin{lemma}
Let $q\geq 1$ such that $\mathcal{C}_q$ is nonempty. Then, $\mathcal{C}_q$ is an open subset of the linear space
\begin{equation}\label{eq:def_Vq}
\mathcal{V}_q\coloneqq \left\{
\left.\begin{bmatrix}
    D_{\mK} & C_{\mK} \\
    B_{\mK} & A_{\mK}
    \end{bmatrix}
    \in \mathbb{R}^{(m+q) \times (p+q)} \right|\;  D_{\mK} = 0_{m\times p}\right\}.
\end{equation}
\end{lemma}

The following two lemmas give useful characterizations of the LQG cost function $J_q$. These results are known in the literature; we provide a short proof in~\cref{appendix:real_analytic} for completeness.
\begin{lemma}\label{lemma:LQG_cost_formulation1}
Fix $q\in\mathbb{N}$ such that $\mathcal{C}_q\neq\varnothing$. Given $\mK\in\mathcal{C}_q$, we have
%{\color{blue}Given $\mK\in\mathcal{C}_q$, if $\mK$ is controllable and observable,} we have
\begin{equation}\label{eq:LQG_cost_formulation1}
J_q(\mK)
=
\operatorname{tr}
\left(
\begin{bmatrix}
Q & 0 \\ 0 & C_{\mK}^\tr R C_{\mK}
\end{bmatrix} X_\mK\right)
=
\operatorname{tr}
\left(
\begin{bmatrix}
W & 0 \\ 0 & B_{\mK} V B_{\mK}^\tr
\end{bmatrix} Y_\mK\right),
\end{equation}
where $X_{\mK}$ and $Y_{\mK}$ are the unique positive semidefinite
solutions to the following Lyapunov equations
\begin{subequations}
\begin{align}
\begin{bmatrix} A &  BC_{\mK} \\ B_{\mK} C & A_{\mK} \end{bmatrix}X_{\mK} + X_{\mK}\begin{bmatrix} A &  BC_{\mK} \\ B_{\mK} C & A_{\mK} \end{bmatrix}^\tr +  \begin{bmatrix} W & 0 \\ 0 & B_{\mK}VB_{\mK}^\tr  \end{bmatrix}
& = 0, \label{eq:LyapunovX}
\\
\begin{bmatrix} A &  BC_{\mK} \\ B_{\mK} C & A_{\mK} \end{bmatrix}^\tr Y_{\mK} +  Y_{\mK}\begin{bmatrix} A &  BC_{\mK} \\ B_{\mK} C & A_{\mK} \end{bmatrix} +   \begin{bmatrix} Q & 0 \\ 0 & C_{\mK}^\tr R C_{\mK} \end{bmatrix}
& = 0. \label{eq:LyapunovY}
\end{align}
%As a consequence, $J_q$ is a real analytic function on $\mathcal{C}_q$.
\end{subequations}
\end{lemma}
\begin{lemma}\label{lemma:LQG_cost_analytical}
    Fix $q \in \mathbb{N}$ such that $\mathcal{C}_q\neq\varnothing$. Then, $J_q$ is a real analytic function on $\mathcal{C}_q$. %, and $J_q$ is non-***.
\end{lemma}
%might not approach to infinity as $\mK$ goes to the boundary of $\mathcal{C}_q$

Now, given the dimension $n$ of the plant's state variable, the LQG problem~\eqref{eq:LQG} can be reformulated into a constrained optimization problem:
%\begin{subequations}
\begin{equation} \label{eq:LQG_reformulation_KX}
    \begin{aligned}
        \min_{\mK} \quad &J_n(\mK) \\
        \text{subject to} \quad& \mK \in \mathcal{C}_n. % \;~\eqref{eq:LyapunovX}.
    \end{aligned}
\end{equation}
% \begin{equation} \label{eq:LQG_reformulation_KY}
%     \begin{aligned}
%         \min_{\mK} \quad &J_n(\mK) \\
%         \text{subject to} \quad& \mK \in \mathcal{C}_n, \;~\eqref{eq:LyapunovY}.
%     \end{aligned}
% \end{equation}
%
%\end{subequations}
After reformulating the LQG~\eqref{eq:LQG} into~\eqref{eq:LQG_reformulation_KX}, it is possible to estimate the gradient of $J_n(\mK)$ from system trajectories, and one may further derive model-free policy gradient algorithms to find a solution to~\eqref{eq:LQG_reformulation_KX}. To characterize the performance of policy gradient algorithms, it is necessary to understand the landscape of~\eqref{eq:LQG_reformulation_KX}. It is well-known that $\mathcal{C}_{n}$ is in general non-convex.~\Cref{lemma:LQG_cost_analytical} indicates that $J_n$ is a real analytical function. However, little is known about their further geometrical and analytical properties, especially those that are fundamental for establishing convergence of gradient-based algorithms. In this paper, we focus on the following two topics of the set $\mathcal{C}_n$ and the LQG cost function $J_n$:
\begin{enumerate}
\item \emph{The connectivity of $\mathcal{C}_n$ and its implications}, which will be studied in~\cref{sec:connectivity}. Connectivity of stabilizing controllers has received increasing attention, but most recent results focus on state-feedback controllers or static output-feedback controllers~\cite{fazel2018global,mohammadi2019convergence,bu2019topological,feng2020connectivity}. It is known that the set of stabilizing state-feedback policies is in general non-convex but connected, and this connectivity is fundamental for gradient-based local search algorithms to find a good solution. It is also known that the set of stabilizing static output-feedback policies can be highly disconnected (there exist cases with an exponential number of connected components~\cite{feng2020connectivity}). The connectivity of dynamical controllers $\mathcal{C}_n$, however, is unknown and has not been discussed before in the literature.

%{\color{red} expand this part further to make it as balanced as the next one.}
\item \emph{The structure of the stationary points and the global optimum of $J_n$}, which will be studied in~\cref{sec:stationary_points}. A classical result in control is that the solution to the LQR problem is unique under mild technical assumptions, which is an important fact in
establishing the \emph{gradient dominant} property of the LQR cost function \cite{fazel2018global,mohammadi2019convergence}. In addition, it has been recently shown that a class of output-feedback controller design problem in finite-time horizon also has a unique stationary point~\cite{furieri2020learning}. However, it is expected that the stationary points of the LQG problem~\eqref{eq:LQG_reformulation_KX} are not unique due to the non-uniqueness of globally optimal solutions in the state-space domain. We aim to reveal further structural properties of stationary points of the LQG problem~\eqref{eq:LQG_reformulation_KX}.
\end{enumerate}

\section{Connectivity of the Set of Stabilizing Controllers}
\label{sec:connectivity}

In this section, we characterize the connectivity of the set of full-order stabilizing controllers $\mathcal{C}_{n}$. %  as well as its corresponding version of proper stabilizing controllers.
%
%\subsection{Connectivity of $\mathcal{C}_{n}$}
%We note that
%
%\subsection{Strictly proper controllers}
We first have the following observation.
\begin{lemma} \label{lemma:unbounded_proper}
    Under \cref{assumption:stabilizability}, the set ${\mathcal{C}}_{n}$ is non-empty, unbounded, and can be non-convex.
\end{lemma}
\begin{proof}
     It is a well-known fact in control theory that $\mathcal{C}_{n} \neq \varnothing$ under~\cref{assumption:stabilizability}. In particular, any pole assignment algorithm or solving the Ricatti equations~\eqref{eq:Riccati_P} and~\eqref{eq:Riccati_S} can find a feasible point in $\mathcal{C}_{n}$.
     To show the unboundedness of $\mathcal{C}_n$, we introduce the following set
     $$
        \mathcal{S}_n = \left\{ \mK=\begin{bmatrix}
    0 & C_{\mK} \\
    B_{\mK} & A_{\mK}
    \end{bmatrix}
    \in \mathbb{R}^{(m+n) \times (p+n)} \; \left| \; \begin{aligned}A_{\mK} = A - BK - LC, \; B_{\mK} = L, \; C_{\mK} = -K, \\ A-BK \; \text{and}\; A - LC \; \text{are stable}\end{aligned} \right. \right\}.
        $$
    It has been established in classical control theory that $\mathcal{S}_n \subset \mathcal{C}_{n}$~\cite[Chapter 3.5]{zhou1996robust} and the set $\{K \mid A - BK\; \text{is stable}\}$ is unbounded (see, e.g., \cite[Observation 3.6]{bu2019topological}). Thus, the set $\mathcal{S}_n$ is unbounded, and so is ${\mathcal{C}}_{n}$. Non-convexity of $\mathcal{C}_{n}$ is also known and can be illustrated by the explicit counterexample in~\cref{example:disconnectivity}.
\end{proof}

\begin{example}[Non-convexity of stabilizing controllers] \label{example:disconnectivity}
Consider a dynamical system~\eqref{eq:Dynamic} with
\begin{equation*} %\label{eq:CounterExample}
   A = 1, \;\; B = 1, \;\;  C = 1.
\end{equation*}
The set of stabilizing controllers $\mathcal{C}_n=\mathcal{C}_1$ is given by
$$
\mathcal{C}_n
=\left\{
\left.
\mK = \begin{bmatrix} 0 & C_{\mK} \\
                          B_{\mK} & A_{\mK}\end{bmatrix} \in \mathbb{R}^{2 \times 2}
\right|
\begin{bmatrix}
1 & C_{\mK} \\
B_{\mK} & A_{\mK}
\end{bmatrix}
\text{ is stable}
\right\}.
$$
It is easy to verify that the following dynamical controllers
$$
    \mK^{(1)} = \begin{bmatrix} 0 & 2 \\
                          -2 & -2\end{bmatrix}, \qquad     \mK^{(2)} = \begin{bmatrix} 0 & -2 \\
                          2 & -2\end{bmatrix}
$$
internally stabilize the plant and thus belong to $\mathcal{C}_1$. However,
$
   \hat{\mK} = \frac{1}{2}\left(\mK^{(1)} + \mK^{(2)}\right) = \begin{bmatrix} 0 & 0 \\
                          0 & -2\end{bmatrix}
$
fails to stabilize the plant. \hfill\qed
%
%Since~\eqref{eq:CounterExample} is open-loop unstable, we know $\mathcal{C}_0 = \varnothing$, \emph{i.e.}, there is no reduced-order strictly proper stabilizing controller. By \cref{Theo:connectivity_conditions}, $\mathcal{C}_1$ has two path-connected components.
\end{example}

\subsection{Main Results}

We first introduce the notion of similarity transformation that has been widely-used in control theory. Given $q\geq 1$ such that $\mathcal{C}_q\neq\varnothing$, we define the mapping $\mathscr{T}_q:\mathrm{GL}_q\times\mathcal{C}_q\rightarrow\mathcal{C}_q$ that represents similarity transformations on $\mathcal{C}_q$ by
\begin{equation}\label{eq:def_sim_transform}
\mathscr{T}_q(T,\mK)
\coloneqq \begin{bmatrix}
I_m & 0 \\
0 & T
\end{bmatrix}\begin{bmatrix}
D_{\mK} & C_{\mK} \\
B_{\mK} & A_{\mK}
\end{bmatrix}\begin{bmatrix}
I_p & 0 \\
0 & T
\end{bmatrix}^{-1} = \begin{bmatrix}
D_{\mK} & C_{\mK}T^{-1} \\
TB_{\mK} & TA_{\mK}T^{-1}
\end{bmatrix}.
\end{equation}
It is not hard to verify that for any invertible matrix $T\in\mathrm{GL}_q$ and $\mK\in\mathcal{C}_q$, $\mathscr{T}_q(T,\mK)$ is indeed a stabilizing controller of order $q$ and thus is in $\mathcal{C}_q$. We can also check that $\mathscr{T}_q$ is indefinitely differentiable on $\mathrm{GL}_q\times \mathcal{C}_q$, and that
\begin{equation}\label{eq:group_action}
\mathscr{T}_q(T_2,\mathscr{T}_q(T_1,\mK))
=
\mathscr{T}_q(T_2T_1,\mK)
\end{equation}
for any invertible $T_1,T_2 \in\mathrm{GL}_q$. This implies that for any fixed $T\in\mathrm{GL}_n$, the map
$$
\mK\mapsto \mathscr{T}_q(T,\mK)
$$
admits an inverse given by $\mK\mapsto\mathscr{T}_q(T^{-1},\mK)$. Therefore, we have the following result (see~\cref{app:manifold} for a review of manifolds and diffeomorphism).
\begin{lemma} \label{lemma:Cq_invariant}
Given $q\geq 1$ such that $\mathcal{C}_q\neq\varnothing$, for any invertible matrix $T \in\mathrm{GL}_q$, the map
$$
\mK\mapsto \mathscr{T}_q(T,\mK)
$$
is a diffeomorphism from $\mathcal{C}_q$ to itself.
\end{lemma}

Our main technical results in this section are on the path-connectivity of ${\mathcal{C}}_{n}$. Recall that $\mathscr{T}_n(T,\mK)$ is defined by \eqref{eq:def_sim_transform}. For notational simplicity, for any fixed $T\in\mathrm{GL}_n$, we let $\mathscr{T}_T:\mathcal{C}_n\rightarrow\mathcal{C}_n$ denote the mapping given by
$$
\mathscr{T}_T(\mK)
\coloneqq
\mathscr{T}_n(T,\mK)
=\begin{bmatrix}
D_{\mK} & C_{\mK}T^{-1} \\
TB_{\mK} & TA_{\mK}T^{-1}
\end{bmatrix}.
$$
We are now ready to present the main technical results.

\begin{theorem} \label{Theo:disconnectivity}
Under~\cref{assumption:stabilizability}, the set ${\mathcal{C}}_{n}$ has at most two path-connected components.
\end{theorem}

\begin{theorem} \label{Theo:Cn_homeomorphic}
{If ${\mathcal{C}}_{n}$ has two path-connected components $\mathcal{C}_n^{(1)}$ and $\mathcal{C}_n^{(2)}$, then $\mathcal{C}_n^{(1)}$ and $\mathcal{C}_n^{(2)}$ are diffeomorphic under the mapping $\mathscr{T}_T$, for any invertible matrix $T\in\mathbb{R}^{n\times n}$ with $\det T<0$. % restricted on $\mathcal{C}_n^{(1)}$ gives a homeomorphism from $\mathcal{C}_n^{(1)}$ to $\mathcal{C}_n^{(2)}$.
}
\end{theorem}

\Cref{Theo:Cn_homeomorphic} shows that even if ${\mathcal{C}}_{n}$ has two path-connected components, there exists a linear bijection mapping defined by a similarity transformation $\mathscr{T}_T$ between these two components.  In the following theorem, we present a sufficient condition under which $\mathcal{C}_n$ is path-connected. This condition becomes necessary for a class of dynamical systems with single input or single output.

\begin{theorem} \label{Theo:connectivity_conditions}
Under~\cref{assumption:stabilizability}, the following statements are true.
\begin{enumerate}
    \item
$\mathcal{C}_n$ is path-connected if there exists a reduced-order stabilizing controller, \emph{i.e.}, ${\mathcal{C}}_{n-1} \neq \varnothing$.

\item Suppose the plant \eqref{eq:Dynamic} is single-input or single-output, i.e., $m=1$ or $p=1$. Then the set ${\mathcal{C}}_{n}$ is path-connected if and only if ${\mathcal{C}}_{n-1} \neq \varnothing$.
\end{enumerate}
\end{theorem}

One main idea in our proofs is based on a classical change of variables for dynamical controllers (see, e.g.,~\cite{scherer1997multiobjective}). We adopt the change of variables to construct a set with a convex projection and a surjective mapping from that set to $\mathcal{C}_n$, and then path-connectivity
results generally follow from the fact that a convex set is path-connected.
The potential disconnectivity of $\mathcal{C}_n$ comes from the fact that the set of real invertible matrices $\mathrm{GL}_n=\{\Pi\in\mathbb{R}^{n\times n}\mid\,\det \Pi\neq 0\}$ has two path-connected components \citep{lee2013introduction}:
$
\mathrm{GL}^+_n=\{\Pi\in\mathbb{R}^{n\times n}\mid\,\det \Pi> 0\},
\,
\mathrm{GL}^-_n=\{\Pi\in\mathbb{R}^{n\times n}\mid\,\det \Pi< 0\}.
$
The full proofs are technically involved, and we postpone them to~\Cref{subsection:proof_connectivity}---~\ref{subsection:proof_connectivity_2}.

Here, we note that given any open-loop unstable first-order dynamical system, \emph{i.e.}, $n = 1$, and $A > 0$ in~\eqref{eq:Dynamic}, it is easy to see that there exist no reduced-order stabilizing controllers, \emph{i.e.},   $\mathcal{C}_{n-1} = \varnothing$. Thus,~\cref{Theo:connectivity_conditions} indicates that its associated set of stabilizing controllers $\mathcal{C}_n$ is not path-connected.
%
%we have that $\mathcal{C}_1$ is disconnected and has two path-connected components.
We provide an explicit single-input and single-output (SISO) example below. % for which reduced-order controllers does not exist, and consequently the set ${\mathcal{C}}_{n}$ is not path-connected.

\begin{example}[Disconectivity of stabilizing controllers]\label{example:SISO1}
 Consider the dynamical system in~\cref{example:disconnectivity}:
$$
A=1,\quad B=1,\quad C=1.
$$
Since it is open-loop unstable and only has state of dimension $n = 1$, we know $\mathcal{C}_{n-1} = \varnothing$. Thus, \cref{Theo:connectivity_conditions} indicates that its associated set of stabilizing controllers $\mathcal{C}_n$ is not path-connected.

Indeed, using the Routh--Hurwitz stability criterion, it is straightforward to derive that
\begin{equation}\label{eq:region_example}
    \begin{aligned}
    \mathcal{C}_1 &= \left\{\mK = \begin{bmatrix} 0 & C_{\mK} \\
                          B_{\mK} & A_{\mK}\end{bmatrix} \in \mathbb{R}^{2 \times 2} \left| \begin{bmatrix}
    A & BC_{\mK} \\
    B_{\mK}C & A_{\mK}
\end{bmatrix}\; \text{is stable}\right. \right\} \\
&= \left\{ \left. \mK = \begin{bmatrix} 0 & C_{\mK} \\
                          B_{\mK} & A_{\mK}\end{bmatrix} \in \mathbb{R}^{2 \times 2} \right| A_{\mK} < -1, \; B_{\mK}C_{\mK} < A_{\mK} \right\}. \\
\end{aligned}
\end{equation}
This set has two path-connected components: $\mathcal{C}_1 = \mathcal{C}_1^+ \cup \mathcal{C}_1^-$ with $\mathcal{C}_1^+ \cap \mathcal{C}_1^- = \emptyset$, where
$$
\begin{aligned}
    \mathcal{C}_1^+ &:= \left\{ \left. \mK = \begin{bmatrix} 0 & C_{\mK} \\
                          B_{\mK} & A_{\mK}\end{bmatrix} \in \mathbb{R}^{2 \times 2} \right| A_{\mK} < -1,\, B_{\mK}C_{\mK} < A_{\mK}, \, B_{\mK} > 0 \right\},
                          \\
                          \mathcal{C}_1^- &:= \left\{ \left. \mK = \begin{bmatrix} 0 & C_{\mK} \\
                          B_{\mK} & A_{\mK}\end{bmatrix} \in \mathbb{R}^{2 \times 2} \right| A_{\mK} < -1, \, B_{\mK}C_{\mK} < A_{\mK}, \, B_{\mK} < 0 \right\}. \\
\end{aligned}
$$
In addition, as expected by \cref{Theo:Cn_homeomorphic}, it is easy to verify that $\mathcal{C}_1^{+}$ and $\mathcal{C}_1^{-}$ are homeomorphic under the mapping $\mathscr{T}_T$, for any $T<0$. \cref{fig:feasible_region_disconnected} illustrates the region of the set $\mathcal{C}_1$ in~\eqref{eq:region_example}. \hfill\qed
\end{example}

\begin{figure}
\centering
\begin{subfigure}{.4\textwidth}
    \includegraphics[width =\textwidth]{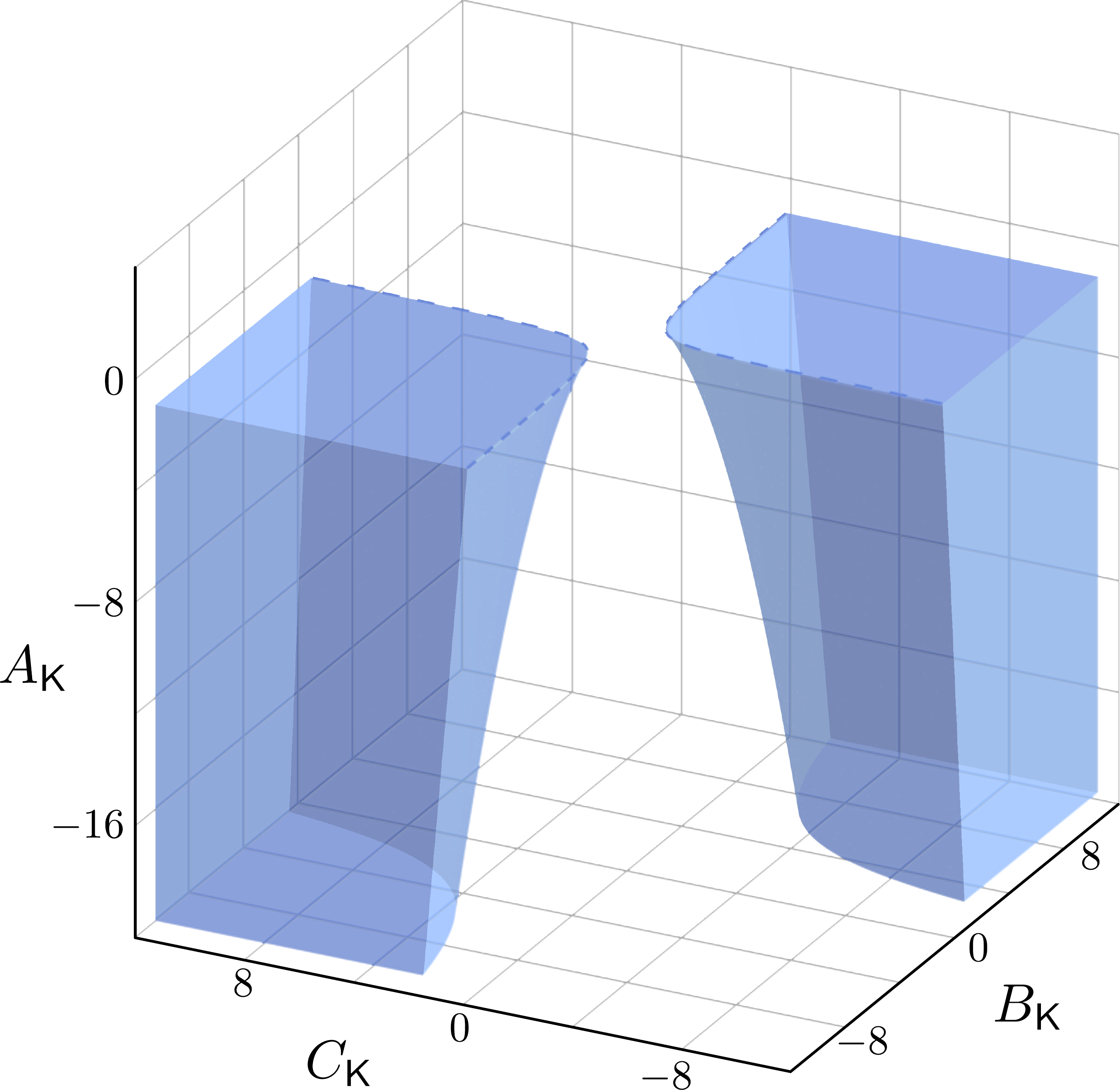}
    \caption{$\mathcal{C}_1$ for \cref{example:SISO1}}
    \label{fig:feasible_region_disconnected}
\end{subfigure}
    \hspace{10mm}
\begin{subfigure}{.4\textwidth}
\includegraphics[width=\textwidth]{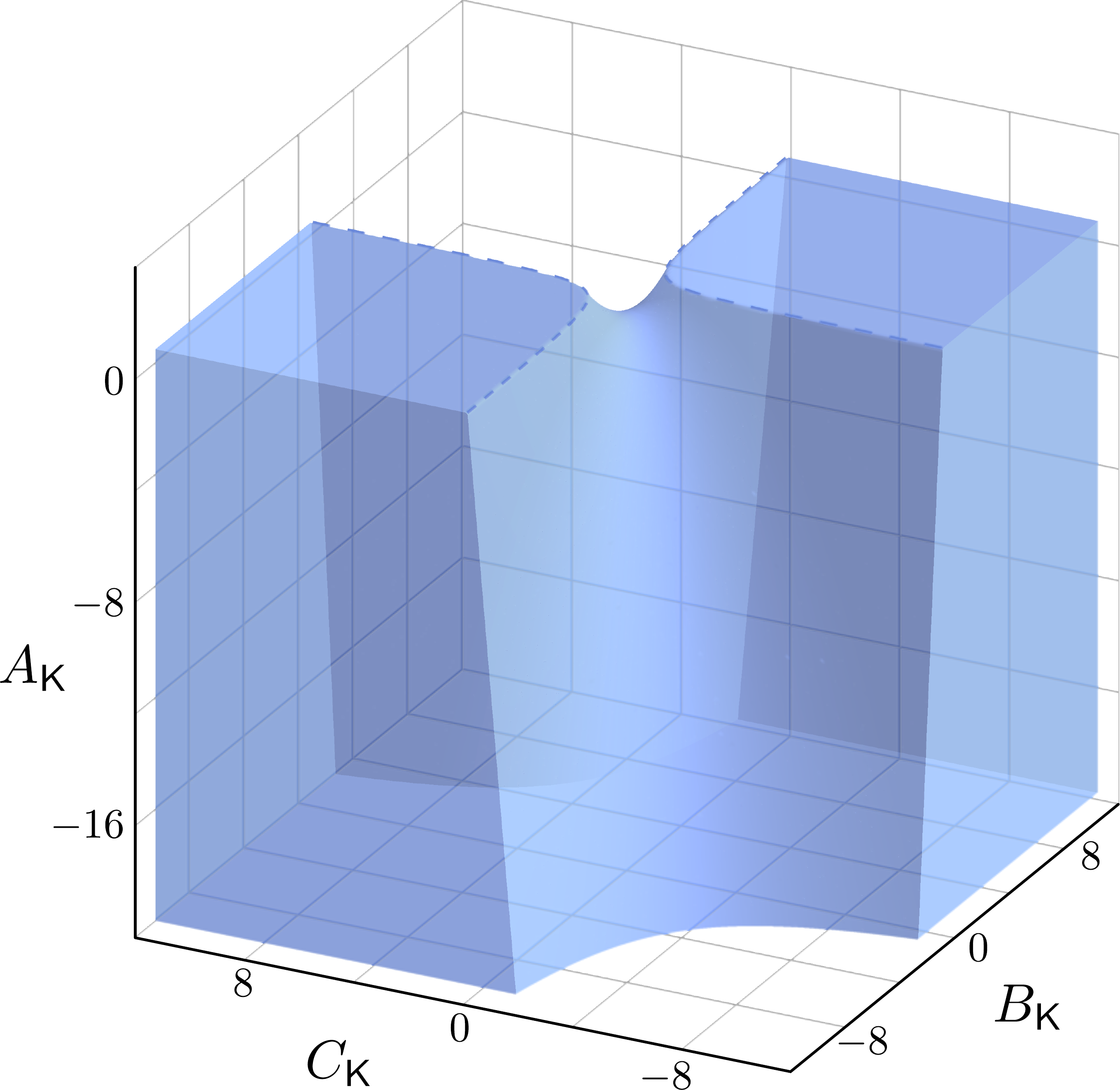}
\caption{$\mathcal{C}_1$ for \cref{example:SISO_connected}}
\label{fig:feasible_region_connected}
\end{subfigure}
    \caption{The set of stabilizing controllers $\mathcal{C}_1$ for Examples~\ref{example:SISO1} and~\ref{example:SISO_connected}: (a) For~\cref{example:SISO1}, the set $\mathcal{C}_1$ given by~\eqref{eq:region_example} has two path-connected components; (b) For~\cref{example:SISO_connected}, the set $\mathcal{C}_1$ given by~\eqref{eq:region_example_connected} is path-connected.}
    \label{fig:feasible_region}
\end{figure}

In \cref{appendix:eg_disconnectivity}, we present a nontrivial second-order SISO system, for which $\mathcal{C}_{n-1} = \varnothing$ and $\mathcal{C}_{n}$ is disconnected. \cref{Theo:connectivity_conditions} also suggests the following corollary.

\begin{corollary} \label{corollary:stable}
    Given any open-loop stable dynamical system~\eqref{eq:Dynamic}, \emph{i.e.}, $A$ is stable, we have that $\mathcal{C}_n$ is path-connected.
\end{corollary}
\begin{proof}
    Since the dynamical system~\eqref{eq:Dynamic} is open-loop stable, thus
    $$
        \mK = \begin{bmatrix} 0_{m\times p} & 0_{m\times (n-1)} \\ 0_{(n-1)\times p} & -I_{n-1} \end{bmatrix} \in \mathcal{C}_{n-1},
    $$
    and $\mathcal{C}_{n-1} \neq \varnothing$. By \cref{Theo:connectivity_conditions}, $\mathcal{C}_n$ is path-connected.
\end{proof}

\begin{example}[Stabilizing controllers for an open-loop stable system]\label{example:SISO_connected}
Consider an open-loop stable dynamical system~\eqref{eq:Dynamic} with
$$
A=-1,\;\;  B=1, \;\; C=1.
$$
Since it is open-loop stable, \cref{corollary:stable} indicates that its associated set of stabilizing controllers $\mathcal{C}_n$ is path-connected. Using the Routh--Hurwitz stability criterion, it is straightforward to derive that
\begin{equation}\label{eq:region_example_connected}
    \begin{aligned}
    \mathcal{C}_1 &= \left\{ \left. \mK = \begin{bmatrix} 0 & C_{\mK} \\
                          B_{\mK} & A_{\mK}\end{bmatrix} \in \mathbb{R}^{2 \times 2} \right| A_{\mK} < 1, B_{\mK}C_{\mK} < -A_{\mK} \right\}. \\
\end{aligned}
\end{equation}
This set is path-connected, as illustrated in~\cref{fig:feasible_region_connected}. %We note that the set of stabilizing controllers in~\eqref{eq:region_example} is a strict subset of $\mathcal{C}_1$ in~\eqref{eq:region_example_connected}. \hfill\qed
\end{example}

Before presenting the technical proofs, we note that the controllers of $\mathcal{C}_n$ in~\eqref{eq:internallystabilizing} are always \emph{strictly proper}, which is sufficient for the LQG problem~\eqref{eq:LQG}. For closed-loop stability, we can also consider \emph{proper} dynamical controllers. We provide this discussion in \cref{appendix:proper_controller}: Unlike $\mathcal{C}_n$ that might be disconnected, the set of proper stabilizing dynamical controllers is always connected (see \cref{prop:connectivity_Cn}).

\begin{remark}[Connectivity of the feasible region of LQR/LQG and gradient-based algorithms]
Motivated by the success of data-driven RL techniques, some recent studies revisited the classical LQR problem from a modern optimization perspective and designed policy gradient algorithms~\cite{fazel2018global, mohammadi2019convergence, zhang2019policy}. The connectivity of feasible region (i.e., the set of stabilizing controllers) becomes important to local search algorithms (e.g., policy gradient) since they typically cannot jump between different connected components. It is known that the set of stabilizing static state-feedback policies $\{K\in \mathbb{R}^{m \times n} \mid A - BK \text{~is stable}\}$ is connected \cite{bu2019topological}, and this is one important factor in justifying the performance of the algorithms in~\cite{fazel2018global,mohammadi2019convergence, zhang2019policy}. On the other hand, the set of stabilizing static output feedback policies $\{D_{\mK} \in \mathbb{R}^{m \times p} \mid A - BD_{\mK}C \; \text{is stable}\}$ can be highly disconnected~\cite{feng2020connectivity}, posing a significant challenge for local search algorithms. In Theorems~\ref{Theo:disconnectivity},~\ref{Theo:Cn_homeomorphic} and~\ref{Theo:connectivity_conditions}, we have shown that the set of stabilizing controllers $\mathcal{C}_n$ in LQG problem has at most two path-connected components that are diffeomorphic to each other under a particular similarity transformation. Since similarity transformation does not change the input/output behavior of a controller (see \cref{App:control_basics}), it makes no difference to search over any path-connected component in $\mathcal{C}_n$ even if $\mathcal{C}_n$ is not path-connected. This brings positive news to gradient-based local search algorithms for the LQG problem.
\end{remark}

\subsection{Proof of \cref{Theo:disconnectivity}}
\label{subsection:proof_connectivity}
The following Lyapunov stability criterion~\cite{boyd1994linear} plays a central role in our proof: A square real matrix $M$ is stable if and only if the Lyapunov inequality
$$
    MP+PM^\tr \prec 0
$$
has a positive definite solution $P\succ 0$.

The analysis of the path-connectivity of $\mathcal{C}_n$ is similar with analyzing the connectivity of the set of stabilizing static state feedback policies: We first adopt a classical change of variables that has been used for developing convex reformulation of controller synthesis problems, and then path-connectivity results generally follow from the fact that \emph{a convex set is path-connected}; see~\Cref{remark:connectivity} for details.

\begin{remark}[Connectivity of stabilizing static state-feedback policies] \label{remark:connectivity}
The path-connectivity of the set of stabilizing static state-feedback policies $\{K\in \mathbb{R}^{m \times n} \mid A - BK \text{~is stable}\}$ is easy to show:
    \begin{equation} \label{eq:staticpolicies_connect}
        \begin{aligned}
         &\{K\in \mathbb{R}^{m \times n} \mid A - BK \text{~is stable}\} \\
         \; \Longleftrightarrow \; &\{K \in \mathbb{R}^{m \times n} \mid \exists P \succ 0,  P(A-BK)^\tr+ (A - BK)P \prec 0 \} \\
        \; \Longleftrightarrow \; &\{K \in \mathbb{R}^{m \times n} \mid \exists P \succ 0,  PA^\tr- L^\tr B^\tr + AP - BL \prec 0, L = KP \} \\
        \; \Longleftrightarrow \; &\{K = LP^{-1} \in \mathbb{R}^{m \times n} \mid \exists P \succ 0,  PA^\tr- L^\tr B^\tr + AP - BL \prec 0\}.
        \end{aligned} \\
    \end{equation}
    Since the set
    \begin{equation} \label{eq:convexPL}
        \{(P,L)\mid P \succ 0,  PA^\tr - L^\tr B^\tr + AP - BL \prec 0\}
    \end{equation}
    is convex and the map
    %\begin{equation}\label{eq:change_of_variables_state}
     $   K = LP^{-1}$
    %\end{equation}
    is continuous for the elements in~\eqref{eq:convexPL}, we know $\{K\in \mathbb{R}^{m \times n} \mid A - BK \text{~is stable}\}$ is path-connected.
    The second equivalence in~\eqref{eq:staticpolicies_connect} utilizes a well-known \emph{change of variables}
    %\begin{equation} %\label{eq:change_of_variables_state}
     $K = LP^{-1}.$
    %\end{equation}
    This trick is essential to derive \emph{convex reformulations} for designing state-feedback policies in various setups~\cite{boyd1994linear}. %, even though the original controller synthesis problems are non-convex.
    We note that the trick~\eqref{eq:staticpolicies_connect} has been used in~\cite{bu2019topological,mohammadi2019convergence}.

    The main strategy in the proof of \cref{Theo:disconnectivity} is similar to~\eqref{eq:staticpolicies_connect}, but we need to use a more complicated change of variables for dynamical controllers in the state-space domain~\cite{scherer1997multiobjective}. To see the difficulty, applying the Lyapunov stability result leads to\footnote{We explicitly include the matrix $D_{\mK}$ in the Lypuanov inequality~\eqref{eq:LyapunovOutput}: $D_{\mK} = 0$ corresponds to strictly proper controllers and $D_{\mK} \neq 0$ corresponds to proper controllers; see \cref{appendix:proper_controller}.}
    \begin{equation} \label{eq:LyapunovOutput}
    \begin{aligned}
    &\begin{bmatrix}
    A+BD_{\mK}C & BC_{\mK} \\ B_{\mK}C & A_{\mK}
    \end{bmatrix}\; \text{is stable }  \\
    \Longleftrightarrow \quad &\exists P \succ 0, \, P\begin{bmatrix}
    A+BD_{\mK}C & BC_{\mK} \\ B_{\mK}C & A_{\mK}
    \end{bmatrix}^\tr + \begin{bmatrix}
    A+BD_{\mK}C & BC_{\mK} \\ B_{\mK}C & A_{\mK}
    \end{bmatrix}P \prec 0,
    \end{aligned}
    \end{equation}
    where the coupling between the auxiliary variable $P$ and the controller parameters $A_{\mK}, B_{\mK}, C_{\mK}, D_{\mK}$ are much more involved.
\end{remark}

In our proof, we adopt the change of variables presented in \cite{scherer1997multiobjective}. Given the system dynamics $(A, B, C)$ in~\eqref{eq:Dynamic}, we first introduce the following convex set\footnote{
We explicitly include the zero matrix $G$ in the definition of $\mathcal{F}_n$, for which the purpose will become clear when studying the set of proper stabilizing controllers; see \cref{appendix:proper_controller}.
}
\begin{equation} \label{eq:SetF}
\begin{aligned}
\mathcal{F}_n
\coloneqq
\bigg\{
(X,Y,M,&\,G,H,F) \mid
X,Y\in\mathbb{S}^n,\
M \in \mathbb{R}^{n \times n}, G=0_{m\times p}, H\in \mathbb{R}^{n \times p}, F \in \mathbb{R}^{m \times n},\\
& \begin{bmatrix}
X & I \\ I & Y
\end{bmatrix} \! \succ 0,\
\begin{bmatrix}
AX \!+\! BF & A \!+\! BGC \\ M & YA \!+\! HC
\end{bmatrix}
+
\begin{bmatrix}
AX \!+\! BF & A \!+\! BGC \\ M & YA \!+\! HC
\end{bmatrix}^{\!\top}
\!\! \prec 0
\bigg\},
\end{aligned}
\end{equation}
and the extended set
\begin{equation} \label{eq:SetG}
\mathcal{G}_n
:=\left\{\mZ=(X,Y,M,G,H,F,\Pi,\Xi) \left|\; \begin{aligned}(X,Y,M,G,H,F)\in\mathcal{F}_n, \\
\Pi,\Xi\in\mathbb{R}^{n\times n},\
\Xi\Pi =I-YX
\end{aligned}
\right. \right\}.
\end{equation}
We shall later see that there exists a continuous surjective map from $\mathcal{G}_n$ to $\mathcal{C}_n$, and the path-connectivity of the convex set $\mathcal{F}_n$ plays a key role in analyzing the path-connected components of $\mathcal{C}_n$. Before proceeding, we note the following observation for each element in $\mathcal{G}_n$.
\begin{lemma}\label{lemma:connectivity_preliminary}
For any $(X,Y,M,G,H,F,\Pi,\Xi)\in\mathcal{G}_n$, $\Pi$ and $\Xi$ are always invertible, and consequently, the block triangular matrices
$
\begin{bmatrix} I & 0 \\ YB & \Xi
\end{bmatrix}$ and $\begin{bmatrix} I & CX\\
0 & \Pi  \end{bmatrix}
$
are invertible.
\end{lemma}
\begin{proof}
 By definition,  for all $ (X,Y,W,G,H,F,\Pi,\Xi)\in\mathcal{G}_n$, we have
 $\begin{bmatrix} X & I \\ I & Y \end{bmatrix}\succ 0$,
 implying that
$$
\det(YX-I)
=\det X\det(Y-X^{-1})
=\det\begin{bmatrix}
X & I \\ I & Y
\end{bmatrix}>0.
$$
Thus, $\text{det}(\Pi) \neq 0$ and $\text{det}(\Xi) \neq 0$, indicating they are both invertible. The invertibility of the other two block triangular matrices is straightforward.
\end{proof}

We now define a mapping from $\mathcal{G}_n$ to a subset of $\mathbb{R}^{(m+n)\times(p+n)}$. % that is essential to our proof.
\begin{definition}[Change of variables via nonlinear mapping]
    For each $\mZ=(X,Y,M,G,H,F,\Pi,\Xi)$ in $\mathcal{G}_n$, let
    \begin{equation} \label{eq:change_of_variables_output}
        \Phi(\mZ)=
        \begin{bmatrix}
        \Phi_D(\mZ) & \Phi_C(\mZ) \\
        \Phi_B(\mZ) & \Phi_A(\mZ)
        \end{bmatrix}
        \coloneqq
        \begin{bmatrix} I & 0 \\ YB & \Xi
        \end{bmatrix}^{-1} \begin{bmatrix}
        G & H \\
        F & M-YAX
        \end{bmatrix}\begin{bmatrix} I & CX\\
        0 & \Pi  \end{bmatrix}^{-1}.
    \end{equation}
\end{definition}
It is easy to see that $\Phi_D(\mZ) \equiv G \equiv 0$ for $\mZ\in\mathcal{G}_n$. % by the definition of $\mathcal{G}_n$.
 We point out that this mapping~\eqref{eq:change_of_variables_output} is derived from the change of variables presented in \cite{scherer1997multiobjective}, which is essential to obtain equivalent convex reformulations for a range of output-feedback controller synthesis, including $\mathcal{H}_\infty$ and $\mathcal{H}_2$ optimal control. %(see \cref{appendix:LQG_convex_reformulation} for an equivalent convex reformulation for the LQG problem~\eqref{eq:LQG} using~\eqref{eq:change_of_variables_output}).
The following result builds an explicit connection between $\mathcal{G}_n$ and $\mathcal{C}_n$ via the mapping $\Phi$, and its proof is provided in \cref{app:surjective_proof}.
\begin{proposition}\label{proposition:Phi_surjective}
The mapping $\Phi$ in~\eqref{eq:change_of_variables_output} is a continuous and surjective mapping from $\mathcal{G}_n$ to $\mathcal{C}_n$.
\end{proposition}

After establishing the continuous surjection from $\mathcal{G}_n$ to $\mathcal{C}_n$, it is now clear that we can study the path-connectivity of $\mathcal{C}_n$ via the path-connectivity of $\mathcal{G}_n$: Any continuous path in $\mathcal{G}_n$ will be mapped to a continuous path in $\mathcal{C}_n$, and thus any path-connected component of $\mathcal{G}_n$ has a path-connected image under the mapping $\Phi$. Consequently, the number of path-connected components of $\mathcal{C}_n$ will be no more than the number of path-connected components of $\mathcal{G}_n$.

We now proceed to provide results on the path-connectivity of the set $\mathcal{G}_n$.
\begin{proposition}\label{lemma:Gn_connected_components}
The set $\mathcal{G}_n$ has two path-connected components, given by
$$
\begin{aligned}
\mathcal{G}_n^+
=\ &
\left\{(X,Y,M,G,H,F,\Pi,\Xi)\in\mathcal{G}_n \mid \, \det \Pi>0\right\}, \\
\mathcal{G}_n^-
=\ &
\left\{(X,Y,M,G,H,F,\Pi,\Xi)\in\mathcal{G}_n \mid \, \det \Pi<0\right\}.
\end{aligned}
$$
\end{proposition}
\begin{proof}
First, the convexity of $\mathcal{F}_n$ implies that set $\mathcal{F}_n$ is path-connected. We then notice that the set of real invertible matrices $\mathrm{GL}_n=\{\Pi\in\mathbb{R}^{n\times n}\mid\,\det \Pi\neq 0\}$ has two path-connected components \cite{lee2013introduction}
$$
\mathrm{GL}^+_n=\{\Pi\in\mathbb{R}^{n\times n}\mid\,\det \Pi> 0\},
\qquad
\mathrm{GL}^-_n=\{\Pi\in\mathbb{R}^{n\times n}\mid\,\det \Pi< 0\}.
$$
Therefore the Cartesian product $\mathcal{F}_n\times \mathrm{GL}_n$ has two path-connected components. Finally, it is not hard to verify that the following mapping
$$
(X,Y,M,G,H,F,\Pi)
\mapsto
(X,Y,M,G,H,F,\Pi,(I-YX)\Pi^{-1})
$$
is a continuous bijection from $\mathcal{F}_n\times \mathrm{GL}_n$ to $\mathcal{G}_n$.

Therefore $\mathcal{G}_n$ also has two path-connected components, and their expressions are evident.
\end{proof}

\cref{lemma:Gn_connected_components} then implies that $\mathcal{C}_n$ has at most two path-connected components. Precisely, upon defining
$$
\mathcal{C}_n^{+}=\Phi(\mathcal{G}_n^+),
\qquad
\mathcal{C}_n^{-}=\Phi(\mathcal{G}_n^-),
$$
the two path-connected components of $\mathcal{C}_n$ are just given by $\mathcal{C}_n^{+}$ and $\mathcal{C}_n^{-}$, if $\mathcal{C}_n$ is not path-connected. This completes the proof of \cref{Theo:disconnectivity}.

\subsection{Proof of \cref{Theo:Cn_homeomorphic}}

In the previous subsection, we have already shown that $\mathcal{C}_n^{+}$ and $\mathcal{C}_n^{-}$ are the two path-connected components if $\mathcal{C}_n$ is not connected. In order to prove \cref{Theo:Cn_homeomorphic}, it suffices to show that, regardless of the path-connectivity of $\mathcal{C}_n$,
for any $T\in\mathbb{R}^{n\times n}$ with $\det T<0$, the mapping $\mathscr{T}_T$ restricted on $\mathcal{C}_n^{+}$ gives a diffeomorphism from $\mathcal{C}_n^{+}$ to $\mathcal{C}_n^{-}$.

Since $\mathscr{T}_T$ is a diffeomorphism from $\mathcal{C}_n$ to itself with inverse $\mathscr{T}_{T^{-1}}$, and $\mathcal{C}_n^+$ and $\mathcal{C}_n^-$ are two open subsets of $\mathcal{C}_n$, to complete the proof, we only need to show that
$$\mathscr{T}_T(\mathcal{C}_n^+)\subseteq \mathcal{C}_n^-, \qquad \text{and} \qquad  \mathscr{T}_{T^{-1}}(\mathcal{C}_n^-)\subseteq \mathcal{C}_n^+$$
when $\det T<0$. Consider an arbitrary point
$$
\mK =
\begin{bmatrix}
0 & C_{\mK} \\
B_{\mK} & A_{\mK}
\end{bmatrix}\in\mathcal{C}_n^+.
$$
By the definition of $\mathcal{C}_n^+$, there exists $\mZ=(X,Y,M,G,H,F,\Pi,\Xi)\in\mathcal{G}_n^+$ such that $\Phi(\mZ)=\mK$. Now let
$$
\hat \Pi = T\Pi,\quad \hat \Xi = \Xi T^{-1},
\quad
\hat \mZ = (X,Y,M,G,H,F,\hat\Pi,\hat\Xi).
$$
It is not difficult to verify that $\hat{\mZ}\in\mathcal{G}_n$. Since $\det\hat{\Pi}=\det T\cdot\det \Pi < 0$, we have $\hat\mZ\in \mathcal{G}_n^-$. Then,
$$
\begin{aligned}
\Phi(\hat{\mZ})= \ &
\begin{bmatrix}
\Phi_D(\hat\mZ) & \Phi_C(\hat\mZ) \\
\Phi_B(\hat\mZ) & \Phi_A(\hat\mZ)
\end{bmatrix} \\
=\ &
\begin{bmatrix} I & 0 \\
YB & \hat{\Xi} \end{bmatrix}^{-1} \begin{bmatrix}
G & F \\
H & M-YAX \end{bmatrix}\begin{bmatrix} I & CX \\ 0 & \hat \Pi \end{bmatrix}^{-1} \\
=\ &
\begin{bmatrix}
I & 0 \\ 0 & T
\end{bmatrix}
\begin{bmatrix} \Xi & YB\\ 0 & I \end{bmatrix}^{-1} \begin{bmatrix}
G & F \\
H & M-YAX \end{bmatrix}
\begin{bmatrix} I & CX \\ 0 & \Pi\end{bmatrix}^{-1}
\begin{bmatrix}
I & 0 \\ 0 & T^{-1}
\end{bmatrix} \\
=\ &
\begin{bmatrix}
I & 0 \\ 0 & T
\end{bmatrix}
\begin{bmatrix}
0 & C_\mK \\
 B_\mK & A_\mK
\end{bmatrix}
\begin{bmatrix}
I & 0 \\ 0 & T^{-1}
\end{bmatrix} \\
= \ &
\begin{bmatrix}
0 & {C_\mK} T^{-1} \\
T {B_\mK} & T {A_\mK}T^{-1}
\end{bmatrix} \\
=\ & \mathscr{T}_T(\mK),
\end{aligned}
$$
which implies that $\mathscr{T}_T(\mK)\in \Phi(\mathcal{G}_n^-)=\mathcal{C}_n^-$ and consequently $\mathscr{T}_{T}(\mathcal{C}_n^+)\subseteq \mathcal{C}_n^-$.

The proof of $\mathscr{T}_{T^{-1}}(\mathcal{C}_n^-)\subseteq \mathcal{C}_n^+$ is similar by noting that $\det T^{-1}<0$ if and only if $\det T<0$.

\subsection{Proof of \cref{Theo:connectivity_conditions}}
\label{subsection:proof_connectivity_2}

We first show that the non-emptiness of $\mathcal{C}_{n-1}$ implies the path-connectivity of $\mathcal{C}_{n}$. Indeed, suppose there exists $\tilde{\mK}\in\mathcal{C}_{n-1}$. Then it can be augmented to be a full-order controller in $\mathcal{C}_n$ by
$$
\mK =
\begin{bmatrix}
0 & \tilde{C}{_\mK} & 0 \\
\tilde{B}{_\mK} & \tilde{A}{_\mK} & 0 \\
0 & 0 & -1
\end{bmatrix}
$$
Now define a similarity transformation matrix
$$
T=\begin{bmatrix}
I_{n-1} & 0 \\ 0 & -1
\end{bmatrix}.
$$
By the proof of \cref{Theo:Cn_homeomorphic}, we can see that $\mK\in\mathcal{C}_n^\pm$ implies $\mathscr{T}_T(\mK)\in\mathcal{C}_n^\mp$. On the other hand, we can directly check that $\mathscr{T}_T(\mK)=\mK$. Therefore we have
$$\mK\in \mathcal{C}_n^+\cap \mathcal{C}_n^-,$$
indicating that $\mathcal{C}_n^+\cap \mathcal{C}_n^-$ is nonempty. Consequently, $\mathcal{C}_n$ is path-connected.

We then carry out the analysis for the case when the plant is single-input or single-output. The goal is to find a reduced-order controller in $\mathcal{C}_{n-1}$ when $\mathcal{C}_n$ is connected. Here we only prove the single-out case; the single-input case can be proved similarly, \emph{i.e.}, using the observability matrix or by the duality between controllability and observability.

Let $T$ be any real $n\times n$ matrix with $\det T<0$. Let $\mK^{(0)}\in \mathcal{C}_n$ be arbitrary, and let
$
\mK^{(1)}
= \mathscr{T}_T(\mK^{(0)})
$.
If $\mathcal{C}_n$ is path-connected, then there exists a continuous path
$$
\mK(t)=\begin{bmatrix}
0 & C_\mK(t) \\ B_\mK(t) & A_\mK(t)
\end{bmatrix},\quad t\in[0,1]
$$
in $\mathcal{C}_n$ such that
$$
    \mK(0)=\mK^{(0)}, \qquad  \text{and} \qquad \mK(1)=\mK^{(1)}.
$$
Now for each $t\in[0,1]$, let $\mathsf{C}(t)$ denote the controllability matrix for $(A_{\mK}(t),B_{\mK}(t))$, i.e.,
$$
\mathsf{C}(t) = \begin{bmatrix}
B_{\mK}(t) & A_\mK(t)B_{\mK}(t)
& \cdots & A_\mK(t)^{n-1}B_{\mK}(t)
\end{bmatrix}
\in\mathbb{R}^{n\times n},
$$
where the dimension of $\mathsf{C}(t)$ is $n \times n$ since the plant is single-output (i.e., the controller is single-input).

We then have $\mathsf{C}(1)= T \mathsf{C}(0)$, and thus
$$
    \det \mathsf{C}(1)\cdot \det\mathsf{C}(0)<0.
$$
On the other hand, it can be seen that $\det \mathsf{C}(t)$ is a continuous function over $t\in[0,1]$. Therefore
$$
\det \mathsf{C}(\tau)=0
$$
for some $\tau\in(0,1)$, implying that $(A_{\mK}(\tau),B_{\mK}(\tau))$ is not controllable. This indicates that the transfer function $C_\mK(\tau)(sI_n-A_\mK(\tau))^{-1}B_\mK(\tau)$ can be realized by a state-space representation with dimension at most $n-1$ (see \cref{App:control_basics}), and consequently $\mathcal{C}_{n-1}\neq\varnothing$.

\section{Structure of Stationary Points}
\label{sec:stationary_points}

We have shown that the set of stabilizing controllers $\mathcal{C}_{n}$ might be disconnected, and that the potential disconnectivity has no harm to gradient-based local search algorithms. In this section, we proceed to characterize the stationary points of the cost function in the LQG problem~\eqref{eq:LQG}, which is another important factor for establishing the convergence of gradient-based algorithms.

\cref{subsection:invariance_LQG} discusses the invariance of the LQG cost $J_q$ under similarity transformation and its implications. \cref{subsection:LQG_grad_Hessian} shows how to compute the gradient and the Hessian of the LQG cost $J_q$. In \cref{subsection:nonminimal_stationary}, some results related to non-minimal stationary points are provided. We characterize the minimal stationary points for LQG over~$\mathcal{C}_n$ in \cref{subsection:gradient_LQG}. Finally, in \cref{subsection:hessaion_LQG_minimal}, we discuss the second-order behavior of $J_n(\mK)$ around its minimal stationary points.

\subsection{Invariance of LQG Cost under Similarity Transformation} \label{subsection:invariance_LQG}

As shown in \cref{lemma:Cq_invariant}, the similarity transformation $\mathscr{T}_q(T,\cdot)$ is a diffeomorphism from $\mathcal{C}_q$ to itself for any invertible matrix $T \in \operatorname{GL}_q$. Then together with \eqref{eq:group_action}, we can see that the set of similarity transformation is a group that is isomorphic to $\mathrm{GL}_q$. We can therefore define the \emph{orbit} of $\mK\in\mathcal{C}_q$ by
$$
\mathcal{O}_\mK
\coloneqq \{\mathscr{T}_q(T,\mK)\mid T\in\mathrm{GL}_q\}.
$$
It is known that the LQG cost is invariant under the same similarity transformation, and thus is a constant over an orbit $\mathcal{O}_{\mK}$ for any $\mK\in\mathcal{C}_q$.
\begin{lemma} \label{lemma:Jn_invariance}
   Let $q\geq 1$ such that $\mathcal{C}_{q}\neq\varnothing$. Then we have
   $$
   J_q(\mK) = J_q\!\left(\mathscr{T}_q(T,\mK)\right)
   $$
   for any $\mK\in\mathcal{C}_q$ and any invertible matrix $T\in\mathrm{GL}_q$.
\end{lemma}
 \begin{proof}
 Given any $\mK \in \mathcal{C}_q$ and any invertible $T\in\mathbb{R}^{q\times q}$, we know that $\mathscr{T}(T,\mK) \in \mathcal{C}_q$. Thus, the Lyapunov equation~\eqref{eq:LyapunovX} admits a unique positive semidefinite solution for each of $\mK$ and $\mathscr{T}_q(T,\mK)$ (see \cref{lemma:Lyapunov}).

 Suppose that the solution of~\eqref{eq:LyapunovX} for  $\mK$ is $X_{\mK}$. Then, it is not difficult to verify that the unique solution of~\eqref{eq:LyapunovX} for  $\mathscr{T}_q(T,\mK)$ is
 $$
   \begin{bmatrix}
    I & 0 \\0 &T
    \end{bmatrix}X_{\mK}\begin{bmatrix}
    I & 0 \\0 &T
    \end{bmatrix}^{\tr}.
 $$
 Therefore, we have
 $$
 \begin{aligned}
    J_q(\mathscr{T}_q(T,\mK)) = &
\operatorname{tr}
\left(
\begin{bmatrix}
Q & 0 \\ 0 & (C_{\mK}T^{-1})^\tr R C_{\mK}T^{-1}
\end{bmatrix} \begin{bmatrix}
    I & 0 \\0 &T
    \end{bmatrix}X_{\mK}\begin{bmatrix}
    I & 0 \\0 &T
    \end{bmatrix}^{\tr}\right) \\
    =&
\operatorname{tr}
\left(
\begin{bmatrix}
Q & 0 \\ 0 & C_{\mK}^\tr R C_{\mK}
\end{bmatrix} X_\mK\right) \\
= &J_q(\mK),
\end{aligned}
 $$
 where the second identity applies the trace property $\operatorname{tr}(AB) = \operatorname{tr}(BA)$ for $A, B$ with compatible dimensions.
\end{proof}

The following proposition shows that every orbit $\mathcal{O}_\mK$ corresponding to controllable and observable controllers has dimension $q^2$ with two path-connected components. The proof is given in \cref{app:proof_sim_trans_submanifold}.

\begin{proposition}\label{proposition:sim_trans_submanifold}
Suppose $\mK\in\mathcal{C}_q$ represents a controllable and observable controller. Then the orbit $\mathcal{O}_{\mK}$ is a submanifold of $\mathcal{C}_q$ of dimension $q^2$, and has two path-connected components, given by
$$
\begin{aligned}
\mathcal{O}_{\mK}^+
=\ &
\{\mathscr{T}_q(T,\mK)\mid T\in\mathrm{GL}_q,\det T>0\}, \\
\mathcal{O}_{\mK}^-
=\ &
\{\mathscr{T}_q(T,\mK)\mid T\in\mathrm{GL}_q,\det T<0\}.
\end{aligned}
$$
\end{proposition}

\begin{figure}
    \centering
        \centering
\begin{subfigure}{.4\textwidth}
    \includegraphics[width = 0.8\textwidth]{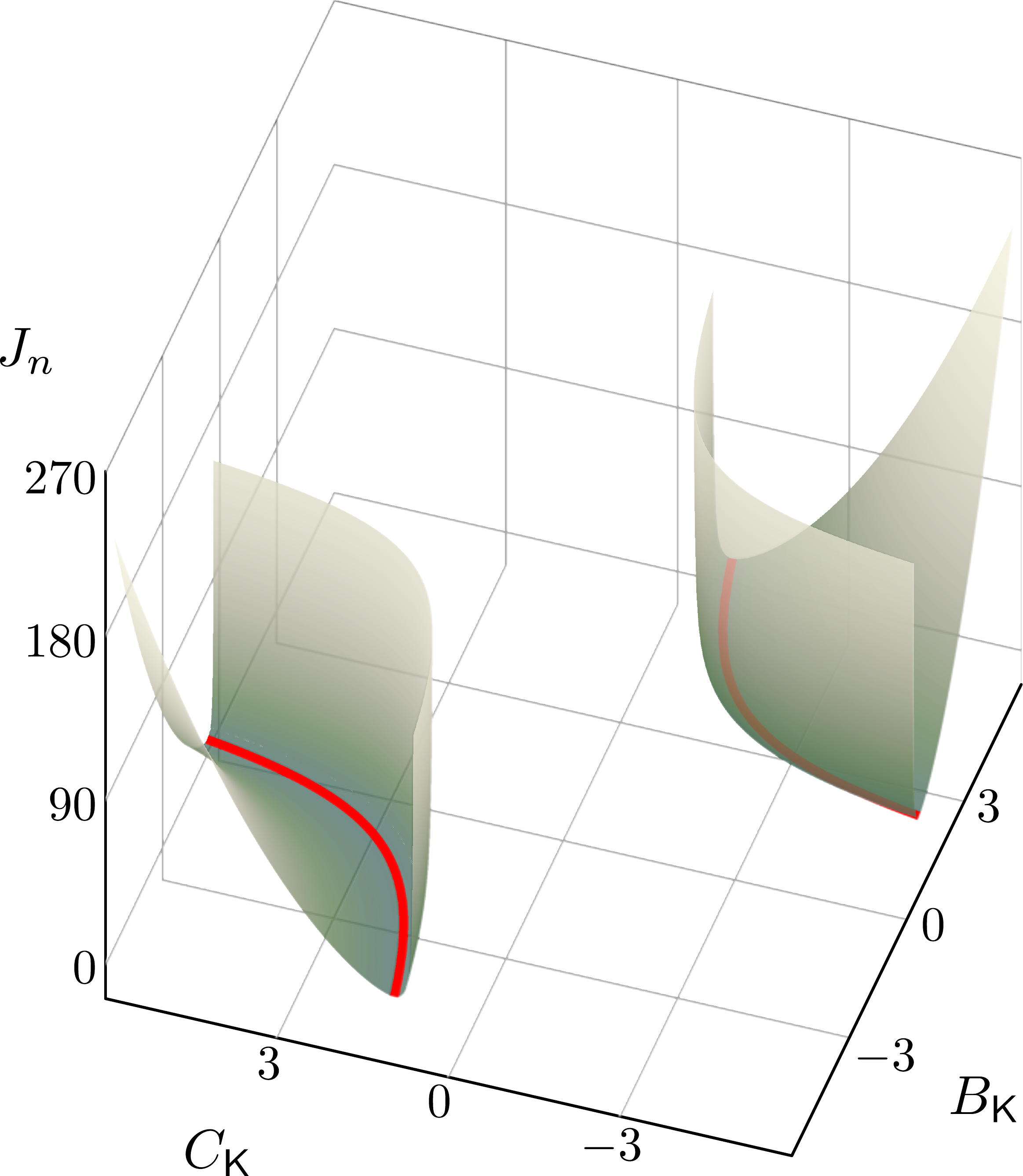}
     \caption{Open-loop unstable system in \cref{example:SISO1}}
    \end{subfigure}
    \hspace{15mm}
\begin{subfigure}{.4\textwidth}
    \includegraphics[width = 0.8\textwidth]{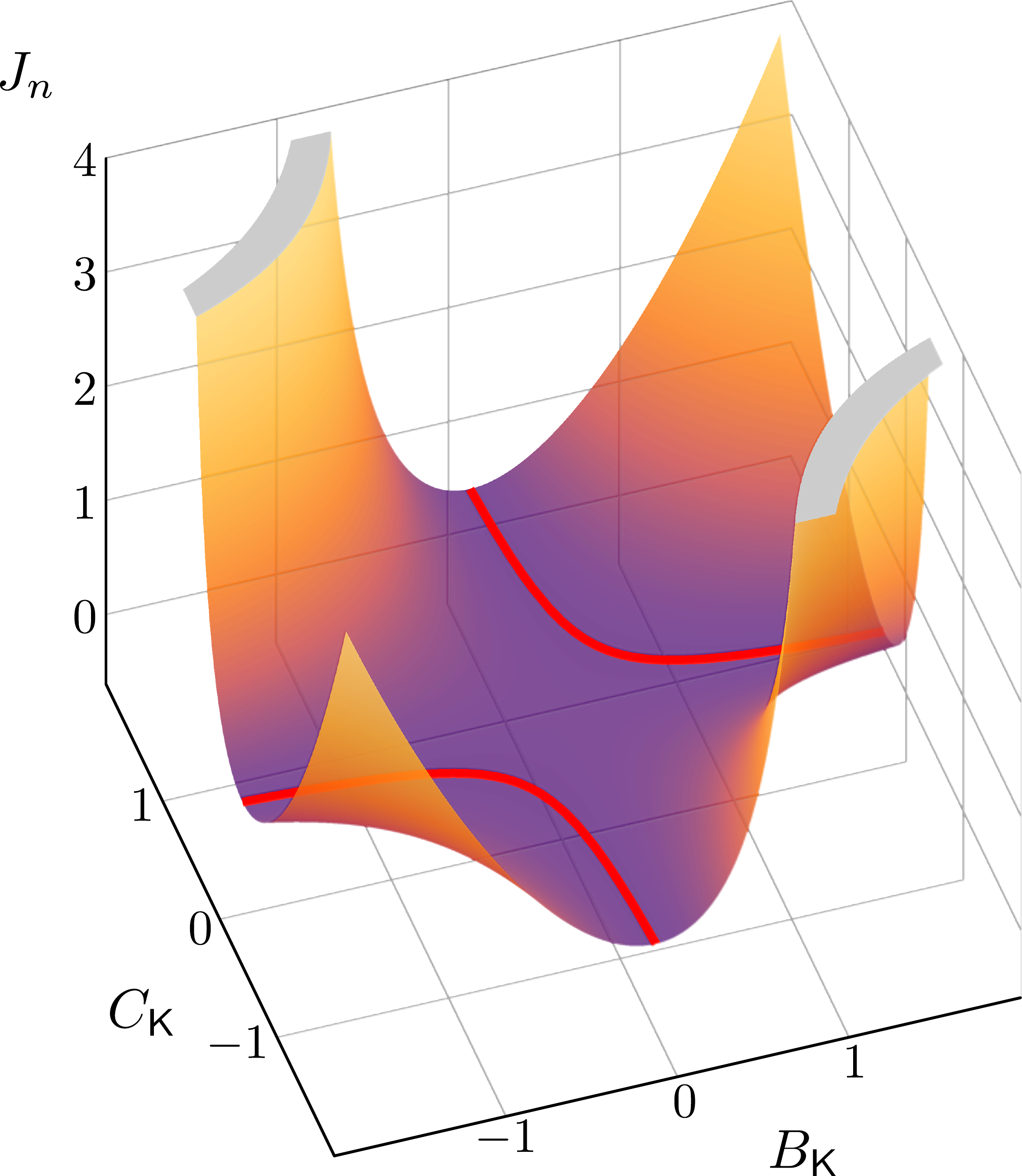}
         \caption{Open-loop stable system in \cref{example:SISO_connected}}
\end{subfigure}
    \caption{Non-isolated and disconnected globally optimal LQG controllers. In both cases, we set $Q = 1, R = 1, V = 1, W = 1 $. (a) LQG cost for the open-loop unstable SISO system in \cref{example:SISO1} when fixing $A_{\mK} = -1-2\sqrt{2}$, for which the set of globally optimal points $\left\{(B_{\mK}, C_{\mK}) \mid B_{\mK} = (1+\sqrt{2})\frac{1}{T}, C_{\mK}=-(1+\sqrt{2})T,\; T \neq 0\right\}$ has two connected components. (b) LQG cost for the open-loop stable SISO system in \cref{example:SISO_connected} when fixing $A_{\mK} = 1-2\sqrt{2}$, for which the set of globally optimal points $\left\{(B_{\mK}, C_{\mK}) \mid B_{\mK} = (-1+\sqrt{2})\frac{1}{T}, C_{\mK}=(1-\sqrt{2})T,\; T \neq 0\right\}$ has two connected components.}
    \label{fig:LQG_landscape_eg3}
\end{figure}

%\begin{remark}[Non-isolated and disconnected globally optimal LQG controllers]
    From \cref{lemma:Jn_invariance} and \cref{proposition:sim_trans_submanifold}, one interesting consequence is that given a globally optimal LQG controller $\mK^* \in \mathcal{C}_n$, then any controller in following orbit is globally optimal
$$
\mathcal{O}_{\mK^*}
\coloneqq \{\mathscr{T}_n(T,\mK^*)\mid T\in\mathrm{GL}_n\}.
$$
If $\mK^*$ is minimal (i.e., controllable and observable), the orbit $\mathcal{O}_{\mK^*}$ is a submanifold in $\mathcal{V}_n$ of dimension $n^2$, and it has two path-connected components. \cref{fig:LQG_landscape_eg3}  demonstrates the orbit of globally optimal LQG controllers for an open-loop unstable system and another open-loop stable system, which shows that the set of globally optimal LQG controllers are non-isolated and disconnected in $\mathcal{C}_n$.  % globally optimal LQG controllers
%\end{remark}

% \begin{figure}
%     \centering
%  %   \subfigure[]{
%     \includegraphics[scale = 0.7]{./figs/Fig4.eps}
%   %  }
%     \caption{LQG cost function of the open-loop stable SISO system in \cref{example:SISO_connected} when fixing $A_{\mK} = 1-2\sqrt{2}$, where $Q = 1, R = 1, V = 1, W = 1 $. The set of globally optimal points $\left\{(B_{\mK}, C_{\mK}) \mid B_{\mK} = (-1+\sqrt{2})\frac{1}{T}, C_{\mK}=(1-\sqrt{2})T,\; T \neq 0\right\}$ has two connected components. }
%     \label{fig:LQG_landscape_eg4}
% \end{figure}

%{\color{red} define tangent space for a manifold}

\begin{figure}[t]
    \centering
    \includegraphics[scale = 0.32]{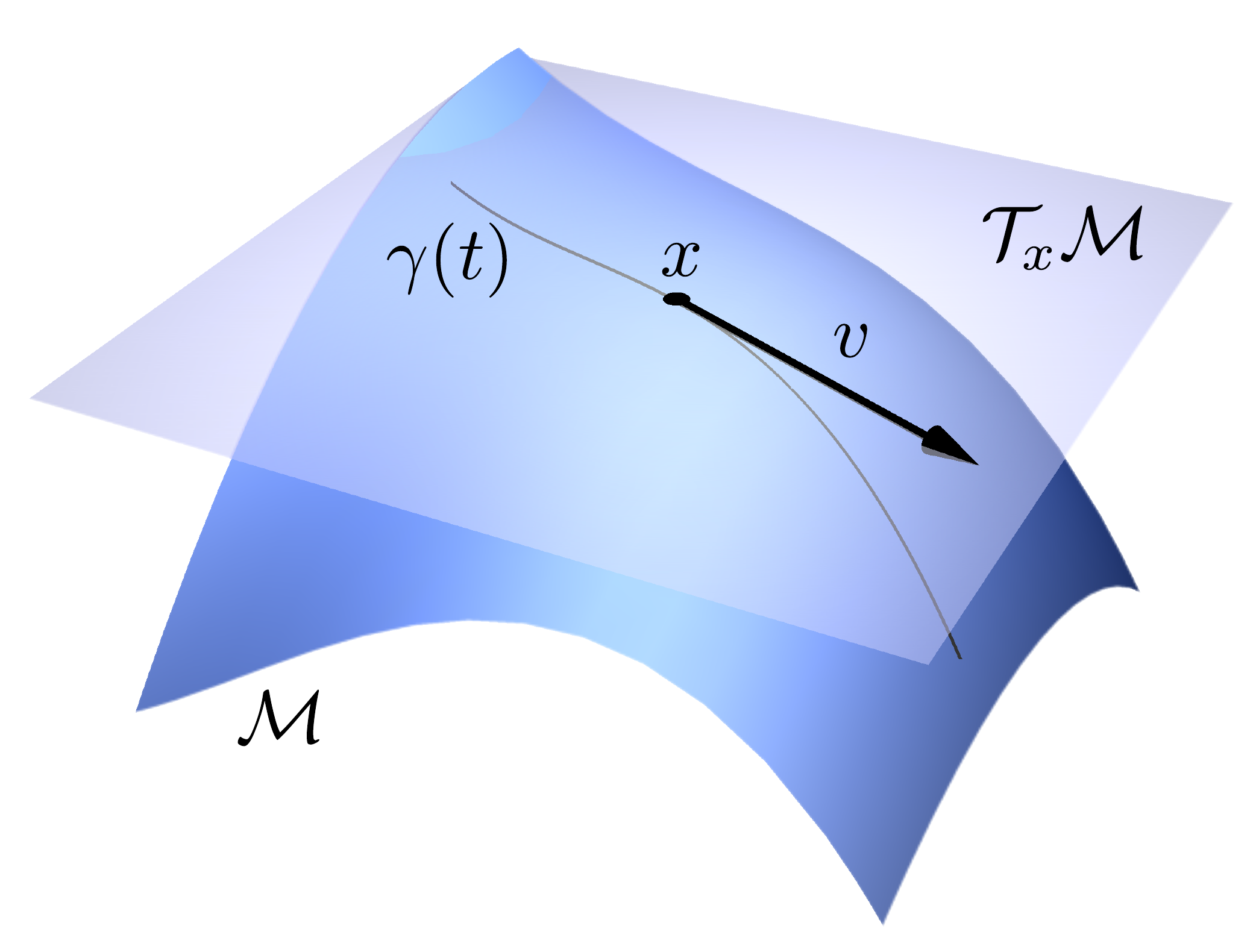}
    \caption{A graphical illustration of a manifold $\mathcal{M}$ and its tanget space $\mathcal{T}_x\mathcal{M}$ at some point $x\in\mathcal{M}$. Here $\gamma(t)$ is an arbitrary $C^\infty$ curve in $\mathcal{M}$ that passes through $x$, and $v$ is the tangent vector of $\gamma(t)$ at $x$. The tangent space $\mathcal{T}_x\mathcal{M}$ consists of all such vectors $v$.}
    \label{fig:tangent_space}
\end{figure}

\cref{proposition:sim_trans_submanifold} guarantees that for any controllable and observable $\mK\in\mathcal{C}_q$, the orbit $\mathcal{O}_{\mK}$ is a submanifold of dimension $q^2$ in $\mathcal{C}_q$, which allows us to define the tangent space of $\mathcal{O}_{\mK}$.\footnote{See \cref{app:manifold} for the definition of tangent spaces. A visualization of a manifold $\mathcal{M}$ and its tangent space $\mathcal{T}_x\mathcal{M}$ at one point $x\in\mathcal{M}$ is provided in \cref{fig:tangent_space}.} For each minimal $\mK\in\mathcal{C}_q$, we use $\mathcal{TO}_{\mK}$ to denote the tangent space of $\mathcal{O}_{\mK}$ at $\mK$, and treat it as a subspace of $\mathcal{V}_q$; recall that $\mathcal{V}_q$ is defined by \eqref{eq:def_Vq}. The dimension of $\mathcal{TO}_{\mK}$ is then
$$
\dim \mathcal{TO}_{\mK}
=\dim \mathcal{O}_{\mK}=q^2.
$$
%
%
% {\color{red}
% \begin{remark}
% Example, not controllable or observable, $\mathcal{O}_{\mK}$ might has only one connected component.
%      %It is kind of related to the over-parameterization  or symmetry~\cite{zhang2020symmetry} in machine learning?
% \end{remark}
% }
%
We denote the orthogonal complement of $\mathcal{TO}_{\mK}$ in $\mathcal{V}_q$ by $\mathcal{TO}_{\mK}^\perp$. The following proposition characterizes the tangent space $\mathcal{TO}_{\mK}$ and its orthogonal complement $\mathcal{TO}_{\mK}^\perp$ at a minimal controller $\mK \in \mathcal{C}_q$.

\begin{proposition}\label{proposition:tangent_orbit}
Let $\mK\in\mathcal{C}_q$ represent a controllable and observable controller. Then
\begin{align*}
\mathcal{TO}_{\mK}
=\ &
\left\{
\left.\begin{bmatrix}
0 & -C_{\mK} H \\
H B_{\mK} & HA_{\mK} - A_{\mK}H
\end{bmatrix}\,
\right| H\in\mathbb{R}^{q\times q}
\right\}, \\
\mathcal{TO}_{\mK}^\perp
=\ &
\left\{
\left.\Delta=\begin{bmatrix}
0 & \Delta_{B_\mK} \\
\Delta_{C_\mK} & \Delta_{A_\mK}
\end{bmatrix}
\in\mathcal{V}_q
\,\right|
\Delta_{A_{\mK}} A_{\mK}^\tr
-A_{\mK}^\tr \Delta_{A_{\mK}}
+ \Delta_{B_{\mK}} B_{\mK}^\tr
- C_{\mK}^\tr \Delta_{C_{\mK}}=0
\right\}.
\end{align*}
\end{proposition}
\begin{proof}
Let $H\in\mathbb{R}^{q\times q}$ be arbitrary. Then for sufficiently small $\epsilon$, we have
$$
\begin{aligned}
\mathscr{T}_q(I+\epsilon H,\mK)
=\ &
\begin{bmatrix}
0 & C_{\mK}(I+\epsilon H)^{-1} \\
(I+\epsilon H)B_{\mK} &
(I+\epsilon H)A_{\mK} (I+\epsilon H)^{-1}
\end{bmatrix} \\
=\ &
\mK
+\epsilon\begin{bmatrix}
0 & -C_{\mK} H \\
HB_{\mK} & HA_{\mK} - A_{\mK} H
\end{bmatrix}
+o(\epsilon),
\end{aligned}
$$
implying that the tangent map of $\mathscr{T}_q(\cdot,\mK)$ at the identity is given by
$$
H\mapsto \begin{bmatrix}
0 & -C_{\mK} H \\
HB_{\mK} & HA_{\mK} - A_{\mK} H
\end{bmatrix}.
$$
Then since $\mathscr{T}_q(\cdot,\mK)$ is a diffeomorphism from $\mathrm{GL}_q$ to $\mathcal{O}_\mK$, the tangent map of $\mathscr{T}_q(\cdot,\mK)$ at the identity is an isomorphism from $\mathbb{R}^{q\times q}$ (the tangent space of $\mathrm{GL}_q$ at the identity) to the tangent space $\mathcal{TO}_{\mK}$. Thus
$$
\mathcal{TO}_{\mK}
=
\left\{
\left.\begin{bmatrix}
0 & -C_{\mK} H \\
H B_{\mK} & HA_{\mK} - A_{\mK}H
\end{bmatrix}\,
\right| H\in\mathbb{R}^{q\times q}
\right\}.
$$
Then the orthogonal complement $\mathcal{TO}_{\mK}^\perp$ is given by
$$
\begin{aligned}
\mathcal{TO}_{\mK}^\perp
= &
\left\{
\Delta\in\mathcal{V}_q
\left|\, \operatorname{tr}(U^\tr \Delta)=0
\textrm{ for all }U\in \mathcal{TO}_{\mK}\right.
\right\} \\
= &
\left\{
\Delta \!=\! \begin{bmatrix}
0 & \Delta_{B_\mK} \\
\Delta_{C_\mK} & \Delta_{A_\mK}
\end{bmatrix}
\!\in\!\mathcal{V}_q
\left|\, \operatorname{tr}\left(
\begin{bmatrix}
0 & -C_{\mK} H \\
HB_{\mK} & HA_{\mK} - A_{\mK} H
\end{bmatrix}^\tr\Delta\right)=0,
\forall H\in \mathbb{R}^{q\times q}
\right.
\right\} \\
= &
\left\{
\left.\Delta \!=\! \begin{bmatrix}
0 & \Delta_{B_\mK} \\
\Delta_{C_\mK} & \Delta_{A_\mK}
\end{bmatrix}
\!\in\!\mathcal{V}_q\,
\right|\operatorname{tr}
H^\tr
\!\left(
\Delta_{A_\mK}A_{\mK}^\tr
-A_{\mK}^\tr\Delta_{A_\mK}
+ \Delta_{B_\mK}B_{\mK}^\tr-C_{\mK}^\tr \Delta_{C_\mK}
\right)\!=0,
\forall H\in \mathbb{R}^{q\times q}
\right\} \\
= &
\left\{
\left.\Delta \!=\! \begin{bmatrix}
0 & \Delta_{B_\mK} \\
\Delta_{C_\mK} & \Delta_{A_\mK}
\end{bmatrix}
\!\in\!\mathcal{V}_q\,
\right|
\Delta_{A_\mK}A_{\mK}^\tr
-A_{\mK}^\tr\Delta_{A_\mK}
+ \Delta_{B_\mK}B_{\mK}^\tr-C_{\mK}^\tr \Delta_{C_\mK}=0
\right\}.
\end{aligned}
$$
This completes the proof.
\end{proof}

We conclude this subsection by noting that the LQG cost function $J_q(\mK)$ is not coercive in the sense that there might exist sequences of stabilizing controllers $\mK_j \in \mathcal{C}_q$ where $\lim_{j \rightarrow \infty} \mK_j = \hat{\mK} \in \partial \mathcal{C}_q$ such that
$$
    \lim_{j \rightarrow \infty} J_q(\mK_j) < \infty, \qquad
$$
and sequences of stabilizing controllers $\mK_j \in \mathcal{C}_q$ where $\lim_{j \rightarrow \infty} \|\mK_j\|_F = \infty$ such that
$$
    \lim_{j \rightarrow \infty} J_q(\mK_j) < \infty.
$$
The latter fact is easy to see from \cref{proposition:sim_trans_submanifold} since the orbit $\mathcal{O}_{\mK}$ can be unbounded and $J_q(\mK)$ is constant for any controller in the same orbit. The following example shows that the LQG cost might converge to a finite value even when the controller $\mK$ goes to the boundary of $\mathcal{C}_q$.

\begin{example}[Non-coercivity of the LQG cost] \label{example:non-coercivity}
    Consider the open-loop stable SISO system in \cref{example:SISO_connected}, and we fix $Q=1, R= 1, V = 1, W =1$ in the LQG formulation. The set of full-order stabilizing controllers $\mathcal{C}_1$ is shown in~\eqref{eq:region_example_connected}. We consider the following stabilizing controller
    $$
    \mK_\epsilon = \begin{bmatrix} 0 & \epsilon \\ -\epsilon & 0\end{bmatrix} \in \mathcal{C}_1, \qquad  \forall \epsilon \neq 0.
    $$
    It is not hard to see that
    $
    \lim_{\epsilon \rightarrow 0} \mK_\epsilon \in \partial\mathcal{C}_1.
    $
    By solving the Lyapunov equation~\eqref{eq:LyapunovX}, we get the unique solution as
    $$
        X_{\mK_{\epsilon}}= \begin{bmatrix} \displaystyle \frac{\epsilon^2+1}{2} &    \displaystyle \frac{\epsilon}{2} \\
        \displaystyle \frac{\epsilon}{2} & \displaystyle  \frac{\epsilon^2}{2}+1 \end{bmatrix},
    $$
    and the corresponding LQG cost as
    $$
        J(\mK_{\epsilon}) =  \frac{1+3\epsilon^2 + \epsilon^4}{2}.
    $$
    Therefore, we have
    $
        \lim_{\epsilon \rightarrow 0} J(\mK_{\epsilon}) = 1/2,
    $
    while
    $
        \lim_{\epsilon \rightarrow 0} \mK_\epsilon \in \partial\mathcal{C}_1.
    $\hfill\qed
\end{example}

%{\color{red} For this particular example, the boundary point corresponds to a stationary point in lower-order stabilizing controllers. the sublevel sets of $J(\mK)$ can be disconnected?}

\subsection{The Gradient and the Hessian of the LQG Cost}\label{subsection:LQG_grad_Hessian}

The following lemma gives a closed-loop form for the gradient of the LQG cost function $J_q$, and its proof is given in \cref{app:gradien_Jq}.
\begin{lemma}[Gradient of LQG cost $J_q$] \label{lemma:gradient_LQG_Jn}
    Fix $q \geq 1$ such that $\mathcal{C}_q \neq \varnothing$.
    For every $\mK = \begin{bmatrix} 0 & C_{\mK} \\B_{\mK} & A_{\mK} \end{bmatrix} \in \mathcal{C}_q$, the gradient of $J_q(\mK)$ is given by
    $$
    \nabla J_q(\mK)
    =\left[\!\begin{array}{cc}
    0 & \mfrac{\partial J_q(\mK)}{\partial C_{\mK}} \\[6pt]
    \mfrac{\partial J_q(\mK)}{\partial B_{\mK}} & \mfrac{\partial J_q(\mK)}{\partial A_{\mK}}
    \end{array}\!\right],
    $$
    with
    \begin{subequations} \label{eq:gradient_Jn}
        \begin{align}
         \frac{\partial J_q(\mK)}{\partial A_{\mK}} &= 2\left(Y_{12}^\tr X_{12} + Y_{22}X_{22}\right), \label{eq:partial_Ak}\\
        \frac{\partial J_q(\mK)}{\partial B_{\mK}} &= 2\left(Y_{22}B_{\mK}V + Y_{22}X_{12}^\tr C^\tr + Y_{12}^\tr X_{11} C^\tr\right), \label{eq:partial_Bk}\\
        \frac{\partial J_q(\mK)}{\partial C_{\mK}} &= 2\left(RC_{\mK}X_{22} + B^\tr Y_{11}X_{12} + B^\tr Y_{12} X_{22}\right), \label{eq:partial_Ck}
        \end{align}
    \end{subequations}
    where $X_{\mK}$ and $Y_{\mK}$, partitioned as
    \begin{equation} \label{eq:LyapunovXY_block}
        X_{\mK} = \begin{bmatrix}
        X_{11} & X_{12} \\ X_{12}^\tr & X_{22}
        \end{bmatrix},  \qquad Y_{\mK} = \begin{bmatrix}
        Y_{11} & Y_{12} \\ Y_{12}^\tr & Y_{22}
        \end{bmatrix}
    \end{equation}
    are the unique positive semidefinite
    solutions to~\eqref{eq:LyapunovX} and~\eqref{eq:LyapunovY}, respectively.
\end{lemma}

We next consider the Hessian of $J_q(\mK)$. %the second-order behavior of $J_q(\mK)$, which is characterized by its Hessian.
Let $\mK$ be any controller in $\mathcal{C}_q$, and we use $\operatorname{Hess}_{\,\mK}:\mathcal{V}_q\times\mathcal{V}_q\rightarrow\mathbb{R}$ to denote the bilinear form of the Hessian of $J_q$ at $\mK$, so that for any $\Delta\in\mathcal{V}_q$, we have
$$
J_n(\mK+\Delta)
=J_n(\mK)+
\operatorname{tr}
\left(
\nabla J_q(\mK)^\tr \Delta\right)
+
\frac{1}{2}
\operatorname{Hess}_{\,\mK}(\Delta,\Delta)
+o(\|\Delta\|_F^2)
$$
as $\|\Delta\|_F\rightarrow 0$. Obviously, $\operatorname{Hess}_{\,\mK}$ is symmetric in the sense that $\operatorname{Hess}_{\,\mK}(x,y)=\operatorname{Hess}_{\,\mK}(y,x)$ for all $x,y\in\mathcal{V}_n$.
The following lemma shows how to compute $\operatorname{Hess}_{\,\mK}(\Delta,\Delta)$ for any $\Delta\in\mathcal{V}_q$ by solving three Lyapunov equations, whose proof is given in \cref{app:gradien_Jq}.
\begin{lemma}\label{lemma:Jn_Hessian}
Fix $q\geq 1$ such that $\mathcal{C}_q\neq\varnothing$.
Let $\mK=\begin{bmatrix}
0 & C_{\mK} \\
B_{\mK} & A_{\mK}
\end{bmatrix} \in \mathcal{C}_q$. Then for any $\Delta=\begin{bmatrix}
0 & \Delta_{C_\mK} \\
\Delta_{B_\mK} & \Delta_{A_\mK}
\end{bmatrix}\in\mathcal{V}_q$, we have
$$
\begin{aligned}
\operatorname{Hess}_{\,\mK}(\Delta,\Delta)
=\ &
2\operatorname{tr}
\Bigg(
2
\begin{bmatrix}
0 & B\Delta_{C_\mK} \\
\Delta_{B_\mK} C & \Delta_{A_\mK}
\end{bmatrix}
X'_{\mK^\ast,\Delta}\cdot Y_{\mK^\ast}
+2\begin{bmatrix}
0 & 0 \\ 0 & {C_{\mK}^\ast}^\tr R \Delta_{C_\mK}
\end{bmatrix}\cdot X'_{\mK^\ast,\Delta}
\\
& \qquad\qquad
+\begin{bmatrix}
0 & 0 \\
0 & \Delta_{B_\mK}V\Delta_{B_\mK}^\tr
\end{bmatrix} Y_{\mK^*}
+
\begin{bmatrix}
0 & 0 \\ 0 & \Delta_{C_\mK}^\tr R \Delta_{C_\mK}
\end{bmatrix}X_{\mK^\ast}
\Bigg),
\end{aligned}
$$
where $X_{\mK^\ast}$ and $Y_{\mK^\ast}$ are the solutions to the Lyapunov equations~\eqref{eq:LyapunovX} and~\eqref{eq:LyapunovY}, and $X'_{\mK^\ast,\Delta}\in\mathbb{R}^{(n+q)\times(n+q)}$ is the solution to the following Lyapunov equation
\begin{equation} \label{eq:Lyapunov_hessian}
    \begin{bmatrix}
A & BC_{\mK}^\ast \\
B_{\mK}^\ast C & A_{\mK}^\ast
\end{bmatrix} X'_{\mK^\ast,\Delta}
+X'_{\mK^\ast,\Delta}
\begin{bmatrix}
A & BC_{\mK}^\ast \\
B_{\mK}^\ast C & A_{\mK}^\ast
\end{bmatrix}^\tr
+M_1(X_{\mK^*},\Delta) = 0,
\end{equation}
with
\begin{align*}
M_1(X_{\mK^*},\Delta)
\coloneqq
\begin{bmatrix}
0 & B\Delta_{C_\mK} \\
\Delta_{B_\mK} C & \Delta_{A_\mK}
\end{bmatrix} X_{\mK^\ast}
+ X_{\mK^\ast}\begin{bmatrix}
0 & B\Delta_{C_\mK} \\
\Delta_{B_\mK} C & \Delta_{A_\mK}
\end{bmatrix}^\tr
\!+
\begin{bmatrix}
0 & 0 \\
0 & B_{\mK}^\ast V\Delta_{B_\mK}^\tr
\!+\!
\Delta_{B_\mK} V {B_{\mK}^\ast}^\tr
\end{bmatrix}.
\end{align*}
\end{lemma}
From \cref{lemma:Jn_Hessian}, one can further compute $\operatorname{Hess}_{\,\mK}(\Delta_1,\Delta_2)$ for any $\Delta_1,\Delta_2\in\mathcal{V}_n$ by
$$
\begin{aligned}
\operatorname{Hess}_{\,\mK}(\Delta_1,\Delta_2)
=\ &
\frac{1}{4}
\left(\operatorname{Hess}_{\,\mK}(\Delta_1+\Delta_2,\Delta_1+\Delta_2)
-\operatorname{Hess}_{\,\mK}(\Delta_1-\Delta_2,\Delta_1-\Delta_2)\right) \\
=\ &
\frac{1}{2}
\left(\operatorname{Hess}_{\,\mK}(\Delta_1+\Delta_2,\Delta_1+\Delta_2)
-\operatorname{Hess}_{\,\mK}(\Delta_1,\Delta_1)
-\operatorname{Hess}_{\,\mK}(\Delta_2,\Delta_2)\right).
\end{aligned}
$$

% We are now interested in the stationary points of $J(\mK)$.

%\subsection{Properties of the gradient for the LQG cost function $J_q(\mK)$}
\subsection{Non-minimal Stationary Points}
\label{subsection:nonminimal_stationary}

In this part, we show that the LQG cost $J_n(\mK)$ over the full-order stabilizing controller $\mathcal{C}_n$ may have many non-minimal stationary points that might be strict saddle points.

We first investigate the gradient of $J_q(\mK)$ under similarity transformation. Given any $T \in \mathrm{GL}_q$, recall the definition of the linear map of similarity transformation $\mathscr{T}_q
\left(T,\mK
\right)$ in~\eqref{eq:def_sim_transform}. The following lemma gives an explicit relationship among the gradients of $J_q(\cdot)$ at $\mK$ and $\mathscr{T}_q
\left(T,\mK
\right)$.

\begin{lemma} \label{lemma:gradient_simi_tran_linear}
Let $\mK=\begin{bmatrix}
0 & C_{\mK} \\ B_{\mK} & A_{\mK}
\end{bmatrix}\in \mathcal{C}_q$ be arbitrary.
For any $T\in\mathrm{GL}_q$, we have

\begin{equation} \label{eq:Gradient_sim_transformation}
\left.\nabla J_q\right|_{\mathscr{T}_q \left(T,\mK \right)}
=\begin{bmatrix}
I_m & 0 \\
0 & T^{-\tr}
\end{bmatrix}
\cdot \left.\nabla J_q\right|_{\mK}
\cdot \begin{bmatrix}
I_p & 0 \\
0 & T^{\tr}
\end{bmatrix}.
\end{equation}
\begin{comment}
\begin{equation} \label{eq:Gradient_sim_transformation}
\begin{aligned}
\left.\frac{\partial J_q}{\partial A_{\mK}}\right|_{\mathscr{T}_q \left(T,\mK \right)}
=\ &T^{-\top}\cdot \left.\frac{\partial J_q}{\partial A_{\mK}}\right|_{\mK} \cdot T^\tr,
\\
\left.\frac{\partial J_q}{\partial B_{\mK}}\right|_{\mathscr{T}_q \left(T,\mK \right)}
=\ &T^{-\top}\cdot \left.\frac{\partial J_q}{\partial B_{\mK}}\right|_{\mK},
\\
\left.\frac{\partial J_q}{\partial C_{\mK}}\right|_{\mathscr{T}_q \left(T,\mK \right)}
=\ &\left.\frac{\partial J_q}{\partial C_{\mK}}\right|_{\mK}\cdot T^{\top}.
\end{aligned}
\end{equation}
\end{comment}
\end{lemma}

\begin{proof}
Let $\Delta\in\mathcal{V}_q$ be arbitrary. We have
$$
\begin{aligned}
& J_q(\mathscr{T}_q\left(T,\mK + \Delta\right))
-J_q(\mathscr{T}_q \left(T,\mK \right)) \\
=\ &
J_q(\mathscr{T}_q\left(T,\mK)+ \mathscr{T}_q(T,\Delta\right))
-J_q(\mathscr{T}_q \left(T,\mK \right)) \\
=\ &
\operatorname{tr}
\left[\left(\left.\nabla J_q\right|_{\mathscr{T}_q \left(T,\mK \right)}\right)^{\!\tr}
\cdot \mathscr{T}_q \left(T,\Delta \right)
\right]
+o(\|\Delta\|)\\
=\ &
\operatorname{tr}
\left[\left(\left.\nabla J_q\right|_{\mathscr{T}_q \left(T,\mK \right)}\right)^{\!\tr}
\cdot \begin{bmatrix}
I_m & 0 \\
0 & T
\end{bmatrix}\Delta
\begin{bmatrix}
I_p & 0 \\
0 & T^{-1}
\end{bmatrix}
\right]
+o(\|\Delta\|) \\
=\ &
\operatorname{tr}
\left[
\left(\begin{bmatrix}
I_m & 0 \\
0 & T
\end{bmatrix}^\tr\cdot \left.\nabla J_q\right|_{\mathscr{T}_q \left(T,\mK \right)}
\cdot \begin{bmatrix}
I_p & 0 \\
0 & T^{-1}
\end{bmatrix}^\tr\right)^{\!\!\tr}
\Delta
\right]
+o(\|\Delta\|).
\end{aligned}
$$
On the other hand, \cref{lemma:Jn_invariance} shows that the LQG cost stays the same when applying similarity transformation. Thus, we have
$$
\begin{aligned}
J_q(\mathscr{T}_T \left(\mK + \Delta\right))
-J_q(\mathscr{T}_T \left(\mK \right))
=\ &
J_q(K+\Delta) - J_q(\mK) \\
=\
&
\operatorname{tr}
\left[\left(\left.\nabla J_q\right|_{\mK}\right)^\tr
\cdot \Delta
\right]
+o(\|\Delta\|).
\end{aligned}
$$
By comparing the two equations, we get
$$
\left.\nabla J_q\right|_{\mK}
=\begin{bmatrix}
I_m & 0 \\
0 & T
\end{bmatrix}^\tr\cdot \left.\nabla J_q\right|_{\mathscr{T}_q \left(T,\mK \right)}
\cdot \begin{bmatrix}
I_p & 0 \\
0 & T^{-1}
\end{bmatrix}^\tr,
$$
which then leads to the relationship~\eqref{eq:Gradient_sim_transformation}.
\end{proof}

As expected, a direct consequence of \cref{lemma:gradient_simi_tran_linear} is that, if $\mK\in\mathcal{C}_q$ is a stationary point of $J_q$, then any controller in the orbit $\mathcal{O}_{\mK}$ is also a stationary point of $J_q$. In addition, \cref{lemma:gradient_simi_tran_linear} allows us to establish an interesting result that any stationary point of $J_{q}$ can be transferred to stationary points of $J_{q+q'}$ for any $q'>0$ with the same objective value.

\begin{comment}
\begin{corollary} \label{corollary:gradient_invirant}
Let $\mK\in \mathcal{C}_q$, and suppose there exists an orthogonal $Q\in\mathbb{R}^{q\times q}$ such that $\mathscr{T}_Q \left(\mK\right) = \mK$. Then for any $\eta>0$ such that $\mK-\eta\nabla J_q(\mK)\in \mathcal{C}_q$, we have $\mathscr{T}_Q(K-\eta \nabla J_q(\mK)) = \mK-\eta \nabla J_q(\mK)$.
\end{corollary}
\end{comment}

\begin{theorem} \label{theorem:non_globally_optimal_stationary_point}
Let $q\geq 1$ be arbitrary. Suppose there exists $\mK^\star=\begin{bmatrix}
0 & C_{\mK}^\star \\ B_{\mK}^\star & A_{\mK}^\star
\end{bmatrix}
\in \mathcal{C}_{q}$ such that $\nabla J_{q}(\mK^\star)=0$.
Then for any $q'\geq 1$ and any stable $\Lambda\in\mathbb{R}^{q'\times q'}$, the following controller
\begin{equation} \label{eq:gradient_nonglobally_K}
        \tilde{\mK}^\star
    =\left[\begin{array}{c:cc}
    0 & C_{\mK}^\star &  0 \\[2pt]
    \hdashline
    B_{\mK}^\star & A_{\mK}^\star & 0 \\[-2pt]
    0 & 0 & \Lambda
    \end{array}\right] \in \mathcal{C}_{q+q'}
\end{equation}
is a stationary point of $J_{q+q'}$ over $\mathcal{C}_{q+q'}$ satisfying $J_{q+q'}\big(\tilde{\mK}^\star\big)
=J_{q}(\tilde{\mK})$.
\end{theorem}
\begin{proof}
Since $\mK^\star \in \mathcal{C}_{q}$, we have  $\tilde{\mK}^\star\in\mathcal{C}_{q+q'}$ by construction. It is straightforward to verify that
$$
\mathscr{T}_{q+q'}\big(T,\tilde{\mK}^\star\big)=\tilde{\mK}^\star \quad \text{with} \quad
T=\begin{bmatrix}
I_{q} & 0 \\ 0 & -I_{q'}
\end{bmatrix}.
$$
Therefore, by \cref{lemma:gradient_simi_tran_linear}, we have
$$
\left.\nabla J_{q+q'}\right|_{\tilde{\mK}^\star}
=
\left.\nabla J_{q+q'}\right|_{\mathscr{T}_{q+q'}\big(T,\tilde{\mK}^\star\big)}
=\begin{bmatrix}
I_{m+q} & 0 \\ 0 & -I_{q'}
\end{bmatrix}
\cdot \left.\nabla J_{q+q'}\right|_{\tilde{\mK}^\star}
\cdot
\begin{bmatrix}
I_{p+q} & 0 \\ 0 & -I_{q'}
\end{bmatrix},
$$
which implies that, excluding the the bottom right $q'\times q'$ block, the last $q'$ rows and the last $q'$ columns of $\left.\nabla J_{q+q'}\right|_{\tilde{\mK}^\star}$ are zero.
On the other hand, it can be checked that
$$
J_{q+q'}\!\left(\begin{bmatrix}
\mK & 0 \\ 0 & \Lambda
\end{bmatrix}\right) = J_{q}(\mK),
\qquad\forall\,\mK\in \mathcal{C}_{q},
$$
and since $\nabla J_{q}(\mK^\star)=0$,
we can see that the upper left $(m+q)\times (p+q)$ block of $\left.\nabla J_{q+q'}\right|_{\tilde{K}^\star}$ is equal to zero.
Then, from \cref{lemma:LQG_cost_formulation1}, it is not difficult to verify that the value $J_q(\tilde{\mK}^\ast)$ is independent of the $q'\times q'$ stable matrix $\Lambda$, and thus the bottom right $q'\times q'$ block of $\left.\nabla J_{q+q'}\right|_{\tilde{\mK}^\star}$ is zero.

We can now see that $\left.\nabla J_{q+q'}\right|_{\tilde{\mK}^\star}=0$. This completes the proof.
\end{proof}

\cref{theorem:non_globally_optimal_stationary_point} indicates that from any stationary point of $J_{q}$ over lower-order stabilizing controllers in $\mathcal{C}_{q}$, we can construct a family of stationary points of $J_{q+q'}$ over higher-order stabilizing controllers in $\mathcal{C}_{q+q'}$. Moreover, the stationary points constructed by \eqref{eq:gradient_nonglobally_K} are neither controllable nor observable. This indicates that, if the globally optimal controller of $J_n$ is controllable and observable, and if the problem %if finding an optimal reduced-order controller
$$
\min_{\mK\in\mathcal{C}_q} J_q(\mK)
$$
has a solution for some $q<n$, then there will exist many \emph{strictly suboptimal stationary points} of $J_n$ over $\mathcal{C}_n$.

The following theorem explicitly constructs a family of stationary points for $J_n$ with an open-loop stable plant, and also provides a criterion for checking whether the corresponding Hessian is indefinite or vanishing.
\begin{theorem}\label{theorem:zero_stationary_hessian}
Suppose the plant \eqref{eq:Dynamic} is open-loop stable. Let $\Lambda\in\mathbb{R}^{n\times n}$ be stable, and let
$$
\mK^\star = \begin{bmatrix}
0 & 0 \\
0 & \Lambda
\end{bmatrix}.
$$
Then $\mK^\star$ is a stationary point of $J_n(\mK)$ over $\mK\in\mathcal{C}_n$, and the corresponding Hessian $\operatorname{Hess}_{\,\mK^\star}$ is either indefinite or zero.

Furthermore, suppose $\Lambda$ is diagonalizable, and let $\operatorname{eig}(-\Lambda)$ denote the set of (distinct) eigenvalues of $-\Lambda$. Let $X_{\mathrm{op}}$ and $Y_{\mathrm{op}}$ be the solutions to the following Lyapunov equations
\begin{equation} \label{eq:Lyapunov_hessian_vanishing}
A X_{\mathrm{op}} + X_{\mathrm{op}} A^\tr + W=0,
\quad
A^\tr Y_{\mathrm{op}} + Y_{\mathrm{op}} A + Q=0,
\end{equation}
and let
\begin{equation} \label{eq:hessian_vanishing_condition}
\mathcal{Z}
=\left\{
s\in\mathbb{C}\mid
CX_{\mathrm{op}}\big(sI-A^\tr\big)^{-1}Y_{\mathrm{op}} B
=0
\right\}.
\end{equation}
Then, the Hessian of $J_n$ at $\mK^\star$ is indefinite if and only if $\operatorname{eig}(-\Lambda) \nsubseteq \mathcal{Z}$; %$\operatorname{eig}(-\Lambda)\backslash\mathcal{Z} \neq\varnothing$.
the Hessian of $J_n$ at $\mK^\star$ is zero if and only if %$\operatorname{eig}(-\Lambda) \backslash\mathcal{Z}=\varnothing$.
$\operatorname{eig}(-\Lambda) \subseteq \mathcal{Z}$.
\end{theorem}

The fact that $\mK^\star=\begin{bmatrix} 0 & 0 \\ 0 & \Lambda\end{bmatrix}$ is a stationary point can be proved similarly as in \cref{theorem:non_globally_optimal_stationary_point}. Regarding the properties of the Hessian, we exploit its bilinear property and use \cref{lemma:Jn_Hessian} for direct calculation. In particular, the Lyapunov equations~\eqref{eq:LyapunovX} and~\eqref{eq:LyapunovY} are reduced to~\eqref{eq:Lyapunov_hessian_vanishing}, and the transfer function in~\eqref{eq:hessian_vanishing_condition} is obtained when we solve the third Lyapunov equation~\eqref{eq:Lyapunov_hessian}.  The detailed proof is provided in \cref{app:proof_zero_stationary_hessian}.

\cref{theorem:zero_stationary_hessian} constructs a family of non-minimal strict saddle points or stationary points with vanishing Hessian for LQG with open-loop stable systems. We now present two explicit examples illustrating the Hessian of $J_q(\mK)$ at non-minimal stationary points.
\begin{example}[Strict saddle point] \label{example:strict_saddle}
    Consider the open-loop stable SISO system in \cref{example:SISO_connected}. We choose $Q = R = 1, W = V = 1$ for the LQG formulation. By  \cref{theorem:zero_stationary_hessian}, given any negative $a < 0$, the following controller
    $$
        \mK^\star = \begin{bmatrix}
        0 & 0 \\
        0 & a
        \end{bmatrix} \in \mathbb{R}^{2 \times 2}
    $$
     is a stationary point of $J_1(\mK)$ over the set of full-order stabilizing controller $\mathcal{C}_1$. Furthermore, it can be checked that
     $$
     CX_{\mathrm{op}}\big(sI-A^\tr\big)^{-1}Y_{\mathrm{op}}B
     =\frac{1}{4(s+1)}.
     $$
     Therefore the Hessian of $J_1$ at $\mK^\star$ is indefinite by \cref{theorem:zero_stationary_hessian}, indicating that $\mK^\ast$ is a strict saddle point \cite{lee2019first}. Indeed, by using \eqref{eq:LQG_cost_formulation1}, we can directly compute the LQG cost and obtain
     $$
     J_1\!\left(
     \begin{bmatrix}
     0 & C_{\mK} \\
     B_{\mK} & A_{\mK}
     \end{bmatrix}
     \right)
     =
     \frac{A_{\mK}^2 - A_{\mK} (1 + B_{\mK}^2 C_{\mK}^2) -
 B_{\mK} C_{\mK} (1 - 3 B_{\mK} C_{\mK} + B_{\mK}^2 C_{\mK}^2)}{2 (-1 + A_{\mK}) (A_{\mK} + B_{\mK} C_{\mK})}.
     $$
     The Hessian at $\mK^\star$ can then be represented as
     $$
        \left. \begin{bmatrix}
        \frac{\partial J^2(\mK)}{\partial A^2_{\mK}} & \frac{\partial J^2(\mK)}{\partial A_{\mK} \partial B_{\mK}}  & \frac{\partial J^2(\mK)}{\partial A_{\mK} \partial C_{\mK}} \\
                \frac{\partial J^2(\mK)}{\partial B_{\mK}A_{\mK}} & \frac{\partial J^2(\mK)}{\partial B^2_{\mK}}  & \frac{\partial J^2(\mK)}{\partial B_{\mK} \partial C_{\mK}}  \\
                \frac{\partial J^2(\mK)}{\partial C_{\mK}A_{\mK}} & \frac{\partial J^2(\mK)}{\partial C_{\mK}B_{\mK}}  & \frac{\partial J^2(\mK)}{\partial  \partial C^2_{\mK}}
        \end{bmatrix}\right|_{\mK^\star = \begin{bmatrix}
        0 & 0 \\
        0 & a
        \end{bmatrix}} = \frac{1}{2(1 - a)} \begin{bmatrix}
         0 & 0 & 0\\
         0 & 0 & 1\\
         0 & 1 & 0
        \end{bmatrix},
     $$
     which has eigenvalues $0$ and $\pm\frac{1}{2(1-a)}$.
     \hfill\qed
\end{example}

\begin{example}[Stationary point with vanishing Hessian]
\label{example:saddle_point_vanishing_Hessian}
Consider the following SISO system:
$$
A=\begin{bmatrix}
-1 & 0 \\ 1 & -2
\end{bmatrix},
\quad
B=\begin{bmatrix}
-1 \\ 1
\end{bmatrix},
\quad
C=\begin{bmatrix}
-2 & 11
\end{bmatrix},
\quad
W=\begin{bmatrix}
1 & 0 \\ 0 & 1
\end{bmatrix},
\quad
V = 1,
$$
and let
$$
Q=\begin{bmatrix}
1 & 0 \\ 0 & 1
\end{bmatrix},
\quad
R = 1.
$$
It can be checked that
$$
CX_{\mathrm{op}}\big(sI-A^\tr\big)^{-1}Y_{\mathrm{op}} B
=\frac{5(s-1)}{36(s+1)(s+2)}.
$$
By \cref{theorem:zero_stationary_hessian}, the point
$$
\mK^\star=\begin{bmatrix}
0 & 0 & 0 \\
0 & -1 & 0 \\
0 & 0 & -1
\end{bmatrix}
$$
is a stationary point of $J_n$ with a vanishing Hessian. %
In \cref{fig:saddle_point_vanishing_Hessian}, we plot the graph of the function $t\mapsto J_n(\mK^\star + t\Delta)$ for
$$
\Delta = \begin{bmatrix}
0 & 2 & 1/2 \\
-1 & 1 & 3 \\
3 & 0 & 0
\end{bmatrix}.
$$
\cref{fig:saddle_point_vanishing_Hessian} suggests that $\mK^\star$ is a saddle point of $J_n$ with a vanishing Hessian but non-vanishing third-order partial derivatives. \hfill\qed
\begin{figure}[t]
    \centering
    \includegraphics[width=.5\textwidth]{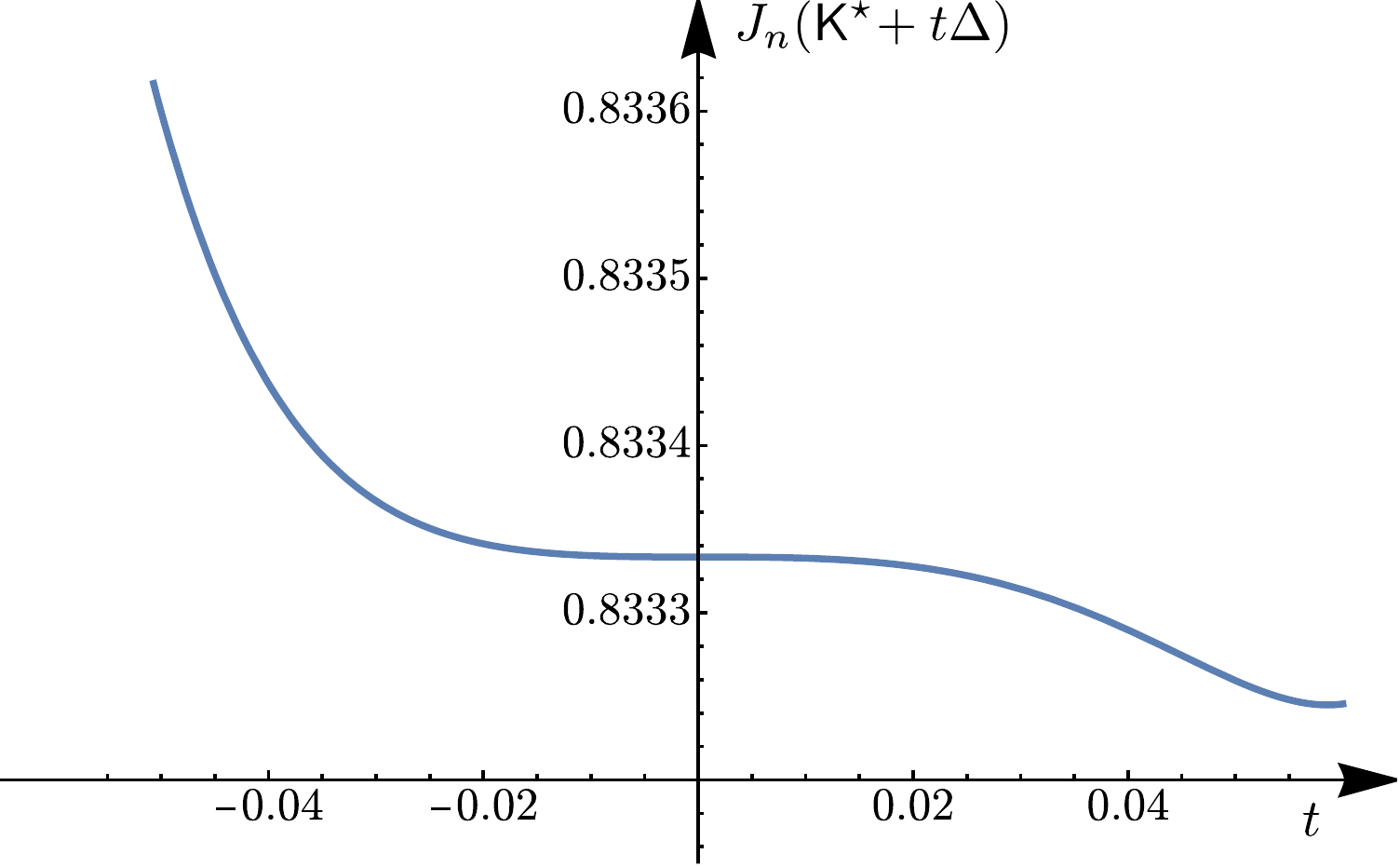}
    \caption{The function $t\mapsto J_n(\mK^\star + t\Delta)$ for \cref{example:saddle_point_vanishing_Hessian}.}
    \label{fig:saddle_point_vanishing_Hessian}
\end{figure}
\end{example}

\begin{remark}
Some recent studies have shown that many gradient-based algorithms can automatically escape strict saddle points under mild conditions~\cite{lee2019first, jin2017escape}. However,~\Cref{example:saddle_point_vanishing_Hessian} shows that the LQG cost function $J(\mK)$ may have non-strict saddle points, and further analysis is required to examine whether gradient-based methods can also escape such stationary points.
\end{remark}

\subsection{Minimal Stationary Points Are Globally Optimal} \label{subsection:gradient_LQG}

As discussed in Theorems~\ref{theorem:non_globally_optimal_stationary_point} and ~\ref{theorem:zero_stationary_hessian}, there may exist many \emph{non-minimal} stationary points for $J_n$ that are not globally optimal. In this section, we aim to show that all \emph{minimal} stationary points are globally optimal to the LQG problem~\eqref{eq:LQG}.

%Under \cref{assumption:stabilizability}, if the LQG problem~\eqref{eq:LQG_reformulation_KX} has \emph{minimal} stationary points, \emph{i.e.}, $\nabla J_n(\mK) = 0$ and $\mK$ is minimal, then they are globally optimal. These minimal stationary points only have a difference up to a similarity transformation.
%

Recall that $\mK = \begin{bmatrix}
0 & C_{\mK} \\
B_{\mK} & A_{\mK}
\end{bmatrix} \in \mathcal{C}_q$ is minimal if it represents a controllable and observable controller. The gradient computation in \cref{lemma:gradient_LQG_Jn} works for both minimal and non-minimal stabilizing controllers in $\mathcal{C}_q$.
For a minimal stabilizing controller $\mK$, we further have the following result (see \cref{app:LyapunovXY_pd} for a proof).
\begin{lemma} \label{lemma:LyapunovXY_pd}
   Fix $q \in \mathbb{N}$ such that $\mathcal{C}_q \neq \varnothing$, and let $\mK \in \mathcal{C}_q$ be minimal. Under \cref{assumption:stabilizability}, the solutions $X_{\mK}$ and $Y_{\mK}$ to~\eqref{eq:LyapunovX} and~\eqref{eq:LyapunovY} are positive definite.
\end{lemma}

By letting the gradient~\eqref{eq:gradient_Jn} equal to zero, \emph{i.e.},
\begin{equation} \label{eq:gradient_zero}
    \frac{\partial J_n(\mK)}{\partial A_{\mK}}=0, \quad \frac{\partial J_n(\mK)}{\partial B_{\mK}} = 0, \quad \frac{\partial J_n(\mK)}{\partial C_{\mK}} = 0,
\end{equation}
we can characterize the stationary points of the LQG problem~\eqref{eq:LQG_reformulation_KX}. In particular, we have closed-loop form expressions for full-order minimal stationary points $\mK \in \mathcal{C}_n$, which turn out to be globally optimal. This result is formally summarized below.

\begin{theorem} \label{theo:stationary_points_globally_optimal}
 %Let $Q \succ 0, R \succ 0, W \succ 0, V \succ 0$ in~\eqref{eq:LQG_reformulation_KX}.
 Under \cref{assumption:stabilizability}, all minimal stationary points $\mK\in\mathcal{C}_n$ to the LQG problem~\eqref{eq:LQG_reformulation_KX} are globally optimal, and they are in the form of
 \begin{equation} \label{eq:stationary_points}
        A_{\mK} = T(A - BK - LC)T^{-1}, \qquad    B_{\mK} = - TL, \qquad   C_{\mK} =  KT^{-1},
 \end{equation}
where $T \in \mathbb{R}^{n \times n}$ is an invertible matrix, and
\begin{equation} \label{eq:LQG_KL}
    K = R^{-1}B^\tr S, \qquad L = PC^\tr V^{-1},
\end{equation}
with $P$ and $S$ being the unique positive definite solutions to the Riccati equations~\eqref{eq:Riccati_P} and~\eqref{eq:Riccati_S}.

\end{theorem}

\cref{theo:stationary_points_globally_optimal} can be viewed as a special case in~\cite[Theorem 20.6]{zhou1996robust},~\cite[Section II]{hyland1984optimal} that presents first-order necessary conditions for optimal reduced-order controllers $\mK \in \mathcal{C}_q$. Following the analysis in~\cite[Chapter 20]{zhou1996robust}, we present an adapted proof for~\Cref{theo:stationary_points_globally_optimal} here.

\begin{proof}
Consider a stationary point $\mK = \begin{bmatrix}
0 & C_{\mK} \\
B_{\mK} & A_{\mK}
\end{bmatrix} \in \mathcal{C}_n$ such that the gradient~\eqref{eq:gradient_Jn} vanishes. If the controller $\mK$ is minimal, we know by \cref{lemma:LyapunovXY_pd} that the solutions $X_{\mK}$ and $Y_{\mK}$ to~\eqref{eq:LyapunovX} and~\eqref{eq:LyapunovY} are unique and positive definite.

Upon partitioning $X_{\mK}$ and $Y_\mK$ in~\eqref{eq:LyapunovXY_block}, by the Schur complement, the following matrices are well-defined and positive definite
\begin{equation} \label{eq:riccatiPS}
    \begin{aligned}
        P : = X_{11} - X_{12}X_{22}^{-1}X_{12}^\tr \succ 0, \qquad
        S : = Y_{11} - Y_{12}Y_{22}^{-1} Y_{12}^\tr \succ 0.
    \end{aligned}
\end{equation}
We further define
$
    T := Y_{22}^{-1} Y_{12}^\tr.
$
 By~\eqref{eq:partial_Ak}, we know that matrix $T$ is invertible, and
$$
T^{-1} = - X_{12}X_{22}^{-1}.
$$
Now, letting $ \frac{\partial J_n(\mK)}{\partial B_{\mK}} = 0$, from~\eqref{eq:partial_Bk}, we have \begin{equation} \label{eq:stationary_Bk}
    \begin{aligned}
         B_{\mK} &= -(X_{12}^\tr + Y_{22}^{-1} Y_{12}^\tr X_{11})C^\tr V^{-1}, \\
         &= -(X_{12}^\tr + T X_{11})C^\tr V^{-1} \\
         &= - T(X_{11} - X_{12}X_{22}^{-1}X_{12}^\tr)C^\tr V^{-1}, \\
         &= - TPC^\tr V^{-1}.
    \end{aligned}
\end{equation}
Similarly, from~\eqref{eq:partial_Ck}, we have
\begin{equation} \label{eq:stationary_Ck}
    \begin{aligned}
        C_{\mK} = -R^{-1}B^\tr (Y_{11}X_{12}X_{22}^{-1} + Y_{12}) = R^{-1}B^\tr S T^{-1}.
    \end{aligned}
\end{equation}

Furthermore, since $X_\mK$ is the solution to the Lyapunov equation~\eqref{eq:LyapunovX}, by plugging in the blocks of $X_{\mK}$ we get
\begin{subequations}
\begin{align}
0=\ & AX_{11}+X_{11}A+BC_{\mK}X_{12}^\tr + X_{12}C_{\mK}^\tr B_{\mK}^\tr + W,
\label{eq:minimal_stationary_Lyapunov_1}
\\
0=\ & AX_{12}+BC_{\mK}X_{22}+X_{11}C^\tr B_{\mK}^\tr + X_{12}A_{\mK}^\tr,
\label{eq:minimal_stationary_Lyapunov_2}
\\
0=\ &
A_{\mK}X_{22}+X_{22}A_{\mK}^\tr+B_{\mK}CX_{12}+X_{12}^\tr C^\tr B_{\mK}^\tr  + B_{\mK}VB_{\mK}^\tr.
\label{eq:minimal_stationary_Lyapunov_3}
\end{align}
\end{subequations}
%
%
%at the stationary point $\mK$
Now, we have~\eqref{eq:minimal_stationary_Lyapunov_3} + $T\times$\eqref{eq:minimal_stationary_Lyapunov_2} leads to
$$
\begin{aligned}
    A_{\mK}X_{22}+X_{22}A_{\mK}^\tr+B_{\mK}CX_{12}+X_{12}^\tr C^\tr B_{\mK}^\tr  &+ B_{\mK}VB_{\mK}^\tr + \\
    & T(AX_{12} +  BC_{\mK} X_{22} + X_{11}C^\tr B_{\mK}^\tr +X_{12}A_{\mK}^\tr) =0,
\end{aligned}
$$
which is the same as
$$
\begin{aligned}
    A_{\mK}X_{22} + X_{22}A_{\mK}^\tr- &TPC^\tr V^{-1}CX_{12}  - X_{12}^\tr C^\tr V^{-1} CPT   + TPC^\tr V^{-1}CPT+ \\
    &T(AX_{12} + BR^{-1}B^\tr ST^{-1} X_{22} - X_{11}C^\tr V^{-1}CPT  +X_{12}A_{\mK}^\tr) =0.
\end{aligned}
$$
By the definition of $T$, we have $TX_{12} = - X_{22}$. Then, the equation above becomes
$$
\begin{aligned}
    A_{\mK}X_{22} - TPC^\tr V^{-1}CX_{12}- &X_{12}^\tr C^\tr V^{-1} CPT  + TPC^\tr V^{-1}CPT+ \\
    &T(AX_{12} + BR^{-1}B^\tr ST^{-1} X_{22} - X_{11}C^\tr V^{-1}CPT) =0,
\end{aligned}
$$
leading to
\begin{equation}\label{eq:stationary_Ak}
    \begin{aligned}
    A_{\mK} &=  TPC^\tr V^{-1}CX_{12}X_{22}^{-1} + X_{12}^\tr C^\tr V^{-1} CPTX_{22}^{-1}   - TPC^\tr V^{-1}CPTX_{22}^{-1} \\
    &\qquad \qquad \qquad \qquad \qquad \qquad - T(AX_{12} + BR^{-1}B^\tr ST^{-1} X_{22} - X_{11}C^\tr V^{-1}CPT)X_{22}^{-1} \\
    %&= T(A - PC^\tr V^{-1}C - BR^{-1}B^\tr S)T^{-1} + (X_{12}^\tr - TP + TX_{11})C^\tr V^{-1} CPTX_{22}^{-1}\\
    &= T(A - PC^\tr V^{-1}C - BR^{-1}B^\tr S)T^{-1}.
    \end{aligned}
\end{equation}

From~\eqref{eq:stationary_Bk},~\eqref{eq:stationary_Ck} and~\eqref{eq:stationary_Ak}, upon defining $K$ and $L$ in~\eqref{eq:LQG_KL}, it is easy to see that
% upon defining $
% K := R^{-1}B^\tr S, L := PC^\tr V^{-1},
% $
the stationary points
% are
% $$
%         A_{\mK} = T(A - BK - LC)T^{-1}, \qquad    B_{\mK} = - TL, \qquad   C_{\mK} =  KT^{-1}.
% $$
% which is
are in the form of~\eqref{eq:stationary_points}. It remains to prove that $P$ and $S$ defined in~\eqref{eq:riccatiPS} are the unique positive definite solutions to
the Riccati equations~\eqref{eq:Riccati_P} and~\eqref{eq:Riccati_S}.

We multiply \eqref{eq:minimal_stationary_Lyapunov_3} by $T^{-1}$ on the left and by $T^{-\tr}$ on the right, and by noting that $B_{\mK}=-TPC^\tr V^{-1}$ and $T^{-1}=-X_{12}X_{22}^{-1}$, we get
$$
\begin{aligned}
0 =\ & X_{12}X_{22}^{-1}A_{\mK}X_{12}^\tr + X_{12} A_{\mK}^\tr X_{22}^{-1}X_{12}^\tr\\
& + PC^\tr V^{-1} CX_{12}X_{22}^{-1}X_{12}^\tr +
X_{12}X_{22}^{-1}X_{12}^\tr C^\tr V^{-1} C P  + PC^\tr V^{-1}CP.
\end{aligned}
$$
Since $P=X_{11}-X_{12}X_{22}^{-1}X_{12}^\tr$, we further get
\begin{equation}\label{eq:minimal_stationary_Lyapunov_2_conseq1}
0 = X_{12}X_{22}^{-1}A_{\mK}X_{12}^\tr+X_{12} A_{\mK}^\tr X_{22}^{-1}X_{12}^\tr
+ PC^\tr V^{-1} CX_{11} +
X_{11} C^\tr V^{-1} C P  - PC^\tr V^{-1}CP.
\end{equation}
Next, we multiply \eqref{eq:minimal_stationary_Lyapunov_2} by $-T^{-\tr}=X_{22}^{-1}X_{12}^\tr$ on the right and get
$$
0 =
AX_{12}X_{22}^{-1}X_{12}^\tr
+BC_{\mK}X_{12}^\tr
+X_{11}C^\tr V^{-1}C^\tr P
+X_{12}A_{\mK}^\tr X_{22}^{-1}X_{12}^\tr.
$$
By plugging this equality into \eqref{eq:minimal_stationary_Lyapunov_2_conseq1}, we get
$$
0 =
-AX_{12}X_{22}^{-1}X_{12}^\tr
-BC_{\mK}X_{12}^\tr
-X_{12}X_{22}^{-1}X_{12}^\tr A
-X_{12}C_{\mK}^\tr B^\tr
-PC^\tr V^{-1}CP.
$$
Then, we plug the above equality into \eqref{eq:minimal_stationary_Lyapunov_1} and get
$$
0 = A(X_{11}-X_{12}X_{22}^{-1}X_{12}^\tr)+(X_{11}-X_{12}X_{22}^{-1}X_{12}^\tr)A-PC^\tr V^{-1}CP+W,
$$
and since $P=X_{11}-X_{12}X_{22}^{-1}X_{12}^\tr$, we can see that $P$ satisfies the Riccati equation~\eqref{eq:Riccati_P}. Through similar steps, we can derive from \eqref{eq:LyapunovY} that $S$ satisfies the Riccati equation~\eqref{eq:Riccati_S}.

Finally, from classical control theory~\cite[Theorem 14.7]{zhou1996robust}, a globally optimal controller to the LQG problem~\eqref{eq:LQG_reformulation_KX} is given by~\eqref{eq:LQGstatespace}, and any similarity transformation leads to another equivalent controller with the same LQG cost. Therefore, any minimal stationary point, given by~\eqref{eq:stationary_points}, is globally optimal.
%
%
%
% Substituting these solutions into~\eqref{eq:XY_first_order}, which can be written as
% \begin{subequations}
%     \begin{align}
%         AX_{11} + X_{11}A^\tr + BC_{\mK} X_{12}^\tr + X_{12}C_{\mK}^\tr B^\tr + W &= 0, \\
%         AX_{12} +  BC_{\mK} X_{22} + X_{11}C^\tr B_{\mK}^\tr +X_{12}A_{\mK}^\tr  &= 0,  \label{eq:X_s2} \\
%          B_{\mK} C X_{12} + A_{\mK}X_{22} + X_{12}^\tr C^\tr B_{\mK}^\tr  + X_{22}A_{\mK}^\tr + B_{\mK}VB_{\mK}^\tr &= 0,  \label{eq:X_s3}
%     \end{align}
% \end{subequations}
% as well as
% \begin{subequations}
%     \begin{align}
%         A^\tr Y_{11} + Y_{11}A + C^\tr B_{\mK}^\tr Y_{12}^\tr + Y_{12}B_{\mK} C + Q &= 0, \\
%         A^\tr Y_{12} +  C^\tr B_{\mK}^\tr  Y_{22} + Y_{11}B C_{\mK} +Y_{12}A_{\mK}  &= 0, \\
%          C_{\mK}^\tr B^\tr Y_{12} + A_{\mK}^\tr Y_{22} + Y_{12}^\tr B C_{\mK}  + Y_{22}A_{\mK} + C_{\mK}^\tr R C_{\mK} &= 0, \label{eq:Y_s3}
%     \end{align}
% \end{subequations}
% %
%
%we can prove that $P$ and $S$ are the unique positive definite solutions to the famous Riccati equations~\eqref{eq:Riccati_P} and~\eqref{eq:Riccati_S}.
\end{proof}

The results in \cref{theo:stationary_points_globally_optimal} indicate that if the LQG problem~\eqref{eq:LQG_reformulation_KX} has a globally optimal solution in $\mathcal{C}_n$ that is also minimal, then the globally optimal controller is unique in $\mathcal{C}_n$ after taking a quotient with respect to similarity transformation. This is expected from the classical result that the globally optimal LQG controller is unique in the frequency domain~\cite[Theorem 14.7]{zhou1996robust}.

We note that \emph{minimal} stationary points are required in the proof of \cref{theo:stationary_points_globally_optimal}, as it guarantees that matrices~\eqref{eq:riccatiPS} are well-defined and the solutions~\eqref{eq:stationary_Bk} and~\eqref{eq:stationary_Ck} are unique. \cref{theo:stationary_points_globally_optimal} allows us to establish the following corollaries.
\begin{corollary}\label{corollary:non-minimal-globally-optimal}
The following statements are true:
\begin{enumerate}
\item     If $J_n(\mK)$ has a minimal stationary point in $\mathcal{C}_n$, then all its non-minimal stationary points $\mK \in \mathcal{C}_n$ are strictly suboptimal.
\item  If $J_n(\mK)$ has a non-minimal stationary point in $\mathcal{C}_n$ that is globally optimal, then all stationary points $\mK \in \mathcal{C}_n$ of $J_n(\mK)$ are non-minimal.
\end{enumerate}
\end{corollary}

We have already seen LQG cases with non-minimal stationary points that are strictly suboptimal in \cref{example:strict_saddle} and \cref{example:saddle_point_vanishing_Hessian}. It should be noted that, even with \cref{assumption:stabilizability}, the LQG problem~\eqref{eq:LQG_reformulation_KX} might have no minimal stationary points, i.e., all the solutions $\mK$ for~\eqref{eq:gradient_zero} may be non-minimal; this happens if the controller from the Ricatti equations~\eqref{eq:Riccati} is not minimal.

\begin{example}[Non-minimal globally optimal controllers]
\label{example:non-minimal-LQG}
Here we give an example from \cite{yousuff1984note}, whose optimal LQG controller does not have a minimal realization in $\mathcal{C}_n$.
Consider the linear system \eqref{eq:Dynamic} with
$$
A=\begin{bmatrix}
0 & -1 \\ 1 & 0
\end{bmatrix},
\qquad
B = \begin{bmatrix}
1 \\ 0
\end{bmatrix},
\qquad
C = \begin{bmatrix}
1 & -1
\end{bmatrix},
\qquad
W = \begin{bmatrix}
1 & -1 \\ -1 & 16
\end{bmatrix},
\qquad
V = 1,
$$
and let the LQG cost be defined by
$$
Q=\begin{bmatrix}
4 & 0 \\ 0 & 0
\end{bmatrix},
\qquad
R = 1.
$$
This LQG problem satisfies \cref{assumption:stabilizability}. The positive definite solutions to the Riccati equations \eqref{eq:Riccati} are given by
$$
P=\begin{bmatrix}
1 & 0 \\ 0 & 4
\end{bmatrix},
\quad
S
=\begin{bmatrix}
2 & 0 \\ 0 & 2
\end{bmatrix},
$$
and the globally optimal controller is given by
\begin{equation} \label{eq:LQGcontroller_Nonminimal_case}
A_{\mK} = \begin{bmatrix}
-3 & 0 \\
5 & -4
\end{bmatrix},
\quad
B_{\mK}=L
=
\begin{bmatrix}
1 \\ -4
\end{bmatrix},
\quad
C_{\mK}=-K =
\begin{bmatrix}
-2 & 0
\end{bmatrix}.
\end{equation}
It is not hard to see that $(C_{\mK},A_{\mK})$ is not observable. Therefore, the controller obtained from the Riccati equations is not minimal in this example. Consequently, by \cref{corollary:non-minimal-globally-optimal}, all stationary points of $J_n$ are not minimal for this example.

In this case, the globally optimal controllers in $\mathcal{C}_n$ are not all connected by similarity transformations. For example, it can be verified that the following two non-minimal controllers are both globally optimal:
     \begin{equation*}
         \mK_1 = \begin{bmatrix}
            0 & -2 & 0\\
            1 & -3& 0\\
            -4 & 5 & -4
         \end{bmatrix}, \qquad \mK_2 = \begin{bmatrix}
                 0 & -2 & 0\\
            1 & -3& 0\\
            0 & 0 & -1
         \end{bmatrix},
     \end{equation*}
     but there exists no similarity transformation between $\mK_1$ and $\mK_2$ since $\begin{bmatrix}
     -3 & 0 \\ 5 & -4
     \end{bmatrix}$ and $\begin{bmatrix}-3 & 0 \\ 0 & -1\end{bmatrix}$ have different sets of eigenvalues (recall that similarity transformation does not change eigenvalues).
     \hfill\qed
\end{example}

If a sequence of gradient iterates converges to a point, \cref{theo:stationary_points_globally_optimal} also allow us to check whether the limit point is a globally optimal solution to the LQG problem.

\begin{corollary} \label{corollary:Gradient_Descent_Convergence}
    Consider a gradient descent algorithm $\mK_{t+1} = \mK_{t} - \alpha_t \nabla J_n(\mK_t)$ for the LQG problem~\eqref{eq:LQG_reformulation_KX}, where $\alpha_t$ is a step size. Suppose the iterates $\mK_{t}$ converge to a point $\mK^*$, \emph{i.e.}, $\lim_{t\rightarrow \infty} \mK_t = \mK^*$. If $\mK^*$ is a controllable and observable controller, then it is globally optimal.
\end{corollary}

\begin{remark}
\Cref{corollary:Gradient_Descent_Convergence} proposes checking the controllability and observability of $\mK^\ast$ for verifying global optimality when the gradient descent iterates converge to $\mK^\ast$. In practice, the limit $\mK^\ast$ cannot be directly computed, and one tentative approach to check its controllability (observability) is to check whether the smallest singular value of the controllability (observability) matrix of the last iterate $\mK_T$ is sufficiently bounded away from zero. A rigorous justification of this approach will be of interest for future work.
\end{remark}

\begin{remark}
Note that~\Cref{corollary:Gradient_Descent_Convergence} does not discuss under what conditions will the gradient descent iterates converge. The results in~\cite{absil2005convergence} guarantee that if the cost function is analytic over the whole Euclidean space, then the gradient descent with step sizes satisfying the Wolfe conditions will either converge to a stationary point or diverge to infinity. In our case, however, the cost function $J_n(\mK)$ is only analytic over a subset $\mathcal{C}_n \subset \mathcal{V}_n$. Furthermore, $J_n(\mK)$ is not coercive as shown in Example~\ref{example:non-coercivity}. Whether the gradient descent with properly chosen step sizes can converge to a stationary point of $J_n(\mK)$ requires further investigation. %still  an open problem worth investigation.

%{\color{blue}
%$\nabla J(\mK_t) \rightarrow 0$ or not
%}
\end{remark}

%{\color{red} some comments of the convergence conditions, and how to check controllable/observable controller, rank conditions -- the $n$-th minimal singular value. }

%{\color{red} can gradient-based algorithms escape non-minimal stationary points, in the presence of minimal stationary points?}

%{\color{red}
%existence of minimal stationary points? Can we identify some situations where we always have minimal stationary points/ minimal globally optimal controllers?
%}

%{\color{red} Gradient = 0 may not --> admit minimal solutions. If it admits minimal solutions, they are globally optimal;} {\color{red} in this case, any lower-order controller is sub-optimal. }

\subsection{Hessian of $J_n(\mK)$ at Minimal Stationary Points} \label{subsection:hessaion_LQG_minimal}

Finally, we turn to characterizing the second-order behavior of $J_n$ around a globally optimal controller $\mK^\ast$. Throughout this subsection, we will assume that $\mK^\ast$ is controllable and observable. We focus on the eigenvalues and eigenspaces of the Hessian $\operatorname{Hess}_{\,\mK^\ast}$. The null space of $\operatorname{Hess}_{\,\mK^\ast}$ is
$$
\operatorname{null} \operatorname{Hess}_{\,\mK^\ast}
=\{x\in\mathcal{V}_n \mid \operatorname{Hess}_{\,\mK^\ast}(x,y)=0,\ \forall y\in\mathcal{V}_n\}.
$$
The following lemma shows that the tangent space $\mathcal{TO}_{\mK^\ast}$ is a subspace of the null space of $\operatorname{Hess}_{\,\mK^\ast}$, which is a direct corollary of \cite[Theorem 2]{li2019symmetry}.
\begin{lemma}\label{lemma:null_hessian_tangent}
Suppose $\mK^\ast$ is controllable and observable. Then
$$
\mathcal{TO}_{\mK^\ast}\subseteq\operatorname{null}\operatorname{Hess}_{\,\mK^\ast}.
$$
\end{lemma}
This lemma can be viewed as a local version of \cref{lemma:Jn_invariance} indicating the invariance of $J_n$ along the orbit $\mathcal{O}_{\mK}$. Consequently, the dimension of the null space of $\operatorname{Hess}_{\mK^\ast}$ is at least $q^2$.
On the other hand, we also have the following result.

\begin{lemma}\label{lemma:positive_orthogonal_tangent}
Suppose $\mK^\ast$ is controllable and observable, and let $\Delta\in\mathcal{TO}_{\mK^\ast}^\perp$. Then for all sufficiently small $t>0$,
$$
J_n(\mK^\ast+t\Delta)-J_n(\mK^\ast)>0.
$$
\end{lemma}
\begin{proof}
We prove by contradiction. Suppose for any sufficiently small $\delta>0$, there always exists $t\in(0,\delta)$ such that $J_n(\mK^\ast+t\Delta)=J_n(\mK^\ast)$. Then we can find a positive sequence $(t_j)_{j\geq 1}$ such that $t_j\rightarrow 0$ and $J_n(\mK^\ast+t_j\Delta)=J_n(\mK^\ast)$. Denote $\mK_j=\mK^\ast+t_j\Delta$. Since $\Delta$ is orthogonal to $\mathcal{TO}_{\mK^\ast}$, there must exists some $j\geq 1$ such that $\mK_j\notin \mathcal{TO}_{\mK^\ast}$. By \cite[Theorem 3.17]{zhou1996robust}, we can see that the transfer function of $\mK_j$ will be different from the transfer function of $\mK_j$. Then by the uniqueness of the transfer function solution to the LQG problem, $\mK_j$ cannot be a global minimum of $J_n$, contradicting $J_n(\mK_j)=J_n(\mK^\ast)$.
\end{proof}

Combining the observations from~\Cref{lemma:null_hessian_tangent} and~\ref{lemma:positive_orthogonal_tangent}, we can see that, while the Hessian $\operatorname{Hess}_{\,\mK^\ast}$ is degenerate and its null space has a nontrivial subspace $\mathcal{TO}_{\mK^\ast}$, the degeneracy associated with $\mathcal{TO}_{\mK^\ast}$ does not cause much trouble for optimizing $J_n$, as the directions in $\mathcal{TO}_{\mK^\ast}$ correspond to similarity transformations that lead to other globally optimal controllers, while along the directions orthogonal to $\mathcal{TO}_{\mK^\ast}$, the optimal controller of $J_n$ is locally unique.

We are therefore interested in the behavior of $\operatorname{Hess}_{\,\mK^\ast}$ restricted to the subspace $\mathcal{TO}_{\mK^\ast}^\perp$. Specifically, we let $\operatorname{rcond}_{\,\mK^\ast}$ denote the reciprocal condition number of $\operatorname{Hess}_{\,\mK^\ast}$ restricted to the subspace $\mathcal{TO}_{\mK^\ast}^\perp$, i.e.,
\begin{equation} \label{eq:reciprocal_condition_number}
\operatorname{rcond}_{\mK^\ast}
\coloneqq
\frac{\min_{\Delta\perp \mathcal{TO}_{\mK^\ast}} \operatorname{Hess}_{\,\mK^\ast}(\Delta,\Delta)/\|\Delta\|_F^2}{\max_{\Delta\perp \mathcal{TO}_{\mK^\ast}} \operatorname{Hess}_{\,\mK^\ast}(\Delta,\Delta)/\|\Delta\|_F^2}.
\end{equation}
Intuitively, if $\operatorname{rcond}_{\mK^\ast}$ is bounded away from zero, then we can expect gradient-based methods to achieve good local convergence behavior for optimizing $J_n$. However, we give an explicit example below showing that $\operatorname{rcond}_{\mK^\ast}$ can be arbitrarily bad even if the original plant seems entirely normal.

\begin{example}
\label{example:bad_Hessian}
Let $\epsilon>0$ be arbitrary, and let
$$
A=\frac{3}{2}\begin{bmatrix}
-1 & 0 \\ 0 & -1-\epsilon
\end{bmatrix},
\qquad
B = \begin{bmatrix}
1 \\ 1+\epsilon
\end{bmatrix},
\qquad
C = \begin{bmatrix}
1 & 1
\end{bmatrix},
$$
and
$$
Q=\begin{bmatrix}
4 & 1 \\ 1 & 4
\end{bmatrix},
\qquad
W=\begin{bmatrix}
4 & 1+\epsilon \\ 1+\epsilon & 4(1+\epsilon)^2
\end{bmatrix},
\qquad
V=R = 1.
$$
For this plant, the positive definite solutions to the Riccati equations \eqref{eq:Riccati} are given by
$$
P = \begin{bmatrix}
1 & 0 \\ 0 & 1+\epsilon
\end{bmatrix},
\qquad
S = \begin{bmatrix}
1 & 0 \\ 0 & \frac{1}{1+\epsilon}
\end{bmatrix},
$$
and we have
$$
K = R^{-1}B^\tr S
=\begin{bmatrix}
1 & 1
\end{bmatrix},
\qquad
L = PC^\tr V^{-1}
=\begin{bmatrix}
1 \\ 1+\epsilon
\end{bmatrix}.
$$
The optimal controller $\mK^\ast$ is then given by
$$
\mK^\ast=
\begin{bmatrix}
0 & -K \\ L & A-BK-LC
\end{bmatrix}
=\begin{bmatrix}
0 & -1 & -1 \\
1 & -\frac{7}{2} & -2 \\
1+\epsilon & -2(1+\epsilon) & -\frac{7}{2}(1+\epsilon)
\end{bmatrix}.
$$
It can be checked that the optimal controller provided by the Riccati equations is controllable and observable when $\epsilon\neq 0$. In \cref{proposition:Hessian_example}, we provide an asymptotic upper bound on the reciprocal condition number $\operatorname{rcond}_{\mK^\ast}$. We also provide numerical results on $\operatorname{Hess}_{\mK^\ast}$ for $\epsilon\in[0.002,0.5]$ in \cref{fig:Hessian_example}.
It can be seen that the upper bound~\eqref{eq:example_rcond_upperbound} on $\operatorname{rcond}_{\mK^\ast}$ is on the order of $O(\epsilon^4)$, indicating that $\operatorname{rcond}_{\mK^\ast}$ degrades rapidly as $\epsilon$ approaches zero. Moreover, it can be numerically checked via \cref{lemma:gradient_LQG_Jn} that, even if we set $\epsilon=0.5$, the reciprocal condition number $\operatorname{rcond}_{\mK^\ast}$ is still below $1.7\times 10^{-6}$. On the other hand, if we plug in $\epsilon=0.5$, the resulting plant's parameters as well as the controllability and observability matrices
$$
\begin{bmatrix}
B & AB
\end{bmatrix}
=\begin{bmatrix}
1 & -1.5 \\
1.5 & -3.375
\end{bmatrix},
\qquad
\begin{bmatrix}
C \\ CA^\tr
\end{bmatrix}
=\begin{bmatrix}
1 & 1 \\
-1.5 & -2.25
\end{bmatrix}
$$
seem entirely normal.\hfill\qed
\end{example}

\begin{comment}
$P=\begin{bmatrix}
p_{11} & p_{12} \\ p_{12} & p_{22}
\end{bmatrix}$ and $S=\begin{bmatrix}
s_{11} & s_{12} \\
s_{12} & s_{22}
\end{bmatrix}$ with
$$
p_{12} = -(1-\epsilon)
\cdot\frac{4\epsilon^2-8\epsilon+17-3\sqrt{8\epsilon^2-16\epsilon+25}}{12(\epsilon^2-2\epsilon+2)^2},
\quad
p_{11} = \frac{1}{3}+\frac{p_{12}}{1-\epsilon},
\quad
p_{22} = \frac{1}{3}+(1-\epsilon)p_{12},
$$
$$
s_{12} = -(1+\epsilon)
\cdot\frac{4\epsilon^2+8\epsilon+17-3\sqrt{8\epsilon^2+16\epsilon+25}}{12(\epsilon^2+2\epsilon+2)^2},
\quad
s_{11} = \frac{1}{3}+\frac{s_{12}}{1+\epsilon},
\quad
s_{22} = \frac{1}{3}+(1+\epsilon)s_{12},
$$
and we have
\begin{align*}
K =\ &
R^{-1}B^\tr S
=
\frac{-3+\sqrt{8\epsilon^2+16\epsilon+25}}{2(\epsilon^2+2\epsilon+2)}
\cdot
\begin{bmatrix}
1 & 1+\epsilon
\end{bmatrix}, \\
L
=\ &
PC^\tr V^{-1}
=\frac{-3+\sqrt{8\epsilon^2-16\epsilon+25}}{2(\epsilon^2-2\epsilon+2)}
\cdot
\begin{bmatrix}
1 \\ 1-\epsilon
\end{bmatrix},
\end{align*}
The optimal controller $\mK^\ast$ is then given by
$$
\mK^\ast = \begin{bmatrix}
0 & -K \\ L & A-BK-L C
\end{bmatrix},
$$
and the closed-loop system matrix \eqref{eq:closedloopmatrix} is
$$
\begin{bmatrix}
A & -BK \\ L C &
A-BK-L C
\end{bmatrix}.
$$
It can be checked that the optimal controller provided by the Riccati equations is controllable and observable when $\epsilon\neq 0$.
\end{comment}

%{\color{red} this example only says the convergence of gradient-based algorithm might be very slow, since the conditional number can be arbitrarily bad, but it is still converging. }

\begin{figure}[t]
\centering
\begin{subfigure}{.32\linewidth}
\centering
\includegraphics[width=\linewidth]{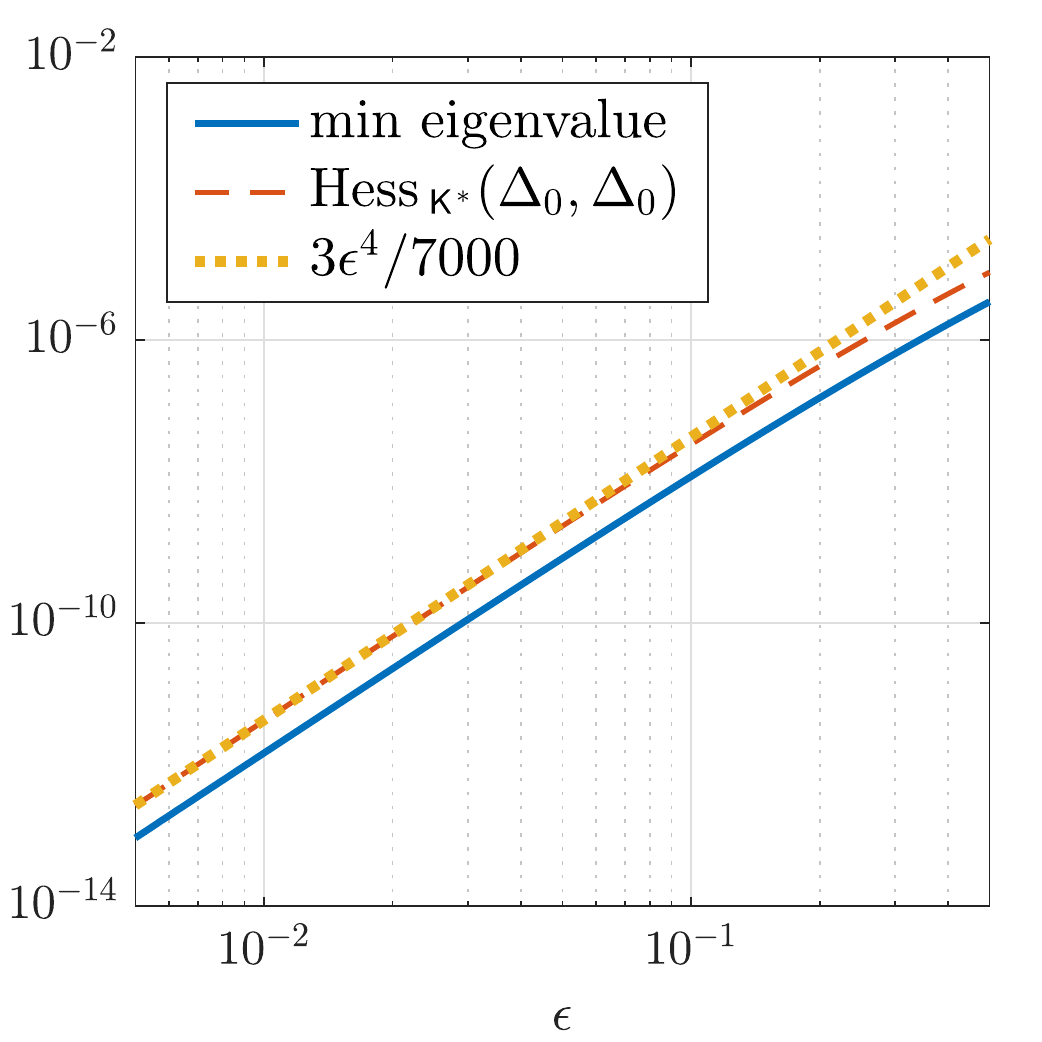}
\caption{}
\end{subfigure}
\begin{subfigure}{.32\linewidth}
\centering
\includegraphics[width=\linewidth]{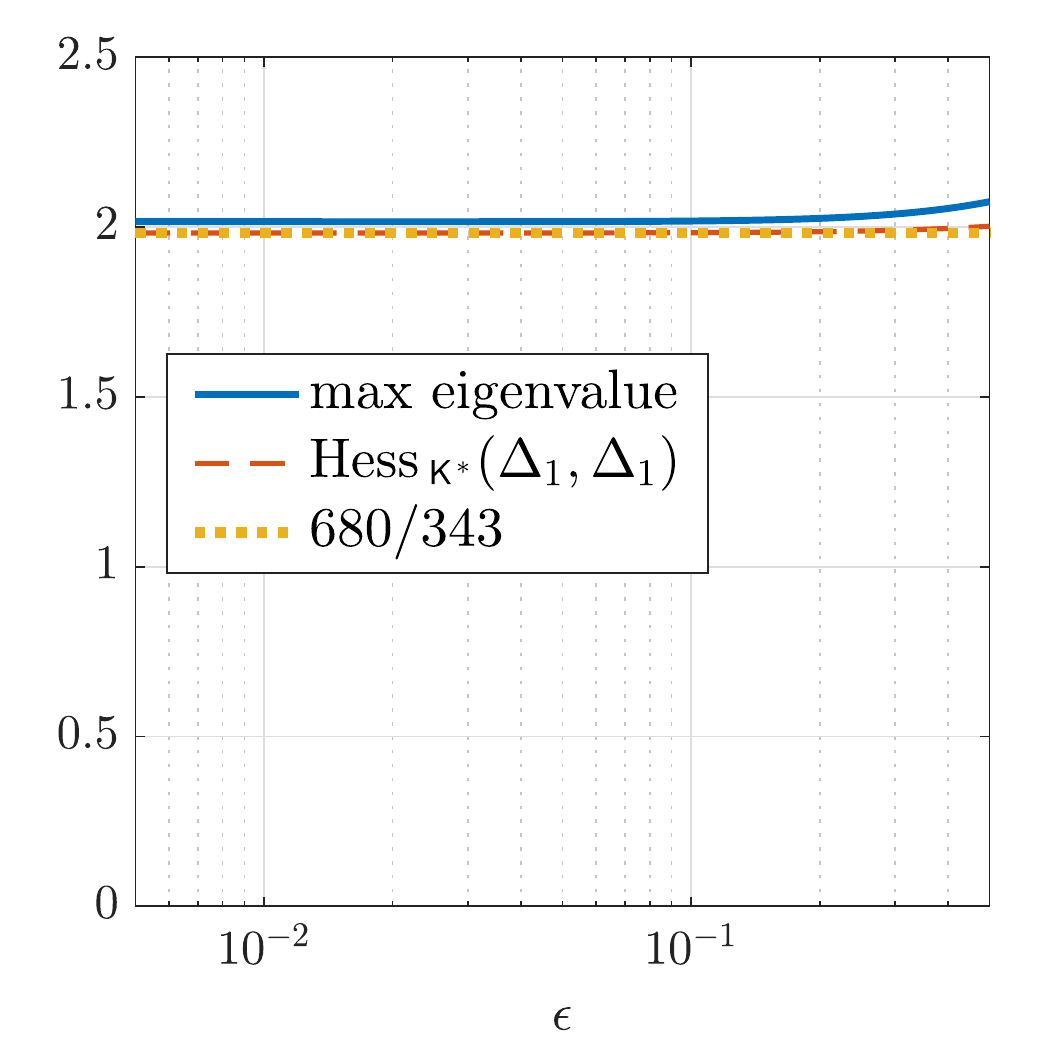}
\caption{}
\end{subfigure}
\begin{subfigure}{.32\linewidth}
\centering
\includegraphics[width=\linewidth]{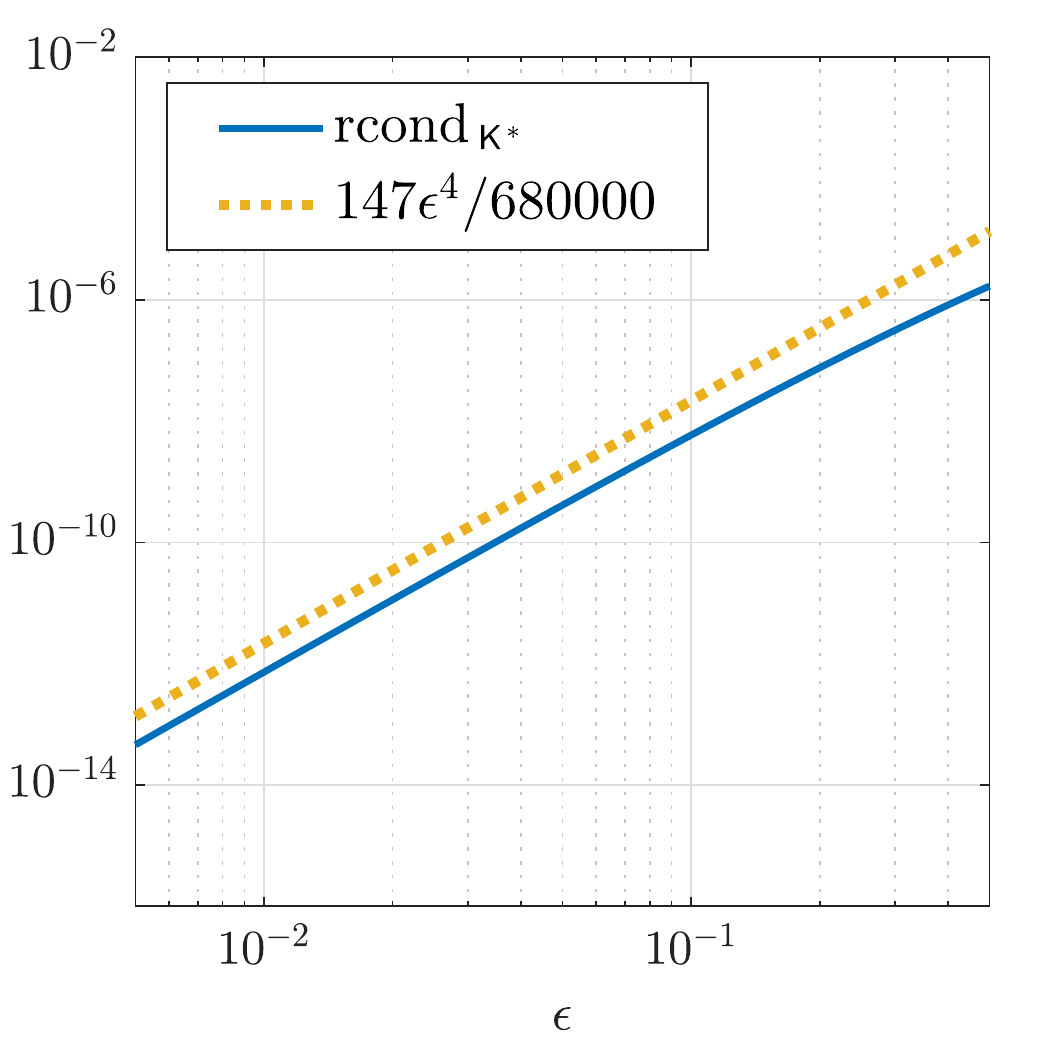}
\caption{}
\end{subfigure}
\caption{Numerical results on the behavior of $\operatorname{Hess}_{\,\mK^\ast}$ in~\cref{example:bad_Hessian}: (a) The minimum eigenvalue of $\operatorname{Hess}_{\mK^\ast}$ restricted on $\mathcal{TO}_{\mK^\ast}^\perp$, the value of $\operatorname{Hess}_{\mK^\ast}(\Delta_0,\Delta_0)$, and the asymptotic upper bound given by \eqref{eq:example_mineig_upperbound}. (b) The maximum eigenvalue of $\operatorname{Hess}_{\mK^\ast}$ restricted on $\mathcal{TO}_{\mK^\ast}^\perp$, the value of $\operatorname{Hess}_{\mK^\ast}(\Delta_1,\Delta_1)$, and the asymptotic lower bound given by \eqref{eq:example_maxeig_upperbound}. (c) The reciprocal condition number $\operatorname{rcond}_{\mK^\ast}$ and its asymptotic upper bound given by \eqref{eq:example_rcond_upperbound}. }
\label{fig:Hessian_example}
\end{figure}

\begin{theorem}\label{proposition:Hessian_example}
Consider the LQG problem in \cref{example:bad_Hessian}. Let $\epsilon>0$ be arbitrary. Let
$$
\Delta_0
=\begin{bmatrix}
0 & 0 & 0 \\
0 & -1/2 & 1/2 \\
0 & 1/2 & -1/2
\end{bmatrix},
\qquad
\Delta_1
=\begin{bmatrix}
0 & -1/2 & -1/2 \\
1/2 & 0 & 0 \\
1/2 & 0 & 0
\end{bmatrix}.
$$
Then, as $\epsilon\rightarrow 0$, we have
\begin{align*}
\operatorname{Hess}_{\,\mK^\ast}(\Delta_0,\Delta_0)
=\ &
\left.\frac{d^2 J(\mK^\ast+t\Delta_0)}{dt^2}\right|_{t=0}
=\frac{3}{7000}\epsilon^4+o(\epsilon^4), \\
\operatorname{Hess}_{\,\mK^\ast}(\Delta_1,\Delta_1)
=\ &
\left.\frac{d^2 J(\mK^\ast+t\Delta_1)}{dt^2}\right|_{t=0}
=
\frac{680}{343}
+o(1),
\end{align*}
and
$$
\left\|\operatorname{Proj}_{\mathcal{TO}_{\mK^\ast}}[\Delta_0]\right\|_F
=O(\epsilon).
$$
Consequently, as $\epsilon\rightarrow 0$,
\begin{subequations}
\begin{align}
\min_{\Delta\perp \mathcal{TO}_{\mK^\ast}}
\frac{\operatorname{Hess}_{\mK^\ast}(\Delta,\Delta)}{\|\Delta\|_F^2}
\leq\ &
\frac{\operatorname{Hess}_{\mK^\ast}(\Delta_0,\Delta_0)}{\|\Delta_0\|_F^2-\left\|\mathcal{P}_{\mathcal{TO}_{\mK^\ast}}[\Delta_0]\right\|_F^2}
=
\frac{3}{7000}\epsilon^4+o(\epsilon^4),
\label{eq:example_mineig_upperbound}\\
\max_{\Delta\perp \mathcal{TO}_{\mK^\ast}}
\frac{\operatorname{Hess}_{\mK^\ast}(\Delta,\Delta)}{\|\Delta\|_F^2}
\geq\ &
\frac{\operatorname{Hess}_{\mK^\ast}(\Delta_1,\Delta_1)}{\|\Delta_1\|_F^2}
=
\frac{680}{343}+o(1),
\label{eq:example_maxeig_upperbound}
\end{align}
and the reciprocal condition number of $\operatorname{Hess}_{\,\mK^\ast}$ restricted on $\mathcal{TO}_{\mK^\ast}^\perp$ can be upper bounded by
\begin{equation}\label{eq:example_rcond_upperbound}
\operatorname{rcond}_{\mK^\ast}
\leq \frac{147}{680000}\epsilon^4+o(\epsilon^4)
\approx 2.16\times 10^{-4}\cdot\epsilon^4+o(\epsilon^4).
\end{equation}
\end{subequations}
\end{theorem}

The proof of~\cref{proposition:Hessian_example} is based on a direct but tedious calculation of Hessian via \cref{lemma:Jn_Hessian}. The details are provided in~\cref{appendix:hessian_example}. The observations in \cref{example:bad_Hessian} suggest that, if we apply the vanilla gradient descent algorithm to the optimization problem~\eqref{eq:LQG_reformulation_KX}, it may take a large number of iterations for the iterate to converge to a globally optimal controller for certain LQG problems that appear entirely normal.

% {\color{red}
\begin{remark}[Symmetry structures in LQG control]
Due to the symmetry induced by similarity transformations, the landscape of LQG shares some similarities with the landscapes of non-convex machine learning problems with rotational symmetries such as phase retrieval, matrix factorization~\cite{sun2018geometric,li2019symmetry,zhang2020symmetry}. For example, the stationary points of these non-convex problems are non-isolated, and the tangent space of the orbit associated with the symmetry group is a subspace of the null space of the Hessian (see~\cref{lemma:null_hessian_tangent}). On the other hand, for phase retrieval~\cite{sun2018geometric} and matrix factorization~\cite{li2019symmetry}, the %identification and
classification of all stationary points as well as their local curvatures (Hessian) seem to be relatively well understood, while there remain many open questions regarding
 %the structures of
the stationary points of LQG:  %still require more investigation:
such as the existence of local optimizers that are not globally optimal, whether all non-globally-optimal stationary points have the form of~\eqref{eq:gradient_nonglobally_K} up to similarity transformations. Finally, in addition to the apparent algebraic complication of LQG and control-theoretic notions such as minimal controllers, the non-compactness of the group of similarity transformations may also render the landscape of LQG distinct from the non-convex machine learning problems with rotational symmetries.
\end{remark}

%}

%{\color{red}
%\begin{remark}[Overall comment on the landscape of the LQG problem and gradient-based algorithms]
%A few summaries and comments.
%which can be confirmed by our numerical experiments.
%\end{remark}
%}

\section{Numerical experiments} \label{section:numerical_results}

We have illustrated our main technical results on the connectivity of stabilizing controllers and stationary points through Examples~\ref{example:disconnectivity}-\ref{example:bad_Hessian}. Here, we present some numerical experiments to demonstrate empirical performance of gradient descent algorithms for solving the LQG problem~\eqref{eq:LQG_reformulation_KX}. The scripts for all experiments can be downloaded from {\small \url{https://github.com/zhengy09/LQG_gradient}}.

% \begin{algorithm}[t]
% \SetAlgoLined
%  \caption{Gradient descent for solving LQG~\eqref{eq:LQG_reformulation_KX}}
%  \label{algorithm:gradient_descent}
%  Initialize a stabilizing controller $\mK_0$\;
%  \For{$t = 0, 1, 2, \ldots, t_{\max}$}{
%     Obtain the gradient of $J_n(\mK)$ at $\mK_{t}$ according to~\eqref{eq:gradient_Jn} \;
%     Update the controller $\mK_{t+1}$ as
%     $$
%         A_{\mK,t+1} = A_{\mK,t} - s_t \left.\frac{\partial J(\mK)}{\partial A_{\mK}}\right|_{\mK_t}, \quad B_{\mK,t+1} = B_{\mK,t} - s_t \left.\frac{\partial J(\mK)}{\partial B_{\mK}}\right|_{\mK_t}, \quad C_{\mK,t+1} = C_{\mK,t} - s_t \left.\frac{\partial J(\mK)}{\partial C_{\mK}}\right|_{\mK_t}
%     $$
%     where the step size $s_t$ is determined by the Armijo rule\;
%     \If{$\|\nabla J(\mK_t)\|_F \leq \epsilon$}{
%     \textbf{Break}\;
%     }
%  }
%  \textbf{Armijo rule:} for step size $s_t$\;
%  \,\,\,Set $s_t = 1$, repeat $s_t = \beta s_t$ until
%  $$
%     J(\mK_{t}) - J(\mK_{t+1})    \geq \alpha s_t  \|\nabla J(\mK_t)\|_F,
%  $$
%  \,where $\alpha \in (0,1), \beta \in (0,1)$, e.g., $\alpha = 0.01$ and $\beta = 0.5$.
% \end{algorithm}

\subsection{Gradient Descent Algorithms}
A vanilla gradient descent algorithm for solving~\eqref{eq:LQG_reformulation_KX} is as follows. Upon giving an initial stabilizing controller $\mK \in \mathcal{C}_n$, we update the controller: $t = 0, 1, 2, \ldots$
\begin{equation}\label{eq:gradient_descent}
    A_{\mK,t+1} = A_{\mK,t} - s_t \left.\frac{\partial J(\mK)}{\partial A_{\mK}}\right|_{\mK_t}, \, B_{\mK,t+1} = B_{\mK,t} - s_t \left.\frac{\partial J(\mK)}{\partial B_{\mK}}\right|_{\mK_t}, \, C_{\mK,t+1} = C_{\mK,t} - s_t \left.\frac{\partial J(\mK)}{\partial C_{\mK}}\right|_{\mK_t},
\end{equation}
where the gradient is obtained using~\eqref{eq:gradient_Jn}, until the gradient satisfies $\|\nabla J(\mK_t)\|_F \leq \epsilon$ or the iteration reaches the maximum number $t_{\max}$. In our simulation, the step size $s_t$ in~\eqref{eq:gradient_descent} is determined by  the \emph{Armijo rule}~\cite[Chapter 1.3]{bertsekas1997nonlinear}: Set $s_t = 1$, repeat $s_t = \beta s_t$ until
 $$
    J(\mK_{t}) - J(\mK_{t+1})    \geq \alpha s_t  \|\nabla J(\mK_t)\|_F^2,
 $$
where $\alpha \in (0,1), \beta \in (0,1)$, e.g., $\alpha = 0.01$ and $\beta = 0.5$.

For numerical comparison, we can also reduce the number of controller parameters by considering a controller canonical form. In particular, for any SISO controller, the controllable canonical form of $\mK$ is
\begin{equation} \label{eq:canonical_form}
    A_{\mK} = \begin{bmatrix}
    0 & 1 & 0 & \ldots & 0\\
    0 & 0 & 1 & \ldots & 0\\
    \vdots & \vdots & \vdots & \ddots & \vdots\\
     0 & 0 & 0 & \ldots & 1\\
    -b_0 & -b_1 & -b_2 & \ldots & -b_{n-1}\\
    \end{bmatrix}, \; B_{\mK} = \begin{bmatrix}
    0 \\ 0 \\0 \\ \vdots \\1
    \end{bmatrix}, \, C_{\mK} = \begin{bmatrix}
    a_0 & a_1 & a_2 & \ldots & a_{n-1}
    \end{bmatrix}.
\end{equation}
We now only update the controller parameters $a_i, b_i, i = 0, \ldots, n-1$  by using a partial gradient in~\eqref{eq:gradient_descent}. It is clear that the set of stabilizing controllable controllers is a subset of $\mathcal{C}_n$, but we note that the connectivity of stabilizing controllable controllers is unclear and cannot be deduced from the results in \cref{sec:connectivity}. Here, we further remark a few facts~\cite[Chapter 3]{zhou1996robust}
\begin{itemize}
    \item The controller $\mK$ in~\eqref{eq:canonical_form} is not necessarily minimal, and it may be unobservable. Thus, this parameterization~\eqref{eq:canonical_form} is able to capture some non-minimal globally optimal controllers, e.g., the LQG problem in \cref{example:non-minimal-LQG}.
    \item For any controllable SISO $\mK$, there is a unique similarity transformation such that $\mathscr{T}_T(\mK)$ is in the form of \eqref{eq:canonical_form}. Conversely, given $\mK$ in the form of \eqref{eq:canonical_form}, all the controllers in the orbit $\mathcal{O}_{\mK}$ are controllable.
    \item By \cref{theo:stationary_points_globally_optimal}, if the LQG problem~\eqref{eq:LQG_reformulation_KX}  for SISO systems has a minimal stationary point, then it admits a unique globally optimal controller in the form of~\eqref{eq:canonical_form}.
\end{itemize}

In our experiments, we set the maximum iteration number $t_{\max} = 10^4$ and the stopping criterion $\epsilon = 10^{-6}$. %For open-loop unstable systems,
To investigate the influence of initial stabilizing controllers on the convergence performance of gradient descent algorithms, we used two different initialization strategies:
\begin{enumerate}
    \item \textit{Random initialization:} We used a pole placement method to get an initial stabilizing controller $\mK$, and the closed-loop poles were chosen randomly from $(-2,-1)$.
%Note that for open-loop stable systems, the controller $\mK$ with $
%    A_{\mK,0} = -I_n, C_{\mK,0} = 0_{m\times n}
% 5$
% and a random $B_{\mK,0}$ is also a valid initial stabilizing point.
\item \textit{Initialization around a globally optimal point:} We also considered initialization around the globally optimal controller from Riccati equations, i.e.,
$$
    A_{\mK,0} \sim \mathcal{N}(A_{\mK}^{\star},\delta I),  \qquad B_{\mK,0} \sim \mathcal{N}(B_{\mK}^{\star},\delta I), \qquad C_{\mK,0} \sim \mathcal{N}(C_{\mK}^{\star},\delta I),
$$
where $(A_{\mK}^{\star}, B_{\mK}^{\star}, C_{\mK}^{\star})$ is the optimal LQG controller~\eqref{eq:LQGstatespace} from solving Riccati equations, and we chose $\delta = 10^{-2}$ in the simulations.
\end{enumerate}

Throughout this section, we denote the vanilla gradient descent algorithm~\eqref{eq:gradient_descent} as \texttt{Vanilla GD$_A$} and call the gradient descent over the controllable canonical form~\eqref{eq:canonical_form} as \texttt{Vanilla GD$_B$}.

\subsection{Numerical Results I: Performance with random initialization}

We first consider two examples for which \texttt{Vanilla GD$_B$} has good empirical convergence performance. The first one is the famous Doyle's LQG example from~\cite{doyle1978guaranteed}
\begin{subequations}\label{eq:doyle_example}
\begin{equation}
    A = \begin{bmatrix}
        1 & 1\\
        0 & 1
    \end{bmatrix},\; B = \begin{bmatrix}
        0\\
        1
    \end{bmatrix}, \; C = \begin{bmatrix} 1 & 0 \end{bmatrix},\; W = 5\begin{bmatrix}
    1 & 1 \\
    1 & 1
    \end{bmatrix}, \; V = 1
\end{equation}
with performance weights
\begin{equation}
 Q = 5\begin{bmatrix}
    1 & 1 \\
    1 & 1
    \end{bmatrix}, \; R = 1.
\end{equation}
\end{subequations}
The globally optimal LQG controller from Riccati equations is
\begin{equation} \label{eq:DoyleExample_OptimalController}
    A_{\mK} = \begin{bmatrix}
       -4 &    1 \\
  -10 &-4
    \end{bmatrix}, \; B_{\mK} = \begin{bmatrix}
    5\\5
    \end{bmatrix},  C_{\mK} = \begin{bmatrix}
    -5 & -5
    \end{bmatrix},
\end{equation}
and its corresponding LQG cost is $J^\star = 750$. The system~\eqref{eq:doyle_example} is open-loop unstable, so we chose an initial stabilizing controller using pole placement where the poles were randomly selected from $(-2,-1)$ in our simulations. The results are shown in \cref{fig:Doyle_example}. For this LQG case, \texttt{Vanilla GD$_B$} over the controllable canonical form has better convergence performance compared to \texttt{Vanilla GD$_A$}. In particular, \texttt{Vanilla GD$_A$} did not converge within $10^4$ iterations, and  the final iterate in \texttt{Vanilla GD$_A$} has nonzero gradient. Instead, for different initial points, \texttt{Vanilla GD$_B$} converged to the following solution (up to two decimal places)
\begin{equation} \label{eq:Doyle_GD_solution}
    A_{\mK} = \begin{bmatrix} 0 & 1 \\
                               -26.00 & -8.00 \end{bmatrix}, B_{\mK} = \begin{bmatrix} 0 \\ 1 \end{bmatrix}, C_{\mK} = \begin{bmatrix}25.00 &  -50.00 \end{bmatrix}.
\end{equation}
The controller~\eqref{eq:Doyle_GD_solution} from \texttt{Vanilla GD$_B$} is minimal, and the gradient is close to zero (stationary point). By \cref{corollary:Gradient_Descent_Convergence}, it is reasonable to conclude that this controller is globally optimal. Indeed,~\eqref{eq:Doyle_GD_solution} is identical to~\eqref{eq:DoyleExample_OptimalController} via a similarity transformation defined by $T = \begin{bmatrix}   25   &  5 \\
  -30  &  5\end{bmatrix}$. By \cref{lemma:Jn_Hessian}, we can also compute the hessian of $J_2(\mK)$ at~\eqref{eq:Doyle_GD_solution}, for which the minimum eigenvalue is $12.15$ when restricting to the subspace $\mathcal{TO}_{\mK^\ast}^\perp$.

\begin{figure}[t]
    \centering
        \centering
\begin{subfigure}{.4\textwidth}
    \hspace{10mm}
    \includegraphics[scale = 0.6]{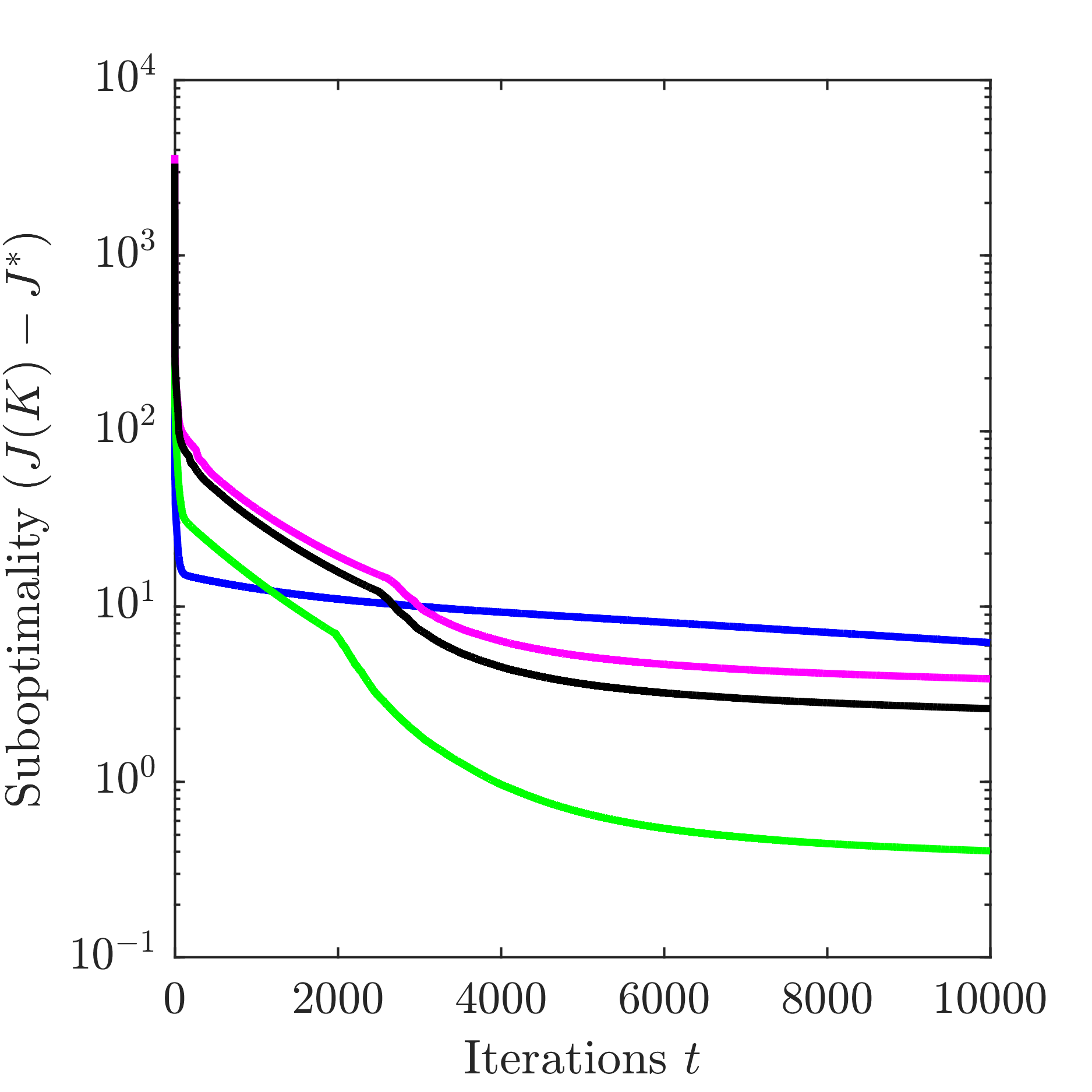}
     \caption{\texttt{Vanilla GD$_A$}}
    \end{subfigure}
    \hspace{20mm}
\begin{subfigure}{.4\textwidth}
    \hspace{10mm}
    \includegraphics[scale = 0.6]{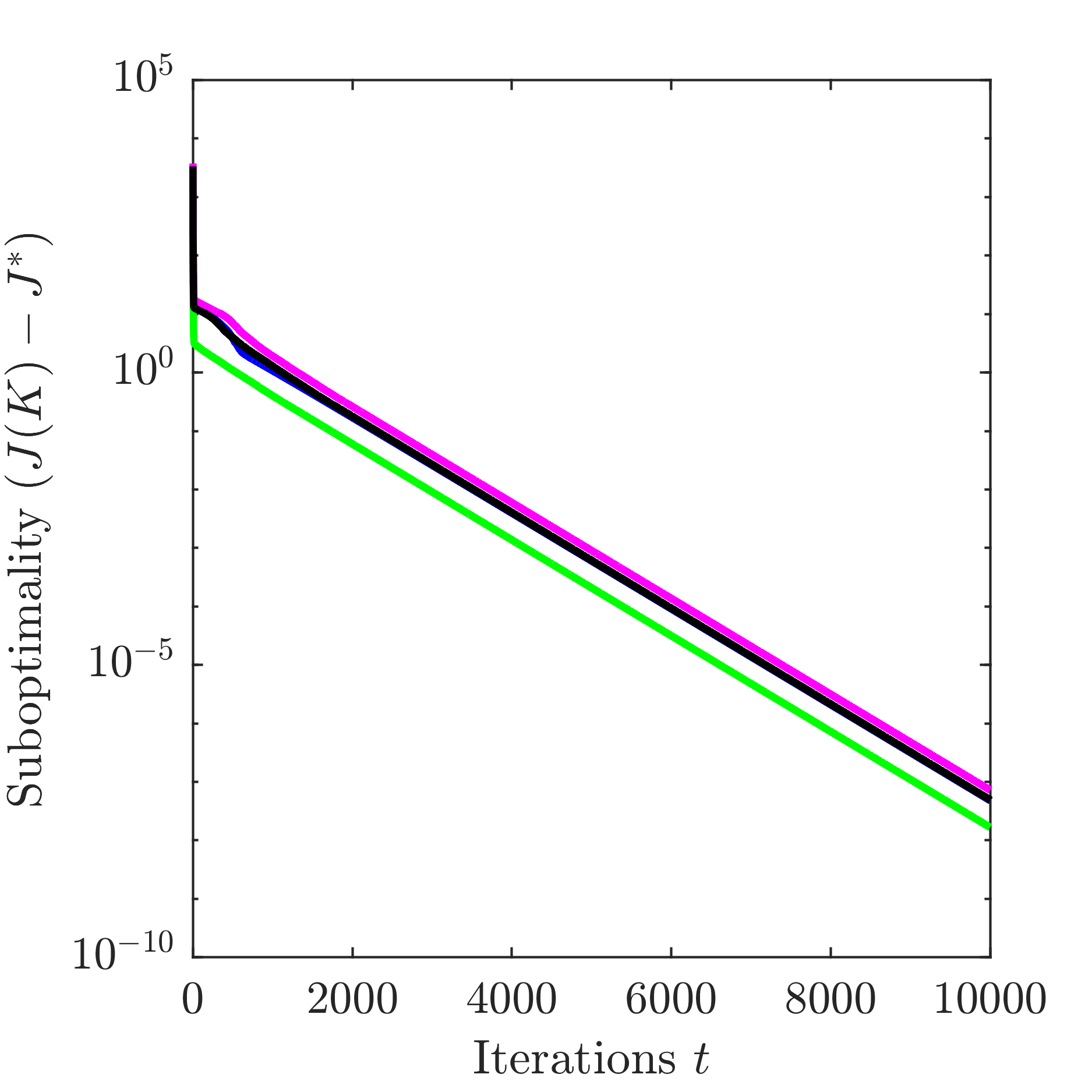}
         \caption{\texttt{Vanilla GD$_B$}}
\end{subfigure}
    \caption{Convergence performance of gradient descent algorithms for Doyle's example in~\eqref{eq:doyle_example} with four different random initialization $\mK_0$.} % (a) Standard gradient descent; (b) Gradient descent over the controllable canonical form.}
    \label{fig:Doyle_example}
\end{figure}

\begin{figure}[t]
    \centering
        \centering
\begin{subfigure}{.4\textwidth}
    \hspace{10mm}
    \includegraphics[scale = 0.6]{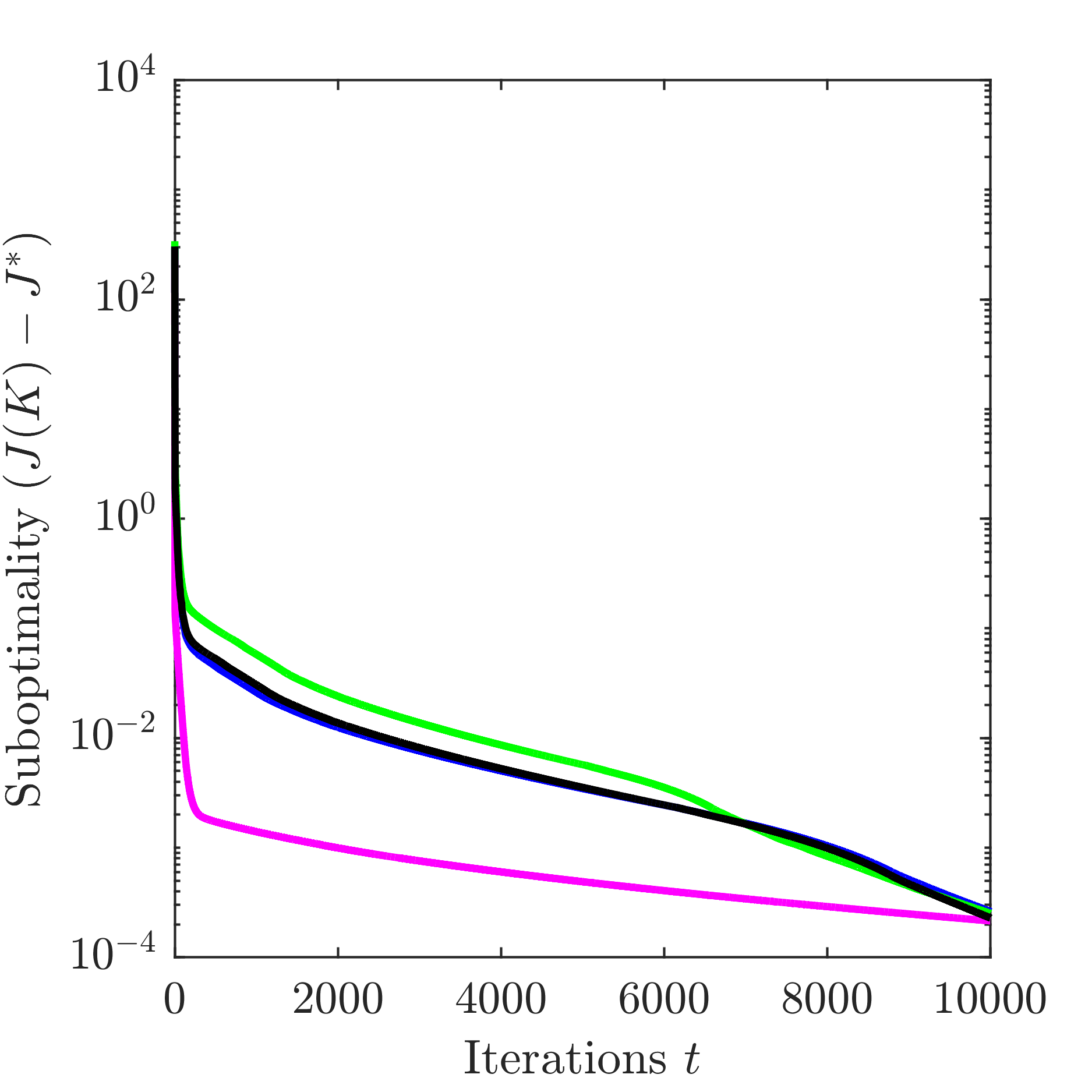}
     \caption{\texttt{Vanilla GD$_A$}}
    \end{subfigure}
    \hspace{20mm}
\begin{subfigure}{.4\textwidth}
    \hspace{10mm}
    \includegraphics[scale = 0.6]{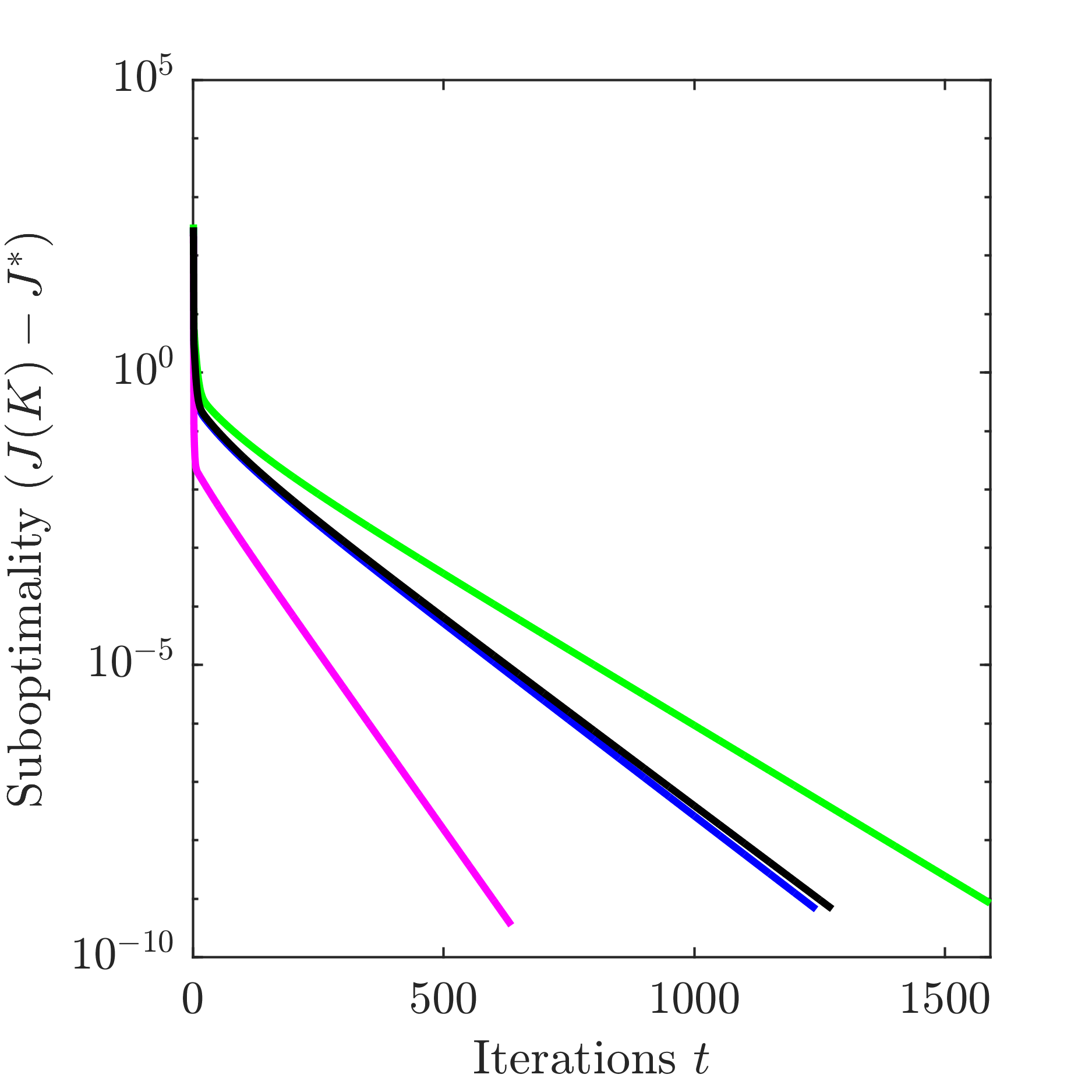}
         \caption{\texttt{Vanilla GD$_B$}}
\end{subfigure}
    \caption{Convergence performance of gradient descent algorithms for \cref{example:non-minimal-LQG} with four different random initialization $\mK_0$.} % (a) Standard gradient descent; (b) Gradient descent over the controllable canonical form.}
    \label{fig:Example7}
\end{figure}

Our second numerical experiment is carried out on  the LQG case in \cref{example:non-minimal-LQG}, for which a globally optimal controller from Riccati equations is non-minimal, shown in~\eqref{eq:LQGcontroller_Nonminimal_case}. The initial controllers were randomly chosen by pole placement from $(-2,-1)$. Similar to the first numerical experiment, \texttt{Vanilla GD$_A$} did not converge within $10^4$ iterations, while \texttt{Vanilla GD$_B$} converged to stationary points (the gradient reached the stopping criterion); see  \cref{fig:Example7}. In this case, the controllers from \texttt{Vanilla GD$_B$} are not minimal, and they have different state-space representations, two of which are
\begin{subequations}
\label{eq:GD_solution_non_minimal}
\begin{align}
    A_{\mK,1} &= \begin{bmatrix} 0 & 1 \\
                               -14.0912 &  -7.6970 \end{bmatrix}, \;\; B_{\mK,1} = \begin{bmatrix} 0 \\ 1 \end{bmatrix}, \;\; C_{\mK,1} = \begin{bmatrix} -9.3941 &  -1.9999 \end{bmatrix}, \\
                                A_{\mK,2} &= \begin{bmatrix} 0 & 1 \\
                               -17.2130 &  -8.7375 \end{bmatrix}, \;\; B_{\mK,2} = \begin{bmatrix} 0 \\ 1 \end{bmatrix}, \;\; C_{\mK,2} = \begin{bmatrix} -11.4753  & -1.9999 \end{bmatrix}.
\end{align}
\end{subequations}
Our theoretical results (\cref{theorem:non_globally_optimal_stationary_point} and \cref{corollary:Gradient_Descent_Convergence}) failed to check whether the controllers~\eqref{eq:GD_solution_non_minimal} from \texttt{Vanilla GD$_B$} are globally optimal. However, after pole-zero cancellation, we can check that the controllers~\eqref{eq:GD_solution_non_minimal} correspond to the same transfer function with~\eqref{eq:LQGcontroller_Nonminimal_case}, which is
$$
    \mathbf{K}_\star = \frac{-2}{s+3}.
$$
Also, we numerically check that the the Hessian of $J_2(\mK)$ at the controllers~\eqref{eq:GD_solution_non_minimal} and~\eqref{eq:LQGcontroller_Nonminimal_case} has a minimum eigenvalue as zero over the subspace $\mathcal{TO}_{\mK^\ast}^\perp$.

\subsection{Numerical Results II: initialization matters}

Here, we present two LQG examples for which \texttt{Vanilla GD$_B$} over the controllable canonical form seems to get stuck around some points when using random initialization. We first consider the LQG in \cref{example:saddle_point_vanishing_Hessian}, for which we have shown there exist stationary points with vanishing Hessian (see \cref{fig:saddle_point_vanishing_Hessian}). Note that this LQG problem has a minimal globally optimal controller, so it admits a unique globally optimal controller in the form of~\eqref{eq:canonical_form}. However, as shown in \cref{fig:Example6}, with random initialization, \texttt{Vanilla GD$_B$} over the controllable canonical form seems to get stuck around different points; \texttt{Vanilla GD$_A$} does make steady improvement over the LQG cost function, but it still failed to converge within $10^4$ iterations.  When using the initialization around a globally optimal point, the convergence performance of both \texttt{Vanilla GD$_A$} and \texttt{Vanilla GD$_B$} has been significantly improved, and both of them reached the stopping criterion within one hundred iterations. We note that the random initialization actually started from a point with a smaller LQG cost compared to the other initialization.

\begin{figure}[t]
    \centering
        \centering
\begin{subfigure}{.49\textwidth}
    \includegraphics[scale = 0.48]{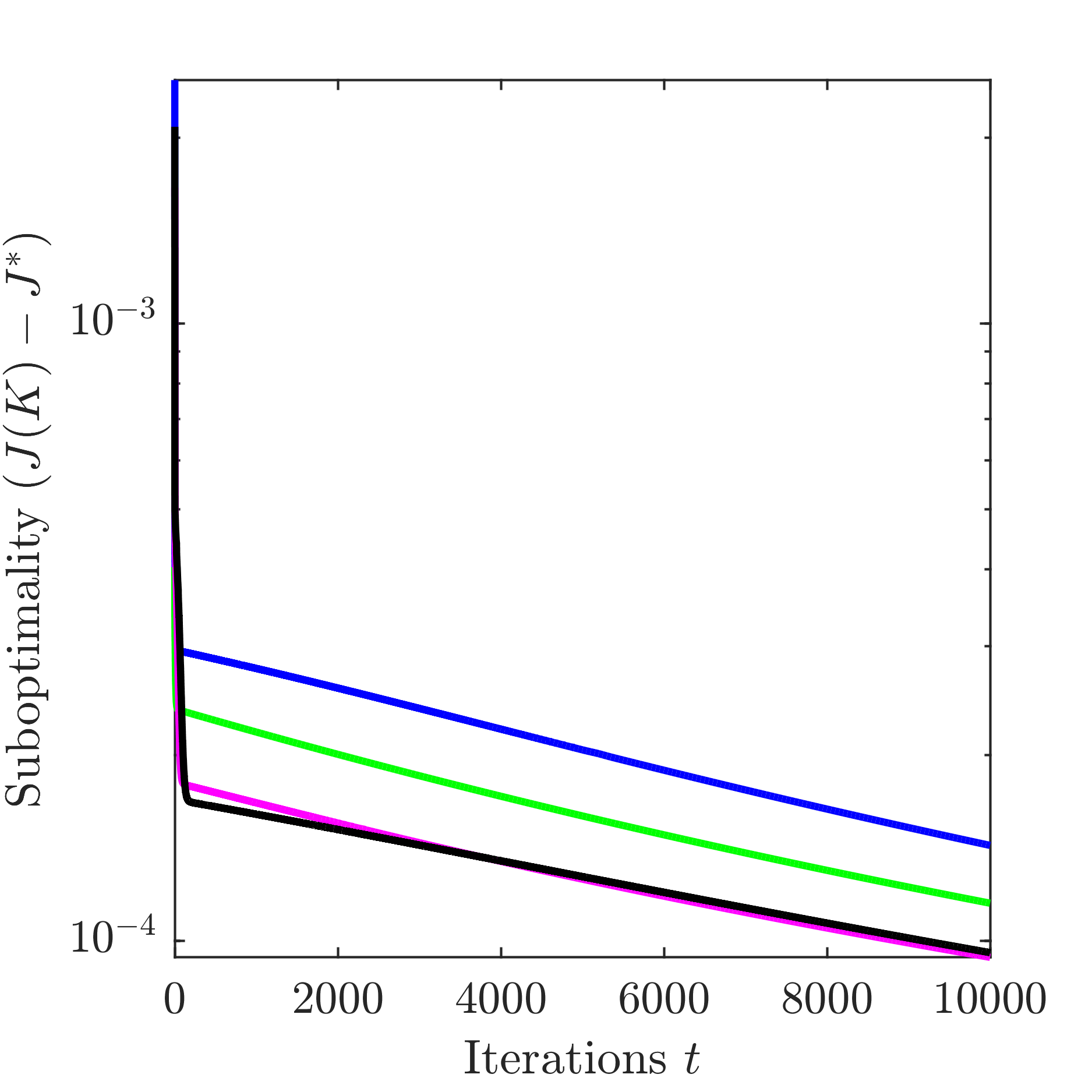}
    \hspace{0mm}
     \includegraphics[scale = 0.48]{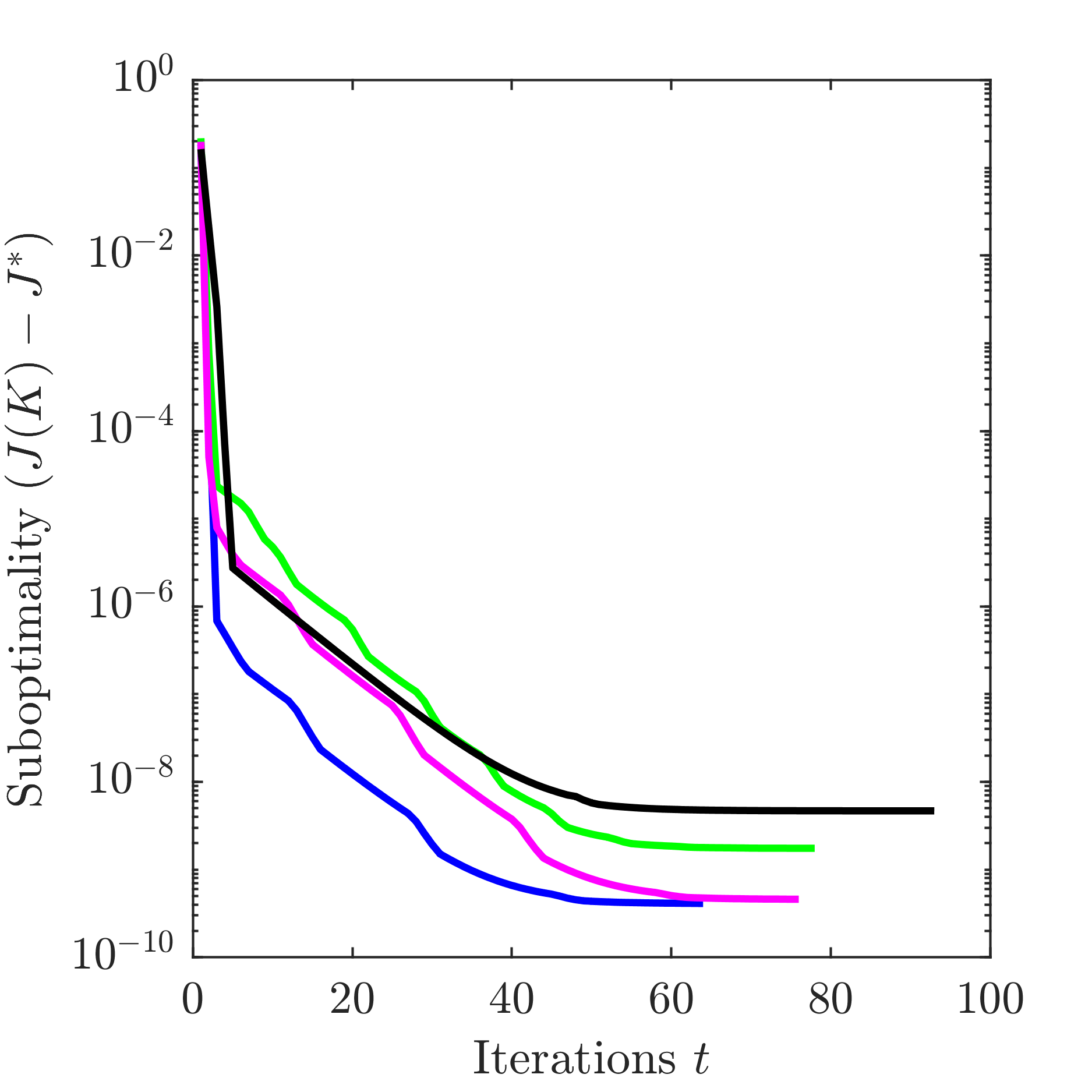}
     \caption{\texttt{Vanilla GD$_A$}}
    \end{subfigure}
    \hspace{1mm}
\begin{subfigure}{.49\textwidth}
    \includegraphics[scale = 0.48]{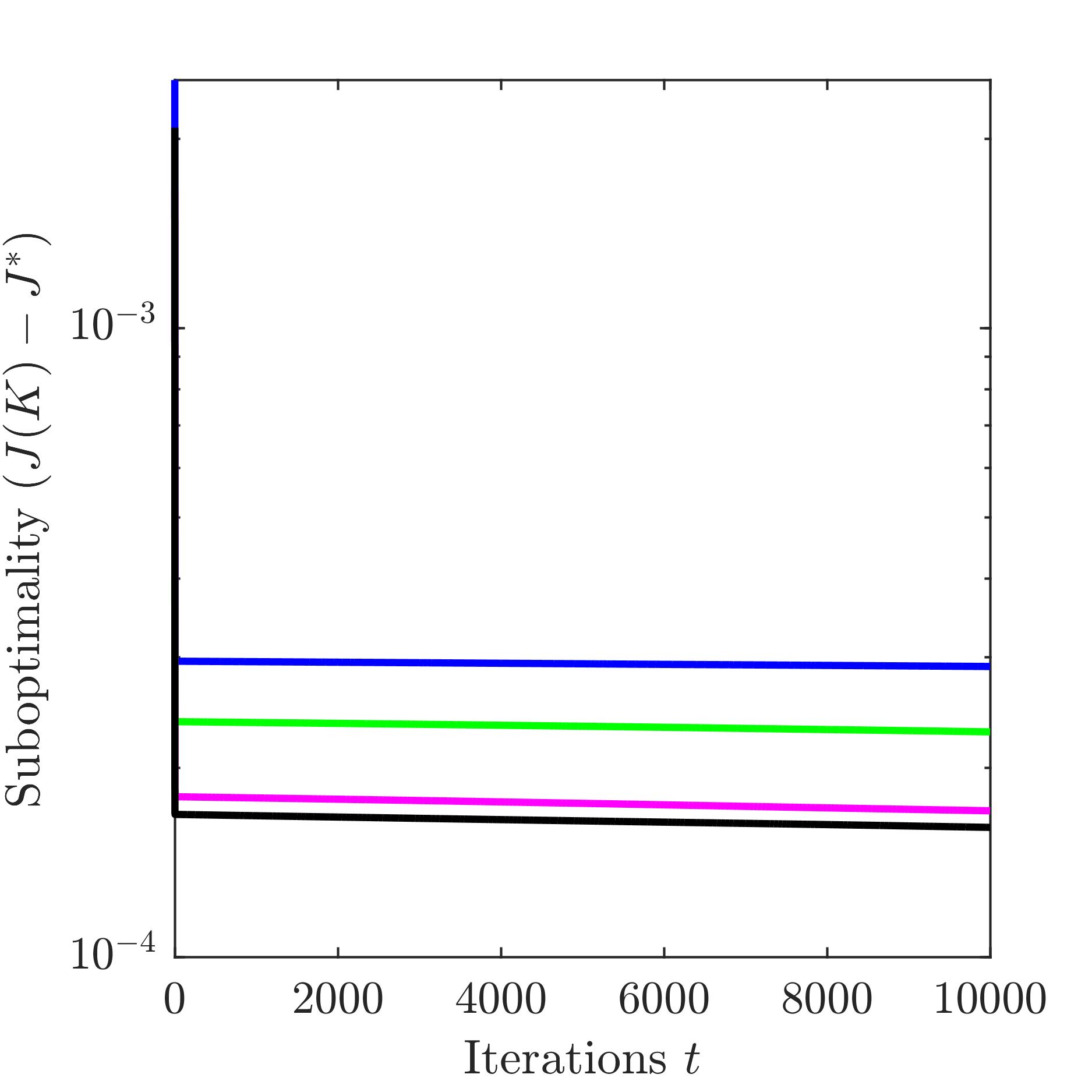}
    \hspace{0mm}
    \includegraphics[scale = 0.48]{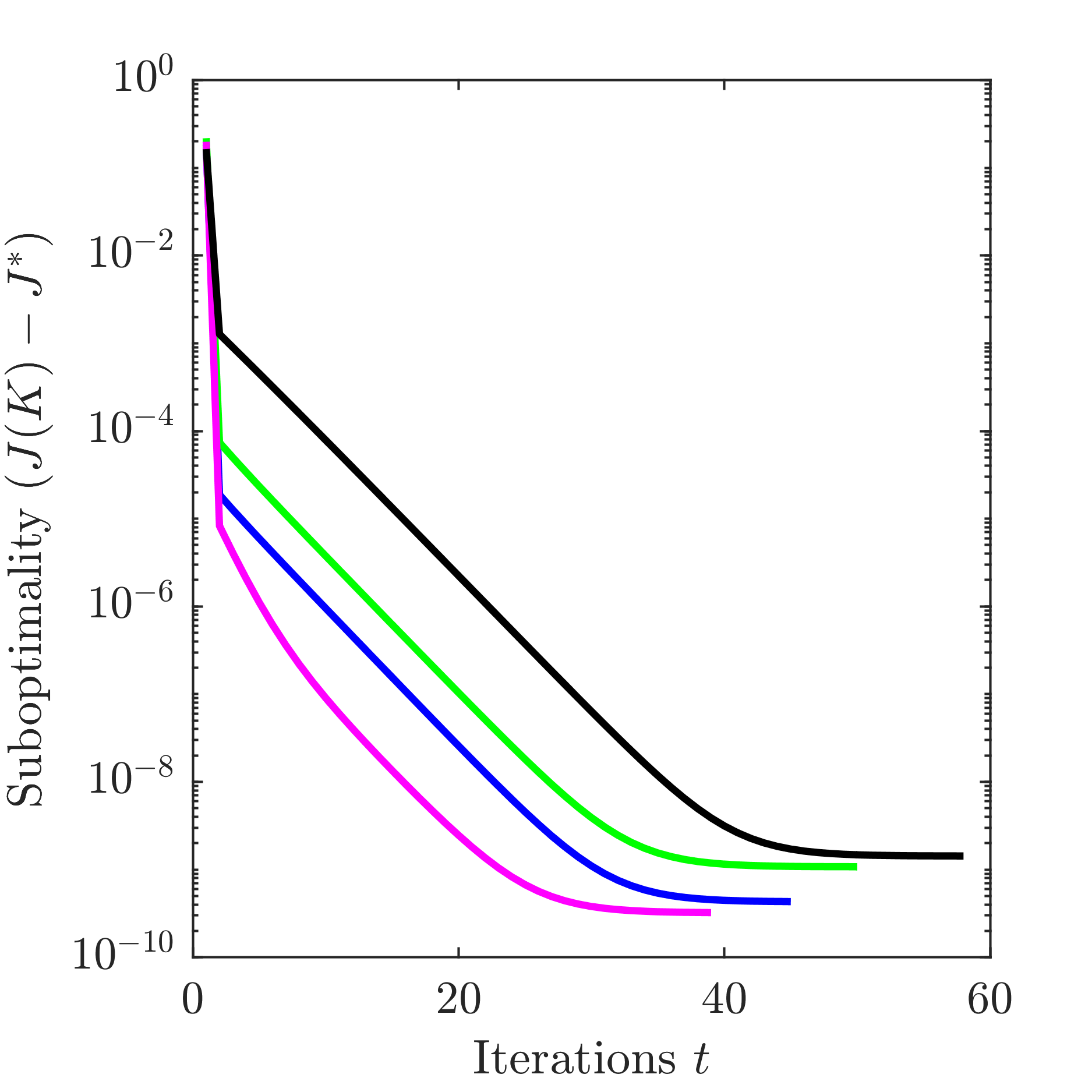}
         \caption{\texttt{Vanilla GD$_B$}}
\end{subfigure}
    \caption{Convergence performance of gradient descent algorithms for \cref{example:saddle_point_vanishing_Hessian} with different initialization strategies. In each subfigure, the left one shows results using random initialization, and the right one show results using initialization around a globally optimal point.} % (a) Standard gradient descent; (b) Gradient descent over the controllable canonical form.}
    \label{fig:Example6}
\end{figure}

%{\color{red} included different initialization}

\begin{figure}[t]
    \centering
        \centering
\begin{subfigure}{.49\textwidth}
    \includegraphics[scale = 0.47]{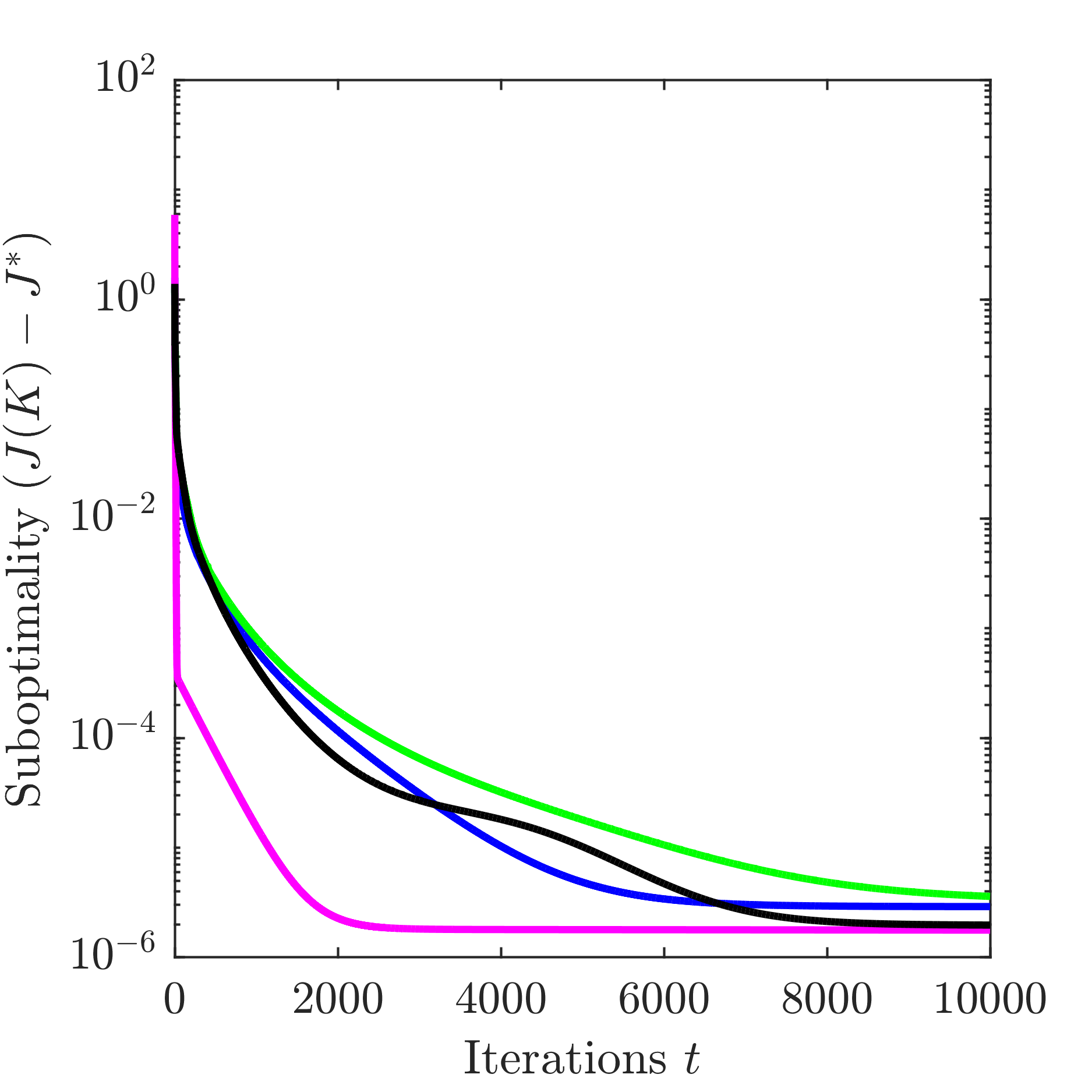}
    \includegraphics[scale = 0.47]{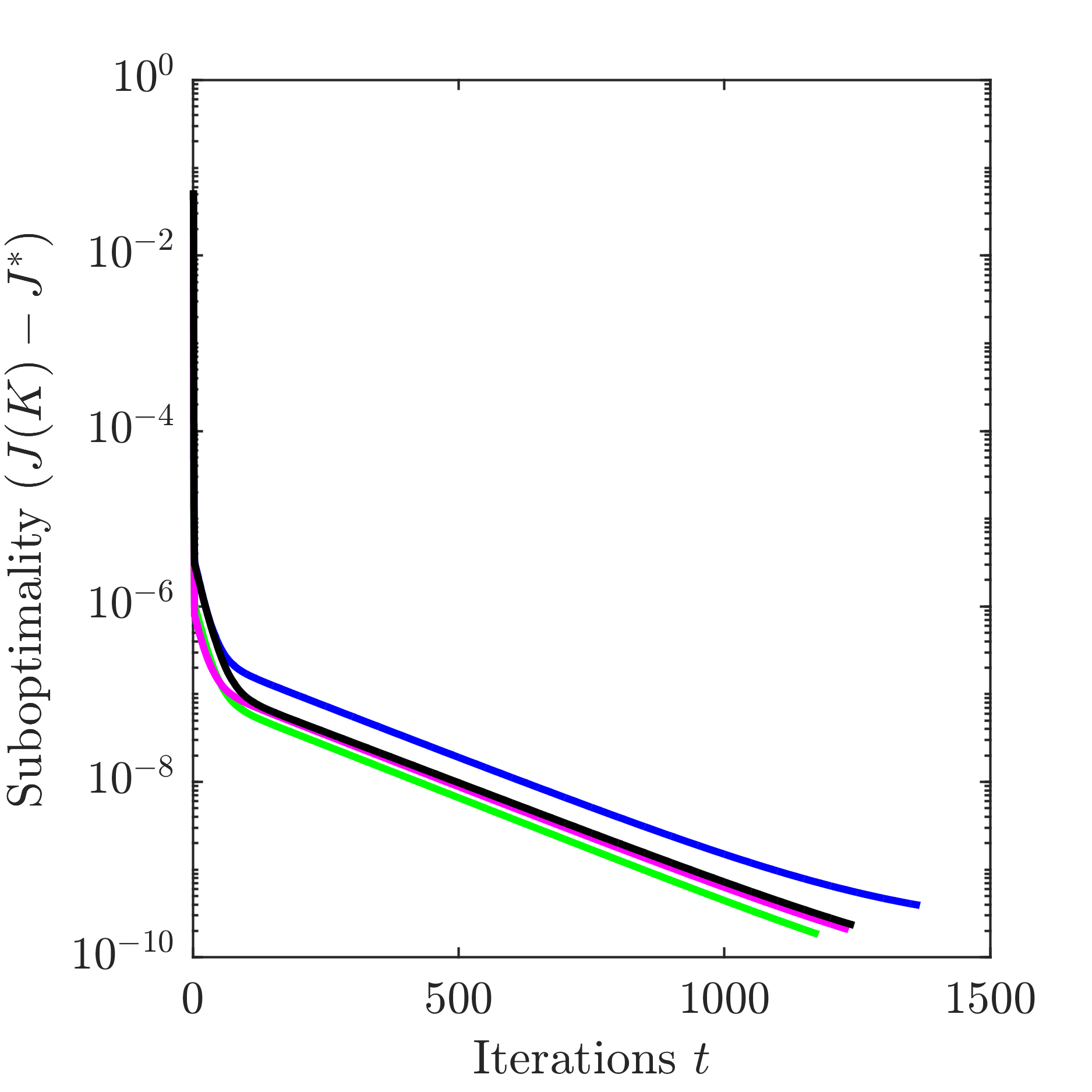}
     \caption{\texttt{Vanilla GD$_A$}}
    \end{subfigure}
    \hspace{1mm}
\begin{subfigure}{.49\textwidth}
    \includegraphics[scale = 0.47]{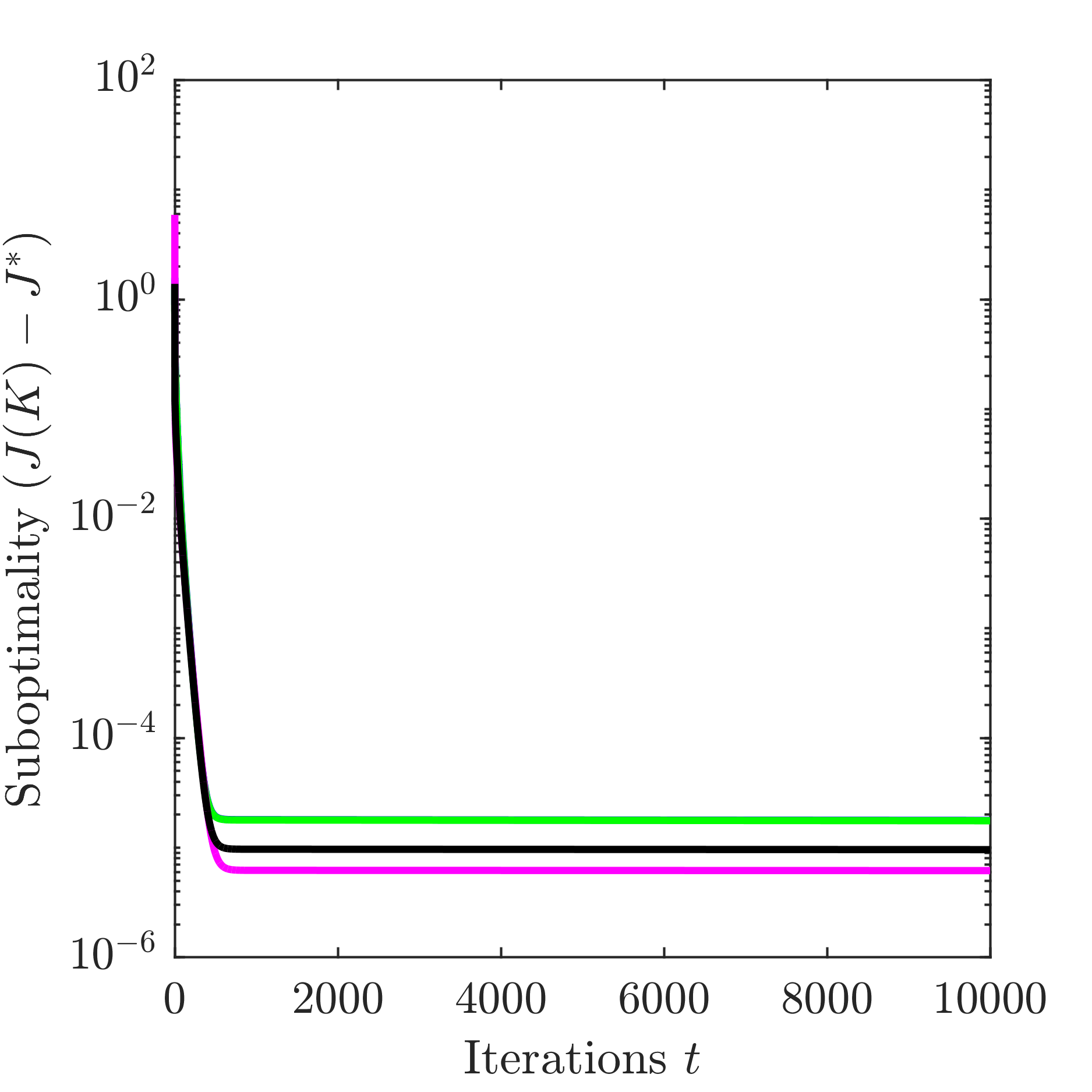}
     \includegraphics[scale = 0.47]{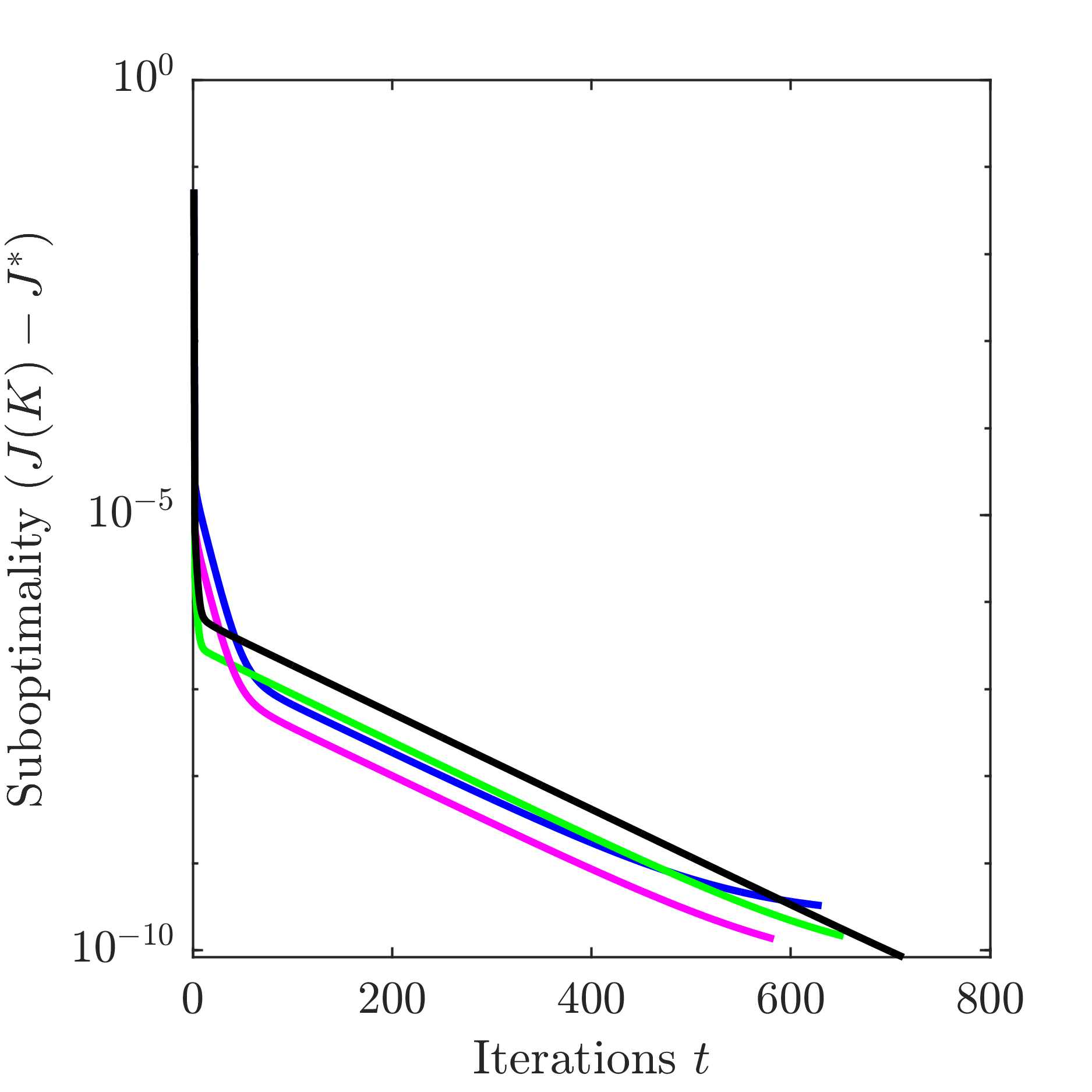}
         \caption{\texttt{Vanilla GD$_B$}}
\end{subfigure}
% \begin{subfigure}{.4\textwidth}
%     \includegraphics[width = 0.8\textwidth]{figs/Fig_Example8_ini_1.png}
%      \caption{\texttt{Vanilla GD$_A$}}
%     \end{subfigure}
%     \hspace{15mm}
% \begin{subfigure}{.4\textwidth}
%     \includegraphics[width = 0.8\textwidth]{figs/Fig_Example8_ini_2.png}
%          \caption{\texttt{Vanilla GD$_B$}}
% \end{subfigure}
    \caption{Convergence performance of gradient descent algorithms for \cref{example:bad_Hessian} ($\epsilon = 0.5$) with four different initialization $\mK_0$. In each subfigure, the left one shows results using random initialization, and the right one show results using initialization around a globally optimal point.} % (a) Standard gradient descent; (b) Gradient descent over the controllable canonical form.}
    \label{fig:Example8}
\end{figure}

Our final numerical experiment is carried out for the LQG in \cref{example:bad_Hessian}, where we chose $\epsilon = 0.5$. The results are shown in \cref{fig:Example8}. Both \texttt{Vanilla GD$_A$} and \texttt{Vanilla GD$_B$} failed to converge with $10^4$ iterations, and they seems to get stuck around different points for very many iterations that are not globally optimal. Similar to the previous case,  using the initialization around a globally optimal point greatly improved the convergence performance of  \texttt{Vanilla GD$_A$} and \texttt{Vanilla GD$_B$}, and both of them reached the stopping criterion within a few hundred iterations.

These two LQG cases show that initialization has a great impact on the performance of gradient algorithms for solving general LQG problems. We also note that for the LQG cases we tested, gradient descent algorithms can reduce the LQG cost quickly in the beginning period of iterations, but might get struck in some region for many iterations. % Further theoretical investigations are required to explain these behaviors.

\section{Conclusion} \label{section:conclusion}

In this paper, we have characterized the connectivity of the set of stabilizing controllers $\mathcal{C}_n$ and provided some structural properties of the LQG cost function. These results reveal rich yet complicated optimization landscape properties of the LQG problem. Ongoing work includes establishing convergence conditions for gradient descent algorithms and investigating whether local search algorithms can escape saddle points of the LQG problem.
We note that the optimization landscape of LQG also depends on the parameterization of dynamical controllers. It will be interesting to look into the LQG problem when parameterizing controllers in a canonical form. % Our numerical experiments in~\cite{ARXIV}
Finally, our analysis reveals that \emph{minimal} stationary points in $\mathcal{C}_n$ are always globally optimal, and it would also be interesting to investigate the existence of minimal stationary points for the LQG problem.

% {\color{red} the reciprocal conditional number in~\eqref{eq:reciprocal_condition_number} might be related to the minimal singular value of the controllability Gramian and Observability Gramian; but the relationship appears to be very complicated}

%{\color{red} Can we identify a class of LQG problems for which the global optimal controllers are always reduced-order? }%for example, in this example, the plant seems very symmetric $B = C^\tr, W = Q, V = R, A = A^\tr$. Can we prove that this kind of systems only has reduced-order global optimal controllers. }

% {\color{red}
% finding and classifying all stationary points

% restriction to controllable/observable canonical form

% natural policy gradient

% other geometry of $\mathcal{C}_n$

% other ways to parametrize the set of stabilizing controllers

% investigate the properties of controllers in canonical form
% }

\vspace{8mm}

\bibliographystyle{unsrt}
\bibliography{biblio.bib}

\newpage
\appendix
%\appendix

%\renewcommand{\theequation}{\thechapter.\arabic{equation}}

\numberwithin{equation}{section}

\vspace{10mm}
\noindent\textbf{\Large Appendix}

\vspace{5mm}
This appendix is divided into four parts:
\begin{itemize}
    \item \cref{appendix:preliminaries} presents some preliminaries in control theory and differential geometry;
    \item \cref{appendix:auxiliary_results} presents auxiliary proofs/results for continuous-time systems;
    \item \cref{appendix:proper_controller} presents the connectivity results for proper stabilizing controllers;
    \item \cref{appendix:discrete_time} presents analogous landscape results for the LQG problem in discrete-time.
\end{itemize}

% \section{LQG in $\mathcal{H}_2$ formulation}

% add a few descriptions
\section{Fundamentals of Control Theory and Differential Geometry} \label{appendix:preliminaries}

For self-completeness, this section reviews some fundamental notions in control theory (see~\cite[Chapter 3]{zhou1996robust} for more details), as well as some basic notions from differential geometry \cite{lee2013introduction,milnor1997topology}.

\subsection{Controllability, Observability, and Minimal Systems} \label{App:control_basics}

Consider a dynamical system, parameterized by $(A, B, C, D) \in \mathbb{R}^{n \times n} \times \mathbb{R}^{n \times m} \times \mathbb{R}^{p \times n} \times \mathbb{R}^{p \times m}$,  as follows
\begin{equation} \label{eq:App_dynamics}
    \begin{aligned}
        \dot{x} &= A x + Bu,\\
        y &= Cx + Du.
    \end{aligned}
\end{equation}
The system~\eqref{eq:App_dynamics} is called \emph{controllable} if the following controllability matrix is of full row rank
$$
    \text{rank}\left(\begin{bmatrix} B & AB & \ldots & A^{n-1}B \end{bmatrix}\right) = n,
$$
and \emph{observable} if the following observability matrix is of full column rank
$$
    \text{rank}\left(\begin{bmatrix} C \\ CA \\ \vdots \\ CA^{n-1}\end{bmatrix}\right) = n.
$$

The input-output behavior of~\eqref{eq:App_dynamics} can also be equivalently described in the frequency domain
\begin{equation} \label{eq:app_transfer_function}
    \mathbf{G}(s) = C(sI - A)^{-1}B + D.
\end{equation}
It is easy to verify that the transfer function $\mathbf{G}(s)$ is invariant under any similarity transformation on the state-space model $(TAT^{-1}, TB, CT^{-1}, D)$.

%\begin{definition}
    System~\eqref{eq:App_dynamics} is called \emph{minimal} if and only if it is controllable and observable.
%\end{definition}
%
This \emph{``minimal''} notion is justified by the following interpretation: if system~\eqref{eq:App_dynamics} is not minimal, then there exists another state-space model with a smaller state dimension $\hat{n} < n$
$$
    \begin{aligned}
        \dot{\hat{x}} &= \hat{A} \hat{x} + \hat{B} u \\
        y &= \hat{C} \hat{x} + D u,
    \end{aligned}
$$
such that the input-output behavior is the same as~\eqref{eq:App_dynamics}, \emph{i.e.},
$
    \mathbf{G}(s) = \hat{C} (sI - \hat{A})^{-1}\hat{B} + D.
$
In this paper, we have used the notions of \emph{``minimal controller''} and \emph{``controllable and observable controller''} in an interchangeabe way.
The following theorem shows that minimal realizations of a transfer matrix are identical up to a similarity transformation.
\begin{theorem}[{\cite[Theorem 3.17]{zhou1996robust}}] \label{theo:minimal_state_space_similarity_transformation}
    Given a real rational transfer matrix $\mathbf{G}(s)$, suppose that $(A_1, B_1, C_1, D_1)$ and $(A_2, B_2, C_2, D_2)$ are two minimal state-space realizations of $\mathbf{G}(s)$. Then, there exists a unique invertible matrix $T$, such that
    $$
        A_2 = TA_1T^{-1}, \qquad B_2 = TB_1, \qquad C_2 = C_1T^{-1}, \qquad D_2=D_1.
    $$
\end{theorem}
Finally, the system~\eqref{eq:App_dynamics} is \emph{proper} in the sense that the degree of the numerator in~\eqref{eq:app_transfer_function} does not exceed the degree of its denominator. The system~\eqref{eq:App_dynamics} becomes \emph{strictly proper} if $D = 0$.

\subsection{Lyapunov Equations}

%The Lyapunov theory is fundamental in many control problems.
Given a real matrix $A \in \mathbb{R}^{n\times n}$ and a symmetric matrix $Q \in \mathbb{S}^n$, we consider the following Lyapunov
equation
\begin{equation} \label{eq:Lyapunov_equation}
    A^\tr X + X A + Q = 0.
\end{equation}
Its vectorized version is
\begin{equation} \label{eq:Lyapunov_equation_vec}
    ( I_n \otimes A^\tr  + A^\tr \otimes I_n)\operatorname{vec}(X) = -\operatorname{vec}(Q),
\end{equation}
where we use $\otimes$ to denote Kronecker product. It can be shown that if $A$ is stable, then $( I_n \otimes A^\tr  + A^\tr \otimes I_n)$ is invertible, and thus from~\eqref{eq:Lyapunov_equation_vec}, the Lyapunov equation~\eqref{eq:Lyapunov_equation} admits a unique solution for any matrix $Q$. Furthermore, we have the following results on the positive semidefiniteness of the solution $X$.
\begin{lemma}[{\cite[Lemma 3.18]{zhou1996robust}}] \label{lemma:Lyapunov}
   Consider the Lyapunov equation~\eqref{eq:Lyapunov_equation}. Assuming that $A$ is stable, the following statements hold.
   \begin{itemize}
       \item The unique solution is
       $$X = \int_0^\infty e^{A^\tr t}Q e^{At} dt.$$
       \item $X \succ 0$ if $Q \succ 0$, and $X \succeq 0$ if $Q \succeq 0$.
       \item If $Q \succeq 0$, then $X \succ 0$ if and only if $(Q^{1/2},A)$ is observable.
   \end{itemize}
\end{lemma}

Given the solution to the Lyapunov equation~\eqref{eq:Lyapunov_equation}, there also exist converse results that establish the stability property of the matrix $A$; see~\cite[Lemma 3.19]{zhou1996robust}.

\subsection{Manifolds and Lie Groups}
\label{app:manifold}

We adopt the following definitions for manifolds in Euclidean spaces. % \cite{milnor1997topology}.
We refer to \cite{lee2013introduction,milnor1997topology} for more details of these definitions and related results.
\begin{definition}[$C^\infty$ maps and diffeomorphism]
Let $\mathcal{E}$ and $\mathcal{F}$ be two real Euclidean spaces, and let $\mathcal{X}\subset\mathcal{E}$ and $\mathcal{Y}\subseteq\mathcal{F}$ be subsets of $\mathcal{E}$ and $\mathcal{F}$ respectively. We say that a map $\phi:\mathcal{X}\rightarrow\mathcal{Y}$ is $C^\infty$, if for any $p\in\mathcal{X}$, there exists an open neighborhood $\mathcal{U}$ of $p$ in $\mathcal{E}$ and an indefinitely differentiable function $\tilde{\phi}:\mathcal{U}\rightarrow\mathcal{F}$ that coincides with $\phi$ on $\mathcal{U}\cap\mathcal{X}$. We say that a $C^\infty$ map $\phi:\mathcal{X}\rightarrow\mathcal{Y}$ is a \emph{diffeomorphism} from $\mathcal{X}$ to $\mathcal{Y}$, if $\phi$ has an inverse map $\phi^{-1}:\mathcal{Y}\rightarrow\mathcal{X}$ that is $C^\infty$. We say that $\mathcal{X}$ and $\mathcal{Y}$ are \emph{diffeomorphic} if there exists a diffeomorphism from $\mathcal{X}$ to $\mathcal{Y}$.
\end{definition}
\begin{definition}[Manifold and submanifold]
Let $\mathcal{E}$ be a real Euclidean space. A subset $\mathcal{M}\subset \mathcal{E}$ is said to be a $C^\infty$ \emph{manifold} of dimension $k$ in $\mathcal{E}$ , if for any $p\in\mathcal{M}$, there exists an open neighborhood $\mathcal{U}$ of $p$ in $\mathcal{E}$, such that $\mathcal{U}\cap\mathcal{M}$ is diffeomorphic to some open subset of $\mathbb{R}^k$.

Let $\mathcal{M}\subseteq\mathcal{E}$ be a $C^\infty$ manifold in the real Euclidean space $\mathcal{E}$. A subset $\mathcal{N}\subseteq\mathcal{M}$ is said to be a $C^\infty$ \emph{(embedded) submanifold} of $\mathcal{M}$ if it is a manifold in the real Euclidean space $\mathcal{E}$.
\end{definition}

\begin{definition}[Tangent space]
    Let $\mathcal{M} \subseteq \mathcal{E}$ be a $C^\infty$ manifold in a real Euclidean space $\mathcal{E}$. Given $x \in \mathcal{M}$, we say that $v \in \mathcal{E}$ is a \emph{tangent vector} of $\mathcal{M}$ at $x$, if there exists a $C^\infty$ curve $\gamma: (-1,1) \rightarrow \mathcal{M}$ with $\gamma(0) = x$ and $v = \gamma'(0)$. The set of tangent vectors of $\mathcal{M}$ at $x$ is called the \emph{tangent space} of $\mathcal{M}$ at $x$, which we denoted by $\mathcal{T}_x\mathcal{M}$.
\end{definition}

It is a known fact in differential geometry that the dimension of the tangent space is equal to the dimension of the manifold.

\begin{definition}[Tangent map]
Let $\mathcal{M}\subseteq\mathcal{E}$ and $\mathcal{N}\subseteq\mathcal{F}$ be two $C^\infty$ manifolds in real Euclidean spaces $\mathcal{E}$ and $\mathcal{F}$ respectively. let $\phi:\mathcal{M}\rightarrow\mathcal{N}$ be a $C^\infty$ map. For any $x\in\mathcal{M}$, the \emph{tangent map} of $\phi$ at $x$ is the linear map $d\phi_x:\mathcal{T}_x\mathcal{M}\rightarrow \mathcal{T}_{\phi(x)}\mathcal{N}$ defined by
$$
d\phi_x (\gamma'(0))
=\left.\frac{d(\phi\circ\gamma(t))}{dt}\right|_{t=0}
$$
for any $C^\infty$ curve $\gamma:(-1,1)\rightarrow\mathcal{M}$ with $\gamma(0)=x$.
\end{definition}

It is known in differential geometry that, if $\phi:\mathcal{M}\rightarrow\mathcal{N}$ is a diffeomorphism, then $d\phi_x$ is an isomorphism (a bijective linear map) from $\mathcal{T}_x\mathcal{M}$ to $\mathcal{T}_{\phi(x)}\mathcal{N}$.

\begin{definition}[Lie group]
A $C^\infty$ manifold $\mathcal{G}$ is said to be a Lie group, if there exists a $C^\infty$ binary operation $\cdot:\mathcal{G}\times\mathcal{G}\rightarrow\mathcal{G}$, such that the following group axioms are satisfied:
\begin{enumerate}
\item associativity: $(x\cdot y)\cdot z=x\cdot(y\cdot z)$ for all $x,y,z\in\mathcal{G}$;
\item identity: there exists $e\in\mathcal{G}$ such that $e\cdot x=x\cdot e=x$ for all $x\in\mathcal{G}$;
\item inverse: for all $x\in\mathcal{G}$ there exists a unique $x^{-1}\in\mathcal{G}$ such that $x\cdot x^{-1}=x^{-1}\cdot x=e$;
\end{enumerate}
and moreover, the inversion $x\mapsto x^{-1}$ is a $C^\infty$ map from $\mathcal{G}$ to $\mathcal{G}$.
\end{definition}

In this paper, we extensively use the Lie group
$\mathrm{GL}_q$ which is the set of $q \times q$ (real) invertible matrices together with the ordinary matrix multiplication. $\mathrm{GL}_q$ is a Lie group whose elements are organized continuously and smoothly. In addition, $\mathrm{GL}_q$ is also a $q^2$-dimensional manifold, where the group operations of multiplication and inversion are smooth maps.

\begin{definition}[Lie group action]
Let $\mathcal{M}$ be a $C^\infty$ manifold, and let $\mathcal{G}$ be a Lie group with identity $e\in\mathcal{G}$. We say that a $C^\infty$ map $\mathscr{T}:\mathcal{G}\times\mathcal{M}\rightarrow\mathcal{M}$ gives a \emph{(left) Lie group action}, if
$
\mathscr{T}(e,x)=x
$
and
$
\mathscr{T}(u\cdot v,x)=\mathscr{T}(u,\mathscr{T}(v,x))
$
for all $x\in\mathcal{M}$ and $u,v\in\mathcal{G}$.
\end{definition}

As an example, the similarity transformation $\mathscr{T}_q(T,K)$ defined in~\eqref{eq:def_sim_transform} gives a Lie group action of $\mathrm{GL}_q$ on $\mathcal{C}_q$.

\section{Auxiliary Results for Continuous-time Systems} \label{appendix:auxiliary_results}

This section presents some auxiliary proofs/results for continuous-time systems.

\subsection{Proofs of Lemmas~\ref{lemma:LQG_cost_formulation1} and~\ref{lemma:LQG_cost_analytical}}\label{appendix:real_analytic}

We first prove the LQG cost formulation in \cref{lemma:LQG_cost_formulation1}. Given a stabilizing controller $\mK \in \mathcal{C}_q$, the closed-loop system is shown in~\eqref{eq:closed-loop_LQG}. Since the controller $\mK$ internally stabilizes the plant, the closed-loop matrix
$$
    A_{\mathrm{cl},\mK} := \begin{bmatrix} A &  BC_{\mK} \\ B_{\mK} C & A_{\mK} \end{bmatrix}
$$
is stable and the state variable $(x(t), \xi(t))$ is a Gaussian process with mean satisfying
$$
    \lim_{t \rightarrow \infty} \mathbb{E}\left(\begin{bmatrix} x(t) \\ \xi(t) \end{bmatrix} \right) = 0,
$$
and covariance satisfying
\begin{equation} \label{eq:co-variance}
\begin{aligned}
    \lim_{t \rightarrow \infty} \mathbb{E}\left(\begin{bmatrix} x(t) \\ \xi(t) \end{bmatrix}\begin{bmatrix} x(t) \\ \xi(t) \end{bmatrix}^\tr \right) &= \lim_{t \rightarrow \infty} \int_{0}^t e^{A_{\mathrm{cl},\mK}(t - \tau)}\begin{bmatrix}W & \\ & B_{\mK} V B_{\mK}^\tr  \end{bmatrix}e^{A_{\mathrm{cl},\mK}^\tr (t - \tau)} d\tau\\
    &= \int_{0}^\infty e^{A_{\mathrm{cl},\mK}t}\begin{bmatrix}W & \\ & B_{\mK} V B_{\mK}^\tr  \end{bmatrix}e^{A_{\mathrm{cl},\mK}^\tr t} dt.
\end{aligned}
\end{equation}
By \cref{lemma:Lyapunov}, the last expression in~\eqref{eq:co-variance} is the same as the unique solution $X_{\mK}$ to the Lyapunov equation~\eqref{eq:LyapunovX}.

Therefore, the corresponding LQG cost is given by
$$
\begin{aligned}
    J_q := \lim_{T \rightarrow \infty }\frac{1}{T}\,\mathbb{E} \!\left[\int_{t=0}^T \left(x^\tr Q x + u^\tr R u\right)dt\right] %& = \lim_{T \rightarrow \infty }\frac{1}{T}\mathbb{E} \left[\int_{t=0}^T \left(\begin{bmatrix}x \\ \xi \end{bmatrix}^\tr  \begin{bmatrix}Q & \\ & C_{\mK}^\tr RC_{\mK} \end{bmatrix}\begin{bmatrix}x \\ \xi \end{bmatrix}\right)dt\right]  \\
    &=  \lim_{t\rightarrow \infty} \mathbb{E}\!\left(\begin{bmatrix}x \\ \xi \end{bmatrix}^\tr  \begin{bmatrix}Q & \\ & C_{\mK}^\tr RC_{\mK} \end{bmatrix}\begin{bmatrix}x \\ \xi \end{bmatrix}\right) \\
    &= \lim_{t\rightarrow \infty} \mathbb{E} \operatorname{tr}\left( \begin{bmatrix}Q & \\ & C_{\mK}^\tr RC_{\mK} \end{bmatrix}\begin{bmatrix}x \\ \xi \end{bmatrix}\begin{bmatrix}x \\ \xi \end{bmatrix}^\tr \right) \\
    &= \operatorname{tr} \left( \begin{bmatrix}Q & \\ & C_{\mK}^\tr RC_{\mK} \end{bmatrix} \lim_{t\rightarrow \infty} \mathbb{E} \left(\begin{bmatrix}x \\ \xi \end{bmatrix}\begin{bmatrix}x \\ \xi \end{bmatrix}^\tr \right) \right)\\
    &= \operatorname{tr} \left( \begin{bmatrix}Q & \\ & C_{\mK}^\tr RC_{\mK} \end{bmatrix} X_{\mK}\right).\\
\end{aligned}
$$
The other expression of the LQG cost in \cref{lemma:LQG_cost_formulation1} follows from the Lyapunov function~\eqref{eq:LyapunovY} by duality between controllability Gramian and observability Gramian.

We now proceed to prove \cref{lemma:LQG_cost_analytical}. First, upon vectorizing the Lyapunov equation~\eqref{eq:LyapunovX}, we have
    $$
       \left(I_{n+q} \otimes A_{\mathrm{cl},\mK} + A_{\mathrm{cl},\mK} \otimes I_{n+q}\right) \operatorname{vec}(X_{\mK}) = -\text{vec}\left( \begin{bmatrix} W & 0 \\ 0 & B_{\mK}VB_{\mK}^\tr  \end{bmatrix}\right).
    $$
Since $A_{\mathrm{cl},\mK}$ is stable, we know that $I_{n+q} \otimes A_{\mathrm{cl},\mK} + A_{\mathrm{cl},\mK} \otimes I_{n+q}$ is invertible, and thus we have
$$
    \operatorname{vec}(X_{\mK}) = - \left(I_{n+q} \otimes A_{\mathrm{cl},\mK} + A_{\mathrm{cl},\mK} \otimes I_{n+q}\right)^{-1}\text{vec}\left( \begin{bmatrix} W & 0 \\ 0 & B_{\mK}VB_{\mK}^\tr  \end{bmatrix}\right).
$$
It is not difficult to see that each element of $\left(I_{n+q} \otimes A_{\mathrm{cl},\mK} + A_{\mathrm{cl},\mK} \otimes I_{n+q}\right)^{-1}$ is a rational function of the elements of $\mK$. Therefore, the LQG cost function
$$
J_q(\mK)
=
\operatorname{tr}
\left(
\begin{bmatrix}
Q & 0 \\ 0 & C_{\mK}^\tr R C_{\mK}
\end{bmatrix} X_\mK\right)
$$
is a rational function of the elements of $\mK$, which is real analytical.

%Consider a sequence of controllers $\mK_j \in \mathcal{C}_q$ such that $\lim_{j \rightarrow \infty} \mK_j = \hat{\mK} \in \partial\mathcal{C}_q$. Let $X_{\mK_j} \succeq 0$ be the corresponding solutions to the Laypunov equation~\eqref{eq:LyapunovX}.

\subsection{Proof of \cref{proposition:Phi_surjective}} \label{app:surjective_proof}
%\begin{proof}
It is straightforward to see that $\Phi(\cdot)$ is continuous since each element of $\Phi(\mZ)$ is a rational function in terms of the elements of $\mZ$ (a ratio of two polynomials).
To show that $\Phi$ is a mapping onto $\mathcal{C}_n$, we need to prove the following statements:
\begin{enumerate}
    \item For all $ \mK \in \mathcal{C}_n$, there exists $\mZ=(X,Y,M,G,H,F,\Pi,\Xi)\in\mathcal{G}_n$ such that
    $
        \Phi(\mZ) = \mK
    $.
    \item For all $ \mZ=(X,Y,M,G,H,F,\Pi,\Xi)\in\mathcal{G}_n$, we have $\Phi(\mZ)\in \mathcal{C}_n$.
\end{enumerate}

To show the first statement, let $\mK= \begin{bmatrix}
D_{\mK} & C{_\mK} \\ B{_\mK} & A{_\mK}
\end{bmatrix}\in\mathcal{C}_n$ be arbitrary. By definition we have $D_{\mK}=0$, and the stability of the matrix $\begin{bmatrix}
A &  BC{_\mK} \\
B{_\mK} C & A{_\mK}
\end{bmatrix}$ implies that the Lyapunov inequality
\begin{equation} \label{eq:LMIstability}
\begin{bmatrix}
A + BD{_\mK}C &  BC{_\mK} \\
B{_\mK} C & A{_\mK}
\end{bmatrix}
\begin{bmatrix}
X & \Pi^\tr \\ \Pi & \hat{X}
\end{bmatrix}
+
\begin{bmatrix}
X & \Pi^\tr \\ \Pi & \hat{X}
\end{bmatrix}
\begin{bmatrix}
A + BD{_\mK}C &  BC{_\mK} \\
B{_\mK} C & A{_\mK}
\end{bmatrix}^\tr \prec 0
\end{equation}
has a solution $\begin{bmatrix} X & \Pi^\tr \\ \Pi & \hat{X} \end{bmatrix} \succ 0$. Without loss of generality we may assume that $\det \Pi\neq 0$ (otherwise we can add a small perturbation on $\Pi$ to make it invertible while still preserving the inequality \eqref{eq:LMIstability}). Upon defining
$$
\begin{bmatrix} Y & \Xi \\ \Xi^\tr & \hat{Y} \end{bmatrix}
:=\begin{bmatrix} X & \Pi^\tr \\ \Pi & \hat{X} \end{bmatrix}^{-1},
\qquad
T :=
\begin{bmatrix} X & \Pi^\tr \\ \Pi & \hat{X} \end{bmatrix}^{-1}\begin{bmatrix} X & I \\ \Pi & 0 \end{bmatrix}
=\begin{bmatrix} I & Y \\ 0 & \Xi^\tr \end{bmatrix},
$$
we can verify that
\begin{equation}\label{eq:Gset_cond1}
YX+\Xi\Pi=I,
\qquad
T^\tr  \begin{bmatrix} X & \Pi^\tr \\ \Pi & \hat{X} \end{bmatrix} T = \begin{bmatrix} X & I \\ I & Y \end{bmatrix}
\succ 0.
\end{equation}
%and that
% \begin{equation}\label{eq:Phi_surjective_temp1}
% \begin{aligned}
% &T^\tr  \begin{bmatrix} A + BD{_\mK}C &  BC{_\mK} \\ B{_\mK} C & A{_\mK} \end{bmatrix} \begin{bmatrix} X & U^\tr \\ U & \hat{X} \end{bmatrix} T \\
% = &
% \begin{bmatrix} I & 0 \\ Y & V \end{bmatrix}  \begin{bmatrix} A + BD{_\mK}C &  BC{_\mK} \\ B{_\mK} C & A{_\mK} \end{bmatrix} \begin{bmatrix} X & I \\ U & 0 \end{bmatrix} \\
% =  &
% \begin{bmatrix} AX + BD{_\mK}CX + BC{_\mK}U  & A + BD{_\mK}C \\ Y(A + BD{_\mK}C)X + VB{_\mK} CX + Y BC{_\mK}U+ V A{_\mK}U & Y(A + BD{_\mK}C) + VB{_\mK} C\end{bmatrix}.
% \end{aligned}
% \end{equation}
Upon letting
\begin{equation}\label{eq:change_of_var}
\begin{aligned}
M &= Y(A + BD{_\mK}C)X + \Xi B{_\mK} CX + YBC{_\mK}\Pi+ \Xi A{_\mK}\Pi, \\
G &= D{_\mK}, \\
H &= YBD{_\mK} + \Xi B{_\mK}, \\
F & = D{_\mK}CX + C{_\mK}\Pi,\\
\end{aligned}
\end{equation}
we can also verify that
\begin{equation}\label{eq:proof_Phi_map_temp}
T^\tr  \begin{bmatrix} A + BD{_\mK}C &  BC{_\mK} \\ B{_\mK} C & A{_\mK} \end{bmatrix} \begin{bmatrix} X & \Pi^\tr \\ \Pi & \hat{X} \end{bmatrix} T  =  \begin{bmatrix} AX + BF  & A + BGC \\ M & Y A + HC\end{bmatrix}.
\end{equation}
Combining \eqref{eq:proof_Phi_map_temp} with \eqref{eq:LMIstability} and \eqref{eq:Gset_cond1}, we see that $\mZ=(X,Y,M,G,H,F,\Pi,\Xi)\in\mathcal{G}_n$ by the definition of $\mathcal{G}_n$.
Note that the change of variables \eqref{eq:change_of_var} can be compactly represented as
$$
\begin{bmatrix} G & F\\ H & M \end{bmatrix}
=
\begin{bmatrix} I & 0 \\ YB & \Xi \end{bmatrix}
\begin{bmatrix} D{_\mK} & C{_\mK} \\ B{_\mK} & A{_\mK}\end{bmatrix}
\begin{bmatrix} I & CX \\ 0 & \Pi \end{bmatrix}
+
\begin{bmatrix} 0 & 0\\0 & YAX \end{bmatrix},
$$
and with the guarantee in \cref{lemma:connectivity_preliminary}, we see that
$$
\begin{bmatrix}
D{_\mK} & C{_\mK} \\
B{_\mK} & A{_\mK}
\end{bmatrix}
=
\begin{bmatrix} I & 0 \\
YB & \Xi \end{bmatrix}^{-1} \begin{bmatrix}
G & F \\
H & M-YAX \end{bmatrix}\begin{bmatrix} I & CX\\ 0 & \Pi \end{bmatrix}^{-1}
=
\begin{bmatrix}
\Phi_D(\mZ) & \Phi_C(\mZ) \\
\Phi_B(\mZ) & \Phi_A(\mZ)
\end{bmatrix}=\Phi(\mZ).
$$

We then prove the second statement. Let $\mZ=(X,Y,M,G,H,F,\Pi,\Xi)\in\mathcal{G}_n$ be arbitrary. Let $\hat{X}=\Pi(X-Y^{-1})^{-1}\Pi^\tr$, and it's straightforward to see that $\hat{X}\succ 0$ and
$$
\begin{bmatrix}
X & \Pi^\tr \\
\Pi & \hat{X}
\end{bmatrix}
\begin{bmatrix}
I & Y \\
0 & \Xi^\tr
\end{bmatrix}
=
\begin{bmatrix}
X & XY+\Pi^\tr \Xi^\tr \\
\Pi & \Pi Y+\hat{X}\Xi^\tr
\end{bmatrix}
=\begin{bmatrix}
X & I \\
\Pi & 0
\end{bmatrix},
$$
where we used the fact that
$$
\begin{aligned}
\Pi Y+\hat{X}\Xi^\tr =&\Pi Y+\Pi (X-Y^{-1})^{-1}\Pi^\tr \Xi^\tr \\
=&\Pi Y - \Pi(X-Y^{-1})^{-1}(XY-I) \\
=&\Pi Y - \Pi (X-Y^{-1})^{-1}(X-Y^{-1})Y \\
=&0.
\end{aligned}
$$ We also have
$$
\begin{bmatrix} G & F\\ H & M \end{bmatrix}
=
\begin{bmatrix} I & 0\\ YB & \Xi \end{bmatrix}
\begin{bmatrix}
\Phi_D(\mZ) & \Phi_C(\mZ) \\
\Phi_B(\mZ) & \Phi_A(\mZ)
\end{bmatrix}
\begin{bmatrix} I & CX\\ 0 & \Pi \end{bmatrix}
+
\begin{bmatrix} 0 & 0\\0 & YAX \end{bmatrix}
$$
from the definition of $\Phi$. Similarly as showing the equality \eqref{eq:proof_Phi_map_temp}, we can derive that
$$
\begin{aligned}
& \begin{bmatrix}
AX \!+\! BF & A \!+\! BGC \\ M & YA \!+\! HC
\end{bmatrix}
=
\begin{bmatrix} I & Y \\ 0 & \Xi^\tr \end{bmatrix}^\tr
\begin{bmatrix} A + BGC &  B\Phi_C(\mZ) \\ \Phi_B(\mZ) C & \Phi_A(\mZ) \end{bmatrix} \begin{bmatrix} X & \Pi^\tr \\ \Pi & \hat{X} \end{bmatrix}
\begin{bmatrix} I & Y \\ 0 & \Xi^\tr \end{bmatrix}.
\end{aligned}
$$
Then from the definition of $\mathcal{G}_n$, we can further get
$$
\begin{bmatrix} A + BGC &  B\Phi_C(\mZ) \\ \Phi_B(\mZ) C & \Phi_A(\mZ) \end{bmatrix} \begin{bmatrix} X & \Pi^\tr \\ \Pi & \hat{X} \end{bmatrix}
+ \begin{bmatrix} X & \Pi^\tr \\ \Pi & \hat{X} \end{bmatrix}
\begin{bmatrix} A + BGC &  B\Phi_C(\mZ) \\ \Phi_B(\mZ) C & \Phi_A(\mZ) \end{bmatrix}^\tr
\prec 0,
$$
and since $X-\Pi^\tr \hat{X}^{-1} \Pi
=Y^{-1}\succ 0$, the matrix $\begin{bmatrix} X & \Pi^\tr \\ \Pi & \hat{X} \end{bmatrix}$ is positive definite. We can now see that $\begin{bmatrix} A &  B\Phi_C(\mZ) \\ \Phi_B(\mZ) C & \Phi_A(\mZ) \end{bmatrix}$ satisfies the Lyapunov inequality and thus is stable, meaning that $\Phi(\mZ)\in \mathcal{C}_n$.
%\end{proof}

\subsection{A Second-Order SISO System for Which $\mathcal{C}_n$ Is Not Path-Connected}
\label{appendix:eg_disconnectivity}

Consider a second-order SISO plant with
\begin{equation} \label{eq:ExampleSISO2}
A=\begin{bmatrix}
0 & 1 \\ 1 & 0
\end{bmatrix},
\qquad
B=\begin{bmatrix}
0 \\ 1
\end{bmatrix},
\qquad
C=\begin{bmatrix}
0 & 1
\end{bmatrix}.
\end{equation}
For this case, any reduced-order controller in $\mathcal{C}_{1}$ and can be parameterized by
$
\mathsf{K} = \begin{bmatrix}
0 & C_{\mK} \\ B_{\mK} & A_{\mK}
\end{bmatrix}
$
for some $A_{\mK},B_{\mK},C_{\mK}\in\mathbb{R}$. We now show that the matrix \eqref{eq:closedloopmatrix}, given by
$$
\begin{bmatrix}
0 & 1 & 0 \\ 1 & 0 & C_{\mK} \\ 0 & B_{\mK} & A_{\mK}
\end{bmatrix},
$$
is not stable for any $A_{\mK},B_{\mK},C_{\mK}\in\mathbb{R}$, implying that $\mathcal{C}_{1}=\varnothing$. Indeed,
by the Routh–Hurwitz criterion, the characteristic polynomial
$$
\det\left(\lambda I_3 - \begin{bmatrix}
        0 & 1 & 0 \\
        1 & 0 & C_{\mK} \\
        0 & B_{\mK} & A_{\mK}
        \end{bmatrix}\right) = \lambda^3 - A_{\mK}\lambda^2 - (B_{\mK}C_{\mK}+1)\lambda  + A_{\mK}.
$$
has all roots in the open left half plane if and only if
$$
        -A_{\mK} > 0, \qquad A_{\mK} > 0,  \qquad A_{\mK}(B_{\mK}C_{\mK}+1) > A_{\mK},
$$
which are obviously infeasible. We can now conclude that $\mathcal{C}_n$ is not path-connected by \cref{Theo:connectivity_conditions} since the plant is SISO.

We can also directly prove the disconnectivity of $\mathcal{C}_n$ in this example. The set $\mathcal{C}_n=\mathcal{C}_2$ for~\eqref{eq:ExampleSISO2} can be written as
$$
    \mathcal{C}_{n} = \left\{
    \begin{bmatrix}
    0 & C_{\mK,1} & C_{\mK,2} \\
    B_{\mK,1} & A_{\mK,11} & A_{\mK,12} \\
    B_{\mK,2} & A_{\mK,21} & A_{\mK,22}
    \end{bmatrix}
    \in \mathbb{R}^{3\times 3}
    \left|
\begin{bmatrix}
0 & 1 & 0 & 0 \\
1 & 0 & C_{\mK,1} & C_{\mK,2} \\
0 & B_{\mK,1} & A_{\mK,11} & A_{\mK,12} \\
0 & B_{\mK,2} & A_{\mK,21} & A_{\mK,22}
\end{bmatrix}
\textrm{ is stable} \right.
\right\}.
$$
Obviously $B_{\mK}$ cannot be zero for any stabilizing controller in $\mathcal{C}_2$. Since for any $B_{\mK}\in\mathbb{R}^{2}\backslash\{0\}$, there exists $T\in\mathbb{R}^{2\times 2}$ with $\det T>0$ such that $TB_{\mK}=\begin{bmatrix}
0 \\ 1
\end{bmatrix}$, by the path-connectivity of the set $\{T\in\mathbb{R}^{2\times 2}:\det T>0\}$ \cite{lee2013introduction}, we can see that $\mathcal{C}_{n}$ is path-connected if and only if the set
$$
\mathcal{S}
=
\left\{
\hat{\mK}=\begin{bmatrix}
C_{\mK,1} & C_{\mK,2} \\
A_{\mK,11} & A_{\mK,12} \\
A_{\mK,21} & A_{\mK,22}
\end{bmatrix}
\in \mathbb{R}^{3\times 2}
\left|
\begin{bmatrix}
0 & 1 & 0 & 0 \\
1 & 0 & C_{\mK,1} & C_{\mK,2} \\
0 & 0 & A_{\mK,11} & A_{\mK,12} \\
0 & 1 & A_{\mK,21} & A_{\mK,22}
\end{bmatrix}\right.
\textrm{ is stable}
\right\}
$$
is path-connected. The Routh--Hurwitz stability criterion allows establishing an equivalent condition for the set $\mathcal{S}$ as
$$
\mathcal{S}
=\left\{\left.\hat{\mK}=\begin{bmatrix}
C_{\mK,1} & C_{\mK,2} \\
A_{\mK,11} & A_{\mK,12} \\
A_{\mK,21} & A_{\mK,22}
\end{bmatrix} \right|\;
p_1(\hat{\mK})>0, \;
p_2(\hat{\mK})>0, \;
p_3(\hat{\mK})>0, \;
p_4(\hat{\mK})>0 \right\},
$$
where
$$
\begin{aligned}
p_1(\hat{\mK})
=\ &
-A_{\mK,11}-A_{\mK,22}, \\
p_{2} (\hat{\mK})
=\ &
A_{\mK,11}+A_{\mK,22}
+ A_{\mK,11} C_{\mK,2} - A_{\mK,12} C_{\mK,1}, \\
p_{3} (\hat{\mK})
=\ &
(A_{\mK,11}+A_{\mK,22})^2
(A_{\mK,11}A_{\mK,22}-A_{\mK,12}A_{\mK,21}) \\
& - (A_{\mK,11}+A_{\mK,22} + A_{\mK,11} C_{\mK,2} - A_{\mK,12} C_{\mK,1}) \\
&
\ \ \ \ \times
[(A_{\mK,11}+A_{\mK,22})
(A_{\mK,11}A_{\mK,22}-A_{\mK,12}A_{\mK,21}-C_{\mK,2})+A_{\mK,11} C_{\mK,2} - A_{\mK,12}C_{\mK,2}], \\
p_{4}(\hat{\mK})
=\ &
-A_{\mK,11}A_{\mK,22}+A_{\mK,12}A_{\mK,21}.
\end{aligned}
$$
We first show that $A_{\mK,12}\neq 0$ for any $\hat{\mK}\in\mathcal{S}$. Indeed, if $A_{\mK,12}=0$, we then have
$$
p_{2}(\hat{\mK})
=
A_{\mK,11}+A_{\mK,22}
+A_{\mK,11} C_{\mK,2},
\qquad
p_4(\hat{\mK})
=-A_{\mK,11}A_{\mK,22},
$$
and
$$
\begin{aligned}
p_3(\hat{\mK})
=\ &
(A_{\mK,11}+A_{\mK,22})^2
A_{\mK,11}A_{\mK,22} \\
&
-
(A_{\mK,11}+A_{\mK,22}
+A_{\mK,11} C_{\mK,2})
[
(A_{\mK,11}+A_{\mK,22})
(A_{\mK,11}A_{\mK,22}-C_{\mK,1})
+A_{\mK,11}C_{\mK,2}
] \\
=\ &
A_{\mK,22}C_{\mK,2}
(A_{\mK,11} + A_{\mK,22} + A_{\mK,11}C_{\mK,2} - A_{\mK,11}^2(A_{\mK,11} + A_{\mK,22})).
\end{aligned}
$$
From $p_1(\hat{\mK})>0$ and $p_2(\hat{\mK})>0$, we get $A_{\mK,11}C_{\mK,2}>0$, and together with $p_4(\hat{\mK})>0$ and $p_3(\hat{\mK})>0$, we see that $A_{\mK,22}C_{\mK,2}<0$ and
$$
A_{\mK,11}+A_{\mK,22}+A_{\mK,11}C_{\mK,2}
<A_{\mK,11}^2(A_{\mK,11}+A_{\mK,22})<0,
$$
which contradicts $p_{2}(\hat{\mK})>0$. Thus $A_{\mK,12}\neq 0$ for any $\hat{\mK}\in\mathcal{S}$

On the other hand, let
$$
\hat{\mK}^{(1)}
=\begin{bmatrix}
-3/2 & -2 \\
0 & 1 \\ 1/8 & -1
\end{bmatrix},
\qquad
\hat{\mK}^{(2)}
=\begin{bmatrix}
3/2 & -2 \\
0 & -1 \\ -1/8 & -1
\end{bmatrix}.
$$
It can be checked that $\hat{\mK}^{(1)}$ and $\hat{\mK}^{(2)}$ are both in  $\mathcal{S}$. Now we see that $\mathcal{S}$ is not path-connected, since any continuous path connecting $\hat{\mK}^{(1)}$ and $\hat{\mK}^{(2)}$ must pass a point with $A_{\mK,12}=0$. Consequently, the set $\mathcal{C}_{2}$ is not path-connected for this example.

\subsection{The Gradient and the Hessian of $J_q(\mK)$}
\label{app:gradien_Jq}

We first introduce the following lemma.
\begin{lemma}\label{lemma:Taylor_Lyapunov_eq}
Suppose $M:(-\delta,\delta)\rightarrow\mathbb{R}^{k\times k}$ and $G:(-\delta,\delta)\rightarrow\mathbb{S}^k$ are two indefinitely differentiable matrix-valued functions for some $\delta>0$ and $k\in\mathbb{N}\backslash\{0\}$, and suppose $M(t)$ is stable for all $t\in(-\delta,\delta)$. Let $X(t)$ denote the solution to the following Lyapunov equation
$$
M(t)X(t)+X(t)M(t)^\tr+G(t)=0.
$$
Then $X(t)$ is indefinitely differentiable over $t\in(-\delta,\delta)$, and its $j$'th order derivative at $t=0$, denoted by $X^{(j)}(0)$, is the solution to the following Lyapunov equation
\begin{equation}\label{eq:Taylor_Lyapunov_eq}
\begin{aligned}
& M(0)X^{(j)}(0)
+X^{(j)}(0)M(0)^\tr \\
& \qquad+
\left(\sum_{i=1}^j \frac{j!}{i!(j-i)!}
\left(M^{(i)}(0)X^{(j-i)}(0)
+X^{(j-i)}(0)M^{(i)}(0)^\tr
\right)
+G^{(j)}(0)\right)=0.
\end{aligned}
\end{equation}
\end{lemma}
\begin{proof}[Proof of \cref{lemma:Taylor_Lyapunov_eq}]
The differentiability of $X(t)$ follows from the observation that the unique solution to the Lyapunov equation can be written as
% $$
% \operatorname{vec}(X(t))
% =-(I_k\otimes M^\tr(t) + M^\tr(t) \otimes I_k)^{-1}
% \operatorname{vec}(Q(t)).
% $$
$$
\operatorname{vec}(X(t))
=-(I_k\otimes M(t) + M(t) \otimes I_k)^{-1}
\operatorname{vec}(G(t)).
$$
Since $M(t)$, $G(t)$ and $X(t)$ are indefinitely differentiable, they admit Taylor expansions around $t=0$ given by
\begin{align*}
M(t) =\ &
\sum_{j=0}^a\frac{t^j}{j!}
M^{(j)}(0)
+o(t^a), \\
G(t) =\ &
\sum_{j=0}^a\frac{t^j}{j!}
G^{(j)}(0)
+o(t^a), \\
X(t) =\ &
\sum_{j=0}^a\frac{t^j}{j!}
X^{(j)}(0)
+o(t^a)
\end{align*}
for any $a\in\mathbb{N}$. By plugging these Taylor expansions into the original Lyapunov equation, after some algebraic manipulations, we can show that
$$
\sum_{j=0}^a t^j
\left[\sum_{i=0}^j \frac{1}{i!(j-i)!}
\left(M^{(i)}(0)X^{(j-i)}(0)
+X^{(j-i)}(0)M^{(i)}(0)^\tr
\right)
+\frac{1}{j!}G^{(j)}(0)
\right]
+o(t^a)
=0.
$$
Since the above equality holds for all sufficiently small $t$, we get
$$
\sum_{i=0}^j \frac{1}{i!(j-i)!}
\left(M^{(i)}(0)X^{(j-i)}(0)
+X^{(j-i)}(0)M^{(i)}(0)^\tr
\right)
+\frac{1}{j!}G^{(j)}(0)=0,
$$
which is the same as~\eqref{eq:Taylor_Lyapunov_eq}. Thus, $X^{(j)}(0)$ is a solution to the Lyapunov equation~\eqref{eq:Taylor_Lyapunov_eq}.
\end{proof}

% The gradient of $J_q$ can computed using its Lagrangian (see, e.g., ~\cite[Chpater 20]{zhou1996robust}). Here, we present a direct elementary calculation, which provides some additional insights in the LQG cost difference and its relations with Lyapunov equations.

Given any stabilizing controller $\mK \in\mathcal{C}_q$, we denote the closed-loop matrix as
$$
{A}_{\text{cl},\mK} = \begin{bmatrix}
A & BC_{\mK} \\ B_{\mK}C & A_{\mK}
\end{bmatrix} = \begin{bmatrix}
A & 0  \\ 0& 0
\end{bmatrix} + \begin{bmatrix}
B & 0 \\ 0 & I
\end{bmatrix}\mK
\begin{bmatrix}
C & 0 \\ 0 & I
\end{bmatrix}
$$
and recall that the LQG cost is given by
$$
J_q(\mK)
=\operatorname{tr}
\left(
\begin{bmatrix}
Q & 0 \\ 0 & C_{\mK}^\tr R C_{\mK}
\end{bmatrix}X_\mK
\right),
$$
where $X_\mK$ is the unique positive semidefinite solution to the Lyapunov equation~\eqref{eq:LyapunovX}.

Consider an arbitrary direction $\Delta
=\begin{bmatrix}
0 & \Delta_{C_{\mK}} \\
\Delta_{B_{\mK}} & \Delta_{A_{\mK}}
\end{bmatrix}
\in\mathcal{V}_q$. For sufficiently small $t>0$ such that $\mK + t \Delta \in \mathcal{C}_q$, the corresponding closed-loop matrix is
$$
A_{\text{cl},\mK+t\Delta} = {A}_{\text{cl},\mK}
+t
\begin{bmatrix}
B & 0 \\ 0 & I
\end{bmatrix}\Delta
\begin{bmatrix}
C & 0 \\ 0 & I
\end{bmatrix},
$$
and we let $X_{\mK,\Delta}(t)$ denote the solution to the Lyapunov equation~\eqref{eq:LyapunovX} with closed-loop matrix $A_{\mathrm{cl},\mK+t\Delta}$, i.e.,
\begin{equation} \label{eq:Lyapunov_X_next_step}
    \begin{aligned}
\left({A}_{\text{cl},\mK} + t \begin{bmatrix}
B & 0 \\ 0 & I
\end{bmatrix}\Delta \begin{bmatrix}
C & 0 \\ 0 & I
\end{bmatrix} \right)X_{\mK,\Delta}(t)
&+
X_{\mK,\Delta}(t)
\left(A_{\text{cl},\mK}
+t \begin{bmatrix}
B & 0 \\ 0 & I
\end{bmatrix}\Delta
\begin{bmatrix}
C & 0 \\ 0 & I
\end{bmatrix}
\right)^\tr \\
&+ \begin{bmatrix}
W & 0 \\
0 & (B_\mK + t\Delta_{B_\mK})V
(B_\mK + t\Delta_{B_\mK})^\tr
\end{bmatrix} = 0.
\end{aligned}
\end{equation}
By \cref{lemma:Taylor_Lyapunov_eq}, we see that $X_{\mK,\Delta}(t)$ admits a Taylor expansion of the form
\begin{equation} \label{eq:Lyapunov_X_Taylor}
X_{\mK,\Delta}(t)
= X_{\mK}
+
t\cdot X_{\mK,\Delta}'(0)
+\frac{t^2}{2}
\cdot X_{\mK,\Delta}''(0)
+o(t^2),
\cdot
\end{equation}
and the derivatives $X_{\mK,\Delta}'(0)$ and $X_{\mK,\Delta}''(0)$ are the solutions to the following Lyapunov equations
\begin{align}
A_{\mathrm{cl},\mK} X'_{\mK,\Delta}(0)
+X'_{\mK,\Delta}(0) A_{\mathrm{cl},\mK}^\tr
+M_1(X_{\mK},\Delta) =\ & 0,
\\
A_{\mathrm{cl},\mK} X''_{\mK,\Delta}(0)
+X''_{\mK,\Delta}(0) A_{\mathrm{cl},\mK}^\tr
+2M_2\big(X'_{\mK,\Delta}(0),\Delta\big)
=\ & 0,
\end{align}
where
\begin{align*}
M_1(X_{\mK},\Delta)
\coloneqq &
\begin{bmatrix}
B & 0 \\ 0 & I
\end{bmatrix}\Delta
\begin{bmatrix}
C & 0 \\ 0 & I
\end{bmatrix} X_{\mK}
+ X_{\mK}\begin{bmatrix}
C & 0 \\ 0 & I
\end{bmatrix}^\tr
\!\Delta^\tr\!
\begin{bmatrix}
B & 0 \\ 0 & I
\end{bmatrix}^\tr
\!+
\begin{bmatrix}
0 & 0 \\
0 & B_{\mK}V\Delta_{B_\mK}^\tr
\!+\!
\Delta_{B_\mK} V B_{\mK}^\tr
\end{bmatrix},
\\
M_2\big(X_{\mK,\Delta}'(0),\Delta\big)
\coloneqq &
\begin{bmatrix}
B & 0 \\ 0 & I
\end{bmatrix}\Delta
\begin{bmatrix}
C & 0 \\ 0 & I
\end{bmatrix} X'_{\mK,\Delta}(0)
+ X'_{\mK,\Delta}(0)\begin{bmatrix}
C & 0 \\ 0 & I
\end{bmatrix}^\tr
\!\Delta^\tr\!
\begin{bmatrix}
B & 0 \\ 0 & I
\end{bmatrix}^\tr
\!
+\begin{bmatrix}
0 & 0 \\
0 & \Delta_{B_\mK}V\Delta_{B_\mK}^\tr
\end{bmatrix}.
\end{align*}

Now, by plugging the Taylor expansion~\eqref{eq:Lyapunov_X_Taylor} into the expression~\eqref{eq:LQG_cost_formulation1} for $J_q(\mK)$, we get
$$
\begin{aligned}
J_q(\mK+t\Delta)
=\ &
\operatorname{tr}
\left(
\begin{bmatrix}
Q & 0 \\ 0 & (C_{\mK}+t\Delta_{C_\mK})^\tr R (C_{\mK}+t\Delta_{C_\mK})
\end{bmatrix} X_{\mK,\Delta}(t)\right)
\\
=\ &
J_q(\mK)
+
t
\cdot\operatorname{tr}
\left(
\begin{bmatrix}
Q & 0 \\ 0 & C_{\mK}^\tr R C_{\mK}
\end{bmatrix} X_{\mK,\Delta}'(0)
+
\begin{bmatrix}
0 & 0 \\ 0 & C_{\mK}^\tr R \Delta_{C_\mK} + \Delta_{C_\mK}^\tr RC_{\mK}
\end{bmatrix}X_{\mK}
\right) \\
& +
\frac{t^2}{2}
\cdot \operatorname{tr}\Bigg(
\begin{bmatrix}
Q & 0 \\ 0 & C_{\mK}^\tr R C_{\mK}
\end{bmatrix} X_{\mK,\Delta}''(0)
+2
\begin{bmatrix}
0 & 0 \\ 0 & C_{\mK}^\tr R \Delta_{C_\mK} + \Delta_{C_\mK}^\tr RC_{\mK}
\end{bmatrix}X'_{\mK,\Delta}(0) \\
&\qquad\qquad\qquad\qquad
+
2\begin{bmatrix}
0 & 0 \\ 0 & \Delta_{C_\mK}^\tr R \Delta_{C_\mK}
\end{bmatrix}X_{\mK}
\Bigg)
+o(t^2),
\end{aligned}
$$
from which we can directly recognize $\left.\mfrac{dJ_q(\mK+t\Delta)}{dt}\right|_{t=0}$ and $\left.\mfrac{d^2J_q(\mK+t\Delta)}{dt^2}\right|_{t=0}$.

Now suppose $X$ is the solution to the following Lyapunov equation
$$
A_{\mathrm{cl},\mK} X
+X A_{\mathrm{cl},\mK}^\tr
+M=0
$$
for some $M\in\mathbb{S}^{n+q}$. Then, by \cref{lemma:Lyapunov}, we have
$$
X = \int_0^{+\infty}
e^{A_{\mathrm{cl},\mK} s}
M e^{A_{\mathrm{cl},\mK}^\tr s}\,ds,
$$
and consequently
$$
\begin{aligned}
\operatorname{tr}\left(\begin{bmatrix}
Q & 0 \\ 0 & C_{\mK}^\tr R C_{\mK}
\end{bmatrix}
X\right)
=\ &
\int_0^{+\infty}\operatorname{tr}\left(\begin{bmatrix}
Q & 0 \\ 0 & C_{\mK}^\tr R C_{\mK}
\end{bmatrix}
e^{A_{\mathrm{cl},\mK} s}
M e^{A_{\mathrm{cl},\mK}^\tr s}\right)\,ds \\
=\ &
\int_0^{+\infty}\operatorname{tr}\left(
e^{A_{\mathrm{cl},\mK}^\tr s}\begin{bmatrix}
Q & 0 \\ 0 & C_{\mK}^\tr R C_{\mK}
\end{bmatrix}
e^{A_{\mathrm{cl},\mK} s}
M \right)\,ds
=\operatorname{tr}(Y_{\mK}M),
\end{aligned}
$$
in which we recall that $Y_{\mK}$ is the unique positive semidefinite solution to Lyapunov equation~\eqref{eq:LyapunovY}. Therefore the first-order derivative $\left.\mfrac{dJ_q(\mK+t\Delta)}{dt}\right|_{t=0}$ can be alternatively given by
$$
\begin{aligned}
& \left.\frac{d J_q(\mK+t\Delta)}{dt}\right|_{t=0} \\
=\ &
\operatorname{tr}
\left(Y_{\mK} M_1(X_{\mK},\Delta)
+\begin{bmatrix}
0 & 0 \\ 0 & C_{\mK}^\tr R \Delta_{C_\mK} + \Delta_{C_\mK}^\tr RC_{\mK}
\end{bmatrix}X_{\mK}\right) \\
=\ &
2\operatorname{tr}\left[
\left(
\begin{bmatrix}
0 & RC_{\mK} \\ 0 & 0
\end{bmatrix}X_{\mK} \begin{bmatrix}
0 & 0 \\ 0 & I
\end{bmatrix}
+
\begin{bmatrix}
B & 0 \\ 0 & I
\end{bmatrix}^\tr
Y_{\mK} X_{\mK}
\begin{bmatrix}
C & 0 \\ 0 & I
\end{bmatrix}^\tr
+\begin{bmatrix}
0 & 0 \\ 0 & I
\end{bmatrix}
Y_{\mK}
\begin{bmatrix}
0 & 0 \\ B_{\mK}V & 0
\end{bmatrix}
\right)^\tr\Delta\right].
\end{aligned}
$$
One can readily recognize the gradient $\nabla J_q(\mK)$ by noticing that
$$
\left.\frac{dJ_q(\mK+t\Delta)}{dt}\right|_{t=0}
=\operatorname{tr}\left(
\nabla J_q(\mK)^\tr\Delta\right).
$$
Upon partitioning $X_{\mK}$ and $Y_{\mK}$ as~\eqref{eq:LyapunovXY_block}, a few simple calculations lead to the gradient formula of $J_q(\mK)$ in~\eqref{eq:gradient_Jn}.

Similarly, we can show that the second-order derivative $\left.\mfrac{d^2J_q(\mK+t\Delta)}{dt^2}\right|_{t=0}$ can be alternatively given by
\begin{equation*}%\label{eq:Jq_Hessian_discrete}
\begin{aligned}
& \left.\frac{d^2J_q(\mK+t\Delta)}{dt^2}\right|_{t=0} \\
=\ &
2\operatorname{tr}
\left(Y_{\mK} M_2\big(X_{\mK,\Delta}'(0),\Delta\big)
+
\begin{bmatrix}
0 & 0 \\ 0 & C_{\mK}^\tr R \Delta_{C_\mK} + \Delta_{C_\mK}^\tr RC_{\mK}
\end{bmatrix}X'_{\mK,\Delta}(0)
+\begin{bmatrix}
0 & 0 \\ 0 & \Delta_{C_\mK}^\tr R \Delta_{C_\mK}
\end{bmatrix}X_{\mK}\right) \\
=\ &
2\operatorname{tr}
\Bigg(
2\begin{bmatrix}
B & 0 \\ 0 & I
\end{bmatrix}
\Delta\begin{bmatrix}
C & 0 \\ 0 & I
\end{bmatrix}
X'_{\mK,\Delta}(0)Y_{\mK}
+2\begin{bmatrix}
0 & 0 \\ 0 & C_{\mK}^\tr R \Delta_{C_\mK}
\end{bmatrix}X'_{\mK,\Delta}(0)
\\
& \qquad\qquad\qquad\qquad\qquad\qquad
+\begin{bmatrix}
0 & 0 \\
0 & \Delta_{B_\mK}V\Delta_{B_\mK}^\tr
\end{bmatrix} Y_{\mK}
+
\begin{bmatrix}
0 & 0 \\ 0 & \Delta_{C_\mK}^\tr R \Delta_{C_\mK}
\end{bmatrix}X_{\mK}
\Bigg).
\end{aligned}
\end{equation*}
\begin{remark}
If we let $\operatorname{Hess}_{\,\mK}:\mathcal{V}_q\times\mathcal{V}_q\rightarrow\mathbb{R}$ denote the bilinear form of the Hessian of $J_q$ at $\mK\in\mathcal{C}_q$. Then one can compute $\operatorname{Hess}_{\,\mK}(\Delta_1,\Delta_2)$ for any $\Delta_1,\Delta_2\in\mathcal{V}_q$ by noting that
$$
\operatorname{Hess}_{\,\mK}(\Delta_1,\Delta_2)
=
\frac{1}{4}\left(\operatorname{Hess}_{\,\mK}(\Delta_1+\Delta_2,\Delta_1+\Delta_2)
-\operatorname{Hess}_{\,\mK}(\Delta_1-\Delta_2,\Delta_1-\Delta_2)\right),
$$
and that
$$
\operatorname{Hess}_{\,\mK}(\Delta,\Delta)
=\left.\frac{d^2J_q(\mK+t\Delta)}{dt^2}\right|_{t=0}
$$
for any $\Delta\in\mathcal{V}_q$.
\end{remark}

\subsection{Proof of \cref{lemma:LyapunovXY_pd}} \label{app:LyapunovXY_pd}
    By \cref{lemma:Lyapunov}, given a stable matrix $A$, if $(C,A)$ is observable, then the solution $L$ to the Lyapunov equation is positive definite
    $$
        A^\tr L + L A + C^\tr C = 0.
    $$
    Therefore, we only need to prove
    $$
       \left(\begin{bmatrix} Q^{\frac{1}{2}} & 0 \\ 0 & R^{\frac{1}{2}}C_{\mK}\end{bmatrix},\begin{bmatrix} A & BC_{\mK} \\
        B_{\mK} C & A_{\mK}\end{bmatrix} \right)
    $$
    is observable, and this is equivalent to show that the eigenvalues of the following matrix
    $$
         \begin{bmatrix} A & BC_{\mK} \\
        B_{\mK} C & A_{\mK}\end{bmatrix} + \begin{bmatrix} L_{11} & L_{12} \\ L_{21} & L_{22}\end{bmatrix}\begin{bmatrix} Q^{\frac{1}{2}} & 0 \\ 0 & R^{\frac{1}{2}}C_{\mK}\end{bmatrix}  = \begin{bmatrix} A+ L_{11}Q^{\frac{1}{2}} & BC_{\mK} + L_{12}R^{\frac{1}{2}}C_{\mK} \\ B_{\mK}C+ L_{21}Q^{\frac{1}{2}} & A_{\mK} + L_{22}R^{\frac{1}{2}}C_{\mK}\end{bmatrix}
    $$
    can be arbitrarily assigned by choosing $L_{11},L_{12},L_{21},L_{22}$. {This is true by choosing
    $$
        L_{12} = -BR^{-\frac{1}{2}},
    $$
    and observing that $A + L_{11}Q^{\frac{1}{2}} $ and $A_{\mK} + L_{22}R^{\frac{1}{2}}C_{\mK}$ can be arbitrarily assigned since $(Q^{\frac{1}{2}},A), (C_{\mK}, A_{\mK})$ are both observable. }

    Thus, by \cref{lemma:Lyapunov}, the solution $Y_{\mK}$ to~\eqref{eq:LyapunovY} is positive definite.
    Similarly, we can prove $X_{\mK}$ is positive definite.

\subsection{Proof of \cref{proposition:sim_trans_submanifold}}
\label{app:proof_sim_trans_submanifold}

We have already seen that $\mathscr{T}_q$ gives a smooth Lie group action of $\mathrm{GL}_q$ on $\mathcal{C}_q$.
We first show that the isotropy group of $\mK$ under the group actions in $\mathrm{GL}_q$, defined by
$$
\{T\in\mathrm{GL}_q\mid \mathscr{T}_q(T,\mK)=\mK\},
$$
is a trivial group containing only the identity matrix. Let $T\in\mathrm{GL}_q$ satisfy $\mathscr{T}_q(T,\mK)=\mK$, or
$$
\begin{bmatrix}
0 & C_{\mK}T^{-1} \\
TB_{\mK} & TA_{\mK}T^{-1}
\end{bmatrix}
=
\begin{bmatrix}
0 & C_{\mK} \\
B_{\mK} & A_{\mK}
\end{bmatrix}.
$$
Then we have $TA_{\mK}=A_{\mK}T$, and consequently
$$
TA_{\mK}^{j+1}B_{\mK}
=A_{\mK}TA_{\mK}^jB_{\mK}.
$$
By mathematical induction, we can see that $TA_{\mK}^{j}B_{\mK}=A_{\mK}^{j}B_{\mK}$ for all $j=0,\ldots,q-1$, indicating that any column vector of $A_{\mK}^{j}B_{\mK}$ is an eigenvector of $T$ with eigenvalue $1$. On the other hand, the controllability of $\mK$ implies the column vectors of the matrix
$$
\begin{bmatrix}
B_{\mK} & A_{\mK}B_{\mK} & \cdots & A_{\mK}^{q-1}B_{\mK}
\end{bmatrix}
$$
span the whole space $\mathbb{R}^q$. Therefore $\mathbb{R}^q$ is a subspace of the eigenspace of $T$ with eigenvalue $1$, meaning that $T$ is just the identity matrix.

Since the isotropy group $\{T\in\mathrm{GL}_q\mid \mathscr{T}_q(T,\mK)=\mK\}$ only contains the identity, by \cite[Proposition 7.26]{lee2013introduction}, the mapping $T\mapsto \mathscr{T}_q(T,\mK)$ is an immersion and the orbit $\mathcal{O}_{\mK}$ is an immersed submanifold.

We then prove that $\mathcal{O}_{\mK}$ is closed under the original topology of $\mathcal{C}_q$. Suppose $(T_j)_{j=1}^\infty$ is a sequence in $\mathrm{GL}_q$ such that
$$
\mathscr{T}_q(T_j,\mK)
=\begin{bmatrix}
0 & C_{\mK} T_j^{-1} \\
T_jB_{\mK} & T_jA_{\mK} T_j^{-1}
\end{bmatrix}\rightarrow
\begin{bmatrix}
0 & \tilde{C}_{\mK} \\
\tilde{B}_{\mK} & \tilde{A}_{\mK}
\end{bmatrix}
=\tilde{\mK},
\qquad j\rightarrow\infty.
$$
Let $\mathbf{G}(s)$ be the transfer function of $\mK$, i.e.,
$$
\mathbf{G}(s)
=C_{\mK}(sI-A_{\mK})^{-1}B_{\mK}.
$$
We notice that for any $j\geq 1$, the matrix $sI-T_jA_{\mK}T_j^{-1}$ is invertible if and only if $sI-A_{\mK}$ is invertible. Thus for any fixed $s\in\mathbb{C}$ such that $sI-A_{\mK}$ is invertible, we have
$$
\lim_{j\rightarrow\infty}
C_{\mK}T_j^{-1}
(sI-T_jA_{\mK}T_j^{-1})^{-1}
T_jB_{\mK}
=
\tilde{C}_{\mK}(sI-\tilde{A}_{\mK})^{-1}\tilde{B}_{\mK}.
$$
On the other hand, we simply have
$$
C_{\mK}T_j^{-1}
(sI-T_jA_{\mK}T_j^{-1})^{-1}
T_jB_{\mK}
=
C_{\mK}
(sI-A_{\mK})^{-1}B_{\mK}
=\mathbf{G}(s).
$$
This shows that the transfer function of $\tilde{\mK}$ agrees with $\mathbf{G}(s)$ for any $s\in\mathbb{C}$ such that $sI-A_{\mK}$ is invertible, and thus is just equal to $\mathbf{G}(s)$. On the other hand, the controllability and observability of $\mK\in\mathcal{C}_q$ indicates that the transfer function $\mathbf{G}(s)$ has order $q$, and so any two state-space representations of $\mathbf{G}(s)$ with order $q$ will always be similarity transformations of each other (see \cref{theo:minimal_state_space_similarity_transformation}). %\cite[Theorem 3.17]{zhou1996robust}.
In other words, there exists $\tilde{T}\in\mathrm{GL}_q$ such that
$$
\tilde{\mK}
=\begin{bmatrix}
0 & \tilde{C}_{\mK} \\
\tilde{B}_{\mK} & \tilde{A}_{\mK}
\end{bmatrix}
=
\begin{bmatrix}
0 & C_{\mK}\tilde{T}^{-1} \\
\tilde{T}{B}_{\mK} & \tilde{T}{A}_{\mK}\tilde{T}^{-1}
\end{bmatrix}
=\mathscr{T}_q(\tilde{T},\mK),
$$
which implies that $\tilde{\mK}\in\mathcal{O}_\mK$. We can now conclude that $\mathcal{O}_\mK$ is a closed subset of $\mathcal{C}_q$. As a consequence of the closedness of $\mathcal{O}_\mK$, the set $\mathcal{O}_\mK$ equipped with the subspace topology induced from $\mathcal{C}_q$ is a locally compact Hausdorff space.

Now, by combining the above results and applying \cite[Theorem 2.13]{montgomery2018topological}, we can conclude that the mapping $T\mapsto \mathscr{T}_q(T,\mK)$ is a homeomorphism from $\mathrm{GL}_q$ to $\mathcal{O}_\mK$. Therefore, the mapping $T\mapsto \mathscr{T}_q(T,\mK)$ is a diffeomorphism from $\mathrm{GL}_q$ to $\mathcal{O}_{\mK}$, and $\mathcal{O}_\mK$ is an embedded submanifold of $\mathcal{C}_q$ with dimension given by
$$
\dim\mathcal{O}_{\mK}
= \dim\mathrm{GL}_q = q^2.
$$
Finally, the two path-connected components of $\mathcal{O}_{\mK}$ are immediate.

\subsection{Proof of \cref{theorem:zero_stationary_hessian}}
\label{app:proof_zero_stationary_hessian}

We first show that $\mK^\star$ is a stationary point of $J_n(\mK)$ over $\mK\in\mathcal{C}_n$. Since
$$
\mathscr{T}_n(-I_n,\mK^\star)
=\mK^\star,
$$
by \cref{lemma:gradient_simi_tran_linear}, we have
$$
\left.\nabla J_n\right|_{\mK^\star}
=\left.\nabla J_n\right|_{\mathscr{T}_n(-I_n,\mK^\star)}
=\begin{bmatrix}
I_m & 0\\
0 & -I_n
\end{bmatrix}
\cdot
\left.\nabla J_n\right|_{\mK^\star}
\cdot
\begin{bmatrix}
I_p & 0\\
0 & -I_n
\end{bmatrix}.
$$
This equality implies that, excluding the bottom right $n\times n$ block, the last $n$ rows and the last $n$ columns of $\left.\nabla J_n\right|_{\mK^\star}$ are zero. On the other hand, it is not hard to see that
$
J_n(\mK^\star)
$
does not depend on the choice of $\Lambda$ as long as $\Lambda$ is stable. Therefore the bottom right $n\times n$ block of $\left.\nabla J_n\right|_{\mK^\star}$ is zero. We can now see that $\left.\nabla J_n\right|_{\mK^\star}=0$, showing that $\mK^\star$ is a stationary point of $J_n$.

Let $\Delta=\begin{bmatrix}
0 & \Delta_{C_\mK} \\
\Delta_{B_\mK} & \Delta_{A_\mK}
\end{bmatrix}\in\mathcal{V}_n$ be arbitrary, and let
$$
\Delta^{(1)} =
\begin{bmatrix}
0 & \Delta_{C_\mK} \\
0 & 0
\end{bmatrix},
\quad
\Delta^{(2)} =
\begin{bmatrix}
0 & 0 \\
\Delta_{B_\mK} & 0
\end{bmatrix},
\quad
\Delta^{(3)} =
\begin{bmatrix}
0 & 0 \\
0 & \Delta_{A_\mK}
\end{bmatrix}.
$$
By the bilinearity of the Hessian, we have
$$
\begin{aligned}
\operatorname{Hess}_{\,\mK^\star}(\Delta,\Delta)
=\ &
\sum_{1\leq i<j\leq 3}
\operatorname{Hess}_{\,\mK^\star}(\Delta^{(i)}+\Delta^{(j)},\Delta^{(i)}+\Delta^{(j)})
-\sum_{i=1}^3\operatorname{Hess}_{\,\mK^\star}(\Delta^{(i)},\Delta^{(i)}).
\end{aligned}
$$
Since the controllers $\mK^\star
+t\Delta^{(i)}$ for $i=1,2,3$ and $\mK^\star
+t(\Delta^{(i)}+\Delta^{(3)})$ for $i=1,2$ have the same transfer function representation as $\mK^*$, we can see that for all sufficiently small $t$,
$$
\begin{aligned}
J_n(\mK^\star)
=\ &
J_n(\mK^\star
+t\Delta^{(1)})
=J_n(\mK^\star
+t\Delta^{(2)})
=J_n(\mK^\star
+t\Delta^{(3)}) \\
=\ &
J_n(\mK^\star
+t(\Delta^{(1)}+\Delta^{(3)}))
=J_n(\mK^\star
+t(\Delta^{(2)}+\Delta^{(3)})),
\end{aligned}
$$
which implies that
$$
\operatorname{Hess}_{\,\mK^\star}
(\Delta^{(i)},\Delta^{(i)})=0,\qquad\forall i=1,2,3,
$$
and
$$
\operatorname{Hess}_{\,\mK^\star}
(\Delta^{(1)}+\Delta^{(3)},\Delta^{(1)}+\Delta^{(3)})
=\operatorname{Hess}_{\,\mK^\star}
(\Delta^{(2)}+\Delta^{(3)},\Delta^{(2)}+\Delta^{(3)})=0.
$$
Therefore
$$
\operatorname{Hess}_{\,\mK^\star}(\Delta,\Delta)
=\operatorname{Hess}_{\,\mK^\star}
(\Delta^{(1)}+\Delta^{(2)},\Delta^{(1)}+\Delta^{(2)}).
$$

Now, if $\operatorname{Hess}_{\,\mK^\star}(\Delta,\Delta)= 0$ for all $\Delta\in\mathcal{V}_n$, then the Hessian $\operatorname{Hess}_{\,\mK^\star}$ is obviously zero. Otherwise, $\operatorname{Hess}_{\,\mK^\star}(\Delta,\Delta)\neq 0$ for some $\Delta\in\mathcal{V}_n$, which implies that
$$
\begin{aligned}
& \operatorname{Hess}_{\,\mK^\star}
(\Delta^{(1)},\Delta^{(2)}) \\
=\ &
\frac{1}{2}
\!\left(\operatorname{Hess}_{\,\mK^\star}
(\Delta^{(1)}+\Delta^{(2)},\Delta^{(1)}+\Delta^{(2)})
- \operatorname{Hess}_{\,\mK^\star}
(\Delta^{(1)},\Delta^{(1)})
-\operatorname{Hess}_{\,\mK^\star}
(\Delta^{(2)},\Delta^{(2)})
\right) \\
=\ &
\frac{1}{2}\operatorname{Hess}_{\,\mK^\star}
(\Delta,\Delta)\neq 0.
\end{aligned}
$$
Note that $\Delta^{(1)}$ and $\Delta^{(2)}$ are linearly independent (otherwise $\operatorname{Hess}_{\,\mK^\star}
(\Delta^{(1)},\Delta^{(2)})$ will be zero). Together with $\operatorname{Hess}_{\,\mK^\star}
(\Delta^{(i)},\Delta^{(i)})=0$ for $i=1,2$, we see that $\operatorname{Hess}_{\,\mK^\star}$ must be indefinite (a symmetric matrix having a $2 \times 2$ principal submatrix with zero diagonal entries and non-zero off-diagonal entries must be indefinite).

Now we proceed to the situation where $\Lambda$ is diagonalizable. We will use $e_i^{(k)}$ to denote the $k$-dimensional vector where only the $i$th entry is $1$ and other entries are zero. %is the only nonzero entry of $e_i^{(k)}$.
\vspace{5pt}

\noindent\textbf{Part I: $\operatorname{eig}(-\Lambda)\nsubseteq\mathcal{Z}$ $\Longrightarrow$ the Hessian is indefinite. }
Let $\lambda\in \operatorname{eig}(-\Lambda)\backslash\mathcal{Z}$. Since $\lambda\notin \mathcal{Z}$, there exists some $i,j$ such that
$$
G(\lambda)
\coloneqq
{e_i^{(p)}}^\tr CX_{\mathrm{op}}\big(\lambda I-A^\tr\big)^{-1}Y_{\mathrm{op}} B e_j^{(m)}
\neq 0.
$$
We consider three situations:
\begin{enumerate}
\item $\lambda$ is real. In this case, let $T$ be a real invertible matrix such that
$$
T\Lambda T^{-1}
=\begin{bmatrix}
-\lambda & 0 \\ 0 & \ast
\end{bmatrix}.
$$
Let $\Delta^{(1)},\Delta^{(2)}\in\mathcal{V}_n$ be given by
$$
\Delta^{(1)}
=\begin{bmatrix}
0 & \Delta_{C_{\mK}}^{(1)} \\
0 & 0
\end{bmatrix},
\qquad
\Delta^{(2)}
=\begin{bmatrix}
0 & 0 \\
\Delta_{B_{\mK}}^{(2)} & 0
\end{bmatrix},
$$
where
$$
\Delta_{C_{\mK}}^{(1)}
=e_j^{(m)}{e_1^{(n)}}^\tr T^{-1},
\qquad
\Delta_{B_{\mK}}^{(2)}
=Te_1^{(n)} {e_i^{(p)}}^\tr.
$$
Then it's not hard to see that
$$
J_n(\mK^\star+t\Delta^{(1)})
= J_n(\mK^\star+t\Delta^{(2)})
= J_n(\mK^\star)
$$
for any sufficiently small $t$, indicating that
$$
\operatorname{Hess}_{\,\mK^\star}(\Delta^{(1)},\Delta^{(1)})
=\operatorname{Hess}_{\,\mK^\star}(\Delta^{(2)},\Delta^{(2)})=0.
$$
On the other hand, we have that the unique solutions to Lyapunov equations~\eqref{eq:LyapunovX} and~\eqref{eq:LyapunovY} are
$$
X_{\mK^\star} = \begin{bmatrix}
X_{\mathrm{op}} & 0 \\
0 & 0
\end{bmatrix},
\quad
Y_{\mK^\star} = \begin{bmatrix}
Y_{\mathrm{op}} & 0 \\
0 & 0
\end{bmatrix}.
$$
By \cref{lemma:Jn_Hessian}, we can see that
$$
\operatorname{Hess}_{\,\mK^\star}(\Delta^{(1)}+\Delta^{(2)},\Delta^{(1)}+\Delta^{(2)})
=
4\operatorname{tr}
\!\left(
\begin{bmatrix}
0 & B\Delta_{C_\mK}^{(1)} \\
\Delta_{B_\mK}^{(2)}C & 0
\end{bmatrix}
X'_{\mK^\star,\Delta^{(1)}\!+\!\Delta^{(2)}}
\begin{bmatrix}
Y_{\mathrm{op}} & 0 \\ 0 & 0
\end{bmatrix}
\right),
$$
where $X'_{\mK^\star,\Delta^{(1)}\!+\!\Delta^{(2)}}$ is the solution to the following Lyapunov equation
$$
\begin{aligned}
& \begin{bmatrix}
A & 0 \\
0 & \Lambda
\end{bmatrix}
X'_{\mK^\star,\Delta^{(1)}\!+\!\Delta^{(2)}}
+
X'_{\mK^\star,\Delta^{(1)}\!+\!\Delta^{(2)}}
\begin{bmatrix}
A & 0 \\
0 & \Lambda
\end{bmatrix}^\tr \\
&
\qquad\qquad
+
\begin{bmatrix}
0 & B\Delta_{C_\mK}^{(1)} \\
\Delta_{B_\mK}^{(2)}C & 0
\end{bmatrix}
\begin{bmatrix}
X_{\mathrm{op}} & 0 \\
0 & 0
\end{bmatrix}
+
\begin{bmatrix}
X_{\mathrm{op}} & 0 \\
0 & 0
\end{bmatrix}
\begin{bmatrix}
0 & B\Delta_{C_\mK}^{(1)} \\
\Delta_{B_\mK}^{(2)}C & 0
\end{bmatrix}^\tr
=0.
\end{aligned}
$$
Since
$$
\begin{bmatrix}
0 & B\Delta_{C_\mK}^{(1)} \\
\Delta_{B_\mK}^{(2)}C & 0
\end{bmatrix}
\begin{bmatrix}
X_{\mathrm{op}} & 0 \\
0 & 0
\end{bmatrix}
+
\begin{bmatrix}
X_{\mathrm{op}} & 0 \\
0 & 0
\end{bmatrix}
\begin{bmatrix}
0 & B\Delta_{C_\mK}^{(1)} \\
\Delta_{B_\mK}^{(2)}C & 0
\end{bmatrix}^\tr
=\begin{bmatrix}
0 & X_{\mathrm{op}}C^\tr {\Delta_{B_{\mK}}^{(2)}}^\tr \\
\Delta_{B_{\mK}}^{(2)}CX_{\mathrm{op}} & 0
\end{bmatrix},
$$
the matrix $X'_{\mK^\star,\Delta^{(1)}+\Delta^{(2)}}$ can be represented by
$$
\begin{aligned}
X'_{\mK^\star,\Delta^{(1)}+\Delta^{(2)}}
=\ &
\int_0^{+\infty} \exp
\!\left(
\begin{bmatrix}
A & 0 \\
0 & \Lambda
\end{bmatrix}s\right)
\begin{bmatrix}
0 & X_{\mathrm{op}}C^\tr {\Delta_{B_{\mK}}^{(2)}}^\tr \\
\Delta_{B_{\mK}}^{(2)}CX_{\mathrm{op}} & 0
\end{bmatrix}
\exp
\!\left(
\begin{bmatrix}
A & 0 \\
0 & \Lambda
\end{bmatrix}^\tr s\right)\,ds \\
=\ &
\int_0^{+\infty}
\begin{bmatrix}
e^{As} & 0 \\ 0 & e^{\Lambda s}
\end{bmatrix}
\begin{bmatrix}
0 & X_{\mathrm{op}}C^\tr {\Delta_{B_{\mK}}^{(2)}}^\tr \\
\Delta_{B_{\mK}}^{(2)}CX_{\mathrm{op}} & 0
\end{bmatrix}
\begin{bmatrix}
e^{A^\tr s} & 0 \\ 0 & e^{\Lambda^\tr s}
\end{bmatrix}\,ds \\
=\ &
\int_0^{+\infty}
\begin{bmatrix}
0 & e^{As} X_{\mathrm{op}}C^\tr {\Delta_{B_\mK}^{(2)}}^\tr
e^{\Lambda^\tr s} \\
e^{\Lambda s}\Delta_{B_\mK}^{(2)} C X_{\mathrm{op}} e^{A^\tr s} & 0
\end{bmatrix} \,ds.
\end{aligned}
$$
Therefore
$$
\begin{aligned}
& \operatorname{Hess}_{\,\mK^\star}(\Delta^{(1)}+\Delta^{(2)},\Delta^{(1)}+\Delta^{(2)}) \\
=\ &
\int_0^\infty 4\operatorname{tr}
\!\left(
\begin{bmatrix}
0 & B\Delta_{C_\mK}^{(1)} \\
\Delta_{B_\mK}^{(2)}C & 0
\end{bmatrix}
\begin{bmatrix}
0 & e^{As} X_{\mathrm{op}}C^\tr {\Delta_{B_\mK}^{(2)}}^\tr
e^{\Lambda^\tr s} \\
e^{\Lambda s}\Delta_{B_\mK}^{(2)} C X_{\mathrm{op}} e^{A^\tr s} & 0
\end{bmatrix}
\begin{bmatrix}
Y_{\mathrm{op}} & 0 \\ 0 & 0
\end{bmatrix}
\right)\,ds \\
=\ &
\int_0^{+\infty}
4
\operatorname{tr}
\left(
B\Delta_{C_\mK}^{(1)}
e^{\Lambda s} \Delta_{B_\mK}^{(2)}
C X_{\mathrm{op}} e^{A^\tr s}
Y_{\mathrm{op}}
\right)ds.
\end{aligned}
$$
By the construction of $\Delta_{C_\mK}^{(1)}$ and $\Delta_{B_\mK}^{(2)}$, we can see that
$$
\Delta_{C_\mK}^{(1)}
e^{\Lambda s} \Delta_{B_\mK}^{(2)}
= e^{-\lambda s}
e_j^{(m)} {e_i^{(p)}}^\tr.
$$
Thus
$$
\begin{aligned}
\operatorname{Hess}_{\,\mK^\star}(\Delta^{(1)}+\Delta^{(2)},\Delta^{(1)}+\Delta^{(2)})
=\ &
\int_0^{+\infty} 4
{e_i^{(p)}}^\tr C
X_{\mathrm{op}} e^{(A^\tr-\lambda I) s} Y_{\mathrm{op}}Be_j^{(m)}
\,ds \\
=\ &
4
{e_i^{(p)}}^\tr C
X_{\mathrm{op}} \big(\lambda I-A^\tr\big)^{-1} Y_{\mathrm{op}}Be_j^{(m)} \\
=\ &
4 G(\lambda),
\end{aligned}
$$
which is nonzero by assumption. Consequently,
$$
\begin{aligned}
& \operatorname{Hess}_{\,\mK^\star}(\Delta^{(1)},\Delta^{(2)}) \\
=\ &
\frac{1}{2}\!\left(\operatorname{Hess}_{\,\mK^\star}(\Delta^{(1)}\!+\!\Delta^{(2)},\Delta^{(1)}\!+\!\Delta^{(2)})
-\operatorname{Hess}_{\,\mK^\star}(\Delta^{(1)},\Delta^{(1)})
-\operatorname{Hess}_{\,\mK^\star}(\Delta^{(2)},\Delta^{(2)})
\right) \\
=\ &
2G(\lambda)\neq 0.
\end{aligned}
$$
Together with the fact that $\operatorname{Hess}_{\,\mK^\star}(\Delta^{(1)},\Delta^{(1)})=\operatorname{Hess}_{\,\mK^\star}(\Delta^{(2)},\Delta^{(2)})=0$, we can see that neither $\operatorname{Hess}_{\,\mK^\star}$ nor $-\operatorname{Hess}_{\,\mK^\star}$ can be positive semidefinite. Thus $\operatorname{Hess}_{\,\mK^\star}$ has at least one positive eigenvalue and one negative eigenvalue.

\item $\lambda=\lambda_{\mathrm{re}}+\mathrm{i}\lambda_{\mathrm{im}}$ is not real, and $G(\lambda)$ is not purely imaginary. In this case, since $\Lambda$ is real, the complex conjugate of $\lambda$, which we denote by $\overline{\lambda}$, is also an eigenvalue of $\Lambda$. We can find a real invertible matrix $T$ such that
$$
T\Lambda T^{-1}
=\begin{bmatrix}
\begin{bmatrix}
-\lambda_{\mathrm{re}}
&
-\lambda_{\mathrm{im}} \\
\lambda_{\mathrm{im}} &
-\lambda_{\mathrm{re}}
\end{bmatrix} & 0 \\
0 & \ast
\end{bmatrix}.
$$
We still let $\Delta^{(1)},\Delta^{(2)}\in\mathcal{V}_n$ be given by
$$
\Delta^{(1)}
=\begin{bmatrix}
0 & \Delta_{C_{\mK}}^{(1)} \\
0 & 0
\end{bmatrix},
\quad
\Delta^{(2)}
=\begin{bmatrix}
0 & 0 \\
\Delta_{B_{\mK}}^{(2)} & 0
\end{bmatrix},
\quad
\Delta_{C_{\mK}}^{(1)}
=e_j^{(m)}{e_1^{(n)}}^\tr T^{-1},
\quad
\Delta_{B_{\mK}}^{(2)}
=Te_1^{(n)} {e_i^{(p)}}^\tr.
$$
Then similarly as in the previous situation, we have
$$
\operatorname{Hess}_{\,\mK^\star}(\Delta^{(1)},\Delta^{(1)})
=\operatorname{Hess}_{\,\mK^\star}(\Delta^{(2)},\Delta^{(2)})=0,
$$
and
$$
\operatorname{Hess}_{\,\mK^\star}(\Delta^{(1)}+\Delta^{(2)},\Delta^{(1)}+\Delta^{(2)})
=
\int_0^{+\infty}
4
\operatorname{tr}
\left(
B\Delta_{C_\mK}^{(1)}
e^{\Lambda s} \Delta_{B_\mK}^{(2)}
C X_{\mathrm{op}} e^{A^\tr s}
Y_{\mathrm{op}}
\right)ds.
$$
By the construction of $\Delta_{C_\mK}^{(1)}$ and $\Delta_{B_\mK}^{(2)}$, we have
$$
\begin{aligned}
\Delta_{C_\mK}^{(1)}
e^{\Lambda s} \Delta_{B_\mK}^{(2)}
=\ &
e^{-\lambda_{\mathrm{re}} s}
\cos(-\lambda_{\mathrm{im}}s)
e_j^{(m)} {e_i^{(p)}}^\tr \\
=\ &
\frac{e^{-\lambda s}
+e^{-\overline{\lambda}s}}{2\mathrm{i}}
e_j^{(m)} {e_i^{(p)}}^\tr,
\end{aligned}
$$
and therefore
$$
\begin{aligned}
& \operatorname{Hess}_{\,\mK^\star}(\Delta^{(1)}+\Delta^{(2)},\Delta^{(1)}+\Delta^{(2)}) \\
=\ &
\frac{1}{2\mathrm{i}}\left(\int_0^{+\infty}
4{e_i^{(p)}}^\tr C
X_{\mathrm{op}} e^{(A^\tr-\lambda I) s} Y_{\mathrm{op}}Be_j^{(m)}
\,ds
+
\int_0^{+\infty} 4
{e_i^{(p)}}^\tr C
X_{\mathrm{op}} e^{(A^\tr-\overline{\lambda} I) s} Y_{\mathrm{op}}Be_j^{(m)}
\,ds\right) \\
=\ &
-2\mathrm{i}(G(\lambda)+\overline{G(\lambda)}),
\end{aligned}
$$
and since $G(\lambda)$ is not purely imaginary, we have $\operatorname{Hess}_{\,\mK^\star}(\Delta^{(1)}+\Delta^{(2)},\Delta^{(1)}+\Delta^{(2)})\neq 0$. Consequently, $\operatorname{Hess}_{\,\mK^\star}(\Delta^{(1)},\Delta^{(2)})\neq 0$, and together with the fact that $\operatorname{Hess}_{\,\mK^\star}(\Delta^{(1)},\Delta^{(1)})=\operatorname{Hess}_{\,\mK^\star}(\Delta^{(2)},\Delta^{(2)})=0$, we can conclude that $\operatorname{Hess}_{\,\mK^\star}$ has at least one positive eigenvalue and one negative eigenvalue.

\item $\lambda=\lambda_{\mathrm{re}}+\mathrm{i}\lambda_{\mathrm{im}}$ is not real, and $G(\lambda)$ is purely imaginary. In this case, we can still find a real invertible matrix $T$ such that
$$
T\Lambda T^{-1}
=\begin{bmatrix}
\begin{bmatrix}
-\lambda_{\mathrm{re}}
&
-\lambda_{\mathrm{im}} \\
\lambda_{\mathrm{im}} &
-\lambda_{\mathrm{re}}
\end{bmatrix} & 0 \\
0 & \ast
\end{bmatrix}.
$$
We let $\Delta^{(1)},\Delta^{(2)}\in\mathcal{V}_n$ be given by
$$
\Delta^{(1)}
=\begin{bmatrix}
0 & \Delta_{C_{\mK}}^{(1)} \\
0 & 0
\end{bmatrix},
\quad
\Delta^{(2)}
=\begin{bmatrix}
0 & 0 \\
\Delta_{B_{\mK}}^{(2)} & 0
\end{bmatrix},
\quad
\Delta_{C_{\mK}}^{(1)}
=e_j^{(m)}{e_1^{(n)}}^\tr T^{-1},
\quad
\Delta_{B_{\mK}}^{(2)}
=Te_2^{(n)} {e_i^{(p)}}^\tr.
$$
Then we have
$$
\operatorname{Hess}_{\,\mK^\star}(\Delta^{(1)},\Delta^{(1)})
=\operatorname{Hess}_{\,\mK^\star}(\Delta^{(2)},\Delta^{(2)})=0,
$$
and
$$
\operatorname{Hess}_{\,\mK^\star}(\Delta^{(1)}+\Delta^{(2)},\Delta^{(1)}+\Delta^{(2)})
=
\int_0^{+\infty}
4
\operatorname{tr}
\left(
B\Delta_{C_\mK}^{(1)}
e^{\Lambda s} \Delta_{B_\mK}^{(2)}
C X_{\mathrm{op}} e^{A^\tr s}
Y_{\mathrm{op}}
\right)ds.
$$
By the construction of $\Delta_{C_\mK}^{(1)}$ and $\Delta_{B_\mK}^{(2)}$, we have
$$
\begin{aligned}
\Delta_{C_\mK}^{(1)}
e^{\Lambda s} \Delta_{B_\mK}^{(2)}
=\ &
e^{-\lambda_{\mathrm{re}} s}
\sin(-\lambda_{\mathrm{im}}s)
e_j^{(m)} {e_i^{(p)}}^\tr \\
=\ &
\frac{e^{-\lambda s}
-e^{-\overline{\lambda}s}}{2}
e_j^{(m)} {e_i^{(p)}}^\tr,
\end{aligned}
$$
and therefore
$$
\begin{aligned}
& \operatorname{Hess}_{\,\mK^\star}(\Delta^{(1)}+\Delta^{(2)},\Delta^{(1)}+\Delta^{(2)}) \\
=\ &
\frac{1}{2}\left(\int_0^{+\infty}
4{e_i^{(p)}}^\tr C
X_{\mathrm{op}} e^{(A^\tr-\lambda I) s} Y_{\mathrm{op}}Be_j^{(m)}
\,ds
-
\int_0^{+\infty} 4
{e_i^{(p)}}^\tr C
X_{\mathrm{op}} e^{(A^\tr-\overline{\lambda} I) s} Y_{\mathrm{op}}Be_j^{(m)}
\,ds\right) \\
=\ &
2(G(\lambda)-\overline{G(\lambda)}),
\end{aligned}
$$
and since $G(\lambda)$ has a nonzero imaginary part, we have $\operatorname{Hess}_{\,\mK^\star}(\Delta^{(1)}+\Delta^{(2)},\Delta^{(1)}+\Delta^{(2)})\neq 0$. Consequently, $\operatorname{Hess}_{\,\mK^\star}(\Delta^{(1)},\Delta^{(2)})\neq 0$, and together with the fact that $\operatorname{Hess}_{\,\mK^\star}(\Delta^{(1)},\Delta^{(1)})=\operatorname{Hess}_{\,\mK^\star}(\Delta^{(2)},\Delta^{(2)})=0$, we can conclude that $\operatorname{Hess}_{\,\mK^\star}$ has at least one positive eigenvalue and one negative eigenvalue.
\end{enumerate}
\vspace{5pt}

\noindent\textbf{Part II: $\operatorname{eig}(-\Lambda)\subseteq\mathcal{Z}$ $\Longrightarrow$ the Hessian is zero. } In this part, we will show that
$$
\operatorname{Hess}_{\,\mK^\star}(\Delta,\Delta)=0
$$
for any $\Delta\in\mathcal{V}_n$.

Let $\Delta=\begin{bmatrix}
0 & \Delta_{C_\mK} \\
\Delta_{B_\mK} & \Delta_{A_\mK}
\end{bmatrix}\in\mathcal{V}_n$ be arbitrary. Let
$$
\Delta^{(1)} =
\begin{bmatrix}
0 & \Delta_{C_\mK} \\
0 & 0
\end{bmatrix},
\quad
\Delta^{(2)} =
\begin{bmatrix}
0 & 0 \\
\Delta_{B_\mK} & 0
\end{bmatrix},
\quad
\Delta^{(3)} =
\begin{bmatrix}
0 & 0 \\
0 & \Delta_{A_\mK}
\end{bmatrix}.
$$
We have already shown that
$$
\operatorname{Hess}_{\,\mK^\star}(\Delta,\Delta)
=\operatorname{Hess}_{\,\mK^\star}
(\Delta^{(1)}+\Delta^{(2)},\Delta^{(1)}+\Delta^{(2)}).
$$
Let $T$ be an invertible $n\times n$ (complex) matrix that diagonalizes $\Lambda$ as
$$
T\Lambda T^{-1}
=\begin{bmatrix}
-\lambda_1 \\
& \ddots \\
& & -\lambda_n
\end{bmatrix}.
$$
Define
$$
U_{ik}
=e_i^{(m)}{e_k^{(n)}}^\tr T^{-1},
\quad
V_{jk}
=Te_k^{(n)}{e_j^{(p)}}^\tr
$$
for each $1\leq i\leq m$, $1\leq j\leq p$ and $1\leq k\leq n$.
It's not hard to see that $\{U_{ik}\mid 1\leq i\leq m,1\leq k\leq n\}$ forms a basis of $\mathbb{C}^{m\times n}$, and $\{V_{jk}\mid 1\leq j\leq n,1\leq k\leq n\}$ forms a basis of $\mathbb{C}^{n\times q}$. Therefore $\Delta_{C_\mK}$ and $\Delta_{B_\mK}$ can be expanded as
$$
\Delta_{C_\mK}
=\sum_{1\leq i\leq m}\sum_{1\leq k\leq n} \alpha_{ik}U_{ik},
\qquad
\Delta_{B_\mK}
=\sum_{1\leq j\leq q}\sum_{1\leq k\leq n} \beta_{jk}V_{jk}.
$$
By similar derivations as in Case 1, we can get
$$
\begin{aligned}
& \operatorname{Hess}_{\,\mK^\star}(\Delta^{(1)}+\Delta^{(2)},\Delta^{(1)}+\Delta^{(2)}) \\
=\ &
\int_0^{+\infty}
4
\operatorname{tr}
\left(
B\Delta_{C_\mK}
e^{\Lambda s} \Delta_{B_\mK}
C X_{\mathrm{op}} e^{A^\tr s}
Y_{\mathrm{op}}
\right)ds.
\end{aligned}
$$
Then, since
$$
\begin{aligned}
\Delta_{C_\mK}
e^{\Lambda s} \Delta_{B_\mK}
=\ &
\sum_{1\leq i\leq m}
\sum_{1\leq j\leq q}
\sum_{1\leq k\leq n}
\sum_{1\leq k'\leq n}
\alpha_{ik}\beta_{jk'}U_{ik}e^{\Lambda s}V_{jk'} \\
=\ &
\sum_{1\leq i\leq m}
\sum_{1\leq j\leq q}
\sum_{1\leq k\leq n}
\sum_{1\leq k'\leq n}
\alpha_{ik}\beta_{jk'}
e_i^{(m)}{e_k^{(n)}}^\tr
\begin{bmatrix}
e^{-\lambda_1 s} \\
& \ddots \\
& & e^{-\lambda_n s}
\end{bmatrix}
e_{k'}^{(n)}{e_j^{(p)}}^\tr \\
=\ &
\sum_{1\leq i\leq m}
\sum_{1\leq j\leq q}
\sum_{1\leq k\leq n}
\alpha_{ik}\beta_{jk'} e^{-\lambda_k s}
e_i^{(m)}
{e_j^{(p)}}^\tr,
\end{aligned}
$$
we have
$$
\begin{aligned}
& \operatorname{Hess}_{\,\mK^\star}(\Delta^{(1)}+\Delta^{(2)},\Delta^{(1)}+\Delta^{(2)}) \\
=\ &
\sum_{1\leq i\leq m}
\sum_{1\leq j\leq q}
\sum_{1\leq k\leq n}
\int_0^{+\infty}
4\alpha_{ik}\beta_{jk'}
\cdot
{e_j^{(p)}}^\tr
C X_{\mathrm{op}} e^{(A-\lambda_k I)^\tr s}
Y_{\mathrm{op}}B
e_i^{(m)}
\,ds \\
=\ &
\sum_{1\leq i\leq m}
\sum_{1\leq j\leq q}
\sum_{1\leq k\leq n}
4\alpha_{ik}\beta_{jk'}
\cdot
{e_j^{(p)}}^\tr
C X_{\mathrm{op}} \big(\lambda_k I-A^\tr\big)^{-1}
Y_{\mathrm{op}}B
e_i^{(m)}.
\end{aligned}
$$
Since $\operatorname{eig}(-\Lambda)\backslash\mathcal{Z}=\varnothing$, we can see that $C X_{\mathrm{op}} \big(\lambda_k I-A^\tr\big)^{-1}
Y_{\mathrm{op}}B=0$ for any $1\leq k\leq n$. Therefore
$$
\operatorname{Hess}_{\,\mK^\star}(\Delta,\Delta)=
\operatorname{Hess}_{\,\mK^\star}(\Delta^{(1)}+\Delta^{(2)},\Delta^{(1)}+\Delta^{(2)})=0,
$$
which completes the proof.

\subsection{Proof of \cref{proposition:Hessian_example}} \label{appendix:hessian_example}

For each $\epsilon>0$, let the closed-loop system matrix be denoted by
$$
A_{\mathrm{cl}}(\epsilon)
=\begin{bmatrix}
-\frac{3}{2} & 0 & -1 & -1 \\
0 & -\frac{3}{2}(1+\epsilon) & -1-\epsilon & -1-\epsilon \\
1 & 1 & -\frac{7}{2} & -2 \\
1+\epsilon & 1+\epsilon & -2(1+\epsilon) & -\frac{7}{2}(1+\epsilon),
\end{bmatrix}
$$
and let
\begin{align*}
M_{W,V}(\epsilon)
=\ &
\begin{bmatrix}
W & 0 \\
0 &  B_{\mK}^\ast V{B_{\mK}^\ast}^\tr
\end{bmatrix}
=\begin{bmatrix}
4 & 1+\epsilon & 0 & 0 \\
1+\epsilon & 4(1+\epsilon)^2 & 0 & 0 \\
0 & 0 & 1 & 1+\epsilon \\
0 & 0 & 1+\epsilon & (1+\epsilon)^2
\end{bmatrix},
\\
M_{Q,R}(\epsilon)
=\ &
\begin{bmatrix}
Q & 0 \\
0 & {C_{\mK}^\ast}^\tr R C_{\mK}^\ast
\end{bmatrix}
=\begin{bmatrix}
4 & 1 & 0 & 0 \\
1 & 4 & 0 & 0 \\
0 & 0 & 1 & 1 \\
0 & 0 & 1 & 1
\end{bmatrix}
\end{align*}
Let $X_{\mK^\ast}(\epsilon)$ and $Y_{\mK^\ast}(\epsilon)$ denote the solutions to the Lyapunov equations
$$
\begin{aligned}
A_{\mathrm{cl}}(\epsilon)
X_{\mK^\ast}(\epsilon)
+X_{\mK^\ast}(\epsilon)
A_{\mathrm{cl}}(\epsilon)^\tr
+M_{W,V}(\epsilon) & =0, \\
A_{\mathrm{cl}}(\epsilon)^\tr
Y_{\mK^\ast}(\epsilon)
+Y_{\mK^\ast}(\epsilon)
A_{\mathrm{cl}}(\epsilon)
+M_{Q,R}(\epsilon) & =0.
\end{aligned}
$$
By \cref{lemma:Taylor_Lyapunov_eq}, we can compute the Taylor expansions of $X_{\mK^\ast}(\epsilon)$ and $Y_{\mK^\ast}(\epsilon)$, which turn out to be
$$
\begin{aligned}
X_{\mK^\ast}(\epsilon)
=\ &
\frac{1}{7}\begin{bmatrix}
8 & 1 & 1 & 1 \\
1 & 8 & 1 & 1 \\
1 & 1 & 1 & 1 \\
1 & 1 & 1 & 1
\end{bmatrix}
+
\begin{bmatrix}
-1/5 & 1/2 & -1/5 & 1/2 \\
1/2 & 41/5 & 1/2 & 6/5 \\
-1/5 & 1/2 & -1/5 & 1/2 \\
1/2 & 6/5 & 1/2 & 6/5
\end{bmatrix}\frac{\epsilon}{7} \\
& +
\begin{bmatrix}
1/5 & -1/2 & 1/5 & -1/2 \\
-1/2 & 1/5 & -1/2 & 1/5 \\
1/5 & -1/2 & 1/5 & -1/2 \\
-1/2 & 1/5 & -1/2 & 1/5
\end{bmatrix}
\left(\frac{\epsilon^2}{14}
-\frac{\epsilon^3}{28}
+\frac{\epsilon^4}{56}\right)
+o(\epsilon^4),
\end{aligned}
$$
and
$$
\begin{aligned}
Y_{\mK^\ast}(\epsilon)
=\ &
\frac{1}{7}\begin{bmatrix}
8 & 1 & -1 & -1 \\
1 & 8 & -1 & -1 \\
-1 & -1 & 1 & 1 \\
-1 & -1 & 1 & 1
\end{bmatrix}
+
\begin{bmatrix}
-1/5 & -1/2 & 1/5 & 1/2 \\
-1/2 & -39/5 & 1/2 & 4/5 \\
1/5 & 1/2 & -1/5 & -1/2 \\
1/2 & 4/5 & -1/2 & 4/5
\end{bmatrix}\frac{\epsilon}{7} \\
& +
\begin{bmatrix}
1/5 & 1/2 & -1/5 & -1/2 \\
1/2 & 77/5 & -1/2 & -7/5 \\
-1/5 & -1/2 & 1/5 & 1/2 \\
-1/2 & -7/5 & 1/2 & 7/5
\end{bmatrix}
\frac{\epsilon^2}{14}
+
\begin{bmatrix}
-1/5 & -1/2 & 1/5 & 1/2 \\
-1/2 & -153/5 & 1/2 & 13/5 \\
1/5 & 1/2 & -1/5 & -1/2 \\
1/2 & 13/5 & -1/2 & -13/5
\end{bmatrix}\frac{\epsilon^3}{28} \\
& +
\begin{bmatrix}
1/5 & 1/2 & -1/5 & -1/2 \\
1/2 & 61 & -1/2 & -5 \\
-1/5 & -1/2 & 1/5 & 1/2 \\
-1/2 & -5 & 1/2 & 5
\end{bmatrix}
+\frac{\epsilon^4}{56}
+o(\epsilon^4).
\end{aligned}
$$
Next, we let
$$
\begin{aligned}
M_1^{(0)}(\epsilon)
=\ &
\begin{bmatrix}
0 & 0 & 0 & 0 \\
0 & 0 & 0 & 0 \\
0 & 0 & -1/2 & 1/2 \\
0 & 0 & 1/2 & -1/2
\end{bmatrix} X_{\mK^\ast}(\epsilon)
+ X_{\mK^\ast}(\epsilon)\begin{bmatrix}
0 & 0 & 0 & 0 \\
0 & 0 & 0 & 0 \\
0 & 0 & -1/2 & 1/2 \\
0 & 0 & 1/2 & -1/2
\end{bmatrix},
\end{aligned}
$$
which corresponds to the matrix $M_1(X_{\mK^\ast},\Delta_0)$ in \cref{lemma:Jn_Hessian}. Let $X_{\mK^\ast}^{\prime(0)}(\epsilon)$ denote the solution to the Lyapunov equation
$$
A_{\mathrm{cl}}(\epsilon)
X_{\mK^\ast}^{\prime(0)}(\epsilon)
+X_{\mK^\ast}^{\prime(0)}(\epsilon)A_{\mathrm{cl}}(\epsilon)^\tr
+M_1^{(0)}(\epsilon)=0.
$$
Then similarly by \cref{lemma:Taylor_Lyapunov_eq}, we can compute the Taylor expansion of $X_{\mK^\ast}^{\prime(0)}(\epsilon)$, which is given by
$$
\begin{aligned}
X_{\mK^\ast}^{\prime(0)}(\epsilon)
=\ &
\begin{bmatrix}
0 & 0 & 1 & -1 \\
0 & 0 & 1 & -1 \\
1 & 1 & 2 & 0 \\
-1 & -1 & 0 & -2
\end{bmatrix}\frac{\epsilon}{100}
+
\begin{bmatrix}
-6 & -6 & -69 & 78 \\
-6 & -6 & -20 & 29 \\
-69 & -20 & -132 & 64 \\
78 & 29 & 64 & 64
\end{bmatrix}\frac{\epsilon^2}{9800} \\
& +
\begin{bmatrix}
12 & 12 & 89 & -107 \\
12 & 12 & -9 & -9 \\
89 & -9 & 166 & -128 \\
-107 & -9 & -128 & -30
\end{bmatrix}
\frac{\epsilon^3}{19600}
+
\begin{bmatrix}
-18 & -18 & -109 & 136 \\
-18 & -18 & 38 & -11 \\
-109 & 38 & -200 & 192 \\
136 & -11 & 192 & -4
\end{bmatrix}
\frac{\epsilon^4}{39200}
+o(\epsilon^4).
\end{aligned}
$$
By \cref{lemma:Jn_Hessian}, we then have
$$
\mathrm{Hess}_{\,\mK^\ast}(\Delta_0,\Delta_0)
=4\operatorname{tr}
\left(
\begin{bmatrix}
0 & 0 & 0 & 0 \\
0 & 0 & 0 & 0 \\
0 & 0 & -1/2 & 1/2 \\
0 & 0 & 1/2 & -1/2
\end{bmatrix}X_{\mK^\ast}^{\prime(0)}(\epsilon) Y_{\mK^\ast}(\epsilon)
\right)
=\frac{3}{7000}\epsilon^4+o(\epsilon^4).
$$

Similarly, to compute the leading term of the Taylor expansion of $\mathrm{Hess}_{\,\mK^\ast}(\Delta_1,\Delta_1)$, we let
$$
\begin{aligned}
M_1^{(1)}(\epsilon)
=\ &
\begin{bmatrix}
0 & 0 & -1/2 & -1/2 \\
0 & 0 & -(1+\epsilon)/2 & -(1+\epsilon)/2 \\
1/2 & 1/2 & 0 & 0 \\
1/2 & 1/2 & 0 & 0
\end{bmatrix}
X_{\mK^\ast}(\epsilon) \\
&
+
X_{\mK^\ast}(\epsilon)
\begin{bmatrix}
0 & 0 & -1/2 & -1/2 \\
0 & 0 & -(1+\epsilon)/2 & -(1+\epsilon)/2 \\
1/2 & 1/2 & 0 & 0 \\
1/2 & 1/2 & 0 & 0
\end{bmatrix}
+
\begin{bmatrix}
0 & 0 & 0 & 0 \\
0 & 0 & 0 & 0 \\
0 & 0 & 1 & (2+\epsilon)/2 \\
0 & 0 & (2+\epsilon)/2 & 1+\epsilon
\end{bmatrix},
\end{aligned}
$$
which corresponds to the matrix $M_1(X_{\mK^\ast},\Delta_1)$ in \cref{lemma:Jn_Hessian}. Let $X_{\mK^\ast}^{\prime(1)}(\epsilon)$ denote the solution to the Lyapunov equation
$$
A_{\mathrm{cl}}(\epsilon)
X_{\mK^\ast}^{\prime(1)}(\epsilon)
+X_{\mK^\ast}^{\prime(1)}(\epsilon)A_{\mathrm{cl}}(\epsilon)^\tr
+M_1^{(1)}(\epsilon)=0.
$$
Then by \cref{lemma:Taylor_Lyapunov_eq}, we have
$$
X_{\mK^\ast}^{\prime(1)}(\epsilon)
=\frac{1}{686}
\begin{bmatrix}
-72 & -72 & 5 & 5 \\
-72 & -72 & 5 & 5 \\
5 & 5 & 82 & 82 \\
5 & 5 & 82 & 82
\end{bmatrix}
+o(1),
$$
and then by \cref{lemma:Jn_Hessian}, we can show that
$$
\mathrm{Hess}_{\,\mK^\ast}(\Delta_1,\Delta_1)
= \frac{680}{343}+o(1).
$$

Finally, we show that $\left\|\operatorname{Proj}_{\mathcal{TO}_{\mK^\ast}}[\Delta_0]\right\|_F=O(\epsilon)$. It can be shown that for any $\Delta=\begin{bmatrix}
0 & \Delta_{C_{\mK}} \\
\Delta_{B_{\mK}} & \Delta_{A_{\mK}}
\end{bmatrix}$, we have
$$
\begin{aligned}
& \Delta_{A_\mK}{A_{\mK}^\ast}^\tr
-{A_{\mK}^\ast}^\tr\Delta_{A_\mK}
+ \Delta_{B_\mK}{B_{\mK}^\ast}^\tr-{C_{\mK}^\ast}^\tr \Delta_{C_\mK}=0 \\
\Longleftrightarrow\quad &
\begin{bmatrix}
A_{\mK}^\ast\otimes I_n-I_n\otimes {A_{\mK}^\ast}^\tr &
B_{\mK}^\ast\otimes I_n &
-I_n\otimes {C_{\mK}^\ast}^\tr
\end{bmatrix}
\begin{bmatrix}
\operatorname{vec}(\Delta_{A_\mK}) \\
\operatorname{vec}(\Delta_{B_\mK}) \\
\operatorname{vec}(\Delta_{C_\mK})
\end{bmatrix}
=0.
\end{aligned}
$$
Denoting
$$
\begin{aligned}
\mathsf{M}
=\ &
\begin{bmatrix}
A_{\mK}^\ast\otimes I_n-I_n\otimes {A_{\mK}^\ast}^\tr &
B_{\mK}^\ast\otimes I_n &
-I_n\otimes {C_{\mK}^\ast}^\tr
\end{bmatrix} \\
=\ &
\begin{bmatrix}
0 & 2(1+\epsilon) & -2 & 0 & 1 & 0 & 1 & 0 \\
2 &	7\epsilon/2 & 0 & -2 & 0 & 1 & 1 & 0 \\
-2(1+\epsilon) & 0 & -7\epsilon/2 & 2(1+\epsilon) & 1+\epsilon & 0 & 0 & 1 \\
0 & -2(1+\epsilon) & 2 & 0 & 0 & 1+\epsilon & 0 & 1
\end{bmatrix}.
\end{aligned}
$$
Since for $\epsilon>0$, $\dim\mathcal{TO}_{\mK^\ast}=n^2=4$, we can see that
$$
\operatorname{rank}\mathsf{M}=n^2+nm+np-\dim\operatorname{ker}\mathsf{M}
=8-\dim\mathcal{TO}_{\mK^\ast}
=4.
$$
By \cref{proposition:tangent_orbit}, we can obtain $\|\operatorname{Proj}_{\mathcal{TO}_{\mK^\ast}}[\Delta_0]\|_F$ by computing
$$
\left\|\mathsf{M}^\tr(\mathsf{M}\mathsf{M}^\tr)^{-1}\mathsf{M}v_0\right\|
$$
where
$v_0=
\begin{bmatrix}
-1/2 & 1/2 & 1/2 & -1/2 & 0 & 0 & 0 & 0
\end{bmatrix}^\tr
$. We note that
$$
\mathsf{M}\mathsf{M}^\tr
=\begin{bmatrix}
10+8\epsilon+4\epsilon^2 &
1+7\epsilon+7\epsilon^2 &
1+8\epsilon &
-8-8\epsilon-4\epsilon^2 \\
1+7\epsilon+7\epsilon^2 &
10+49\epsilon^2/4 &
-8-8\epsilon &
1-6\epsilon-7\epsilon^2 \\
1+8\epsilon &
-8-8\epsilon &
10 + 18\epsilon + 85 \epsilon^2/4 &
1-7\epsilon \\
-8-8\epsilon-4\epsilon^2 &
1-6\epsilon-7\epsilon^2 &
1-7\epsilon &
10+10\epsilon+5\epsilon^2
\end{bmatrix}.
$$
It can be checked that
$$
\begin{aligned}
\mathsf{M}\mathsf{M}^\tr
\begin{bmatrix}
-72 - 84 \epsilon + 3 \epsilon^2 + 29 \epsilon^3 + 49 \epsilon^4/2 \\
72 + 126 \epsilon + 60 \epsilon^2 + 13 \epsilon^3 - 14 \epsilon^4 \\
72 - 30 \epsilon - 36 \epsilon - 35 \epsilon^3 \\
-72 - 12 \epsilon- 27 \epsilon^2 + 14 \epsilon^3
\end{bmatrix}
=\ &
(432 + 840\epsilon + 1122 \epsilon^2 + 702 \epsilon^3 + 249 \epsilon^4)\begin{bmatrix}
\epsilon \\
7\epsilon/4 \\
-7\epsilon/4 \\
-\epsilon
\end{bmatrix} \\
=\ &
(432 + 840\epsilon + 1122 \epsilon^2 + 702 \epsilon^3 + 249 \epsilon^4)\mathsf{M}v_0,
\end{aligned}
$$
implying that
$$
\begin{aligned}
(\mathsf{M}\mathsf{M}^\tr)^{-1}
\mathsf{M}v_0
=\ &
\frac{1}{432 + 840\epsilon + 1122 \epsilon^2 + 702 \epsilon^3 + 249 \epsilon^4}
\begin{bmatrix}
-72 - 84 \epsilon + 3 \epsilon^2 + 29 \epsilon^3 + 49 \epsilon^4/2 \\
72 + 126 \epsilon + 60 \epsilon^2 + 13 \epsilon^3 - 14 \epsilon^4 \\
72 - 30 \epsilon - 36 \epsilon - 35 \epsilon^3 \\
-72 - 12 \epsilon- 27 \epsilon^2 + 14 \epsilon^3
\end{bmatrix} \\
=\ &
\begin{bmatrix}
-1/6 \\ 1/6 \\ 1/6 \\ -1/6
\end{bmatrix}
+\begin{bmatrix}
28 \\ -7 \\ -85 \\ 64
\end{bmatrix}
\frac{\epsilon}{216}
+o(\epsilon).
\end{aligned}
$$
Since
$$
\mathsf{M}^\tr
\begin{bmatrix}
-1/6 \\ 1/6 \\ 1/6 \\ -1/6
\end{bmatrix}=O(\epsilon),
$$
we can see that
$$
\mathsf{M}^\tr(\mathsf{M}\mathsf{M}^\tr)^{-1}
\mathsf{M}v_0
=O(\epsilon),
$$
which completes the proof.

% {\color{red}
% \subsection{Non-minimal LQG optimal controller}

% a brief discussion.

% Especially when system is only detectable and stabilizable
% }

% \subsection{Equivalent convex reformulation for the LQG problem~\eqref{eq:LQG}}
% \label{appendix:LQG_convex_reformulation}

%copy from a textbook.
\section{Connectivity of the Set of Proper Stabilizing Controllers} \label{appendix:proper_controller}
%We conclude this section by showing

We present connectivity results for the set of proper stabilizing controllers. The dynamical controller in~\eqref{eq:Dynamic_Controller} is \emph{strictly proper} as it does not contain a direct feedback term from the output measurement. We note that the optimal solution for the LQG problem~\eqref{eq:LQG} is always strictly proper.

For closed-loop stability, we can also consider a \emph{proper} dynamical controller as follows
\begin{equation}
    \begin{aligned}
        \dot \xi(t) &= A_{\mK}\xi(t) + B_{\mK}y(t), \\
        u(t) &= C_{\mK}\xi(t) + D_{\mK}y(t),
    \end{aligned}
\end{equation}
parameterized by four matrices $A_{\mK}, B_{\mK}, C_{\mK}, D_{\mK}$ with compatible dimensions. Similarly, we define the set of \emph{proper stabilizing controllers} as
\begin{equation}\label{eq:sys_closed_loop}
\hat{\mathcal{C}}_{q}:= \left\{\mK = \begin{bmatrix} D_{\mK} & C_{\mK} \\  B_{\mK} & A_{\mK} \end{bmatrix} \in \mathbb{R}^{(p+q) \times (m+q)}   \left|
\begin{bmatrix}
A+BD_{\mK}C & BC_{\mK} \\ B_{\mK}C & A_{\mK}
\end{bmatrix}\; \text{is stable }\right. \right\}.
\end{equation}
By this definition, we always have $\mathcal{C}_q \subseteq \hat{\mathcal{C}}_q$, which is consistent with the fact that the set of strictly proper stabilizing controllers is a subset of the set of proper stabilizing controllers. But we note that $\forall \mK \in \hat{\mathcal{C}}_{q}$ with $D_{\mK} \neq 0$, the resulting LQG cost $J(\mK)$ in~\eqref{eq:LQG} is infinite, despite that $\mK$ internally stabilizes the plant.

Similar to \cref{lemma:unbounded_proper}, the observation in \cref{lemma:unbounded_proper2} is obvious. Unlike $\mathcal{C}_n$ that might have two path-connected components, $\hat{\mathcal{C}}_n$ is always path-connected, as stated in \cref{prop:connectivity_Cn}.
\begin{lemma} \label{lemma:unbounded_proper2}
   Under \cref{assumption:stabilizability}, the set $\hat{\mathcal{C}}_{n}$ is non-empty, open, unbounded and non-convex.
\end{lemma}

\begin{theorem} \label{prop:connectivity_Cn}
    Under \cref{assumption:stabilizability}, $\hat{\mathcal{C}}_{n}$ is always path-connected.
\end{theorem}

The proof of \cref{prop:connectivity_Cn} is almost identical to \cref{Theo:disconnectivity}. By replacing the constraint $R=0_{m\times p}$ with $R \in \mathbb{R}^{m\times p}$ in the definitions of $\mathcal{F}_n$, $\mathcal{G}_n$ and $\Phi(\cdot)$, it is not difficult to verify that the results in \cref{proposition:Phi_surjective} and \cref{lemma:Gn_connected_components} still hold for  $\hat{\mathcal{C}}_{n}$.
%
%Unlike the set ${\mathcal{C}}_{n}$ that might have two connected components, the set $\hat{\mathcal{C}}_{n}$ is always path-connected.
Unlike ${\mathcal{C}}_{n-1}$ might be empty, we always have $\hat{\mathcal{C}}_{n-1} \neq \varnothing$ under \cref{assumption:stabilizability}~\cite{brasch1970pole}. By adapting the proof in \cref{Theo:connectivity_conditions}, \cref{prop:connectivity_Cn} is now obvious.

\begin{example}[Connectivity of proper stabilizing controllers]
    Consider the linear system~\eqref{eq:ExampleSISO2} in \cref{appendix:eg_disconnectivity}. We have shown that $\mathcal{C}_{n-1} = \varnothing$ for \emph{strictly proper} reduced-order dynamical controllers. Here, it is easy to verify that the following \emph{proper} reduced-order  dynamical controller
    $$
    A_{\mK} = 1, \quad B_{\mK} = -3, \quad C_{\mK} = 2, \quad D_{\mK} = -2,
    $$
    internally stabilizes the system, \emph{i.e.}, the eigenvalues of
     $$
            \begin{bmatrix}
            0 & 1 & 0 \\
            1 & D_{\mK} & C_{\mK} \\
            0 & B_{\mK} & A_{\mK}
            \end{bmatrix} %\qquad \text{where}\quad a = 1, b = -3, c = 2, d = -2,
    $$
    have all negative real parts, indicating that $\hat{\mathcal{C}}_{n-1} \neq \emptyset$. Thus, $\hat{\mathcal{C}}_n$ is path-connected.

    Similarly, one can verify that the set of proper stabilizing controllers for the system in \cref{example:SISO1} is path-connected. Indeed, using the Routh--Hurwitz stability criterion, we derive that
\begin{equation*}%\label{eq:region_example}
    \begin{aligned}
    \hat{\mathcal{C}}_1 &= \left\{\mK = \begin{bmatrix} D_{\mK} & C_{\mK} \\
                          B_{\mK} & A_{\mK}\end{bmatrix} \in \mathbb{R}^{2 \times 2} \left| \begin{bmatrix}
    A+BD_{\mK}C & BC_{\mK} \\
    B_{\mK}C & A_{\mK}
\end{bmatrix}\; \text{is stable}\right. \right\} \\
&= \left\{ \left. \mK = \begin{bmatrix} D_{\mK} & C_{\mK} \\
                          B_{\mK} & A_{\mK}\end{bmatrix} \in \mathbb{R}^{2 \times 2} \right| A_{\mK} + D_{\mK}< -1, \;\; B_{\mK}C_{\mK} < A_{\mK} + A_{\mK}D_{\mK} \right\}. \\
\end{aligned}
\end{equation*}
This set is path-connected.
\end{example}

%% {\color{red}
 \section{Results for Discrete-Time Systems} \label{appendix:discrete_time}

In this section, we discuss some landscape properties for the discrete-time LQG problem. As we will see, most results are analogous to the continuous-time case. We will slightly abuse the use of notation, and adopt the same notation for both continuous-time and discrete-time cases.

Consider a discrete-time partially observed LTI system
\begin{equation} \label{eq:dynamics_discrete}
    \begin{aligned}
        x_{t+1} &= A x_t + B u_t + w_t, \\
        y_t &= Cx_t + v_t,
    \end{aligned}
\end{equation}
where $x_t \in \mathbb{R}^n, u_t \in \mathbb{R}^m, y_t \in \mathbb{R}^p$ are the system state, input, and output measurement at time $t$, and $w_t \sim \mathcal{N}(0, W), v_t \sim \mathcal{N}(0, V)$ are Gaussian process and measurement noises, respectively.  It is assumed that the covariance matrices satisfy $W \succeq 0, V \succ 0$.
Given performance weight matrices $Q \succeq 0, R \succ 0$, the discrete-time LQG problem is defined as
\begin{equation}~\label{eq:LQG_discrete}
    \begin{aligned}
        \min_{u_0,u_1, \ldots} &\quad \lim_{T \rightarrow \infty} \mathbb{E} \left[ \frac{1}{T} \sum_{t=1}^T\!\left(x_t^\tr Q x_t + u_t^\tr R u_t\right) \right] \\
        \text{subject to} &\quad~\eqref{eq:dynamics_discrete}.
    \end{aligned}
\end{equation}
The control input $u_t$ at time $t$ is allowed to depend on the history $\mathcal{H}_t:=(u_0, \ldots, u_{t-1}, y_0, \ldots, y_{t-1})$. \footnote{The one step delay in $y_t$ is a standard assumption that simplifies the Kalman filtering expressions. We leave the case where the history contains the current measurement $y_t$ for future discussions.} We make the following standard assumption.

\begin{assumption} \label{assumption:stabilizability_discrete}
$(A,B)$ and $(A,W^{1/2})$ are controllable, and $(C,A)$ and $(Q^{1/2},A)$ are observable.
% The dynamics of the plant \eqref{eq:Dynamic} are controllable and observable. %In addition, $Q$ and $R$ are positive definite.
\end{assumption}

Under Assumption~\ref{assumption:stabilizability_discrete}, the optimal solution to~\eqref{eq:LQG_discrete} is a dynamical controller, given by  %$u_t = -K \xi_t$, where $\xi_t$ is equal to the optimal state estimation $\mathbb{E}[x_t | \mathcal{H}_t] $, and $K \in \mathbb{R}^{m \times n}$ is the optimal LQR feedback gain. Precisely, the optimal controller is given by
\begin{equation} \label{eq:LQGcontroller_discrete-time}
    \begin{aligned}
        \xi_{t+1} &= A\xi_{t} + Bu_{t} + L(y_t - C \xi_t), \\
        u_t &= -K \xi_t.
    \end{aligned}
\end{equation}
Here the matrix $L$ is called the \emph{Kalman gain}, computed as $L = APC^\tr (CPC^\tr + V)^{-1}$ where $P$ is the unique positive semidefinite solution to
\begin{equation}\label{eq:Riccati_P_discrete}
    P = APA^\tr - APC^\tr(CPC^\tr+V)^{-1}CPA^\tr +W,
\end{equation}
and  the matrix $K$ is called the LQR \emph{feedback gain}, computed as $  K = (B^\tr SB + R)^{-1}B^\tr SA
$ where $S$ is the unique positive semidefinite solution to
\begin{equation} \label{eq:Riccati_S_discrete}
    S = A^\tr SA - A^\tr SB(B^\tr SB+R)^{-1}B^\tr SA +Q.
\end{equation}

\subsection{Controller Parameterization and the LQG Cost Function}
Similar to the continuous-time case, we consider the following parameterization of dynamical controllers
\begin{equation} \label{eq:controller_discrete}
    \begin{aligned}
        \xi_{t+1} &= A_{\mK} \xi_t + B_{\mK} y_t, \\
        u_t &= C_{\mK} \xi_{t},
    \end{aligned}
\end{equation}
where $\xi_t \in \mathbb{R}^q$ is the controller state at time $t$, and $(A_{\mK},B_{\mK},C_{\mK}) \in \mathbb{R}^{q \times q} \times \mathbb{R}^{q \times p} \times \mathbb{R}^{m \times q}$ specify the controller dynamics. The optimal LQG controller~\eqref{eq:LQGcontroller_discrete-time} can be written in the form of~\eqref{eq:controller_discrete}, where the controller state has dimension $q = n$, and
$$
 \quad A_{\mK} = A - BK - LC, \quad B_{\mK} = L, \quad C_{\mK} = - \mK.
$$

Combining~\eqref{eq:controller_discrete} with~\eqref{eq:dynamics_discrete} leads to the closed-loop system as
\begin{equation*}
\begin{aligned}
   \begin{bmatrix} x_{t+1} \\ \xi_{t+1} \end{bmatrix} &= \begin{bmatrix}
    A & BC_{\mK} \\
    B_{\mK}C & A_{\mK}
    \end{bmatrix} \begin{bmatrix} x_t \\ \xi_t \end{bmatrix}  +  \begin{bmatrix} I & 0 \\ 0 & B_{\mK}  \end{bmatrix}\begin{bmatrix} w_t \\ v_t \end{bmatrix} \\
    \begin{bmatrix} y_t \\ u_t \end{bmatrix}& = \begin{bmatrix} C & 0 \\ 0& C_{\mK} \end{bmatrix} \begin{bmatrix} x_t \\ \xi_t \end{bmatrix} + \begin{bmatrix} v_t \\0 \end{bmatrix}.
\end{aligned}
\end{equation*}
The set of stabilizing controllers with order $q \in \mathbb{N}$ is defined as
\begin{equation} \label{eq:internallystabilizing_discrete}
    \mathcal{C}_{q} := \left\{
    \left.\mK=\begin{bmatrix}
    0_{m\times p} & C_{\mK} \\
    B_{\mK} & A_{\mK}
    \end{bmatrix}
    \in \mathbb{R}^{(m+q) \times (p+q)} \right|\;  \rho\left(\begin{bmatrix}
    A & BC_{\mK} \\
    B_{\mK}C & A_{\mK}
    \end{bmatrix}\right) < 1\right\},
\end{equation}
where $\rho(\cdot)$ denotes the spectral radius of a square matrix.
Let $J_q(\mK):\mathcal{C}_q\rightarrow\mathbb{R}$ denote the function that maps a parameterized dynamical controller in $\mathcal{C}_q$ to its corresponding LQG cost for each $q\in\mathbb{N}$.
Analogous to the continuous time case, we have the following two lemmas characterizing the LQG cost function $J_q$.
\begin{lemma}\label{lemma:LQG_cost_formulation_discrete}
Fix $q\in\mathbb{N}$ such that $\mathcal{C}_q\neq\varnothing$. Given $\mK\in\mathcal{C}_q$, we have
%{\color{blue}Given $\mK\in\mathcal{C}_q$, if $\mK$ is controllable and observable,} we have
\begin{equation}\label{eq:LQG_cost_formulation_discrete}
J_q(\mK)
=
\operatorname{tr}
\left(
\begin{bmatrix}
Q & 0 \\ 0 & C_{\mK}^\tr R C_{\mK}
\end{bmatrix} X_\mK\right)
=
\operatorname{tr}
\left(
\begin{bmatrix}
W & 0 \\ 0 & B_{\mK} V B_{\mK}^\tr
\end{bmatrix} Y_\mK\right),
\end{equation}
where $X_{\mK}$ and $Y_{\mK}$ are the unique positive semidefinite
solutions to the following Lyapunov equations
\begin{subequations}
\begin{align}
X_{\mK} &= \begin{bmatrix} A &  BC_{\mK} \\ B_{\mK} C & A_{\mK} \end{bmatrix}X_{\mK}\begin{bmatrix} A &  BC_{\mK} \\ B_{\mK} C & A_{\mK} \end{bmatrix}^\tr +  \begin{bmatrix} W & 0 \\ 0 & B_{\mK}VB_{\mK}^\tr  \end{bmatrix}, \label{eq:LyapunovX_discrete}
\\
Y_{\mK} &= \begin{bmatrix} A &  BC_{\mK} \\ B_{\mK} C & A_{\mK} \end{bmatrix}^\tr Y_{\mK}\begin{bmatrix} A &  BC_{\mK} \\ B_{\mK} C & A_{\mK} \end{bmatrix} +   \begin{bmatrix} Q & 0 \\ 0 & C_{\mK}^\tr R C_{\mK} \end{bmatrix}. \label{eq:LyapunovY_discrete}
\end{align}
%As a consequence, $J_q$ is a real analytic function on $\mathcal{C}_q$.
\end{subequations}
\end{lemma}
\begin{lemma}\label{lemma:LQG_cost_analytical_discrete}
    Fix $q \in \mathbb{N}$ such that $\mathcal{C}_q\neq\varnothing$. Then, $J_q$ is a real analytic function on $\mathcal{C}_q$.
\end{lemma}
%might not approach to infinity as $\mK$ goes to the boundary of $\mathcal{C}_q$
The LQG cost function being real analytical is a direct consequence of~\cref{lemma:LQG_cost_formulation_discrete}, and the proof is identical to~\cref{lemma:LQG_cost_analytical}.
Given the dimension $n$ of the plant's state variable, the discrete-time LQG problem~\eqref{eq:LQG_discrete} can be reformulated into a constrained optimization problem:
%\begin{subequations}
\begin{equation} \label{eq:LQG_reformulation_KX_discrete}
    \begin{aligned}
        \min_{\mK} \quad &J_n(\mK) \\
        \text{subject to} \quad& \mK \in \mathcal{C}_n.
    \end{aligned}
\end{equation}

\subsection{Connectivity of the Feasible Region $\mathcal{C}_n$}
We now characterize the connectivity of the set of full-order stabilizing controllers $\mathcal{C}_n$.

\begin{lemma} \label{lemma:unbounded_proper_discrete}
    Under \cref{assumption:stabilizability_discrete}, the set ${\mathcal{C}}_{n}$ in~\eqref{eq:internallystabilizing_discrete} is non-empty, unbounded, and can be non-convex.
\end{lemma}

\begin{example}[Non-convexity of stabilizing controllers] \label{example:nonconvexity_discrete}
Consider an open-loop unstable dynamical system~\eqref{eq:dynamics_discrete} with
\begin{equation*} %\label{eq:CounterExample}
   A = 1.1, \;\; B = 1, \;\;  C = 1.
\end{equation*}
The set of stabilizing controllers $\mathcal{C}_n=\mathcal{C}_1$ is given by
$$
\mathcal{C}_n
=\left\{
\left.
\mK = \begin{bmatrix} 0 & C_{\mK} \\
                          B_{\mK} & A_{\mK}\end{bmatrix} \in \mathbb{R}^{2 \times 2}
\right|
\begin{bmatrix}
1.1 & C_{\mK} \\
B_{\mK} & A_{\mK}
\end{bmatrix}
\text{ is stable}
\right\}.
$$
It is easy to verify that the following dynamical controllers
$$
    \mK^{(1)} = \begin{bmatrix} 0 & 0.5 \\
                          -0.5 & 0\end{bmatrix}, \qquad     \mK^{(2)} = \begin{bmatrix} 0 & -0.5 \\
                          0.5 & 0\end{bmatrix}
$$
stabilize the plant and thus belong to $\mathcal{C}_1$. However,
$
   \hat{\mK} = \frac{1}{2}\left(\mK^{(1)} + \mK^{(2)}\right) = \begin{bmatrix} 0 & 0 \\
                          0 & 0\end{bmatrix}
$
fails to stabilize the plant. \hfill\qed
\end{example}

The notion of similarity transformation for the discrete-time case remains the same as that of the continuous-time case. We thus use the same mapping  $\mathscr{T}_q:\mathrm{GL}_q\times\mathcal{C}_q\rightarrow\mathcal{C}_q$, defined in~\eqref{eq:def_sim_transform}, to represent similarity transformations on $\mathcal{C}_q$. It is not difficult to see that~\cref{lemma:Cq_invariant} also holds in discrete-time. Indeed, All the connectivity results in~\cref{Theo:disconnectivity,Theo:Cn_homeomorphic,Theo:connectivity_conditions} have their counterparts in the discrete-time case.

\begin{theorem} \label{Theo:disconnectivity_discrete}
Under~\cref{assumption:stabilizability_discrete}, the set ${\mathcal{C}}_{n}$ in~\eqref{eq:internallystabilizing_discrete} has at most two path-connected components.
\end{theorem}

\begin{theorem} \label{Theo:Cn_homeomorphic_discrete}
{If ${\mathcal{C}}_{n}$ in~\eqref{eq:internallystabilizing_discrete} has two path-connected components $\mathcal{C}_n^{(1)}$ and $\mathcal{C}_n^{(2)}$, then $\mathcal{C}_n^{(1)}$ and $\mathcal{C}_n^{(2)}$ are diffeomorphic under the mapping $\mathscr{T}_T$, for any invertible matrix $T\in\mathbb{R}^{n\times n}$ with $\det T<0$. % restricted on $\mathcal{C}_n^{(1)}$ gives a homeomorphism from $\mathcal{C}_n^{(1)}$ to $\mathcal{C}_n^{(2)}$.
}
\end{theorem}

\begin{theorem} \label{Theo:connectivity_conditions_discrete}
Under~\cref{assumption:stabilizability}, the following statements are true.
\begin{enumerate}
    \item
$\mathcal{C}_n$ is path-connected if there exists a reduced-order stabilizing controller, \emph{i.e.}, ${\mathcal{C}}_{n-1} \neq \varnothing$.

\item Suppose the plant \eqref{eq:Dynamic} is single-input or single-output, i.e., $m=1$ or $p=1$. Then the set ${\mathcal{C}}_{n}$ is path-connected if and only if ${\mathcal{C}}_{n-1} \neq \varnothing$.
\end{enumerate}
\end{theorem}

\begin{example}[Disconectivity of stabilizing controllers in discrete-time]\label{example:disconnectivity_discrete}
 Consider the discrete-time dynamical system in~\cref{example:nonconvexity_discrete}:
$$
A=1.1,\quad B=1,\quad C=1.
$$
Since it is open-loop unstable and only has state of dimension $n = 1$, we know $\mathcal{C}_{n-1} = \varnothing$. Thus, \cref{Theo:connectivity_conditions_discrete} indicates that its associated set of stabilizing controllers $\mathcal{C}_n$ is not path-connected.

Indeed, using the Jury stability criterion~\cite{jury1964theory,keel1999new} (the discrete-time analogue of the Routh–Hurwitz stability criterion), we can derive that
\begin{equation*}%\label{eq:region_example}
    \begin{aligned}
    \mathcal{C}_1 &= \left\{\mK = \begin{bmatrix} 0 & C_{\mK} \\
                          B_{\mK} & A_{\mK}\end{bmatrix} \in \mathbb{R}^{2 \times 2} \left| \rho\left(\begin{bmatrix}
    1.1 & C_{\mK} \\
    B_{\mK} & A_{\mK}
\end{bmatrix}\right) < 1 \right. \right\} \\
&= \left\{ \left. \mK = \begin{bmatrix} 0 & C_{\mK} \\
                          B_{\mK} & A_{\mK}\end{bmatrix} \in \mathbb{R}^{2 \times 2} \right|  \begin{aligned}B_{\mK}C_{\mK} &< 1 + 1.1 A_{\mK} - |1.1 + A_{\mK}| \\
                          B_{\mK}C_{\mK} &> 1.1 A_{\mK} - 1\end{aligned}\right\}. \\
\end{aligned}
\end{equation*}
Note that we should have $1.1A_{\mK}-1<1+1.1A_{\mK}-|1.1+A_\mK|$ to guarantee $\mathcal{C}_1 \neq \emptyset$, which gives $-3.1 < A_{\mK} < 0.9$. Furthermore, it is not difficult to verify that $1 + 1.1 A_{\mK} - |1.1 + A_{\mK}| < 0, \forall A_{\mK} \in (-3.1,0.9)$.
One can then verify that the set $\mathcal{C}_1$ has two path-connected components: $\mathcal{C}_1 = \mathcal{C}_1^+ \cup \mathcal{C}_1^-$ with $\mathcal{C}_1^+ \cap \mathcal{C}_1^- = \emptyset$, where
$$
\begin{aligned}
    \mathcal{C}_1^+ &:= \left\{ \left. \mK = \begin{bmatrix} 0 & C_{\mK} \\
                          B_{\mK} & A_{\mK}\end{bmatrix} \in \mathbb{R}^{2 \times 2} \right|  \begin{aligned}B_{\mK}C_{\mK} &< 1 + 1.1 A_{\mK} - |1.1 + A_{\mK}| \\
                          B_{\mK}C_{\mK} &> 1.1 A_{\mK} - 1, \quad B_{\mK} > 0\end{aligned}\right\} ,
                          \\
                          \mathcal{C}_1^- &:=  \left\{ \left. \mK = \begin{bmatrix} 0 & C_{\mK} \\
                          B_{\mK} & A_{\mK}\end{bmatrix} \in \mathbb{R}^{2 \times 2} \right|  \begin{aligned}B_{\mK}C_{\mK} &< 1 + 1.1 A_{\mK} - |1.1 + A_{\mK}| \\
                          B_{\mK}C_{\mK} &> 1.1 A_{\mK} - 1, \quad B_{\mK} < 0\end{aligned}\right\}.
                         \\
\end{aligned}
$$
In addition, as expected from \cref{Theo:Cn_homeomorphic_discrete}, it can be verified that $\mathcal{C}_1^{+}$ and $\mathcal{C}_1^{-}$ are homeomorphic under the mapping $\mathscr{T}_T$ for any $T<0$. \hfill\qed
\end{example}

\subsubsection{Proofs of~\cref{Theo:disconnectivity_discrete,Theo:Cn_homeomorphic_discrete,Theo:connectivity_conditions_discrete} }

The proof ideas are almost identical to those for the continuous-time case, and we highlight the main steps here. We utilize the following discrete-time Lyapunov stability criterion: given a square matrix $M \in \mathbb{R}^{n \times n}$, we have $\rho(M) < 1$ if and only if the following LMI is feasible
\begin{equation} \label{eq:Lyapunov_discrete}
    M^\tr X M - X \prec\,0, \quad X \succ\, 0.
\end{equation}
Upon defining $P = X^{-1}$ and using the Schur complement, the discrete-time Lyapunov LMI~\eqref{eq:Lyapunov_discrete} is equivalent to
\begin{equation} \label{eq:Lyapunov_discrete_v2}
    \begin{bmatrix}
         P & MP \\
         PM^\tr & P
    \end{bmatrix} \succ 0
\end{equation}
%Following the argument in~\Cref{remark:connectivity}, it is easy to see that the set of stabilizing static state-feedback policies in discrete-time
%$
%\{K \in \mathbb{R}^{m \times n} \mid \rho(A - BK) < 1\}
%$ is path-connected.

Now, we can use the same change of variables, defined in~\eqref{eq:change_of_variables_output}, to prove~\cref{Theo:disconnectivity_discrete}. Given the system dynamics $(A, B, C)$ in~\eqref{eq:dynamics_discrete}, we first introduce the following convex set

\begin{equation*} %\label{eq:SetF}
\begin{aligned}
\mathcal{F}_n
\coloneqq
\bigg\{
(X,Y,M,&\,G,H,F) \mid
X,Y\in\mathbb{S}^n,\
M \in \mathbb{R}^{n \times n}, G=0_{m\times p}, H\in \mathbb{R}^{n \times p}, F \in \mathbb{R}^{m \times n},\\
&
\left[\begin{array}{c:c}
\begin{bmatrix}
X & I \\ I & Y
\end{bmatrix} & \begin{bmatrix}
AX \!+\! BF & A \!+\! BGC \\ M & YA \!+\! HC
\end{bmatrix} \\
    \hdashline  %\\ \vspace{-3pt}
\begin{bmatrix}
AX \!+\! BF & A \!+\! BGC \\ M & YA \!+\! HC
\end{bmatrix}^{\!\tr}
 &  \begin{bmatrix}
X & I \\ I & Y
\end{bmatrix}
\end{array}\right] \succ 0
\bigg\},
\end{aligned}
\end{equation*}
and the extended set
\begin{equation*} %\label{eq:SetG_discrete}
\mathcal{G}_n
:=\left\{\mZ=(X,Y,M,G,H,F,\Pi,\Xi) \left|\; \begin{aligned}(X,Y,M,G,H,F)\in\mathcal{F}_n, \\
\Pi,\Xi\in\mathbb{R}^{n\times n},\
\Xi\Pi =I-YX
\end{aligned}
\right. \right\}.
\end{equation*}
We then see that there exists a continuous surjective map from $\mathcal{G}_n$ to $\mathcal{C}_n$, as summarized in the following proposition. The proof follows the idea in~\cref{proposition:Phi_surjective}, and we omit the details here.
\begin{proposition}\label{proposition:Phi_surjective_discrete}
The mapping $\Phi$ in~\eqref{eq:change_of_variables_output} is a continuous and surjective mapping from $\mathcal{G}_n$ to $\mathcal{C}_n$.
\end{proposition}

The rest of the proof for~\cref{Theo:disconnectivity_discrete} is identical to the proof of~\cref{Theo:disconnectivity}. Meanwhile, the proofs of~\cref{Theo:Cn_homeomorphic_discrete,Theo:connectivity_conditions_discrete} follow the same arguments as those for~\cref{Theo:Cn_homeomorphic,Theo:connectivity_conditions}.

\subsection{Stationary Points of the LQG Cost Function}

It is not difficult to see that~\cref{lemma:Jn_invariance} and~\cref{proposition:sim_trans_submanifold} hold for the discrete-time LQG cost function $J_q(\mK)$ as well. Thus, the globally optimal solutions to the discrete-time LQG problem~\eqref{eq:LQG_reformulation_KX_discrete} are non-isolated and disconnected in $\mathcal{C}_n$.

%{\color{red}An example for Non-coercivity of LQG cost function}

Here, we present formulas for computing the gradient and Hessian of  the LQG cost function in~\eqref{eq:LQG_cost_formulation_discrete}, and briefly discuss the non-minimal and minimal stationary points of~\eqref{eq:LQG_reformulation_KX_discrete}.

\subsubsection{Gradient and Hessian of the Discrete-Time LQG Cost Function}

The following two lemmas give closed forms for the gradient and Hessian of the discrete-time LQG cost function $J_q$.

\begin{lemma}[Gradient of LQG cost $J_q$] \label{lemma:gradient_LQG_Jn_discrete}
    Fix $q \geq 1$ such that $\mathcal{C}_q \neq \varnothing$.
    For every $\mK = \begin{bmatrix} 0 & C_{\mK} \\B_{\mK} & A_{\mK} \end{bmatrix} \in \mathcal{C}_q$, the gradient of $J_q(\mK)$ is given by
    $$
    \nabla J_q(\mK)
    =\left[\!\begin{array}{cc}
    0 & \mfrac{\partial J_q(\mK)}{\partial C_{\mK}} \\[6pt]
    \mfrac{\partial J_q(\mK)}{\partial B_{\mK}} & \mfrac{\partial J_q(\mK)}{\partial A_{\mK}}
    \end{array}\!\right],
    $$
    with
    \begin{subequations} \label{eq:gradient_Jn_discrete}
        \begin{align}
         \frac{\partial J_q(\mK)}{\partial A_{\mK}} &= 2\left(Y_{12}^\tr(AX_{12}+BC_{\mK}X_{22}) + Y_{22}A_{\mK}X_{22} + Y_{22}B_{\mK}CX_{12}\right), \label{eq:partial_Ak_discrete}\\
        \frac{\partial J_q(\mK)}{\partial B_{\mK}} &= 2\left(Y_{12}^\tr(AX_{11} + BC_{\mK}X_{12}^\tr)C^\tr +  Y_{22} A_{\mK} X_{12}^\tr C^\tr +  Y_{22}B_{\mK}(CX_{11}C^\tr+V)\right), \label{eq:partial_Bk_discrete}\\
        \frac{\partial J_q(\mK)}{\partial C_{\mK}} &= 2\left(B^\tr Y_{12}(A_{\mK}X_{22} + B_{\mK} CX_{12}) + B^\tr Y_{11}AX_{12} + (B^\tr Y_{11}B + R)C_{\mK}X_{22}\right), \label{eq:partial_Ck_discrete}
        \end{align}
    \end{subequations}
    where $X_{\mK}$ and $Y_{\mK}$, partitioned as
    \begin{equation} \label{eq:LyapunovXY_block_discrete}
        X_{\mK} = \begin{bmatrix}
        X_{11} & X_{12} \\ X_{12}^\tr & X_{22}
        \end{bmatrix},  \qquad Y_{\mK} = \begin{bmatrix}
        Y_{11} & Y_{12} \\ Y_{12}^\tr & Y_{22}
        \end{bmatrix}
    \end{equation}
    are the unique positive semidefinite
    solutions to~\eqref{eq:LyapunovX_discrete} and~\eqref{eq:LyapunovY_discrete}, respectively.
\end{lemma}

\begin{lemma}\label{lemma:Jn_Hessian_discrete}
Fix $q\geq 1$ such that $\mathcal{C}_q\neq\varnothing$.
Let $\mK=\begin{bmatrix}
0 & C_{\mK} \\
B_{\mK} & A_{\mK}
\end{bmatrix} \in \mathcal{C}_q$. Then for any $\Delta=\begin{bmatrix}
0 & \Delta_{C_\mK} \\
\Delta_{B_\mK} & \Delta_{A_\mK}
\end{bmatrix}\in\mathcal{V}_q$, we have
$$
\begin{aligned}
\operatorname{Hess}_{\,\mK}(\Delta,\Delta)
=\ &
2\operatorname{tr}
\Bigg(
2\begin{bmatrix}
0 & B\Delta_{C_\mK} \\
\Delta_{B_\mK} C & \Delta_{A_\mK}
\end{bmatrix}
\hat{X}_{\mK,\Delta}A^\tr_{{\mathrm{cl},\mK}}Y_{\mK}
+\begin{bmatrix}
0 & B\Delta_{C_\mK} \\
\Delta_{B_\mK} C & \Delta_{A_\mK}
\end{bmatrix} {X}_{\mK}\begin{bmatrix}
0 & B\Delta_{C_\mK} \\
\Delta_{B_\mK} C & \Delta_{A_\mK}
\end{bmatrix}^\tr Y_{\mK} \\
&\qquad \qquad  +  2\begin{bmatrix}
0 & 0 \\ 0 & C_{\mK}^\tr R \Delta_{C_\mK}
\end{bmatrix}\hat{X}_{\mK,\Delta} +
+\begin{bmatrix}
0 & 0 \\
0 & \Delta_{B_\mK}V\Delta_{B_\mK}^\tr
\end{bmatrix} Y_{\mK}
+
\begin{bmatrix}
0 & 0 \\ 0 & \Delta_{C_\mK}^\tr R \Delta_{C_\mK}
\end{bmatrix}X_{\mK}
\Bigg).
\end{aligned}
$$
where $X_{\mK}$ and $Y_{\mK}$ are the solutions to the Lyapunov equations~\eqref{eq:LyapunovX_discrete} and~\eqref{eq:LyapunovY_discrete}, and $\hat{X}_{\mK,\Delta}\in\mathbb{R}^{(n+q)\times(n+q)}$ is the solution to the following Lyapunov equation
\begin{equation*} %\label{eq:Lyapunov_hessian_discrete}
 \hat{X}_{\mK,\Delta}=   A_{\mathrm{cl},\mK}  \hat{X}_{\mK,\Delta}A_{\mathrm{cl},\mK}^\tr
+M_1,
\end{equation*}
with $ A_{\mathrm{cl},\mK} :=\begin{bmatrix}
A & BC_{\mK} \\
B_{\mK} C & A_{\mK}
\end{bmatrix}$ and
\begin{align*}
M_1
\coloneqq &
\begin{bmatrix}
0 & B\Delta_{C_{\mK}} \\ \Delta_{B_{\mK}}C & \Delta_{A_{\mK}}
\end{bmatrix} X_{\mK}A_{{\mathrm{cl},\mK}}^\tr
+ A_{{\mathrm{cl},\mK}}X_{\mK}\begin{bmatrix}
0 & B\Delta_{C_{\mK}} \\ \Delta_{B_{\mK}}C & \Delta_{A_{\mK}}
\end{bmatrix}^\tr
 +
\begin{bmatrix}
0 & 0 \\
0 & B_{\mK}V\Delta_{B_\mK}^\tr
\!+\!
\Delta_{B_\mK} V B_{\mK}^\tr
\end{bmatrix}.
\end{align*}
\end{lemma}

For the proof of~\cref{lemma:gradient_LQG_Jn_discrete} and~\cref{lemma:Jn_Hessian_discrete}, we first need the following lemma.
\begin{lemma} \label{lemma:Taylor_Lyapunov_discrete}
    Let $M:(-\delta,\delta)\rightarrow\mathbb{R}^{k\times k}$ and $G:(-\delta,\delta)\rightarrow\mathbb{S}^k$ be two indefinitely differentiable matrix-valued functions for some $\delta>0$ and $k\in\mathbb{N}\backslash\{0\}$. Suppose $\rho(M(t)) < 1, \forall t\in(-\delta,\delta)$, and let $X(t)$ denote the solution to the following Lyapunov equation
$$
X(t) = M(t)X(t)M(t)^\tr+G(t).
$$
Then $X(t)$ is indefinitely differentiable over $t\in(-\delta,\delta)$, and its first order derivative and second-order derivative at $t=0$, denoted by $\dot{X}(0)$ and $\ddot{X}(0)$, are the solutions to the following Lyapunov equations
\begin{equation}\label{eq:Taylor_Lyapunov_eq_discrete}
\begin{aligned}
\dot{X}(0) &= M(0)\dot{X}(0)M^\tr(0) + \left( M(0)X(0)\dot{M}^\tr(0) + \dot{M}(0)X(0)M^\tr(0) + \dot{G}(0) \right) \\
\ddot{X}(0)&= M(0)\ddot{X}(0)M^\tr(0) + \bigg( \ddot{M}(0)X(0)M^\tr(0) + M(0)X(0)\ddot{M}^\tr(0) \\
        & \qquad \qquad  + 2\dot{M}(0)\dot{X}(0)M^\tr(0) + 2{M}(0)\dot{X}(0)\dot{M}^\tr(0) + 2\dot{M}(0)X(0)\dot{M}^\tr(0) + \ddot{G}(0)\bigg).
\end{aligned}
\end{equation}
\end{lemma}
\begin{proof}
    The differentiability of $X(t)$ follows from the observation that the unique solution to the Lyapunov equation can be written as
    $$
        \operatorname{vec}(X(t)) =(I-M(t) \otimes M(t))^{-1} \operatorname{vec}(G(t)),
    $$
    where we have applied the fact that $I-M(t) \otimes M(t)$ is invertible thanks to $\rho(M(t)) < 1$.
    Since $M(t), G(t)$ and $X(t)$ are indefinitely differentiable, they admit Taylor expansions around $t = 0$ as
    \begin{align*}
M(t) = \ &  M(0) + t\dot{M}(0) + \frac{t^2}{2}\ddot{M}(0) + o(t^2),\\%\sum_{j=0}^a\frac{t^j}{j!} M^{(j)}(0) +o(t^a), \\
G(t) = \ &  G(0) + t\dot{G}(0) + \frac{t^2}{2}\ddot{G}(0) + o(t^2), \\%\ & \sum_{j=0}^a\frac{t^j}{j!} G^{(j)}(0) +o(t^a), \\
X(t) = \ &  X(0) + t\dot{X}(0) + \frac{t^2}{2}\ddot{X}(0) + o(t^2). %\ & \sum_{j=0}^a\frac{t^j}{j!} X^{(j)}(0) +o(t^a)
\end{align*}
%for any $a\in\mathbb{N}$.
Upon plugging these Taylor expansions into the original Lyapunov equation, some algebraic manipulations lead to
$$
    \begin{aligned}
         - X(0) + M(0)X(0)M^\tr(0) + G(0) \\
        + t\left(-\dot{X}(0) +  M(0)\dot{X}(0)M^\tr(0) + \dot{M}(0)X(0)M^\tr(0) + M(0)X(0)\dot{M}^\tr(0) + \dot{G}(0)\right) \\
        + \frac{t^2}{2} \Bigg(-\ddot{X}(0) + M(0)\ddot{X}(0)M^\tr(0) + \ddot{M}(0)X(0)M^\tr(0) + M(0)X(0)\ddot{M}^\tr(0) + 2\dot{M}(0)\dot{X}(0)M^\tr(0) \\
        + 2{M}(0)\dot{X}(0)\dot{M}^\tr(0) + 2\dot{M}(0)X(0)\dot{M}^\tr(0) + \ddot{G}(0) \Bigg) + o(t^2) = 0.
    \end{aligned}
$$
The equation above holds for all sufficiently small $t$. Therefore, we know that~\eqref{eq:Taylor_Lyapunov_eq_discrete} holds true.

\end{proof}

% Given any stabilizing controller $\mK \in\mathcal{C}_q$, we denote the closed-loop matrix as
% $$
% {A}_{\text{cl},\mK} = \begin{bmatrix}
% A & BC_{\mK} \\ B_{\mK}C & A_{\mK}
% \end{bmatrix} = \begin{bmatrix}
% A & 0  \\ 0& 0
% \end{bmatrix} + \begin{bmatrix}
% B & 0 \\ 0 & I
% \end{bmatrix}\mK
% \begin{bmatrix}
% C & 0 \\ 0 & I
% \end{bmatrix}
% $$
% and
Recall that the discrete-time LQG cost is given by
$$
J_q(\mK)
=\operatorname{tr}
\left(
\begin{bmatrix}
Q & 0 \\ 0 & C_{\mK}^\tr R C_{\mK}
\end{bmatrix}X_\mK
\right),
$$
where $X_\mK$ is the unique positive semidefinite solution to the Lyapunov equation~\eqref{eq:LyapunovX_discrete}.
Consider an arbitrary direction $\Delta
=\begin{bmatrix}
0 & \Delta_{C_{\mK}} \\
\Delta_{B_{\mK}} & \Delta_{A_{\mK}}
\end{bmatrix}
\in\mathcal{V}_q$. For sufficiently small $t>0$ such that $\mK + t \Delta \in \mathcal{C}_q$, the corresponding closed-loop matrix is
$$
A_{\text{cl},\mK+t\Delta} = {A}_{\text{cl},\mK}
+t
\begin{bmatrix}
B & 0 \\ 0 & I
\end{bmatrix}\Delta
\begin{bmatrix}
C & 0 \\ 0 & I
\end{bmatrix},
$$
and we let $X_{\mK,\Delta}(t)$ denote the solution to the Lyapunov equation~\eqref{eq:LyapunovX_discrete} with closed-loop matrix $A_{\mathrm{cl},\mK+t\Delta}$, i.e.,
\begin{equation*} %\label{eq:Lyapunov_X_next_step}
    \begin{aligned}
X_{\mK,\Delta}(t) &= \left({A}_{\text{cl},\mK} + t \begin{bmatrix}
B & 0 \\ 0 & I
\end{bmatrix}\Delta \begin{bmatrix}
C & 0 \\ 0 & I
\end{bmatrix} \right)X_{\mK,\Delta}(t) \left(A_{\text{cl},\mK}
+t \begin{bmatrix}
B & 0 \\ 0 & I
\end{bmatrix}\Delta
\begin{bmatrix}
C & 0 \\ 0 & I
\end{bmatrix}
\right)^\tr \\
&\qquad \qquad \qquad\qquad \qquad \qquad\qquad \qquad \qquad + \begin{bmatrix}
W & 0 \\
0 & (B_\mK + t\Delta_{B_\mK})V
(B_\mK + t\Delta_{B_\mK})^\tr
\end{bmatrix}.
\end{aligned}
\end{equation*}
By \cref{lemma:Taylor_Lyapunov_discrete}, we see that $X_{\mK,\Delta}(t)$ admits a Taylor expansion of the form
\begin{equation} \label{eq:Lyapunov_X_Taylor_discrete}
X_{\mK,\Delta}(t)
= X_{\mK}
+
t\cdot \dot{X}_{\mK,\Delta}(0)
+\frac{t^2}{2}
\cdot \ddot{X}_{\mK,\Delta}(0)
+o(t^2),
\end{equation}
and the derivatives $\dot{X}_{\mK,\Delta}(0)$ and $\ddot{X}_{\mK,\Delta}(0)$ are the solutions to the following Lyapunov equations
\begin{align*}
\dot{X}_{\mK,\Delta}(0) & = A_{\mathrm{cl},\mK} \dot{X}_{\mK,\Delta}(0) A_{\mathrm{cl},\mK}^\tr + M_1,
\\
\ddot{X}_{\mK,\Delta}(0) &= A_{\mathrm{cl},\mK} \ddot{X}_{\mK,\Delta}(0) A_{\mathrm{cl},\mK}^\tr
+2M_2,
\end{align*}
where
\begin{align*}
M_1
\coloneqq &
\begin{bmatrix}
B & 0 \\ 0 & I
\end{bmatrix}\Delta
\begin{bmatrix}
C & 0 \\ 0 & I
\end{bmatrix} X_{\mK}A_{{\mathrm{cl},\mK}}^\tr
+ A_{{\mathrm{cl},\mK}}X_{\mK}\begin{bmatrix}
C & 0 \\ 0 & I
\end{bmatrix}^\tr
\!\Delta^\tr\!
\begin{bmatrix}
B & 0 \\ 0 & I
\end{bmatrix}^\tr
\\
&\qquad \qquad\qquad \qquad\qquad \qquad\qquad \qquad\qquad \qquad +
\begin{bmatrix}
0 & 0 \\
0 & B_{\mK}V\Delta_{B_\mK}^\tr
\!+\!
\Delta_{B_\mK} V B_{\mK}^\tr
\end{bmatrix},
\\
M_2
\coloneqq &
\begin{bmatrix}
B & 0 \\ 0 & I
\end{bmatrix}\Delta
\begin{bmatrix}
C & 0 \\ 0 & I
\end{bmatrix} \dot{X}_{\mK,\Delta}(0)A^\tr_{{\mathrm{cl},\mK}}
+ A_{{\mathrm{cl},\mK}}\dot{X}_{\mK,\Delta}(0)\begin{bmatrix}
C & 0 \\ 0 & I
\end{bmatrix}^\tr
\!\Delta^\tr\!
\begin{bmatrix}
B & 0 \\ 0 & I
\end{bmatrix}^\tr
\\
&\qquad \qquad \qquad + \begin{bmatrix}
B & 0 \\ 0 & I
\end{bmatrix}\Delta
\begin{bmatrix}
C & 0 \\ 0 & I
\end{bmatrix} {X}_{\mK}\begin{bmatrix}
C & 0 \\ 0 & I
\end{bmatrix}^\tr
\!\Delta^\tr\!
\begin{bmatrix}
B & 0 \\ 0 & I
\end{bmatrix}^\tr  + \begin{bmatrix}
0 & 0 \\
0 & \Delta_{B_\mK}V\Delta_{B_\mK}^\tr
\end{bmatrix}.
\end{align*}

By plugging the Taylor expansion~\eqref{eq:Lyapunov_X_Taylor_discrete} into the expression~\eqref{eq:LQG_cost_formulation_discrete} for $J_q(\mK)$, we get
$$
\begin{aligned}
J_q(\mK+t\Delta)
=\ &
\operatorname{tr}
\left(
\begin{bmatrix}
Q & 0 \\ 0 & (C_{\mK}+t\Delta_{C_\mK})^\tr R (C_{\mK}+t\Delta_{C_\mK})
\end{bmatrix} X_{\mK,\Delta}(t)\right)
\\
=\ &
J_q(\mK)
+
t
\cdot\operatorname{tr}
\left(
\begin{bmatrix}
Q & 0 \\ 0 & C_{\mK}^\tr R C_{\mK}
\end{bmatrix} \dot{X}_{\mK,\Delta}(0)
+
\begin{bmatrix}
0 & 0 \\ 0 & C_{\mK}^\tr R \Delta_{C_\mK} + \Delta_{C_\mK}^\tr RC_{\mK}
\end{bmatrix}X_{\mK}
\right) \\
& +
\frac{t^2}{2}
\cdot \operatorname{tr}\Bigg(
\begin{bmatrix}
Q & 0 \\ 0 & C_{\mK}^\tr R C_{\mK}
\end{bmatrix} \ddot{X}_{\mK,\Delta}(0)
+2
\begin{bmatrix}
0 & 0 \\ 0 & C_{\mK}^\tr R \Delta_{C_\mK} + \Delta_{C_\mK}^\tr RC_{\mK}
\end{bmatrix}\dot{X}_{\mK,\Delta}(0) \\
&\qquad\qquad\qquad\qquad
+
2\begin{bmatrix}
0 & 0 \\ 0 & \Delta_{C_\mK}^\tr R \Delta_{C_\mK}
\end{bmatrix}X_{\mK}
\Bigg)
+o(t^2),
\end{aligned}
$$
from which we can directly recognize $\left.\mfrac{dJ_q(\mK+t\Delta)}{dt}\right|_{t=0}$ and $\left.\mfrac{d^2J_q(\mK+t\Delta)}{dt^2}\right|_{t=0}$.

Now suppose $X$ is the solution to the following Lyapunov equation
$$
X = A_{\mathrm{cl},\mK} X A_{\mathrm{cl},\mK}^\tr
+M
$$
for some $M\in\mathbb{S}^{n+q}$. Similar to~\cref{lemma:Lyapunov}, it is known that the unique solution to the Lyapunov equation above is
$$
X = \sum_{k=0}^\infty
{A_{\mathrm{cl},\mK}^k}
\cdot M \cdot {(A_{\mathrm{cl},\mK}^k)^\tr}\,,
$$
and consequently
$$
\begin{aligned}
\operatorname{tr}\left(\begin{bmatrix}
Q & 0 \\ 0 & C_{\mK}^\tr R C_{\mK}
\end{bmatrix}
X\right)
=\ &
\sum_{k=0}^\infty\operatorname{tr}\left(\begin{bmatrix}
Q & 0 \\ 0 & C_{\mK}^\tr R C_{\mK}
\end{bmatrix}
{A^k_{\mathrm{cl},\mK}}
\cdot  M \cdot  {(A^k_{\mathrm{cl},\mK})^\tr}\right) \\
=\ &
\sum_{k=0}^\infty \operatorname{tr}\left(
{(A^k_{\mathrm{cl},\mK})^\tr}\begin{bmatrix}
Q & 0 \\ 0 & C_{\mK}^\tr R C_{\mK}
\end{bmatrix} {A_{\mathrm{cl},\mK}^k}\cdot
M \right)
=\operatorname{tr}(Y_{\mK}M),
\end{aligned}
$$
in which we recall that $Y_{\mK}$ is the unique positive semidefinite solution to Lyapunov equation~\eqref{eq:LyapunovY_discrete}. Therefore, the first-order derivative $\left.\mfrac{dJ_q(\mK+t\Delta)}{dt}\right|_{t=0}$ can be alternatively given by
$$
\begin{aligned}
& \left.\frac{d J_q(\mK+t\Delta)}{dt}\right|_{t=0} \\
=\ &
\operatorname{tr}
\left(Y_{\mK} M_1
+\begin{bmatrix}
0 & 0 \\ 0 & C_{\mK}^\tr R \Delta_{C_\mK} + \Delta_{C_\mK}^\tr RC_{\mK}
\end{bmatrix}X_{\mK}\right) \\
=\ &
2\operatorname{tr}\left[
\left(
\begin{bmatrix}
0 & RC_{\mK} \\ 0 & 0
\end{bmatrix}X_{\mK} \begin{bmatrix}
0 & 0 \\ 0 & I
\end{bmatrix}
+
\begin{bmatrix}
B & 0 \\ 0 & I
\end{bmatrix}^\tr
Y_{\mK} A_{\mathrm{cl},\mK} X_{\mK}
\begin{bmatrix}
C & 0 \\ 0 & I
\end{bmatrix}^\tr
+\begin{bmatrix}
0 & 0 \\ 0 & I
\end{bmatrix}
Y_{\mK}
\begin{bmatrix}
0 & 0 \\ B_{\mK}V & 0
\end{bmatrix}
\right)^\tr\Delta\right].
\end{aligned}
$$
One can readily recognize the gradient $\nabla J_q(\mK)$ by noticing that
$$
\left.\frac{dJ_q(\mK+t\Delta)}{dt}\right|_{t=0}
=\operatorname{tr}\left(
\nabla J_q(\mK)^\tr\Delta\right).
$$
Upon partitioning $X_{\mK}$ and $Y_{\mK}$ as~\eqref{eq:LyapunovXY_block_discrete}, a few calculations lead to the gradient formula of $J_q(\mK)$ in~\eqref{eq:gradient_Jn_discrete}.

Similarly, we can show that the second-order derivative can be alternatively given by
\begin{equation*} %\label{eq:Jq_Hessian}
\begin{aligned}
& \left.\frac{d^2J_q(\mK+t\Delta)}{dt^2}\right|_{t=0} \\
=\ &
2\operatorname{tr}
\left(Y_{\mK} M_2
+
\begin{bmatrix}
0 & 0 \\ 0 & C_{\mK}^\tr R \Delta_{C_\mK} + \Delta_{C_\mK}^\tr RC_{\mK}
\end{bmatrix}\dot{X}_{\mK,\Delta}(0)
+\begin{bmatrix}
0 & 0 \\ 0 & \Delta_{C_\mK}^\tr R \Delta_{C_\mK}
\end{bmatrix}X_{\mK}\right) \\
=\ &
2\operatorname{tr}
\Bigg(
2\begin{bmatrix}
B & 0 \\ 0 & I
\end{bmatrix}
\Delta\begin{bmatrix}
C & 0 \\ 0 & I
\end{bmatrix}
\dot{X}_{\mK,\Delta}(0)A^\tr_{{\mathrm{cl},\mK}}Y_{\mK}
+\begin{bmatrix}
B & 0 \\ 0 & I
\end{bmatrix}\Delta
\begin{bmatrix}
C & 0 \\ 0 & I
\end{bmatrix} {X}_{\mK}\begin{bmatrix}
C & 0 \\ 0 & I
\end{bmatrix}^\tr
\!\Delta^\tr\!
\begin{bmatrix}
B & 0 \\ 0 & I
\end{bmatrix}^\tr Y_{\mK} \\
&\qquad \qquad  +  2\begin{bmatrix}
0 & 0 \\ 0 & C_{\mK}^\tr R \Delta_{C_\mK}
\end{bmatrix}\dot{X}_{\mK,\Delta}(0) +
+\begin{bmatrix}
0 & 0 \\
0 & \Delta_{B_\mK}V\Delta_{B_\mK}^\tr
\end{bmatrix} Y_{\mK}
+
\begin{bmatrix}
0 & 0 \\ 0 & \Delta_{C_\mK}^\tr R \Delta_{C_\mK}
\end{bmatrix}X_{\mK}
\Bigg).
\end{aligned}
\end{equation*}
We now finish the proof of~\cref{lemma:gradient_LQG_Jn_discrete,lemma:Jn_Hessian_discrete}.

\subsubsection{Non-minimal Stationary Points}

Since~\cref{lemma:gradient_simi_tran_linear} and~\cref{theorem:non_globally_optimal_stationary_point} are direct consequences of the similarity transformation $\mathscr{T}_q(T,\mK)$, these two results also naturally hold for the discrete-time LQG cost function. This suggests that the discrete-time LQG cost $J_n(\mK)$ over the full-order stabilizing controller $\mathcal{C}_n$ is likely to have non-minimal stationary points that are strict saddle points. One may further establish similar results in~\cref{theorem:zero_stationary_hessian} for the discrete-time LQG cost $J_n(\mK)$. % and we leave this as exercise for interested readers.

%the same results hold for discrete time

%some examples

\subsubsection{Minimal Stationary Points Are Globally Optimal}

For minimal stationary points, we have the following result.
\begin{theorem} \label{theo:stationary_points_globally_optimal_discrete}
 Under \cref{assumption:stabilizability_discrete}, all minimal stationary points $\mK\in\mathcal{C}_n$ of the discrete-time LQG problem~\eqref{eq:LQG_reformulation_KX_discrete} are globally optimal, and they are in the form of
 \begin{equation} \label{eq:stationary_points_discrete}
        A_{\mK} = T(A - BK - LC)T^{-1}, \qquad    B_{\mK} = -TL, \qquad   C_{\mK} =  KT^{-1},
 \end{equation}
where $T \in \mathbb{R}^{n \times n}$ is an invertible matrix, and
\begin{equation} \label{eq:LQG_KL_discrete}
     K = (B^\tr SB + R)^{-1}B^\tr SA, \qquad L = APC^\tr (CPC^\tr + V)^{-1},
\end{equation}
with $P$ and $S$ being the unique positive definite solutions to the Riccati equations~\eqref{eq:Riccati_P_discrete} and~\eqref{eq:Riccati_S_discrete}.

\end{theorem}

\begin{proof}
Consider a stationary point $\mK = \begin{bmatrix}
0 & C_{\mK} \\
B_{\mK} & A_{\mK}
\end{bmatrix} \in \mathcal{C}_n$ such that the gradient~\eqref{eq:gradient_Jn_discrete} vanishes. If the controller $\mK$ is minimal, similar to \cref{lemma:LyapunovXY_pd}, we can show that the solutions $X_{\mK}$ and $Y_{\mK}$ to~\eqref{eq:LyapunovX_discrete} and~\eqref{eq:LyapunovY_discrete} are unique and positive definite.

Upon partitioning $X_{\mK}$ and $Y_\mK$ in~\eqref{eq:LyapunovXY_block_discrete}, by the Schur complement, the following matrices are well-defined and positive definite
\begin{equation} \label{eq:riccatiPS_discrete}
    \begin{aligned}
        P &: = X_{11} - X_{12}X_{22}^{-1}X_{12}^\tr \succ 0, \\
        S &: = Y_{11} - Y_{12}Y_{22}^{-1} Y_{12}^\tr \succ 0.
    \end{aligned}
\end{equation}
Now, letting $ \frac{\partial J_n(\mK)}{\partial A_{\mK}} = 0$, from~\eqref{eq:partial_Ak_discrete}, we have
\begin{equation*}
    A_{\mK} = - (Y_{22}^{-1}Y_{12}^\tr A X_{12}X_{22}^{-1} + Y_{22}^{-1}Y_{12}^\tr BC_{\mK} + B_{\mK}CX_{12}X_{22}^{-1}).
\end{equation*}
Similarly, at a stationary point, some algebraic manipulations from~\eqref{eq:partial_Bk_discrete} and~\eqref{eq:partial_Ck_discrete} lead to
$$
\begin{aligned}
    B_{\mK} &= - Y_{22}^{-1}Y_{12}^\tr APC^\tr (V + CPC^\tr)^{-1}, \\
    C_{\mK} &= - (R+ B^\tr S B)^{-1}B^\tr SA X_{12}X_{22}^{-1}.
\end{aligned}
$$
We now define
\begin{equation} \label{eq:similarity_T_discrete}
    T := Y_{22}^{-1} Y_{12}^\tr.
\end{equation}

To show that all minimal stationary points are identical up to a similarity transformation and they are in the form of~\eqref{eq:stationary_points_discrete} and~\eqref{eq:LQG_KL_discrete}, it remains to prove that
\begin{enumerate}
    \item Matrix $T$ in~\eqref{eq:similarity_T_discrete} is invertible and
$
T^{-1} = - X_{12}X_{22}^{-1}.
$
    \item The matrices $P$ and $S$, defined in~\eqref{eq:riccatiPS_discrete}, satisfy the algebraic Riccati equations~\eqref{eq:Riccati_P_discrete} and~\eqref{eq:Riccati_S_discrete}, respectively.
\end{enumerate}

We start with the first claim on the matrix $T$. Since $X_{\mK}$ is the solution to the Lyapunov equation~\eqref{eq:LyapunovX_discrete}, expanding the blocks leads to three matrix equations
\begin{subequations}
\begin{align}
X_{11}=\ & AX_{11}A^{\tr}+ BC_\mK X_{12}^\tr A^\tr + AX_{12}C^\tr_{\mK}B^\tr + BC_{\mK}X_{22}C^\tr_{\mK}B^\tr + W,
\label{eq:minimal_stationary_Lyapunov_d1}
\\
X_{12}=\ & AX_{11}C^{\tr}B_{\mK}^\tr + BC_{\mK}X_{12}^\tr C^\tr B_{\mK}^\tr  +AX_{12}A_{\mK}^\tr + B C_{\mK}X_{22}A_{\mK}^\tr,
\label{eq:minimal_stationary_Lyapunov_d2}
\\
X_{22}=\ &
B_{\mK} C X_{11} C^\tr B_{\mK}^\tr + A_{\mK}X_{12}^\tr C^\tr B_{\mK}^\tr + B_{\mK} CX_{12} A_{\mK}^\tr + A_{\mK} X_{22} A_{\mK}^\tr + B_{\mK}VB_{\mK}^\tr.
\label{eq:minimal_stationary_Lyapunov_d3}
\end{align}
\end{subequations}
To prove the invertibility of $T$, it suffices to establish
$
    Y_{12}^\tr X_{12} + Y_{22}X_{22} = 0.
$
Now, $Y_{12}^\tr \times$~\eqref{eq:minimal_stationary_Lyapunov_d2} $+Y_{22} \times$~\eqref{eq:minimal_stationary_Lyapunov_d3} leads to
\begin{equation} \label{eq:Tinverse_discrete_step1}
\begin{aligned}
     & Y_{12}^\tr X_{12} + Y_{22}X_{22} \\
    =&Y_{12}^\tr (AX_{11}C^{\tr}B_{\mK}^\tr + BC_{\mK}X_{12}^\tr C^\tr B_{\mK}^\tr  +AX_{12}A_{\mK}^\tr + B C_{\mK}X_{22}A_{\mK}^\tr) \\
    &\qquad \qquad  + Y_{22}(B_{\mK} C X_{11} C^\tr B_{\mK}^\tr + A_{\mK}X_{12}^\tr C^\tr B_{\mK}^\tr + B_{\mK} CX_{12} A_{\mK}^\tr + A_{\mK} X_{22} A_{\mK}^\tr + B_{\mK}VB_{\mK}^\tr). \\
\end{aligned}
\end{equation}
From the expression of $A_{\mK}$, we have
\begin{equation} \label{eq:Tinverse_discrete_step2}
    \begin{aligned}
        Y_{22}A_{\mK} &= -Y_{12}^\tr A X_{12}X_{22}^{-1} -Y_{12}^\tr BC_{\mK} - Y_{22}B_{\mK}CX_{12}X_{22}^{-1}
    \end{aligned}
\end{equation}
leading to
$$
   Y_{22}A_{\mK} X_{22} A_{\mK}^\tr  + Y_{12}^\tr AX_{12}A_{\mK}^\tr + Y_{12}^\tr BC_{\mK}X_{22} A_{\mK}^\tr + Y_{22}B_{\mK}CX_{12}A_{\mK}^\tr = 0.
$$
Substituting the equation above to~\eqref{eq:Tinverse_discrete_step1}, we have
\begin{equation} \label{eq:Tinverse_discrete_step3}
    \begin{aligned}
     & Y_{12}^\tr X_{12} + Y_{22}X_{22} \\
    =\,&Y_{12}^\tr (AX_{11}C^{\tr}B_{\mK}^\tr + BC_{\mK}X_{12}^\tr C^\tr B_{\mK}^\tr) + Y_{22}(B_{\mK} C X_{11} C^\tr B_{\mK}^\tr + A_{\mK}X_{12}^\tr C^\tr B_{\mK}^\tr + B_{\mK}VB_{\mK}^\tr) \\
    =\,&(Y_{12}^\tr AX_{11}C^{\tr} + Y_{12}^\tr BC_{\mK}X_{12}^\tr C^\tr  + Y_{22}B_{\mK} C X_{11} C^\tr + Y_{22}A_{\mK}X_{12}^\tr C^\tr + Y_{22}B_{\mK}V)B_{\mK}^\tr \\
    =\,& (Y_{12}^\tr AX_{11}C^{\tr} + Y_{22}B_{\mK} C X_{11} C^\tr+ Y_{22}B_{\mK}V -(Y_{12}^\tr A + Y_{22}B_{\mK}C)X_{12}X_{22}^{-1}X_{12}^\tr C^\tr)B_{\mK}^\tr \\
    =\,& (Y_{12}^\tr APC^{\tr} + Y_{22}B_{\mK}CPC^\tr + Y_{22}B_kV)B_{\mK}^\tr,
\end{aligned}
\end{equation}
where the second to last equation applied the fact in~\eqref{eq:Tinverse_discrete_step2}, and the last equation used the definition $P = X_{11} - X_{12}X_{22}^{-1}X_{12}^\tr$. We now consider the expression of $B_{\mK}$, which leads to
$
Y_{22}B_{\mK} = -Y_{12}^\tr APC^\tr (V + CPC^\tr)^{-1},
$
and then~\eqref{eq:Tinverse_discrete_step3} becomes
$$
    Y_{12}^\tr X_{12} + Y_{22}X_{22} = 0.
$$
Therefore, matrix $T$ in~\eqref{eq:similarity_T_discrete} is invertible and
$
T^{-1} = - X_{12}X_{22}^{-1}.
$

We next proceed to show the matrix $P$ being the solution to~\eqref{eq:Riccati_P_discrete}.
From~\eqref{eq:minimal_stationary_Lyapunov_d1}, by the definition of $P:= X_{11} - X_{12}X_{22}^{-1}X_{12}^\tr$, we have
$$
    \begin{aligned}
   P =&\,  AX_{11}A^{\tr}+ BC_\mK X_{12}^\tr A^\tr + AX_{12}C^\tr_{\mK}B^\tr + BC_{\mK}X_{22}C^\tr_{\mK}B^\tr + W - X_{12}X_{22}^{-1}X_{12}^\tr  \\
   = &\, APA^{\tr} - APC^\tr (CPC^\tr + V)^{-1}CPA^\tr + W  \\
   &  + AX_{12}X_{22}^{-1}X_{12}^\tr A^\tr + APC^\tr (CPC^\tr + V)^{-1}CPA^\tr \\
   &+ BC_\mK X_{12}^\tr A^\tr + AX_{12}C^\tr_{\mK}B^\tr + BC_{\mK}X_{22}C^\tr_{\mK}B^\tr - X_{12}X_{22}^{-1}X_{12}^\tr.
    \end{aligned}
$$
It suffices to prove that
\begin{equation}\label{eq:Riccati_discrete_st1}
\begin{aligned}
   X_{12}X_{22}^{-1}X_{12}^\tr =&\, AX_{12}X_{22}^{-1}X_{12}^\tr A^\tr + APC^\tr (CPC^\tr + V)^{-1}CPA^\tr \\
   &\qquad \qquad + BC_\mK X_{12}^\tr A^\tr + AX_{12}C^\tr_{\mK}B^\tr + BC_{\mK}X_{22}C^\tr_{\mK}B^\tr.
\end{aligned}
\end{equation}
We multiply~\eqref{eq:minimal_stationary_Lyapunov_d2} by $-T^{-\tr} =  X_{22}^{-1}X_{12}^\tr$ on the right and get
\begin{equation} \label{eq:Riccati_discrete_st2}
\begin{aligned}
    X_{12}X_{22}^{-1}X_{12}^{\tr}= &\, AX_{11}C^{\tr} (V + CPC^\tr)^{-1}CPA^\tr  + BC_{\mK}X_{12}^\tr C^\tr  (V + CPC^\tr)^{-1}CPA^\tr \\
    & \qquad \qquad \qquad \qquad \qquad \qquad + AX_{12}A_{\mK}^\tr X_{22}^{-1}X_{12}^\tr + B C_{\mK}X_{22}A_{\mK}^\tr X_{22}^{-1}X_{12}^\tr.\\
\end{aligned}
\end{equation}
Now, it is sufficient to show that the right hand side of~\eqref{eq:Riccati_discrete_st1} equals to the right hand side of~\eqref{eq:Riccati_discrete_st2}. For this, we take a difference
\begin{equation*}
    \begin{aligned}
    &AX_{12}X_{22}^{-1}X_{12}^\tr A^\tr + APC^\tr (CPC^\tr + V)^{-1}CPA^\tr \\
    & \qquad \qquad \qquad \qquad + BC_\mK X_{12}^\tr A^\tr + AX_{12}C^\tr_{\mK}B^\tr + BC_{\mK}X_{22}C^\tr_{\mK}B^\tr \\
    &-\bigg(AX_{11}C^{\tr} (V + CPC^\tr)^{-1}CPA^\tr  + BC_{\mK}X_{12}^\tr C^\tr  (V + CPC^\tr)^{-1}CPA^\tr \\
    & \qquad \qquad \qquad \qquad + AX_{12}A_{\mK}^\tr X_{22}^{-1}X_{12}^\tr + B C_{\mK}X_{22}A_{\mK}^\tr X_{22}^{-1}X_{12}^\tr\bigg)\\
    =& -AT^{-1}X_{12}^\tr A^\tr + APC^\tr L^\tr+ BC_\mK X_{12}^\tr A^\tr + AX_{12}C^\tr_{\mK}B^\tr + BC_{\mK}X_{22}C^\tr_{\mK}B^\tr\\
    &-\bigg(AX_{11}C^{\tr} L^\tr  + BC_{\mK}X_{12}^\tr C^\tr  L^\tr  - AX_{12} (T^{-1}A_{\mK})^\tr  - B C_{\mK}X_{22}(T^{-1}A_{\mK})^\tr\bigg),\\
    \end{aligned}
\end{equation*}
where we applied the definition of $L$ and $T$. Considering the expressions of $A_{\mK}, B_{\mK}, C_{\mK}$ in~\eqref{eq:stationary_points_discrete}, the equation above becomes
$$
\begin{aligned}
    & -AT^{-1}X_{12}^\tr A^\tr + APC^\tr L^\tr + AX_{12}C^\tr_{\mK}B^\tr -\bigg(AX_{11}C^{\tr} L^\tr   - AX_{12} (T^{-1}A_{\mK})^\tr\bigg) \\
    & \qquad \qquad + B C_{\mK}X_{22}\left( X_{22}^{-1}X_{12}^\tr A^\tr + C^\tr_{\mK}B^\tr- X_{22}^{-1}X_{12}^\tr C^\tr  L^\tr  +(T^{-1}A_{\mK})^\tr\right)\\
    =&   -AT^{-1}X_{12}^\tr A^\tr + APC^\tr L^\tr + AX_{12}C^\tr_{\mK}B^\tr -\bigg(AX_{11}C^{\tr} L^\tr  - AX_{12} (T^{-1}A_{\mK})^\tr\bigg) \\
    =&\,  AX_{12}X_{22}^{-1}X_{12}^\tr A^\tr - AX_{12}X_{22}^{-1}X_{12}^\tr C^\tr L^\tr + AX_{12}C^\tr_{\mK}B^\tr + AX_{12} (T^{-1}A_{\mK})^\tr\\
    =&\,AX_{12}(X_{22}^{-1}X_{12}^\tr A^\tr - X_{22}^{-1}X_{12}^\tr C^\tr L^\tr + C^\tr_{\mK}B^\tr +  (T^{-1}A_{\mK})^\tr)\\
    =&\,0.
    \end{aligned}
$$
This proves~\eqref{eq:Riccati_discrete_st1}, and thus $P:= X_{11} - X_{12}X_{22}^{-1}X_{12}^\tr$ satisfies the Riccati equation~\eqref{eq:Riccati_P_discrete}. Through similar steps, we can derive from~\eqref{eq:LyapunovY_discrete} that $S$ satisfies the Riccati equation~\eqref{eq:Riccati_S_discrete}.
\end{proof}

Finally, from~\cref{theo:stationary_points_globally_optimal_discrete}, it is easy to see that~\cref{corollary:non-minimal-globally-optimal} and~\cref{corollary:Gradient_Descent_Convergence} also hold for the discrete-time LQG problem~\eqref{eq:LQG_reformulation_KX_discrete}.

\end{document}